\documentclass[a4paper,10pt]{article}
\usepackage[utf8]{inputenc}
\usepackage[T1]{fontenc}

\usepackage[xindy]{imakeidx}

\PassOptionsToPackage{hyphens}{url}
\usepackage[
	colorlinks,
	bookmarksnumbered
	]{hyperref}
\hypersetup{
	linkcolor=blue,
	citecolor=red,
	urlcolor=teal,
	pdftitle = {Stable conjugacy and epipelagic L-packets for Brylinski-Deligne covers of Sp(2n)},
	pdfauthor = {Wen-Wei Li}
}

\usepackage{amsthm}
\usepackage{amssymb}
\usepackage{mathtools}
\usepackage{stmaryrd}  \SetSymbolFont{stmry}{bold}{U}{stmry}{m}{n}	
\usepackage{bbm}
\usepackage{bm}
\usepackage{nccrules}
\usepackage[nottoc]{tocbibind}
\usepackage{amsmath}
\usepackage{mathrsfs}
\usepackage{tikz-cd}
\usetikzlibrary{shapes.geometric}

\usepackage{paralist} 
\usepackage{manfnt}

\usepackage{geometry}
\DeclareSymbolFont{bbold}{U}{bbold}{m}{n}
\DeclareSymbolFontAlphabet{\mathbbold}{bbold}

\newcommand{\Z}{\ensuremath{\mathbb{Z}}}
\newcommand{\Q}{\ensuremath{\mathbb{Q}}}
\newcommand{\R}{\ensuremath{\mathbb{R}}}
\newcommand{\CC}{\ensuremath{\mathbb{C}}}

\newcommand{\A}{\ensuremath{\mathbb{A}}}
\newcommand{\Weil}{\ensuremath{\mathcal{W}}}	
\newcommand{\Gal}[1]{\ensuremath{\mathrm{Gal}(#1)}}	

\newcommand{\Aut}{\operatorname{Aut}}
\newcommand{\Tr}{\operatorname{tr}}
\newcommand{\topwedge}{\ensuremath{\bigwedge^{\mathrm{max}}}}

\newcommand{\cInd}{\operatorname{c-Ind}}

\newcommand{\dd}{\mathop{}\!\mathrm{d}}

\newcommand{\lrangle}[1]{\ensuremath{\langle #1 \rangle}}
\newcommand{\sgn}{\ensuremath{\mathrm{sgn}}}
\newcommand{\Stab}{\ensuremath{\mathrm{Stab}}}
\renewcommand{\emptyset}{\ensuremath{\varnothing}}
\newcommand{\twomatrix}[4]{ \ensuremath{\bigl(\begin{smallmatrix} #1 & #2 \\ #3 & #4 \end{smallmatrix}\bigr)} }
\newcommand{\twobigmatrix}[4]{ \ensuremath{\begin{pmatrix} #1 & #2 \\ #3 & #4 \end{pmatrix}} }

\newcommand{\cate}[1]{\ensuremath{\mathsf{#1}}}
\newcommand{\identity}{\ensuremath{\mathrm{id}}}

\newcommand{\Hom}{\operatorname{Hom}}
\newcommand{\End}{\operatorname{End}}
\newcommand{\rightiso}{\ensuremath{\stackrel{\sim}{\rightarrow}}}
\newcommand{\longrightiso}{\ensuremath{\stackrel{\sim}{\longrightarrow}}}
\newcommand{\leftiso}{\ensuremath{\stackrel{\sim}{\leftarrow}}}
\newcommand{\dtimes}[1]{\ensuremath{\underset{#1}{\times}}}
\newcommand{\quoted}[1]{\ensuremath{\text{\textquotedblleft} #1 \text{\textquotedblright}}}	

\newcommand{\Ker}{\operatorname{ker}}
\newcommand{\Coker}{\operatorname{coker}}
\newcommand{\Image}{\operatorname{im}}
\newcommand{\dotimes}[1]{\ensuremath{\underset{#1}{\otimes}}}

\newcommand{\Ad}{\operatorname{Ad}}
\newcommand{\ad}{\operatorname{ad}}
\newcommand{\Spec}{\operatorname{Spec}}
\newcommand{\Gm}{\ensuremath{\mathbb{G}_\mathrm{m}}}
\newcommand{\Ga}{\ensuremath{\mathbb{G}_\mathrm{a}}}
\newcommand{\Res}{\operatorname{Res}}
\newcommand{\utimes}[1]{\ensuremath{\overset{#1}{\times}}}

\newcommand{\Frob}{\operatorname{Fr}}

\newcommand{\Gmm}[1]{\ensuremath{\mathbb{G}_{\mathrm{m}, #1}}}	
\newcommand{\shK}{\ensuremath{\mathbf{K}}}	

\newcommand{\GL}{\operatorname{GL}}
\newcommand{\SO}{\operatorname{SO}}
\newcommand{\gl}{\ensuremath{\mathfrak{gl}}}
\newcommand{\SL}{\operatorname{SL}}
\newcommand{\Sp}{\operatorname{Sp}}
\newcommand{\Lgrp}[1]{\ensuremath{{}^{\mathrm{L}} #1}}	
\newcommand{\Or}{\ensuremath{\mathrm{O}}}
\newcommand{\SU}{\ensuremath{\mathrm{SU}}}
\newcommand{\PGL}{\ensuremath{\mathrm{PGL}}}

\newcommand{\so}{\ensuremath{\mathfrak{so}}}

\newcommand{\syp}{\ensuremath{\mathfrak{sp}}}
\newcommand{\GSp}{\ensuremath{\mathrm{GSp}}}

\theoremstyle{plain}
\newtheorem{proposition}{Proposition}[subsection]
\newtheorem{lemma}[proposition]{Lemma}
\newtheorem{theorem}[proposition]{Theorem}
\newtheorem{corollary}[proposition]{Corollary}

\theoremstyle{definition}
\newtheorem{definition}[proposition]{Definition}
\newtheorem{definition-theorem}[proposition]{Definition--Theorem}
\newtheorem{definition-proposition}[proposition]{Definition--Proposition}
\newtheorem{hypothesis}[proposition]{Hypothesis}
\newtheorem{example}[proposition]{Example}

\theoremstyle{remark}
\newtheorem{remark}[proposition]{Remark}

\newtheorem{notation}[proposition]{Notation}

\numberwithin{equation}{section}

\newcommand{\bmu}{\ensuremath{\bm{\mu}}}

\newcommand{\noyau}{\ensuremath{\bm{\varepsilon}}} 

\newcommand{\Cali}{\ensuremath{\mathbf{C}}}	
\newcommand{\CaliAd}{\ensuremath{\mathbf{C}\mathrm{Ad}}}	

\usepackage{lmodern}

\title{Stable conjugacy and epipelagic $L$-packets for Brylinski--Deligne covers of $\Sp(2n)$}
\author{Wen-Wei Li}
\date{}

\makeindex

\begin{document}

\maketitle

\begin{abstract}
	Let $F$ be a local field of characteristic not $2$. We propose a definition of stable conjugacy for all the covering groups of $\mathrm{Sp}(2n,F)$ constructed by Brylinski and Deligne, whose degree we denote by $m$. To support this notion, we follow Kaletha's approach to construct genuine epipelagic $L$-packets for such covers in the non-archimedean case with $p \nmid 2m$, or some weaker variant when $4 \mid m$; we also prove the stability of packets when $F \supset \mathbb{Q}_p$ with $p$ large. When $m=2$, the stable conjugacy reduces to that defined by J.\ Adams, and the epipelagic $L$-packets coincide with those obtained by $\Theta$-correspondence. This fits within Weissman's formalism of $L$-groups. For $n=1$ and $m$ even, it is also compatible with the transfer factors proposed by K.\ Hiraga and T.\ Ikeda.
\end{abstract}

{\small
\begin{tabular}{ll}
	\textbf{MSC (2010)} & Primary 22E50; Secondary 11F27 \\
	\textbf{Keywords} & stable conjugacy, epipelagic supercuspidal representations, covering groups
\end{tabular}}

\vspace{1em}
\textdbend\;\textbf{\textsc{Note}} (added in December 2020): Please have a look of the Errata in \S\ref{sec:Errata}.

\tableofcontents

\section{Introduction}
\subsection*{Background}
Let $G$ be a reductive group over a local field $F$. The covering groups in question are topological central extensions $1 \to \bmu_m \to \tilde{G} \xrightarrow{\bm{p}} G(F) \to 1$ where $m \in \Z_{\geq 1}$ and $\bmu_m := \{z \in \CC^\times: z^m=1 \}$. Harmonic analysis on $\tilde{G}$ aims at studying its genuine representations, that is, the representations $(V_\pi, \pi)$ of $\tilde{G}$ satisfying $\pi(\noyau) = \noyau \cdot \identity_{V_\pi}$ for all $\noyau \in \bmu_m$. In contrast with the case of reductive groups ($m=1$), an important issue here is to single out a wide class of coverings admitting extra algebraic structures, and then try to formulate the local Langlands correspondence, etc.

Denote by $\mu(F)$ (resp. $\mu_m(F)$) the group of roots of unity (resp. $m$-th roots of unity) in $F$, and set $N_F := |\mu(F)|$ when $F \neq \CC$, otherwise $N_F := 1$. Brylinski and Deligne built a functorial framework in \cite{BD01} that produces central extensions of $G(F)$ by $K_2(F)$, or more precisely central extensions of $G$ by Quillen's $\shK_2$ as Zariski sheaves. This includes the $K_2(F)$-extensions constructed by Matsumoto \cite{Mat69} for simply connected split groups. To obtain a cover, take $m \mid N_F$ and write $\mu_m = \mu_m(F)$. Pushing-out first by the Hilbert symbol $(\cdot, \cdot)_{F,m}: K_2(F) \to \mu_m$ and then by a chosen $\epsilon: \mu_m \rightiso \bmu_m \subset \CC^\times$, we obtain a topological central extension $\bmu_m \hookrightarrow \tilde{G} \twoheadrightarrow G(F)$ as above. Covers arising in this way are called \emph{BD-covers} in this article, and one has the notion of $\epsilon$-genuine representations. This theory has a global counterpart and works over more general bases than $\Spec(F)$

Variants of Langlands' program in this perspective have been proposed and studied in \cite{GG,Weis18}, among others. In particular, Weissman \cite{Weis18} proposed a definition of $L$-group $\Lgrp{\tilde{G}}$ for BD-covers as an extension $\tilde{G}^\vee \hookrightarrow \Lgrp{\tilde{G}} \twoheadrightarrow \Weil_F$ of locally compact groups. One of his insights is that $\Lgrp{\tilde{G}}$ is not always splittable into $\tilde{G}^\vee \rtimes \Weil_F$; even when splittings exist, they often depend on auxiliary data.

One expects that the genuine irreducible admissible representations are organized into $L$-packets. Denote by $G_\text{reg}$ the open dense subset of strongly regular semisimple elements in $G$. It seems that the internal structure of these conjectural packets is related to the following issues.
\begin{enumerate}[\bfseries (A)]
	\item Of particular importance in harmonic analysis on $\tilde{G}$ are the \emph{good elements}, i.e. those $\tilde{\delta} \in \tilde{G}$ with image $\delta$ satisfying $Z_{\tilde{G}}(\tilde{\delta}) = \bm{p}^{-1}\left( Z_{G(F)}(\delta) \right)$; this property depends only on $\delta$. For BD-covers of a torus $T$, there is a canonical isogeny $\iota_{Q,m}: T_{Q,m} \to T$; it is known that all elements of $\Image(\iota_{Q,m})$ are good, and the converse holds when $T$ is split; this leads to a description of good elements in $T_\text{reg} := T \cap G_\text{reg}$ when $T \subset G$ is a split maximal torus. This is no longer true for non-split tori.
	\item For reductive groups, the internal structure of $L$-packets is elucidated by \emph{endoscopy}, which originates from the difference between ordinary and \emph{stable conjugacy}, i.e. conjugacy over the separable closure $\bar{F}$, for elements of $G_\text{reg}(F)$ at least. It is unclear if this can be lifted to BD-covers. It is more reasonable to do this on pull-backs $\tilde{T}_{Q,m}$ of the BD-cover via $\iota_{Q,m}: T_{Q,m} \to T$, for various maximal tori $T$, or some translates thereof in order to include all good elements.
	\item One should be able to form the \emph{stable character} attached to an $L$-packet: it is a sum of characters therein with certain multiplicities (often one). Desideratum: the stable character should be a genuine \emph{stable distribution}, in an appropriate sense for the BD-cover in question.
\end{enumerate}

Cf. \cite[\S\S 14---15]{GG} for further discussions. To the author's knowledge, only two non-trivial cases have been discovered.
\begin{itemize}
	\item Let $G = \Sp(W)$ where $(W, \lrangle{\cdot|\cdot})$ is a symplectic $F$-vector space of dimension $2n$, $n \in \Z_{\geq 1}$. Fix an additive character $\psi$, then we have Weil's metaplectic covering $\overline{G}_\psi^{(2)} \twoheadrightarrow G(F)$ with $m=2$. It is a distinguished instance of BD-covers, namely
	\begin{inparaenum}[(a)]
		\item it carries the \emph{Weil representation} $\omega_\psi = \omega_\psi^+ \oplus \omega_\psi^-$, and
		\item all elements are good; in fact $\iota_{Q,2} = \identity_T$ in this case, for all $T \subset G$.
	\end{inparaenum}
	Thus the issue \textbf{(A)} disappears, and the structure of packets should be explicated solely in terms of stable conjugacy.

	J.\ Adams defined stable conjugacy on $\overline{G}_\psi^{(2)}$ in terms of the characters of $\omega_\psi^\pm$. This is then developed by D.\ Renard and the author into a fully-fledged theory of endoscopy, as summarized in \cite{Li11}. Due to the usage of Weil representations, it cannot be ported to other BD-covers.
	
	The endoscopic character relations for this BD-cover are the topic of ongoing works of Caihua Luo,
	\item In an unpublished note by K.\ Hiraga and T.\ Ikeda \cite{HI}, they defined the transfer factors for BD-covers of $\SL(2)$ with $m \in 2\Z$ and established the transfer of orbital integrals; one still has to choose $\psi$ when $m \equiv 2 \pmod 4$. They also classified the good elements in $\SL_\text{reg}(2, F)$ and stabilized the regular elliptic part of the trace formula. This is ultimately based on Flicker's theory \cite{Fl80} for $\widetilde{\GL}(2,F)$ and makes use of Kubota's cocycles; both are unavailable in higher ranks.
	
	This offers a testing ground for notions of stable conjugacy, since the transfer factors should transform under stable conjugacy by some explicit character, as in the case of reductive groups \cite{LS1}.  
\end{itemize}
The aim of this article is to explore these issues for BD-covers of $G = \Sp(W)$, for general $n$ and $m \mid N_F$. Note that the representation theory for these covers has also been considered in the recent work \cite{Le19}.

\subsection*{Main results}
Henceforth we let $G := \Sp(W)$ for a symplectic $F$-vector space $(W, \lrangle{\cdot|\cdot})$ of dimension $2n$.
\paragraph*{Stable conjugacy}
Assume $\text{char}(F) \neq 2$ and let $G = \Sp(W)$ as above. The first goal of this article is to address \textbf{(A)} and \textbf{(B)} for all BD-covers of $G$. This is inspired, and in turn generalizes some aspects of the formalism of Adams and Hiraga--Ikeda. By \cite{BD01}, the $\shK_2$-extensions of $G$ are classified by the Weyl-invariant quadratic forms $Q: Y \to \Z$ where $Y = X_*(T_0)$ for some maximal torus $T_0 \subset G$; here we take the split one. There is a generator $Q$ that corresponds to Matsumoto's central extension $K_2(F) \hookrightarrow E_G(F) \twoheadrightarrow G(F)$, and we shall stick to the BD-covers of degree $m \mid N_F$ arising from $E_G(F)$. For harmonic analysis, this is not a real restriction as explained in Remark \ref{rem:rescaling-Q}: let $k \in \Z \smallsetminus \{0\}$, rescaling $Q$ to $kQ$ amounts to reducing $m$ to $m'$ and replacing $\epsilon$ by $\epsilon' := \epsilon^{k'}$, where $m' = m/\text{gcd}(k,m)$ and $k' = k/\text{gcd}(k,m)$. The same reduction applies to Weissman's $L$-groups, as explained in Theorem \ref{prop:rescale-L}.

The definition of stable conjugacy is sketched as follows. Let $\delta \in G_\text{reg}(F)$ so that $T := Z_G(\delta)$ is a maximal torus. When $n=1$ and $G \simeq \SL(2)$, stable conjugacy can always be realized by $G_\text{ad}(F)$-actions. This is no longer true when $n>1$. Nonetheless, we may construct an intermediate group $T \subset G^T \subset G$ such that
\[ \prod_{i \in I} R_{K_i^\sharp/F}\left( K_i^1 \right) \simeq T \hookrightarrow G^T \simeq \prod_{i \in I} R_{K_i^\sharp/F}\left( \SL(2)\right) \]
where
\begin{compactitem}
	\item $K_i^\sharp$ are finite separable extensions of $F$,
	\item $K_i$ is a quadratic étale $K_i^\sharp$-algebra, whose norm-one torus is denoted by $K_i^1$, and
	\item $R_{K_i^\sharp/F}$ stands for the Weil restriction functor.
\end{compactitem}
Consequently, all stable conjugates of $\delta$ can be obtained by $G^T_\text{ad}(F)$-action plus an ordinary conjugation by $G(F)$. These data are canonically defined in terms of the absolute root system of $T \subset G$.

Given a stable conjugacy $\Ad(g)(\delta) = \eta$ in $G(F)$, we may thus decompose $g = g'' g'$ with $g' \in G^T_\text{ad}(F)$ and $g'' \in G(F)$. Call $(g',g'')$ a \emph{factorization pair} for $\Ad(g)$. As an easy byproduct of \cite{BD01}, the adjoint action of $g'$ lifts uniquely to the $\shK_2$-extension pulled back to $G^T$, thus to the associated BD-cover. By inspecting the compatibility between Brylinski--Deligne classification and Weil restriction, the pull-back of $\tilde{G}$ to $G^T(F)$ turns out to be isomorphic the contracted product of BD-covers of $\SL(2, K_i^\sharp)$, all arising from Matsumoto's central extensions with the same $m$ (Theorem \ref{prop:G-T-reduction}). The determination of this pull-back is similar to the scenario in \cite[3.9]{D96}, but somewhat simpler. The upshot is a systematic reduction to the $\SL(2)$-case, modulo Weil restrictions.

We therefore obtain an embryonic stable conjugacy on $\tilde{G}_\text{reg} := \bm{p}^{-1}(G_\text{reg}(F))$ lifting $\Ad(g)$. However, it deviates from what is known for $m=2$ and is not well-defined: in fact, $(T/Z_G)(F)$ acts trivially on $T(F)$ but not so on its cover. As a calibration, we introduce a sign $\Cali_m(\nu(g'), \delta_0) \in \mu_{\text{gcd}(2,m)}$ into stable conjugacy, and the resulting operation is denoted by $\CaliAd(g)$. The price is that we need to prescribe $\delta_0 \in \iota_{Q,m}^{-1}(\delta)$; that is, we have to work in $\tilde{T}_{Q,m} \ni (\tilde{\delta}, \delta_0)$. Here $\delta = \bm{p}(\tilde{\delta})$ as usual.

More generally, for $\sigma \in \{\pm 1\}^I \hookrightarrow T(F)$, one forms $\tilde{T}^\sigma_{Q,m} := \left\{ (\tilde{\delta}, \delta_0): \iota_{Q,m}(\delta_0) = \sigma \cdot \delta \right\}$, and define $\CaliAd^\sigma(g): \tilde{T}^\sigma_{Q,m} \to \tilde{S}^\sigma_{Q,m}$ where $S := gTg^{-1}$ with $\sigma$ transported to $S$ via $\Ad(g)$. Here we assume $\sigma = (+, \ldots, +)$ unless $4 \mid m$. This is Definition \ref{def:st-conj}, which is independent of the choice of factorization pairs and satisfies the natural properties below (Proposition \ref{prop:CAd-prop}).
\begin{compactenum}[\bfseries{AD}.1.\;]
	\item $\CaliAd^\sigma(g)(\noyau\tilde{t}) = \noyau \CaliAd^\sigma(g)(\tilde{t})$ for all $\tilde{t} \in \tilde{T}^\sigma_{Q,m}$ and $\noyau \in \mu_m$.
	\item $\CaliAd(g): \tilde{T}_{Q,m} \rightiso \tilde{S}_{Q,m}$ is an isomorphism of topological groups. In general, $\CaliAd^\sigma(g)$ is a $\CaliAd(g)$-equivariant map between torsors.
	\item If $g \in G(F)$ then $\CaliAd^\sigma(g) = \Ad(g)$; if $g \in T(\bar{F})$, it equals $\identity$.
	\item Transitivity under composition of stable conjugations of maximal $F$-tori.
\end{compactenum}

We have arrived at a possible answer to \textbf{(B)}. As for \textbf{(A)}, Proposition \ref{prop:good-T} characterizes the good elements in $G_\text{reg}(F)$ completely: $\delta \in T_\text{reg}(F)$ is good if and only if $\delta = \sigma \cdot \iota_{Q,m}(\delta_0)$ for some $\delta_0 \in T_{Q,m}(F)$, with $\sigma$ needed only when $4 \mid m$. This boils down to the case of $\SL(2)$ by a similar reduction, which has been solved by Hiraga--Ikeda. It also gives some evidence for our formalism of stable conjugacy.

For potential applications to the trace formula, we show in Theorem \ref{prop:adelic-conj} that when $F$ is a global field, the local $\CaliAd(g_v)$ patch into an adélic one, and coincides with the usual $\Ad(g)$ on the image of $G(F)$ in the adélic BD-cover. A similar characterization of adélically good elements is obtained in Proposition \ref{prop:good-T-global}.

Theorem \ref{prop:st-conj-twofold} asserts that when $m=2$, the notion of stable conjugacy coincides with that of Adams via Weil representations. When $F=\R$, the only nontrivial case is $m=2$; Adams gave in \cite[Definition 3.4]{Ad98} a Lie-theoretic definition of stable conjugacy for real metaplectic groups, namely by lifting the action of stable Weyl groups. His recipe can still be reduced to $\SL(2, \R)$, and is seen to coincide with our recipe by Proposition \ref{prop:CaliAd-simple} since the calibration factor $\Cali_2(\cdots)$ disappears.

We also show in Theorem \ref{prop:HI-cocycle} that when $n=1$ and $m \in 2\Z$, the Hiraga--Ikeda transfer factors satisfy a transformation property with respect to $\CaliAd^\sigma(g)$, which justifies our formalism in that case.

\paragraph*{\texorpdfstring{$L$}{L}-groups}
Fix $\epsilon: \mu_m \rightiso \bmu_m$. It will be explained in \S\ref{sec:Weissman} that $\tilde{G}^\vee \simeq \Sp(2n,\CC)$ (resp. $\SO(2n+1, \CC)$) when $m \in 2\Z$ (resp. $m \notin 2\Z$), with trivial Galois action; furthermore, there is an $L$-isomorphism $\Lgrp{\tilde{G}} \rightiso \tilde{G}^\vee \times \Weil_F$. Such a splitting of $L$-group is canonical when $m \notin 2\Z$, otherwise it depends on
\begin{compactitem}
	\item an additive character $\psi$ when $m \equiv 2 \pmod 4$;
	\item a $G(F)$-conjugacy class of $F$-pinnings of $G$, or more concretely the symplectic form $\lrangle{\cdot|\cdot}$ on $W$ as explained below.
\end{compactitem}
Given $\lrangle{\cdot|\cdot}$, one chooses a symplectic basis for $W$ to produce a standard $F$-pinning, which is well-defined up to $G(F)$-conjugacy. The $F$-pinnings form a torsor under $G_\text{ad}(F)$, and the homomorphism $G_1 := \GSp(W) \twoheadrightarrow G_\text{ad}$ induces
\[ G_1(F)/( Z_{G_1}(F) G(F)) \rightiso G_\text{ad}(F) \big/ \Image\left[ G(F) \to G_\text{ad}(F) \right]. \]
Note that $Z_{G_1}(F) G_1(F)$ is the stabilizer of $\lrangle{\cdot|\cdot} \bmod F^{\times 2}$ under $G_1(F)$. Summing up, one may identify the $G(F)$-conjugacy classes of $F$-pinnings with the dilated symplectic forms $c\lrangle{\cdot|\cdot}$ where $c$ is taken up to $F^{\times 2}$. In this article, symplectic forms will usually be preferred over $F$-pinnings.

When $m \in 2\Z$, the $L$-isomorphism $\Lgrp{\tilde{G}} \simeq \tilde{G}^\vee \times \Weil_F$ undergoes an explicit quadratic twist $\chi_c: \Gamma_F \to \bmu_2$ when $\psi$ or $\lrangle{\cdot|\cdot}$ is rescaled by $c \in F^\times$; see \eqref{eqn:metaGalois-section-twist}.

\paragraph*{Epipelagic \texorpdfstring{$L$}{L}-packets}
In order to justify the formalism of stable conjugacy, we construct \emph{epipelagic supercuspidal $L$-packets} for $\tilde{G}$ when $F$ is non-archimedean with residual characteristic $p \nmid 2m$. This is a tractable class of supercuspidal representations, yet sufficiently rich to probe the genuine spectrum of $\tilde{G}$. The original definition is due to Reeder--Yu \cite{RY14}, but the approach here follows Kaletha \cite{Kal15,Kal19}. We shall take Yu's construction \cite{Yu01} for BD-covers with $p \nmid m$ as granted, in the epipelagic case at least.

On the dual side, there is a well-defined notion of epipelagic supercuspidal $L$-parameters $\phi$. Given $\phi$, Kaletha's recipe produces a stable class of embeddings $j: S \hookrightarrow G$ of certain elliptic maximal tori, which we call of type (ER). For reductive groups we also obtain a character $\theta$ of $S(F)$. For BD-covers, there is an isogeny $\iota_{Q,m}: S_{Q,m} \to S$ and what can be expected is just a character $\theta^\flat$ of $S_{Q,m}(F)$. The first task is to lift it to genuine character(s) $\theta_j$ of $\widetilde{jS} := \bm{p}^{-1}\left( jS(F) \right)$ for various $j$. Since the genuine representation theory of BD-covers of anisotropic tori is not yet fully developed, we do this in an \textit{ad hoc} way, namely
\begin{compactitem}
	\item we give explicit splittings of $\bm{p}$ over $jS(F)$ upon enlarging the cover (see \S\ref{eg:section-S}), thereby showing its commutativity;
	\item when $4 \mid m$, we take all $\theta_j$ with pro-$p$ component determined by $\theta^\flat$;
	\item when $4 \nmid m$, the desiderata on $\theta_j$ are encapsulated into the notion of \emph{stable system} (Definition \ref{def:stable-system}). Again, the pro-$p$ component of $\theta_j$ is pinned down by $\theta^\flat$.
\end{compactitem}
By varying $j$, we construct the packet $\Pi_\phi := \Pi(S, \theta^\flat)$ by taking compact-induction of $\theta_j$ from $\widetilde{jS} \cdot G(F)_{x, 1/e}$ where $x \in \mathcal{B}(G,F)$ is determined by $jS$ and $\frac{1}{e}$ stands for the depth of $\theta^\flat$. Specifically:
\begin{enumerate}[(i)]
	\item Kaletha's construction involves certain quadratic characters $\epsilon_{jS}$ of $jS(F)$ needed for stability (see \cite[\S 4.6]{Kal15}), and here we multiply $\theta_j$ by the same character before inducing.
	\item When $4 \mid m$, some packets might be empty and $\Pi_\phi$ depends on the splitting of $\Lgrp{\tilde{G}}$. As a workaround, we consider all $\chi_c$-twists of $\phi$ at once, for various $c \in F^\times$, to obtain the \emph{pre-$L$-packet} $\Pi_{[\phi]}$. By Remark \ref{rem:pre-packets}, there is at most one datum in the orbit $[\phi]$ under twists that yields a nonempty $\Pi(S, \theta^\flat)$.
	\item When $m \notin 2\Z$, Proposition \ref{prop:std-stable-system} furnishes a standard stable system, whilst for $m \equiv 2 \pmod 4$, the stability and the canonicity (i.e. independence of the splitting of $\Lgrp{\tilde{G}}$) of $\Pi_\phi$ impose stringent constraints on the stable system. In Theorem \ref{prop:dagger-SS}, a stable system for $m \equiv 2 \pmod 4$ will be constructed by reverse-engineering and extrapolation from the case $m=2$, with the help from the description of $\Theta$-lifting by Loke--Ma--Savin \cite{LMS16,LM18}. The main features of this stable system include
	\begin{compactitem}
		\item the use of \emph{moment maps} that relate $G$ and the pure inner forms $\SO(V,q)$ of the split $\SO(2n+1)$, which already appeared in \cite{Ad98};
		\item a sign character measuring the ratio between Kaletha's quadratic characters for $G$ and $\SO(V,q)$;
		\item it turns out that the ratio, as interpreted by Theorem \ref{prop:interplay}, depends not only on the embedded tori, but also on the pro-$p$ part of inducing data, which is related to the unrefined minimal $\mathsf{K}$-type \cite{MP94,MP96}.
	\end{compactitem}
\end{enumerate}
Theorem \ref{prop:central-character} shows that the elements in a packet need not share the same central character. Theorem \ref{prop:L-packet-prop} lists a few properties of $\Pi_\phi$ and $\Pi_{[\phi]}$, including cardinality and disjointness. We do not pursue the further issues such as basepoints (genericity) or formal degrees in this article.

By a straightforward comparison with \cite{LMS16}, we show in Theorem \ref{prop:Theta-compatibility} that when $\text{char}(F)=0$ and $m=2$, the $L$-packets $\Pi_\phi$ are obtained via $\Theta$-lifting from the epipelagic Vogan $L$-packets of $\SO(2n+1)$ constructed by Kaletha \cite{Kal15}. Our formalism is thus compatible with the local Langlands correspondence for $p$-adic metaplectic groups set forth by Gan--Savin \cite{GS1}.

The construction hinges on certain properties of maximal tori of type (ER), eg. Theorem \ref{prop:S-p'}. It would be beneficial to develop a theory for toral supercuspidals, or more generally for the regular supercuspidal representations considered in \cite{Kal19}.

\paragraph*{Stability of packets}
Let $\Xi$ be an invariant genuine distribution on $\tilde{G}$ represented by a locally $L^1$ function that is smooth over $\tilde{G}_{\mathrm{reg}}$. In Definition \ref{def:stability}, we say that $\Xi$ is \emph{stable} if for any maximal torus $T$, good $\tilde{\delta} \in \tilde{T}_{\mathrm{reg}}$ and $\sigma \in \{\pm 1\}^I$ (trivial when $4 \nmid m$), there exists $\delta_0 \in T_{Q,m}(F)$ such that $(\tilde{\delta}, \delta_0) \in \tilde{T}^\sigma_{Q,m}$ and
\[ \CaliAd^\sigma(g)(\tilde{\delta}, \delta_0) = \left( \tilde{\eta}, \Ad(g)(\delta_0)\right) \implies \Xi(\tilde{\delta}) = \Xi(\tilde{\eta}) \]
where $\Ad(g): \delta \mapsto \eta$ is a stable conjugacy. If $4 \nmid m$, the choice of $\delta_0$ is immaterial. If $4 \mid m$, the dependence of this condition on $\delta_0$ is quantified by Proposition \ref{prop:CAd-minus-1}.

Now assume $\text{char}(F)=0$ with large residual characteristic $p$ as in Hypothesis \ref{hyp:p-large}. Let $\phi$ be an epipelagic supercuspidal $L$-parameter for $\tilde{G}$. We show in Theorems \ref{prop:stability-1}, \ref{prop:stability-2} that the character-sum $S\Theta_\phi$ (resp. $S\Theta_{[\phi]}$) of $\Pi_\phi$ when $4 \nmid m$ (resp. $\Pi_{[\phi]}$ when $4 \mid m$) is a stable distribution in the sense above. As in \cite{Kal15}, the basic ingredients are
\begin{inparaenum}[(a)]
	\item the Adler--Spice character formula \cite{AS09} for epipelagic supercuspidal representations of $\tilde{G}$, whose validity is taken for granted, and
	\item Waldspurger's results \cite{Wa97} relating transfer and Fourier transform on Lie algebras, now completed by B.\ C.\ Ngô's proof of the Fundamental Lemma.
\end{inparaenum}

We remark that when $4 \nmid m$, the \textbf{SS.2} in Definition \ref{def:stable-system} is the key to stability. It stipulates that for stably conjugate embeddings $j: S \hookrightarrow G$ and $j' = \Ad(g)j$ in the given stable class, the stable system must satisfy
\[ \theta_{j'}\left( \CaliAd(g)(\tilde{\gamma}) \right) = \theta_j(\tilde{\gamma}), \quad \tilde{\gamma} \in \widetilde{jS}; \]
the reference to $\widetilde{jS}_{Q,m}$ being dropped according to Corollary \ref{prop:st-conj-elements-canonical}. This gives a preliminary answer to the issue \textbf{(C)} alluded to above.

\subsection*{Organization}
The assumptions on $F$ will be summarized in the beginning of each section and subsection. We caution the reader that they are not necessarily optimal.

In \S\ref{sec:BD-theory}, we review the basic formalism of Brylinski--Deligne theory and collect several results concerning Weil restrictions, which are not entirely trivial and seem to be missing in the literature.

In \S\ref{sec:Sp-gen}, we recall some structural results on symplectic groups, including the parameterization of regular semisimple classes and the formation of the intermediate group $T \subset G^T \subset G$. The parameterization of classes originates from \cite{SS70} and has been used in \cite{Wa01, Li11}; here we allow $F$ of any characteristic $\neq 2$.

In \S\ref{sec:BD-covers-Sp} we introduce Matsumoto's central extensions for $G = \Sp(W)$, and define stable conjugacy for the corresponding BD-covers. We also collect some results from Hiraga--Ikeda in the case $\dim_F W = 2$, with complete proofs. The properties studied in \S\ref{sec:further-properties} will play a crucial role in the subsequent sections.

Weissman's theory of $L$-groups is summarized in \S\ref{sec:Weissman}. We provide explicit splittings for $\Lgrp{\tilde{G}}$, cf. \cite{GG}, and clarify their dependence on auxiliary data. In \S\ref{sec:rescaling} we study the $L$-groups attached to Baer multiples of Matsumoto's central extension. The conclusion is that on the dual side, it is legitimate to confine ourselves to Matsumoto's central extension; this is unsurprising, but seem to be undocumented hitherto.

We fix the notation for Moy--Prasad filtrations, etc.\ in \S\ref{sec:construction-rep}. The construction of genuine epipelagic supercuspidal representations is also reviewed there. The only new ingredient is a description of maximal tori of type (ER) in $G$ and their preimages in $\tilde{G}$.

In \S\ref{sec:packets} we present the notions of inducing data and stable systems, then construct the $L$-packets and pre-$L$-packets and establish their basic properties (Theorems \ref{prop:L-packet-prop}, \ref{prop:central-character}). These packets are shown to be independent of the splitting of $\Lgrp{\tilde{G}}$ in \S\ref{sec:independence}. Stability is proven in \S\ref{sec:stability}.

In \S\ref{sec:stable-system} we construct a stable system for $m \equiv 2 \pmod 4$ using moment map correspondences. Auxiliary results in \S\ref{sec:MM} and \S\ref{sec:interplay} will be used in the later comparison with $\Theta$-lifting.

The compatibility of the foregoing constructions with existing theories (Weil's metaplectic groups, Hiraga--Ikeda theory) is resolved in \S\ref{sec:compatibilities}. The necessary backgrounds are also reviewed there.

\subsection*{Acknowledgments}
The author is grateful to Wee Teck Gan, Hung Yean Loke, Jiajun Ma and Gopal Prasad for helpful answers. He would also like to express his appreciation to the anonymous referee(s) for meticulous reading and suggestions. This work is supported by NSFC-11922101.

\subsection*{Conventions}
\paragraph{Fields}
For any field $F$, we denote by $\bar{F}$ a chosen separable closure and denote $\Gamma_F := \Gal{\bar{F}/F}$. Let $\mu_m$ be the group functor \index{mu_m@$\mu_m, \bmu_m$} $R \mapsto \mu_m(R) := \{r \in R^\times: r^m=1 \}$ for commutative rings $R$, and put $\bmu_m := \mu_m(\CC)$. The norm and trace in a finite extension $K/F$ is denoted by $N_{K/F}$ and $\Tr_{K/F}$, respectively.

Suppose $F$ is a non-archimedean local field. The normalized valuation on $F$ is denoted by $v_F$. For a finite extension $L/F$ of non-archimedean local fields, the ramification degree is denoted by $e(L/F)$. The residue field of $F$ is written as $\kappa_F = \mathfrak{o}_F/\mathfrak{p}_F = \mathfrak{o}_F/(\varpi_F)$. Let $\{ \Gamma_F^r \}_{r \geq -1}$ be the upper ramification filtration, and set $\Gamma_F^{r+} := \bigcap_{s > r} \Gamma_F^s$. In particular $\Gamma_F^{-1} = \Gamma_F$, $\Gamma_F^0 = I_F$ (inertia) and $\Gamma_F^{0+} = P_F$ (wild inertia). Denote by $F^\text{nr}$ the maximal unramified extension of $F$ inside $\bar{F}$.

The Weil group of a local or global field $F$ is written as $\Weil_F$.\index{W_F@$\Weil_F$} By an additive character over a local or finite field $F$ (resp. the adèle ring $\A_F$ of a global field), we mean a non-trivial continuous unitary homomorphism $\psi: F \to \CC^\times$ (resp. $\psi: \A_F/F \to \CC^\times$). For $c \in F^\times$ we write $\psi_c: x \mapsto \psi(cx)$. \index{$\psi_c$}

Denote by $\mu(F)$ the group of all roots of unity in $F^\times$, for any field $F$. When $F$ is a local or global field, we denote \index{N_F@$N_F$}
\[ N_F := \begin{cases} |\mu(F)|, & F \neq \CC \\ 1, & F = \CC. \end{cases} \]
If $F \neq \CC$ then $\mu(F) = \mu_{N_F}(F)$, and $N_F$ is always invertible in $F$. For $m \mid N_F$, we write $\mu_m = \mu_m(F)$ and denote by $(\cdot, \cdot)_{F,m}: F^\times \times F^\times \to \mu_m$ the Hilbert symbol of degree $m$.

\paragraph{Cohomology}
For a topological group $\Gamma$, we denote by $H^\bullet(\Gamma, \cdot)$ the continuous cohomology groups. The Galois cohomology over a field $F$ is written as $H^\bullet(F, \cdot)$. If $E/F$ is a separable quadratic extension of local fields, we denote by $\sgn_{E/F}: F^\times \to \mu_2 \simeq \bmu_2$ the corresponding quadratic character; $\sgn_{E/F}(x)=1$ if and only if $x \in N_{E/F}(E^\times)$.

\paragraph{Quadratic spaces}
Let $F$ be a field of characteristic $\neq 2$. By a quadratic $F$-vector space we mean a pair $(V,q)$ where $V$ is a finite-dimensional $F$-vector space and $q: V \times V \to F$ is a non-degenerate symmetric form; $q$ is determined by the function $q(v) := q(v|v)$. We write $\lrangle{a}$ for the quadratic $F$-vector space on $V=F$ with $q: x \mapsto ax^2$, and put $\lrangle{d_1, \ldots, d_n}$ for the orthogonal direct sum of the $\lrangle{d_i}$.

The (signed) \emph{discriminant} of $(V,q)$ is defined as $d^\pm(V,q) := (-1)^{n(n-1)/2} d_1 \cdots d_n \in F^\times/F^{\times 2}$ whenever $(V,q) \simeq \lrangle{d_1, \ldots, d_n}$. We have
\[ d^\pm((V,q) \oplus (V', q')) = d^\pm(V,q) d^\pm(V', q'). \] \index{dpm@$d^\pm(V,q)$}
The meaning of $\Or(V,q)$ and $\SO(V,q)$ is then clear. By $\SO(2n+1)$ we always mean the split $\SO(V,q)$ with $\dim_F V = 2n+1$.

Let $F$ be local with an additive character $\psi$. \emph{Weil's constant} associated to the character of second degree $\psi \circ \frac{q}{2}: V \to \CC^\times$ is denoted by $\gamma_\psi(q) \in \bmu_8$, see \cite[\S 24]{Weil64} or \cite[\S 1.3]{Per81}; this extends to a homomorphism into $\bmu_8$ of the Witt group of quadratic $F$-vector spaces. For $a \in F^\times$, we write $\gamma_\psi(a) := \gamma_\psi(\lrangle{a})$. The \emph{Hasse invariant} of $(V,q)$ is denoted by $\epsilon(V,q) \in \mu_2$. \index{epsilon(V,q)@$\epsilon(V,q)$}

In a similar vein, a symplectic $F$-vector space means a pair $(W, \lrangle{\cdot|\cdot})$, where $\dim_F W = 2n$ is finite and $\lrangle{\cdot|\cdot}$ is non-degenerate anti-symmetric. Dropping $\lrangle{\cdot|\cdot}$ from the notation, we have the groups $\Sp(W) \subset \GSp(W)$; we will occasionally write $\Sp(2n)$. For both quadratic and symplectic vector spaces, $\oplus$ stands for the orthogonal direct sum.

More generally, for a commutative ring $R$, a quadratic form or space, possibly singular, is the datum $(M,Q)$ where
\begin{inparaenum}[(a)]
	\item $M$ is an $R$-module,
	\item the map $Q: M \to R$ satisfies $Q(tx) = t^2 Q(x)$, for all $t \in R$ and $x \in M$, and
	\item $B_Q(x,y) := Q(x+y) - Q(x) - Q(y)$ is a symmetric $R$-bilinear form on $M$.
\end{inparaenum}
This general notion will mainly be applied for $R=\Z$.

\paragraph{Group schemes}
For any $S$-scheme $X$ and a morphism $U \to S$, we write $X_U := X \dtimes{S} U$ for its base-change to $T$. When $S = \Spec(F)$ and $U = \Spec(A)$, we write $X_A = X_{\Spec(A)}$ and $X(B)$ denotes $X(\Spec(B))$ if $B$ is a commutative ring. Denote by $S_\text{Zar}$ the \emph{big Zariski-site} and by $S_\text{ét}$ the \emph{small étale site} over $S$. Then $\mu_m$ is a sheaf for both sites.

The group $S$-schemes will be designated by letters $G$, $H$, etc.; for Lie algebras we use $\mathfrak{gothic}$ letters. The dual of $\mathfrak{g}$ is denoted as $\mathfrak{g}^*$. The identity connected component of a smooth $S$-group scheme $G$ (see \cite[Exposé $\mathrm{VI}_{\mathrm{B}}$ Théorème 3.10]{SGA3-1}) is denoted by $G^\circ$. We use $N_G(\cdot)$, $Z_G(\cdot)$ to indicate the normalizers and centralizers in $G$, respectively. The center of $G$ is written as $Z_G$. The same notation also pertains to abstract and topological groups.

When $F$ is a local field, $G(F)$ is endowed with the topology from $F$, and all the representations of $G(F)$ are over $\CC$.

\paragraph{Reductive groups}
The reductive groups $G$ over a field $F$ are assumed to be connected. The simply connected covering of the derived subgroup is $G_\text{sc} \to G_\text{der} \subset G$. The adjoint group (resp. derived subgroup) of $G$ is $G_\text{ad} := G/Z_G$ (resp. $G_\text{der}$). The adjoint action of $G$ or $G_\text{ad}$ on $G$ is written as $\Ad(g)(\delta) = g\delta g^{-1}$; on Lie algebras we have $\ad(X)(Y) = [X,Y]$. The coroot corresponding to root $\alpha$ is denoted by $\check{\alpha}$ or $\alpha^\vee$.

We say a regular semisimple element $\delta \in G$ is \emph{strongly regular} if $Z_G(\delta) = Z_G(\delta)^\circ$. Denote by $G_\text{reg} \subset G$ the strongly regular locus, which is Zariski-open. For any maximal torus $T \subset G$, we write $T_\text{reg} := T \cap G_\text{reg}$. The same notation pertains to Lie algebras: we have $\mathfrak{g}_\text{reg}$, etc.

For any $\gamma \in G(F)$, denote $G_\gamma := Z_G(\gamma)^\circ$; the same notation pertains to $X \in \mathfrak{g}(F)$. When $\gamma \in G(F)$ or $X \in \mathfrak{g}(F)$ is semi-simple, define the Weyl discriminants
\begin{align*}
	D^G(\gamma) & := \det\left( 1 - \Ad(\gamma) \big| \mathfrak{g}/\mathfrak{g}_\gamma \right), \\
	D^G(X) & := \det\left( \ad(X) \big| \mathfrak{g}/\mathfrak{g}_X \right).
\end{align*}

The Weyl group of a maximal $F$-torus $T \subset G$ is $\Omega(G,T) := N_G(T)/T$\index{$\Omega(G,T)$}. This is to be regarded as a sheaf over $\Spec(F)_\text{ét}$; thus $\Omega(G,T)(F) \supsetneq N_G(T)(F)/T(F)$ in general. For a diagonalizable $F$-group $D$, we write $X^*(D) = \Hom(D, \Gm)$ and $X_*(D) = \Hom(\Gm, D)$, again considered as sheaves over $\Spec(F)_\text{ét}$.

We say $G$ is pinned when it is endowed with an $F$-pinning. We denote the Langlands dual group of $G$ by $\check{G}$ or $G^\vee$ as a pinned $\CC$-group with $\Gamma_F$-action. The $L$-group is denoted by $\Lgrp{G}$: it usually means the Weil form unless otherwise specified. The relevant notation for coverings will be introduced in \S\ref{sec:L-group}.

\section{Theory of Brylinksi--Deligne}\label{sec:BD-theory}
\subsection{Multiplicative \texorpdfstring{$\shK_2$-torsors}{K2-torsors}}\label{sec:torsors-generalities}
The main reference is \cite{BD01}. Fix a base scheme $S$. Let $G, A$ be sheaves of groups over $S_\text{Zar}$ with $(A,+)$ commutative. Consider the \emph{central extensions} of $G$ by $A$
\[ 0 \to A \to E \stackrel{p}{\to} G \to 1. \]
This means that $p$ is an epimorphism between sheaves and $A \simeq \Ker(p)$. It is known \cite[Exp VII, 1.1.2]{SGA7-1} that $E$ is an $A$-torsor over $G$ in this context; in particular, $E \to G$ is Zariski-locally trivial. We will make frequent use of the shorthand $A \hookrightarrow E \twoheadrightarrow G$. 

As in the set-theoretical case, the adjoint action of $G$ on itself lifts to $E$, which we still denote as $\tilde{x} \mapsto g\tilde{x}g^{-1}$ where $g \in G$ and $\tilde{x} \in E$: it is the conjugation by any preimage of $g$ in $E$. To any central extension we may associate the \emph{commutator pairing} $G \times G \to A$. Set-theoretically, it is simply
\begin{equation}\label{eqn:commutator}
	[x,y] := \tilde{x}\tilde{y}\tilde{x}^{-1}\tilde{y}^{-1}
\end{equation} \index{[x,y]@$[x,y]$}
where $\tilde{x}, \tilde{y}$ are preimages of $x,y$ in $E$.

We have the following basic operations on $E$:
\begin{itemize}
	\item the pull-back $f^* E$ by a homomorphism $f: G_1 \to G$;
	\item the push-out $h_* E$ by a homomorphism $h: A \to A_1$, as torsors this amounts to $A_1 \overset{A,h}{\wedge} E$.
\end{itemize}
Up to a canonical isomorphism, the order of pull-back and push-out can be changed.
\begin{itemize}
	\item Given two central extensions $E_1 \to G_1$ and $E_2 \to G_2$ by the same sheaf $A$, we have the \emph{contracted product}\index{contracted product} $A \hookrightarrow E_1 \utimes{A} E_2 \twoheadrightarrow G_1 \times G_2$: it is the push-out of central extension $E_1 \times E_2 \to G_1 \times G_2$ (by $A \times A$) by $+: A \times A \to A$;
	\item when $G_1 = G_2$, the \emph{Baer sum} $E_1 + E_2$ of $E_1, E_2$ can be realized by pulling $E_1 \utimes{A} E_2$ back via the diagonal $G \hookrightarrow G \times G$.
\end{itemize}
 
Given $f: G_1 \to G$, a homomorphism $\varphi: E_1 \to E$ covering (or lifting) $f$ is a commutative diagram of sheaves of groups
\[\begin{tikzcd}
	0 \arrow{r} & A \arrow{r} \arrow{d}[right]{\identity} & E_1 \arrow{r} \arrow{d}[right]{\varphi} & G_1 \arrow{d}[right]{f} \arrow{r} & 1 \\
	0 \arrow{r} & A \arrow{r} & E \arrow{r} & G \arrow{r} & 1
\end{tikzcd}\]
Denote by $\cate{CExt}(G,A)$ \index{CExt@$\cate{CExt}(G,A)$}the category of central extensions of $G$ by $A$, with morphisms being the homomorphisms $E \to E_1$ covering $\identity_G$. This makes $\cate{CExt}(G, A)$ a groupoid equipped with the ``addition'' given by Baer sum, subject to the usual functorial constraints; such a structure is called a \emph{Picard groupoid}. In general, giving $\varphi: E_1 \to E$ covering $f: G_1 \to G$ is the same as giving a morphism $E_1 \to f^* E$ in $\cate{CExt}(G_1, A)$.

In \cite{BD01}, the framework of central extensions is reformulated in terms of \emph{multiplicative $A$-torsors} as in \cite[Exp VII]{SGA7-1}. These are $A$-torsors $p: E \to G$ equipped with a ``multiplication'' $m: \text{pr}_1^* E + \text{pr}^*_2 E \rightiso \mu^* E$, where
\begin{compactitem}
	\item $\mu: G \times G \to G$ is the multiplication,
	\item $G \xleftarrow{\text{pr}_1} G \times G \xrightarrow{\text{pr}_2} G$ are the projections,
	\item the $+$ signifies the Baer sum of $A$-torsors over $G \times G$.
\end{compactitem}
Furthermore, $m$ is required to render the following diagram of $A$-torsors over $G \times G \times G$ commutative
\[\begin{tikzcd}
	& \text{pr}_1^* E + \text{pr}_2^* E + \text{pr}_3^* E \arrow{rd}{m \times \identity} \arrow{ld}[swap]{\identity \times m} & \\
	\text{pr}_1^* E + \text{pr}_{23}^* \mu^* E \arrow{d}[swap]{\simeq} & & \text{pr}_{12}^* \mu^* E + \text{pr}_3^* E \arrow{d}{\simeq} \\
	\text{pr}_1^* E + \mu_{23}^* E \arrow{d}[swap]{(\identity \times \mu)^*(m)} & & \mu_{12}^* E + \text{pr}_3^* E \arrow{d}{(\mu \times \identity)^*(m)} \\
	(\identity \times \mu)^* \mu^* E \arrow{r}[above]{\sim} & \mu_{123}^* E & (\mu \times \identity)^* \mu^* E \arrow{l}[above]{\sim}
\end{tikzcd}\]
where $\mu_{123} = \mu \circ \mu_{ij}: G \times G \times G \to G$ are the morphisms that multiply the slots in the subscript. In forming the diagram we used the compatibility between pull-back and Baer sum. The resulting groupoid of multiplicative torsors is denoted by $\cate{MultTors}(G,A)$.

Given a central extension $A \hookrightarrow E \twoheadrightarrow G$, taking $m$ to be the group law of $E$ gives rise to a multiplicative $A$-torsor $E$ over $G$. By \cite[Exp VII, 1.6.6]{SGA7-1}, this establishes an equivalence $\cate{CExt}(G,A) \rightiso \cate{MultTors}(G,A)$.

Let $\shK_n$ be the Zariski sheaves associated to Quillen's $K$-groups $K_n$, for $n \in \Z_{\geq 0}$; note that $\shK_0 \simeq \Z$ and $\shK_1 \simeq \Gm$ canonically. Several observations are in order.
\begin{enumerate}
	\item By \cite[IV.6.4]{Wei13}, $K_n$ transforms products of rings to products of abelian groups. In parallel, $\shK_n(U_1 \sqcup U_2) = \shK_n(U_1) \times \shK_n(U_2)$ for disjoint union of schemes: in fact, the latter holds for any sheaf.
	\item The stalk of $\shK_n$ at any point $s$ of $S$ equals $K_n(\mathcal{O}_{S,s})$, where $\mathcal{O}_{S,s}$ stands for the local ring. Indeed, this is readily reduced to the case of affine $S$. Then one can use the fact that $K_n$ commutes with filtered $\varinjlim$, see \cite[IV. 6.4]{Wei13}.
	\item Consider a multiplicative $\shK_2$-torsor $E \to G$. When $S = \Spec(R)$ where $R$ is a field or a discrete valuation ring, there is a central extension
		\[ 1 \to K_2(S) \to E(S) \to G(S) \to 1 \]
		of groups. Indeed, $\shK_2(S) = K_2(S)$ by the previous observation; when $R$ is a field we have $H^1(S, \shK_2)=0$ for dimension reasons (true if $\shK_2$ is replaced by any $A$), and the same holds if $R$ is a discrete valuation ring by \cite[Corollary 3.5]{Weis16} (only for $\shK_2$).
\end{enumerate}

\subsection{Classification of Brylinski--Deligne}\label{sec:BD-classification}
In this subsection, $S$ is assumed to be regular of finite type over a field. Fix a reductive group $G$ over $S$ and let
\[ 0 \to \shK_2 \to E \to G \to 1 \]
be an object of $\cate{CExt}(G,\shK_2)$. We set out to review the classification in \cite{BD01} of such objects in terms of triplets $(Q,\mathcal{D},\varphi)$ with étale descent data, formulated as an equivalence of categories. Note that the classification below has been extended to some other rings in \cite{Weis16}, such as the spectrum of discrete valuation rings with finite residue fields.

\begin{asparaenum}[\bfseries (A)]
	\item Let $G=T$ be a split torus. Set $Y := X_*(T)$, a constant sheaf over $S_\text{Zar}$. One constructs in \cite[\S 3]{BD01} a central extension
		\begin{equation}\label{eqn:D-extension}
			1 \to \Gm \to \mathcal{D} \to Y \to 0
		\end{equation} \index{D@$\mathcal{D}$}
		as follows (in \textit{loc. cit.} one writes $\mathcal{E}$ instead of $\mathcal{D}$). First, we introduce a variable $\mathbf{t}$ and consider the base-change of $E$ to $\Gmm{S} = S[\mathbf{t}, \mathbf{t}^{-1}]$. It can be regarded as a central extension over $S$ by Sherman's theorem \cite[(3.1.2)]{BD01} that asserts for all $n$,
			\[ R^i(\Gmm{S} \to S)_* \shK_n = \begin{cases}
					\shK_n \oplus \shK_{n-1}, & i=0 \\
					0, & i > 0.
			\end{cases}	\]
			Pull the resulting extension back via $Y \to T[\mathbf{t}, \mathbf{t}^{-1}]$ that maps $\chi \in Y$ to $\chi(\mathbf{t}) \in T[\mathbf{t}, \mathbf{t}^{-1}]$, and push it out by $\shK_2(\cdots [\mathbf{t}, \mathbf{t}^{-1}]) \to \shK_1 = \Gm$ using Sherman's theorem, we obtain the required $\Gm \hookrightarrow \mathcal{D} \twoheadrightarrow Y$. Alternatively, we can also base-change to $S(\!(\mathbf{t})\!)$, and push-out via the tame symbol \cite[III. Lemma 6.3]{Wei13} $\shK_2(\cdots (\!(\mathbf{t})\!)) \to \shK_1 = \Gm$ instead.
			
			Moreover, one obtains a quadratic form $Q: Y \to \Z$ so that
			\[ B_Q(y_1, y_2) := Q(y_1 + y_2) - Q(y_1) - Q(y_2) \]
			satisfies $[y_1, y_2] = (-1)^{B_Q(y_1, y_2)}$ for all $y_1, y_2 \in Y$, where $[\cdot, \cdot]$ is the commutator pairing of $\Gm \hookrightarrow \mathcal{D} \twoheadrightarrow Y$. The classification in this case \cite[Proposition 3.11]{BD01} says that $\cate{CExt}(T,\shK_2)$ is equivalent to the groupoid of pairs $(Q, \mathcal{D})$ satisfying $[y_1, y_2] = (-1)^{B_Q(y_1, y_2)}$; by stipulation, $\Aut(\mathcal{D}, Q) = \Hom(Y, \Gm)$.
			
			It is important to note that for central extensions of $Y$ by $\Gm$, working over $S_\text{Zar}$ is the same as over $S_\text{ét}$.
	\item Let $G$ be simply connected and split. Fix any split maximal torus $T \subset G$ and put $Y = X_*(T)$. In \cite[\S 4]{BD01} one obtains a Weyl-invariant quadratic form
		\[Q: Y \to \Z \]
		from $E$. Furthermore, there are no non-trivial automorphisms of $E$. The classification in this case \cite[Theorem 4.7]{BD01} says that $\cate{CExt}(G, \shK_2)$ is equivalent to the groupoid of Weyl-invariant quadratic forms $Q: Y \to \Z$; by stipulation, $\Aut(Q) = \{\identity\}$.
		
		We may also pull-back $E$ by $T \hookrightarrow G$, obtaining a central extension of $T$ by $\shK_2$. The resulting quadratic form $Y \to \Z$ furnished by the case of tori is the same as the $Q$ above: this fact has been established in \cite[4.9]{BD01}.
		
	\item The case of a split reductive group $G$ is obtained by patching the cases above. Fix a split maximal torus $T$ with $Y = X_*(T)$. Then $\cate{CExt}(G, \shK_2)$ is equivalent to the groupoid of triplets $(Q, \mathcal{D}, \varphi)$ such that
		\begin{compactitem}
			\item $Q: Y \to \Z$ is a Weyl-invariant quadratic form, whose restriction to $Y_\text{sc} = X_*(T_\text{sc}) \hookrightarrow Y$ defines a central extension of $G_\text{sc}$, hence a central extension
				\[ 1 \to \Gm \to \mathcal{D}_\text{sc} \to Y_\text{sc} \to 0 \]
				by the theory over $T_\text{sc}$;
			\item $\mathcal{D}$ is a central extension of $Y$ by $\Gm$ satisfying $[y_1, y_2] = (-1)^{B_Q(y_1, y_2)}$;
			\item $\varphi: \mathcal{D}_\text{sc} \to \mathcal{D}$ covers $Y_\text{sc} \to Y$, i.e. there is a commutative diagram with exact rows
			\begin{equation}\label{eqn:varphi-diagram}\begin{tikzcd}
				1 \arrow{r} & \Gm \arrow{r} \arrow{d}[left]{\identity} & \mathcal{D}_\text{sc} \arrow{r} \arrow{d}[right]{\varphi} & Y_\text{sc} \arrow{d} \arrow{r} & 1 \\
				1 \arrow{r} & \Gm \arrow{r} & \mathcal{D} \arrow{r} & Y \arrow{r} & 1.
			\end{tikzcd}\end{equation}
			the morphisms are of the form $f: (Q, \mathcal{D}' ,\varphi') \rightiso (Q, \mathcal{D}, \varphi'')$, where $f: \mathcal{D}' \rightiso \mathcal{D}''$ is an isomorphism in $\cate{CExt}(Y, \Gm)$ and renders
			\[\begin{tikzcd}[row sep=small, column sep=small]
				& \mathcal{D}_\text{sc} \arrow{ld}[swap]{\varphi'} \arrow{rd}{\varphi''} & \\
				\mathcal{D}' \arrow{rr}[swap]{f} & & \mathcal{D}''
			\end{tikzcd}\]
			commutative.
		\end{compactitem}
	\item The case of general reductive groups $G$ is obtained in \cite[Theorem 7.2]{BD01} by rephrasing the description above étale-locally, by interpreting $\mathcal{D}, Y$, etc. as sheaves over $S_\text{ét}$. This is legitimate since split maximal tori exist locally over the finite étale site of $S$, by \cite[Exp XXII. Proposition 2.2]{SGA3-3}. For example, when $S$ is the spectrum of a field $F$, we may fix a maximal $F$-torus $T$ and $\cate{CExt}(G, \shK_2)$ is equivalent to the groupoid of triples $(Q, \mathcal{D}, \varphi)$ such that $Q: Y \to \Z$ is Weyl- and Galois-invariant, and so forth. Here $Y$ is viewed as a free $\Z$-module endowed with Galois action.
	\item The Baer sum $E_1 + E_2$ gives rise to $(Q_1 + Q_2, \mathcal{D}_1 + \mathcal{D}_2, \varphi_1 + \varphi_2)$ with obvious notations. This is implicit in \cite{BD01} and has been explicitly stated in \cite[Theorem 2.2]{Weis16}.
\end{asparaenum}

The classification immediately leads to the following statement. The assignment $T \mapsto \cate{CExt}(G_T, \shK_2)$, for $T \to S$ finite étale, becomes naturally a Picard stack over $S_{\text{ét}}$.

\begin{remark}\label{rem:Matsumoto}
	For a split simple and simply connected group $G$ over a field $F$, the space of Weyl-invariant quadratic forms $Y \otimes \R \to \R$ is one-dimensional. By \cite[Proposition 4.15]{BD01}, \emph{Matsumoto's central extension}\index{Matsumoto's central extension} \cite{Mat69} of $G$ by $\shK_2$ corresponds to the unique Weyl-invariant $Y \to \Z$ taking value $1$ one short coroots. The multiplicative $\shK_2$-torsors over $G$ are therefore classified by integers.
\end{remark}

\begin{remark}\label{rem:shared-torus}
	Let $H$ be a subgroup of $G$. Assume that $G,H$ are both simply connected and share a maximal torus $T$. For any object $E$ of $\cate{CExt}(G, \shK_2)$, its pull-back to $H$ and $E$ itself are classified by the same datum $Q: Y \to \Z$. Indeed, both pull back to the same object of $\cate{CExt}(T, \shK_2)$, classified by some $(Q, \mathcal{D})$.
\end{remark}

We record a result of Brylinski and Deligne for later use.
\begin{proposition}\label{prop:BD-adjoint-action}
	Suppose that $G = G_\mathrm{sc}$, and let $E$ be an object of $\cate{CExt}(G, \shK_2)$. Then the adjoint action of $Z_G$ on $E$ is trivial, and the action of $G_\mathrm{ad}$ on $G$ lifts uniquely to an action on $E$.
\end{proposition}
\begin{proof}
	When $G$ is split, this is stated in \cite[4.10]{BD01}; see also \cite[Proposition 5.13]{Mat69}. The general case follows by étale descent in view of the uniqueness of the lift. Note that the triviality of $Z_G$-action on $E$ is straightforward: it covers $\identity_G$, hence trivial since $G = G_\text{sc}$.
\end{proof}

\subsection{Weil restriction}\label{sec:Weil-restriction}
Given a morphism $f: T \to S$, we have the functor $f^*: \cate{Shv}(S_\mathrm{Zar}) \to \cate{Shv}(T_\mathrm{Zar})$ given by $f^* \mathcal{F}(X \xrightarrow{s} T) = \mathcal{F}(X \xrightarrow{fs} S)$. The functor of \emph{Weil restriction} $f_*\mathcal{F}(Y \to S) = \mathcal{F}(Y \times_S T \to T)$ is initially defined for presheaves, but turns out to yield a functor $\cate{Shv}(T_\mathrm{Zar}) \to \cate{Shv}(S_\mathrm{Zar})$ by \cite[p.194, Proposition 3]{BLR90}. Furthermore, $f_*$ is the right adjoint of $f^*$. In particular $f_*$ preserves $\varprojlim$, thus maps group objects to group objects. \index{Weil restriction}

\begin{remark}\label{rem:easy-Weil}
	Here is an easy case of Weil restriction. Suppose that $T = S^{\sqcup I}$ for some set $I$ (disjoint union of $I$ copies). One readily checks that for a family of sheaves $\{\mathcal{F}_i\}_{i \in I}$ on $S$, the sheaf $\mathcal{G} := (\mathcal{F}_i)_{i \in I}$ on $T$ satisfies $f_* (\mathcal{G}) = \prod_{i \in I} \mathcal{F}_i$.
\end{remark}

\begin{notation}\label{nota:Weil-restriction}\index{R_B/A@$R_{B/A}$}
	When $T=\Spec(B)$, $S=\Spec(A)$ and $f: T \to S$ corresponds to a ring homomorphism $A \to B$, it is customary to write $f_* = R_{B/A}$.
\end{notation}

In what follows, the morphism $f: T \to S$ is assumed to be finite and locally free of constant rank $d$. The treatment below is inspired by \cite{113891}.
\begin{definition}\label{def:well-behaved}
	Let $Y$ be a $T$-scheme and $E \to Y$ be an $A$-torsor, where $E,A$ are objects of $\cate{Shv}(T_\mathrm{Zar})$ and $A$ is a sheaf of abelian groups. We say $E \to Y$ is \emph{well-behaved} if there exists an open covering $\mathcal{U} = \{U_i\}_{i \in I}$ of $Y$ such that
	\begin{compactitem}
		\item every $U_i$ is affine,
		\item every $d$-tuple $(x_1, \ldots, x_d)$ of closed points of $Y$ lies in some $U_i$;
		\item there exists a section $s_i: U_i \to E$ of $E \to Y$ for every $i \in I$.
	\end{compactitem}
\end{definition}

For any $T$-group scheme $G$, we denote by $\cate{CExt}(G,A)_0 \subset \cate{CExt}(G,A)$ the full subcategory of $A \hookrightarrow E \twoheadrightarrow G$ such that $E \to G$ is a well-behaved $A$-torsor. When $G$ is affine, it is known that $f_* G$ is represented by an affine group $S$-scheme by \cite[p.194, Theorem 4]{BLR90}.

\begin{proposition}\label{prop:well-behaved}
	Let $G$ be an affine $T$-group scheme and let $A$ be a sheaf of abelian groups over $T_\mathrm{Zar}$. Then $f_*$ induces a functor $\cate{CExt}(G,A)_0 \to \cate{CExt}(f_* G, f_* A)$.
\end{proposition}
\begin{proof}
	Let $p: E \to G$ be an object of $\cate{CExt}(G,A)_0$. Since $f_*$ preserves $\varprojlim$,
	\[ 1 \to f_* A \to f_* E \xrightarrow{f_*(p)} f_* G \]
	is exact. Furthermore,
	\begin{align*}
		f_* E \dtimes{f_* G} f_* A & \longrightarrow f_* E \dtimes{f_* G} f_* E \\
		(x,a) & \longmapsto (xa,a)
	\end{align*}
	is an isomorphism. To get a central extension, it suffices to show that $f_*(p): f_* E \to f_* G$ locally admits sections. Take an open covering $\mathcal{U} = \{U_i\}_{i \in I}$ of $G$ as in Definition \ref{def:well-behaved}. Now \cite{113891} or the proof of \cite[p.194, Theorem 4]{BLR90} implies that $\{ f_* U_i\}_{i \in I}$ with the evident Zariski gluing data yields $f_* G$. Furthermore, for each $i$ we deduce a section $f_*(s_i): f_* U_i \to f_* E$. This shows the local triviality of $f_* E \to f_* G$.
\end{proof}

\begin{example}\label{eg:well-behaved-reductive}
	Suppose $T = \Spec(L)$ where $K$ is an infinite field, and $Y = G$ is a reductive group over $L$. It is well-known that $G(L)$ is Zariski-dense in $G$. For every Zariski $A$-torsor $E \to G$, there exists an open affine subscheme $U \subset G$, $U \neq \emptyset$, together with a section $s: U \to E$. For every $g \in G(L)$, since $H^1(\Spec L, A) = 0$ (see the end of \S\ref{sec:torsors-generalities}), $E(L) \twoheadrightarrow G(L)$ and we may choose a preimage $\tilde{g} \in E(L)$, thus obtain a section $s_g: gU \to E$ as the composite $gU \xrightarrow{\lambda_{g^{-1}}} U \xrightarrow{s} E \xrightarrow{\lambda_{\tilde{g}}} E$ where $\lambda_\bullet$ means left translation. We claim that
	\[ \mathcal{U} := \left\{ gU: g \in G(L) \right\} \]
	forms an open covering satisfying the requirements of Definition \ref{def:well-behaved}. Let $g_1, \ldots, g_d$ be closed points of $G$. Observe that $g^{-1}g_k \in U$ is an open condition on $g$ for any $k$, therefore those $g$ satisfying $g_k \in gU$ for $k=1, \ldots, d$ form a nonempty open subset of $G$. Since $G(L)$ is Zariski dense, we can choose $g \in G(L)$. The required section is given by $s_g: gU \to E$. In particular, we see $\cate{CExt}(G,A) = \cate{CExt}(G,A)_0$.
	
	Notice that given central extensions $E_j$ of $G$ by $A_j$ with $j=1,\ldots,n$, one can find an open affine $U$ as above (hence the covering $\mathcal{U}$) that works for all $E_j$. It follows that the functor in Proposition \ref{prop:well-behaved} preserves Baer sums. Indeed, given objects $E_1, E_2$ of $\cate{CExt}(G, A)$, the natural arrow $f_* E_1 \times f_* E_2 \simeq f_*(E_1 \times E_2) \to f_*(E_1 + E_2)$ induces $f_* E_1 + f_* E_2 \to f_*(E_1 + E_2)$. To show the latter is an isomorphism, we recall from the proof of Proposition \ref{prop:well-behaved} that if $E_j$ is glued from $\{ A \times U_i\}_{i \in I}$ with a system of transition functions $a^{(j)}_{i,i'}$ ($j=1,2$), then $f_* (E_1 + E_2)$ is glued from the transition functions $f_*(a^{(1)}_{i,i'} + a^{(2)}_{i,i'}) = f_* a^{(1)}_{i,i'} + f_* a^{(2)}_{i,i'}$. The same gluing datum defines $f_* E_1 + f_* E_2$.
\end{example}

\begin{example}
	Suppose that $T$ is a local scheme and $Y$ is isomorphic to the constant sheaf over $T_\text{Zar}$ associated to some set (such as $\Z^n$). The requirements in Definition \ref{def:well-behaved} are easily verified.
\end{example}

Still assume that $f: T \to S$ is locally free of finite constant rank $d$. This property is stable under base change, therefore the transfer/norm maps in $K$-theory \cite[V.3.3.2]{Wei13} yields an arrow
\[ f_* \shK_n \to \shK_n, \quad n=0,1,2,\ldots \]

\begin{proposition}\label{prop:BD-restriction-Weil}
	Let $L/F$ be a field extension of finite degree and write $f: \Spec(L) \to \Spec(F)$. Let $G$ be a reductive group over $L$. Then $E \mapsto f_* E$ followed by push-out by $f_* \shK_2 \to \shK_2$ gives a functor
	\[ \cate{CExt}(G, \shK_2) \to \cate{CExt}(f_* G, \shK_2) \]
	which is compatible with Baer sums, i.e. it is a monoidal functor.
\end{proposition}
\begin{proof}
	The first part is a combination of Proposition \ref{prop:well-behaved} and Example \ref{eg:well-behaved-reductive}. The preservation of Baer sums follows from the same property for $f_*: \cate{CExt}(G,A) \to \cate{CExt}(f_* G, f_* A)$ as explained in Example \ref{eg:well-behaved-reductive}.
\end{proof}

In addition, if $L/F$ is separable then $f_* G$ is a reductive group over $F$, otherwise it is \emph{pseudo-reductive}; we refer to \cite[A.5]{CGP15} for the relevant generalities.

Before discussing the effect of Weil restriction on the Brylinski--Deligne classification, we have to review the procedure of Galois descent. Let $L/F$ be a finite extension of fields, with maximal separable subextension $L_s/F$. Choose a finite Galois extension $M/F$ and set
\[ \mathcal{I} := \Hom_F(L_s , M), \quad M_\iota := L \dotimes{L_s, \iota} M, \; (\iota \in \mathcal{I}). \]
By taking sufficiently large $M$ to split $L_s/F$, we have an isomorphism
\begin{equation}\label{eqn:tensor-split}\begin{aligned}
	L \dotimes{F} M & \longrightiso L \dotimes{L_s} M^{\oplus \mathcal{I}} = \bigoplus_{\iota \in \mathcal{I}} M_\iota \\
	cx \otimes t & \longmapsto c \otimes (\iota(x)t)_{\iota \in \mathcal{I}}, \quad c \in L, \; x \in L_s, \; t \in M.
\end{aligned}\end{equation}
We deduce from $L_s \subset L$ a morphism $g_\iota: \Spec(M_\iota) \to \Spec(M)$. The diagram
\[\begin{tikzcd}
	\bigsqcup_{\iota \in \mathcal{I}} \Spec(M_\iota) \arrow{r}{\sim} \arrow{d}[swap]{\bigsqcup_\iota g_\iota } & \Spec(L) \dtimes{\Spec(F)} \Spec(M) \arrow{d}{f'} \arrow{r} & \Spec(L) \arrow{d}{f} \\
	\bigsqcup_{\iota \in \mathcal{I}} \Spec(M) \arrow{r}[swap]{\text{natural}} & \Spec(M) \arrow{r} & \Spec(F).
\end{tikzcd} \]
is commutative and the rightmost square is Cartesian. We have a functor $X \mapsto X_\iota := X \dtimes{L} M_\iota$ from $\cate{Shv}(\Spec(L)_\text{Zar})$ to $\cate{Shv}(\Spec(M_\iota)_\text{Zar})$. The canonical isomorphisms \cite[p.192]{BLR90}, \eqref{eqn:tensor-split} give
\begin{align*}
	f_*(X) \dtimes{\Spec(F)} \Spec(M) & \rightiso f'_* \left( X \dtimes{\Spec(F)} \Spec(M) \right) \\
	& \simeq f'_* \left( X \dtimes{\Spec(L)} \left(\Spec(L) \dtimes{\Spec(F)} \Spec(M)\right) \right) \\
	& \simeq f'_*\left( \bigsqcup_{\iota \in \mathcal{I}} X_\iota \right) \simeq  \prod_{\iota \in \mathcal{I}} (g_\iota)_* X_\iota. \qquad (\because\text{Remark \ref{rem:easy-Weil}})
\end{align*}
Every $\sigma \in \Gal{M/F}$ gives rise to isomorphisms $M_\iota \xrightarrow[\sim]{\sigma} M_{\sigma\iota}$, so $X_\iota \xrightarrow[\sim]{\sigma} X_{\sigma\iota}$. The $\Gal{M/F}$-action on $\prod_{i \in \mathcal{I}} (g_\iota)_* X_\iota$ can be determined from \eqref{eqn:tensor-split}, namely
\begin{equation}\label{eqn:MF-Gal-action}
	(x_\iota)_{\iota \in \mathcal{I}} \xmapsto{\sigma \in \Gal{M/F}} \left( \sigma(x_{\sigma^{-1}\iota}) \right)_{\iota \in \mathcal{I}}.
\end{equation}

\begin{remark}
	If the sheaf $X$ in the formalism above is defined over $\Spec(F)$, such as $\Gm$ and $\shK_2$, then $X_\iota = X \dtimes{L} M_\iota \simeq X \dtimes{F} M_\iota$ canonically, thus we can write $X_\iota = X$ unambiguously.
\end{remark}

\begin{notation}
	From now onwards, we assume $L/F$ to be \emph{separable}, so that $g_\iota = \identity$.
\end{notation}

Next, let $G$ be a reductive group over $L$. Upon enlarging $M$, there exists a maximal $L$-torus $T$ of $G$ that splits over $LM$. For each $\iota \in \mathcal{I}$ there is the $\Z$-module $Y_\iota := X_*(T \times_L M_\iota)$, with $\Gal{M/F}$ operating via $\sigma: Y_{\iota} \rightiso Y_{\sigma\iota}$. At the cost of neglecting Galois actions, one may identify $Y_\iota$ and $Y$. By the discussion above, $f_* T \subset f_* G$ is a maximal $F$-torus (see \cite[Proposition A.5.15]{CGP15}) that splits over $M$ and $X_*((f_* T)_M) = \bigoplus_{\iota \in \mathcal{I}} Y_\iota$ with Galois action given by the recipe \eqref{eqn:MF-Gal-action}. On the other hand, the Weyl group action on $\bigoplus_\iota Y_\iota$ is just the direct sum over $\mathcal{I}$ of the individual ones.

\begin{itemize}
	\item Note that $f_*(G_\text{sc})$ can be identified with $(f_* G)_\text{sc}$. In view of the foregoing discussions, it suffices to notice that $(G_1 \times \cdots \times G_n)_\text{sc} = G_{1,\text{sc}} \times \cdots \times G_{n,\text{sc}}$ for any tuple of reductive groups $G_1, \ldots, G_n$.
	\item We also have $f_*(G_\text{ad}) = (f_* G)_\text{ad}$. To see this, note the identification between $f_*(Z_G)$ and $Z_{f_* G}$ by \cite[Proposition A.5.15]{CGP15}; it remains to apply \cite[Corollary A.5.4 (3)]{CGP15} to see $f_* G / f_*(Z_G) \simeq f_*(G/Z_G)$ canonically.
\end{itemize}

\begin{theorem}\label{prop:BD-restriction}
	In the situation above, suppose that an object $E$ of $\cate{CExt}(G, \shK_2)$ corresponds to the triplet $(Q, \mathcal{D}, \varphi)$ with $\Gal{M/L}$-equivariance, as in \S\ref{sec:BD-classification}. Then $f_* T$ is split over $M$, and the image of $E$ in $\cate{CExt}(f_* G, \shK_2)$ by Proposition \ref{prop:BD-restriction-Weil} corresponds to the triplet $\left( f_* Q, f_* \mathcal{D}, f_* \varphi\right)$ where
	\begin{itemize}
		\item $f_* Q : \bigoplus_{\iota \in \mathcal{I}} Y_\iota \to \Z$ is the direct sum of $Q: Y_\iota = Y \to \Z$, which is automatically Weyl- and Galois-invariant;
		\item $f_* \mathcal{D}$ is the contracted product of the central extensions $\Gm \hookrightarrow \mathcal{D}_\iota \hookrightarrow Y_\iota$, which carries the evident Galois action/descent datum with respect to $M/F$;
		\item $f_* \varphi: f_* \mathcal{D} \to (f_* \mathcal{D})_\mathrm{sc} \simeq f_* (\mathcal{D})_\mathrm{sc}$ is obtained by first taking $\varphi_\iota := \varphi \dtimes{L,\iota} M$ to get
			\[\begin{tikzcd}
				1 \arrow{r} & \Gm^\mathcal{I} \arrow{r} \arrow[-, double equal sign distance]{d} & \bigoplus_{\iota \in \mathcal{I}} \mathcal{D}_{\mathrm{sc},\iota} \arrow{r} \arrow{d}{\oplus_\iota \varphi_\iota} & \bigoplus_{\iota \in \mathcal{I}} Y_{\mathrm{sc}, \iota} \arrow{r} \arrow{d} & 1 \\
				1 \arrow{r} & \Gm^\mathcal{I} \arrow{r} & \bigoplus_{\iota \in \mathcal{I}} \mathcal{D}_\iota \arrow{r} & \bigoplus_{\iota \in \mathcal{I}} Y_\iota \arrow{r} & 1
			\end{tikzcd}\]
			then push-out by product $\Gm^\mathcal{I} \to \Gm$; it makes the analogue of \eqref{eqn:varphi-diagram} commutative.
	\end{itemize}
\end{theorem}
\begin{proof}
	Since $f_* T \subset f_* G$ is a maximal $F$-torus in a reductive $F$-group. We know that $f_* E$ corresponds to some triplet $\left( f_* Q, f_* \mathcal{D}, f_* \varphi\right)$ attached to $f_* T$, together with descent data with respect to $M/F$. By the foregoing discussions,
	\[ f_* E \dtimes{F} M = \left[ \shK_2^\mathcal{I} \hookrightarrow \bigoplus_{\iota \in \mathcal{I}} E_\iota \twoheadrightarrow \bigoplus_{\iota \in \mathcal{I}} G_\iota \right] + \Gal{M/F}-\text{action}. \]
	The norm map in $K$-theory induced by $M \xrightarrow{\text{diag}} M^\mathcal{I}$ (cf. \eqref{eqn:tensor-split}) is $K_2(\cdots)^\mathcal{I} \xrightarrow{\text{sum}} K_2(\cdots)$; a further push-out yields the corresponding $\shK_2$-extension. The corresponding descriptions of $f_* \mathcal{D}$ and $f_* Q$ are then clear: the principle is the same as that for describing Baer sums under the Brylinkski--Deligne classification; we refer to \cite[\S 3]{BD01} for the precise constructions of $f_* \mathcal{D}$ and $f_* Q$.

	The same construction applies to $G_\text{sc}$, the description of $f_* \varphi$ is thus evident. It is also evident by the constructions above that the descent data come from \eqref{eqn:MF-Gal-action}.
\end{proof}

\subsection{The case over local fields: BD-covers}\label{sec:local-BD}
Matsumoto's theorem \cite[III.6.1]{Wei13} says that $K_2(F)$ is the abelian group generated by symbols $\{x,y\}_F$ with $x,y \in F^\times$, subject to the relations
\begin{gather*}
	\{xx', y \}_F = \{x,y\}_F + \{x',y\}_F , \quad \{x,yy'\}_F = \{x,y\}_F + \{x,y'\}_F , \\
	x \neq 0,1 \implies \{x,1-x\}_F = 1.
\end{gather*} \index{$\{x,y\}_F$}
Therefore $K_2(F)$ is a quotient of $F^\times \otimes_\Z F^\times$. In other words $K_2(F)$ equals the Milnor $K$-group $K_2^M(F)$. As is customary, we call $\{x,y\}_F$ the \emph{Steinberg symbol}; they are actually anti-symmetric and satisfy $\{x,-x\}_F = 1$ or $\{x,x\}_F = \{x, -1\}_F$ \cite[p.246]{Wei13}. Bi-multiplicative maps $F^\times \times F^\times \to A$ factorizing through $F^\times \times F^\times \to K_2(F)$ are called $A$-valued symbols, where $A$ is any abelian group.

Henceforth, we assume that $F$ is a local field.

Following \cite[10.1]{BD01}, we consider the symbols that are locally constant with respect to the topology on $F^\times$. By a result of Moore, there is an initial object $F^\times \times F^\times \to K_2^{\text{cont}}(F)$ in the category of locally constant symbols: in fact
\[ K_2^{\text{cont}}(F) = \begin{cases}
	\mu(F), & \mu \neq \CC \\
	\{1\} & \mu=\CC.
\end{cases} \]
The corresponding homomorphism $K_2(F) \to K_2^\text{cont}(F) = \mu(F)$ is just the Hilbert symbol $(\cdot,\cdot)_{F, N_F}$.

Let $G$ be a reductive group over $F$, so $G(F)$ inherits the topology from $F$. For any object $E$ of $\cate{CExt}(G, \shK_2)$, we first take $F$-points to obtain a central extension $0 \to K_2(F) \to E(F) \to G(F) \to 1$, then push-out by $K_2(F) \to K_2^\text{cont}(F)$ to obtain
\[ 1 \to K_2^\text{cont}(F) \to \tilde{G} \to G(F) \to 1. \]
Hence $\tilde{G} \to G(F)$ inherits local sections and transition maps from those of $E \to G$. It becomes a topological $K_2^\text{cont}(F)$-torsor over $G(F)$ by the following fact \cite[Lemma 10.2]{BD01}: for any scheme $X \to \Spec(F)$ of finite type and $s \in H^0(X, \shK_2)$, evaluation gives a locally constant function $s_1: X(F) \to K_2(F) \to K_2^\text{cont}(F)$. Consequently, $\tilde{G}$ is a topological central extension of locally compact groups, also known as a \emph{covering}; here $K_2^\text{cont}(F)$ carries the discrete topology. This is interesting only when $F \neq \CC$. If $m \in \Z_{\geq 1}$ and $m \mid N_F$, a further push-out by
\begin{equation*}\begin{aligned}
	\mu(F) = \mu_{N_F}(F) & \longrightarrow \mu_m(F) \\
	z & \longmapsto z^{N_F/m}
\end{aligned}\end{equation*}
furnishes a topological central extension of $G(F)$ by $\mu_m(F)$. It is the same as the push-out via $(\cdot, \cdot)_{F,m}$ by the following \eqref{eqn:norm-residue-d}.

\begin{remark}\label{rem:Hilb-symbol-cohomology}
	Suppose $m \mid N_F$. Below is a review of the cohomological interpretation of $(\cdot, \cdot)_{F,m}$. The Kummer map gives $\partial: F^\times \to H^1(F, \mu_m)$. Then $(x,y) \mapsto \partial x \cup \partial y$ factors through $K_2(F) \to H^2(F, \mu_m^{\otimes 2})$; this is called the \emph{norm-residue symbol}, see \cite[III.6.10]{Wei13}. Since the $\Gamma_F$ acts trivially on on $\mu_m = \mu_m(F)$,
	\begin{align*}
		H^2(F, \mu_m^{\otimes 2}) & = H^2(F, \mu_m) \otimes \mu_m(F) = {}_m \text{Br}(F) \otimes \mu_m(F) \\
		& = (\Z/m\Z) \otimes \mu_m(F) = \mu_m(F)
	\end{align*}
	where ${}_m \text{Br}(F)$ stands for the $m$-torsion part of $\text{Br}(F)$. The composite $K_2(F) \to \mu_m(F)$ turns out to be $(\cdot, \cdot)_{F,m}$. Moreover, if $d \mid m$, the interpretation above and the commutative diagram with exact rows
	\[\begin{tikzcd}
		1 \arrow{r} & \mu_m \arrow{d}{d} \arrow{r} & \Gm \arrow{d}{d} \arrow{r}{m} & \Gm \arrow[-, double equal sign distance]{d} \arrow{r} & 1 \\
		1 \arrow{r} & \mu_{m/d} \arrow{r} & \Gm \arrow{r}{m/d} \arrow{r} & \Gm \arrow{r} & 1
	\end{tikzcd}\]
	immediately lead to
	\begin{equation}\label{eqn:norm-residue-d}
		(x, y)_{F, m}^d = (x, y)_{F, m/d}, \quad x,y \in F^\times.
	\end{equation}

	Finally, it is known (Merkurjev–-Suslin) that $K_2(F)/m \rightiso \mu_m(F)$ under $(\cdot, \cdot)_{F,m}$: see \cite[III.6.9.3]{Wei13}.
\end{remark}

\begin{definition}[M. Weissman]
	By a \emph{BD-cover} of degree $m$, we mean a covering $p: \tilde{G} \twoheadrightarrow G(F)$ with $\Ker(p)=\mu_m(F)$ arising from the procedure above. It consists of the data $m \mid N_F$ together with a multiplicative $\shK_2$-torsor $E \to G$ over $F$. By convention, when $F=\CC$ the covering splits.
\end{definition}

The next result concerns Weil restrictions. Let $L/F$ be a finite separable extension of local fields, corresponding to $f: \Spec(L) \to \Spec(F)$. Let $G$ be a reductive group over $L$ and $E$ be an object of $\cate{CExt}(G,\shK_2)$. Denote by $E'$ the image in of $\cate{CExt}(f_* G, \shK_2)$ of $E$ furnished by Proposition \ref{prop:BD-restriction-Weil}. Note that $f_* G(F) = G(L)$. If $m \in \Z_{\geq 1}$ and $m \mid N_F$, then $\mu_m(F) = \mu_m(L)$ and the procedure above will yield
\begin{align*}
	\left[ K_2(L) \hookrightarrow E(L) \twoheadrightarrow G(L) \right] & \xrightarrow[\text{push-out}]{(\cdot, \cdot)_{L,m}} \left[ \mu_m(F) \hookrightarrow \tilde{G} \twoheadrightarrow G(L) \right], \\
	\left[ K_2(F) \hookrightarrow E'(F) \twoheadrightarrow G(L) \right] & \xrightarrow[\text{push-out}]{(\cdot, \cdot)_{F,m}} \left[ \mu_m(F) \hookrightarrow \tilde{G}' \twoheadrightarrow G(L) \right].
\end{align*}
In the degenerate case $L=\CC$, we interpret $(\cdot, \cdot)_{L,m}$ as the trivial homomorphism so that $\tilde{G} = G(L) \times \mu_m$.

\begin{proposition}\label{prop:restriction-commutes}
	Under the assumptions above, there is a canonical isomorphism $\tilde{G} \rightiso \tilde{G}'$ between topological central extensions of $G(L)$ by $\mu_m(F)$.
\end{proposition}
\begin{proof}
	By construction, the central extension $\tilde{G}'$ is obtained by first pushing-out $E(L)$ via the norm map $K_2(L) \to K_2(F)$ in $K$-theory, followed by $(\cdot, \cdot)_{F,m}: K_2(F) \to \mu_m(F)$. On the other hand, $\tilde{G}$ is obtained by pushing-out $E(L)$ via $(\cdot, \cdot)_{L,m}: K_2(L) \to \mu_m(L) = \mu_m(F)$. In view of the cohomological interpretation of Hilbert symbols (Remark \ref{rem:Hilb-symbol-cohomology}) and the functoriality of push-out, the existence of a canonical $\tilde{G} \rightiso \tilde{G}'$ is reduced to the commutativity of
	\[\begin{tikzcd}
		K_2(L) \arrow{d}[swap]{\text{norm}} \arrow{r} & H^2(L, \mu_m^{\otimes 2}) \arrow{d}{\text{cor}} \arrow{r} & \mu_m(L) \arrow[-, double equal sign distance]{d} \\
		K_2(F) \arrow{r} & H^2(F, \mu_m^{\otimes 2}) \arrow{r} & \mu_m(F).
	\end{tikzcd}\]
	Indeed, the commutativity of the rightmost square is \cite[Proposition 5.5]{D96}. As regards the leftmost square, see \cite[Lemma 18.2]{Sus85}. To show $\tilde{G} \rightiso \tilde{G}'$ is a homeomorphism, we may argue using \cite[Lemma 10.2]{BD01} as before. Note that all these make sense even when $L=\CC$.
\end{proof}
In summary, the Weil restriction of multiplicative $\shK_2$-torsors (Proposition \ref{prop:BD-restriction-Weil}) does not affect harmonic analysis. \index{Weil restriction}

\begin{remark}
	Not all BD-covers of $(f_* G)(F)$ arise from BD-covers of $G(L)$ via Weil restriction. To see a counter-example, assume $[L:F] = 2$, $F \neq \R$ and $m = 2$. The cocycle of any twofold BD-cover of $\Gm(L) = L^\times$ is given by some power of $(\cdot, \cdot)_{L, 2}$, by the proof of \cite[Proposition 3.11]{BD01}, hence these BD-covers are commutative. On the other hand, one can embed $L^\times$ into $\GL(2, F)$ and restrict Kubota's twofold cover of $\GL(2, F)$ to $L^\times$ to obtain non-abelian BD-covers of $(f_* \Gm)(F)$; see Proposition \ref{prop:Flicker-comm}.
\end{remark}

BD-covers intervene in harmonic analysis in the following way: let $\tilde{G} \twoheadrightarrow G(F)$ be a BD-cover of degree $m \mid N_F$. Upon choosing an embedding $\epsilon: \mu_m(F) \rightiso \bmu_m \subset \CC^\times$\index{epsilon@$\epsilon$}, we obtain a topological central extension of locally compact groups
\[ 1 \to \bmu_m \to \tilde{G} \to G(F) \to 1. \]

\begin{remark}\label{rem:rescaling-Q}
	Let $k \in \Z \smallsetminus \{0\}$ and denote by $\tilde{G}[kQ, m, \epsilon]$ the central extension by $\bmu_m$ obtained from the $k$-fold Baer sum of the $E \to G$ in $\cate{CExt}(G, \shK_2)$. The notation stems from the fact that $kE$ corresponds to the datum $(kQ, \ldots)$, as noted in \S\ref{sec:BD-classification}. Put $m=m'd$, $k=k'd$ with $d := \text{gcd}(k,m)$. In view of \eqref{eqn:norm-residue-d}, $\tilde{G}[kQ, m, \epsilon]$ is the push-out of $K_2(F) \hookrightarrow E(F) \twoheadrightarrow G(F)$ in two equivalent ways:
	\[ \left[ K_2(F) \xrightarrow{(\cdot,\cdot)_{F,m}} \mu_m \xrightarrow{k} \mu_m \xrightarrow{\epsilon} \bmu_m \right] = \left[ K_2(F) \xrightarrow{(\cdot, \cdot)_{F,m'}} \mu_{m'} \xrightarrow{\epsilon^{k'}} \bmu_m \right]. \]
	Hence $\tilde{G}[kQ, m, \epsilon]$ admits a canonical reduction to $\tilde{G}[Q, m', \epsilon']$, where $\epsilon' := \epsilon^{k'}: \mu_{m'} \rightiso \bmu_{m'}$.
\end{remark}

We proceed to review a few generalities on topological central extensions. It is sometimes convenient to allow central extensions by $\CC^\times$, and the results below will carry over.
\begin{notation}
	Let $\bmu_m \hookrightarrow \tilde{G} \twoheadrightarrow G(F)$ be any topological central extension of $G(F)$. For a subset $C \subset G(F)$, it is customary to denote by $\tilde{C}$ its preimage in $\tilde{G}$. In this manner we define $\tilde{G}_{\text{reg}}$, etc.
\end{notation}

\begin{definition}\index{representation!genuine}
	For $\tilde{G}$ as above, we say that a smooth representation $\pi$ of $\tilde{G}$ on a $\CC$-vector space $V$ is \emph{genuine} if $\pi(z) = z \cdot \identity_V$ for all $z \in \bmu_m$. In a similar vein, we have the notion of \emph{$\epsilon$-genuine representations} of a BD-cover of degree $m$ for any embedding $\epsilon: \mu_m(F) \hookrightarrow \CC^\times$.
\end{definition}

Likewise, let $C \subset G(F)$ be a subset, we say a function $f: \tilde{C} \to \CC^\times$ is genuine if $f(z\tilde{x}) = z f(\tilde{x})$ for all $\tilde{x} \in \tilde{C}$ and $z \in \bmu_m$.

As in \eqref{eqn:commutator}, we define the commutators $[x, y] = \tilde{x}\tilde{y}\tilde{x}^{-1}\tilde{y}^{-1}$ for topological central extensions.
\begin{definition}\label{def:good-element}\index{good element}
	Call an element $\gamma \in G(F)$ \emph{good} if
	\[ [\gamma, \eta]=1, \quad \eta \in Z_G(\gamma)(F). \]
	This amounts to $\eta\tilde{\gamma}\eta^{-1} = \tilde{\gamma}$ whenever $\eta \in Z_G(\gamma)(F)$ and $\tilde{G} \ni \tilde{\gamma} \mapsto \gamma$. Call an element of $\tilde{G}$ good if its image in $G(F)$ is.
\end{definition}
Every genuine $G(F)$-invariant function on $\tilde{G}_\text{reg}$ vanishes off the good locus. 

\begin{proposition}\label{prop:lifting-uniqueness}
	Suppose that $G(F)$ is perfect, i.e. equals its own commutator subgroup. Let $\sigma$ be an automorphism of the topological group $G(F)$. Then there exists at most one automorphism $\tilde{\sigma}: \tilde{G} \to \tilde{G}$ of coverings lifting $\sigma$.
\end{proposition}
\begin{proof}
	Let $\tilde{\sigma}_1, \tilde{\sigma}_2$ be two liftings, then $\chi := \tilde{\sigma}_1^{-1} \circ \tilde{\sigma}_2$ is an automorphism of $\tilde{G}$ lifting $\identity_{G(F)}$. Hence $\chi$ is a homomorphism $G(F) \to \bmu_m$, which must be trivial.
\end{proof}


Finally, we record a useful observation on splittings.
\begin{lemma}\label{prop:pro-p-splitting}
	Let $A \subset G(F)$ be a pro-$p$ subgroup, where $p \nmid m$ is a prime number. Then $\tilde{G} \twoheadrightarrow G(F)$ admits a unique splitting over $A$, and all elements in $A$ are good.
\end{lemma}
\begin{proof}
	Write $A \simeq \varprojlim_n A_n$ where $A_n$ is a finite $p$-group for each $n$. The obstructions to the existence and uniqueness of splittings live in $H^2(A, \bmu_m)$ and $H^1(A, \bmu_m)$, respectively. Continuous cohomology satisfies $H^\bullet(A, \bm{\mu_m}) = \varinjlim_n H^\bullet(A_n, \bmu_m)$. Evidently, $H^\bullet(A_n, \bmu_m) = 0$ for all $n$.
	
	Let $a \in A$ and $g \in Z_G(a)(F)$, then $[a^{p^k},g] = [a,g]^{p^k}$ in $\bmu_m$. We have $\lim_{k \to \infty} a^{p^k} = 1$ whereas $\{ p^k \}_{k=1}^\infty$ is periodic mod $m$, hence $[a, g]=1$.
\end{proof}

\subsection{The isogeny \texorpdfstring{$T_{Q,m} \to T$}{TQm to T}}\label{sec:isogeny}
Retain the notation from \S\ref{sec:local-BD}. In particular, we fix a reductive $F$-group $G$ and $m \mid N_F$. Take an object
\[ 1 \to \shK_2 \to E \to G \to 1 \]
of $\cate{CExt}(G, \shK_2)$, and suppose $T \subset G$ is a maximal $F$-torus. Thus $\Gamma_F$ acts on $Y := X_*(T_{\bar{F}})$. The Brylinski--Deligne classification in \S\ref{sec:BD-classification} attaches to $E$ a $\Gamma_F$-invariant quadratic form $Q: Y \to \Z$. Define \index{B_Q@$B_Q$} \index{Y_Qm@$Y_{Q,m}$}\index{X_Qm@$X_{Q,m}$}
\begin{align*}
	B_Q(y_1, y_2) & := Q(y_1 + y_2) - Q(y_1) - Q(y_2), \\
	Y_{Q,m} & := \left\{y \in Y: \forall y' \in Y, \; B_Q(y, y') \in m\Z \right\} \supset mY , \\
	X_{Q,m} & := \left\{x \in X \dotimes{\Z} \Q : \forall y \in Y_{Q,m}, \; \lrangle{x, y} \in \Z \right\} \subset \frac{1}{m} X.
\end{align*}
One should view them as local systems over $\Spec(F)_{\text{ét}}$. From $Y_{Q,m} \hookrightarrow Y$ we deduce an isogeny \index{$\iota_{Q,m}$} \index{T_Qm@$T_{Q,m}$}
\[ \iota_{Q,m}: T_{Q,m} \to T. \]
Note that $\iota_{Q,m}$ is an étale morphism and the scheme-theoretic $\Ker(\iota_{Q,m})$ is a finite étale, since $Y_{Q,m} \supset mY$ and $m$ is invertible in $F$.

The invariance properties allow use to view $Q, B_Q$ as living on $\varprojlim_T X_*(T_{\bar{F}})$, where $T$ ranges over maximal $F$-tori and the arrows are $G(\bar{F})$-conjugation. In particular, we have the commutative diagram below.
\begin{equation}\label{eqn:isogeny-Ad} \begin{tikzcd}[row sep=small]
	T_{Q,m, \bar{F}} \arrow{r}{\Ad(g)} \arrow{d}[swap]{\iota_{Q,m}} & S_{Q,m, \bar{F}} \arrow{d}{\iota_{Q,m}} \\
	T_{\bar{F}} \arrow{r}[swap]{\Ad(g)} & S
\end{tikzcd}\end{equation}
Hence the absolute Weyl group $\Omega(G,T)(\bar{F})$ acts on $T_{Q,m,\bar{F}}$. By taking $\Ad(g)$ defined over $F$, we see that stable conjugacy $\Ad(g): T \to S$ (see \S\ref{sec:Sp-parameters}) induces $T_{Q,m} \to S_{Q,m}$.

Now we take a finite separable extension $L$ of $F$, and write $f: \Spec(L) \to \Spec(F)$. Take $G$ to be a reductive group over $L$ and $T \subset G$ be a maximal $L$-torus. Let $E$ be an object of $\cate{CExt}(G, \shK_2)$ (thus defined over $L$). We turn to the Weil restriction $f_* E$ afforded by Proposition \ref{prop:BD-restriction-Weil}.

\begin{proposition}
	Suppose $m \mid N_F$, so that $m \mid N_E$ and consider the isogeny $\iota_{Q,m}: T_{Q,m} \to T$ between tori over $L$. Then the isogeny attached to $f_* T \subset f_* G$ and $f_* E$ can be identified with
	\[ f_*(\iota_{Q,m}): f_* T_{Q,m} \to f_* T. \]
\end{proposition}
\begin{proof}
	This is clear in view of the effect of $f_*$ on the datum $Q$, see Theorem \ref{prop:BD-restriction}.
\end{proof}
The discussions above work for both local and global fields. Henceforth we shall assume $F$ local. Consider the BD-cover obtained from $E$, denoted by $\tilde{G} \to G(F)$ as usual.

\begin{proposition}[M.\ Weissman \cite{Wei09}]\label{prop:good-in-tori}\index{good element}
	Suppose $G=T$ is a torus, then all elements in $\Image(\iota_{Q,m})$ are good; the converse is true when $T$ is split.
\end{proposition}

In general, we consider the pull-back $\tilde{T} \to T(F)$ of $\tilde{G} \to G(F)$. It is also a BD-cover.
\begin{corollary}\label{prop:isogeny-good}
	Let $\gamma \in T_\mathrm{reg}(F)$. If $\gamma$ lies in the image of $\iota_{Q,m}$, then it is good as an element of $G(F)$; the converse is true when $T$ is split.
\end{corollary}
\begin{proof}
	Since $Z_G(\gamma) = T$, it suffices to notice that the isogeny $\iota_{Q,m}$ is the same for both $\tilde{G}$ and $\tilde{T}$ and apply Proposition \ref{prop:good-in-tori}.
\end{proof}

\section{The symplectic group}\label{sec:Sp-gen}
Unless otherwise specified, $F$ is a local field of characteristic $\neq 2$.

\subsection{Conjugacy classes and maximal tori}\label{sec:Sp}
Let $(W, \lrangle{\cdot|\cdot})$ be a symplectic $F$-vector space with $\dim_F W = 2n > 0$. Every orthogonal decomposition $W = \bigoplus_{i=1}^k W_i$ yields to a natural embedding $\prod_{i=1}^k \Sp(W_i) \hookrightarrow \Sp(W)$.

There always exists a basis $e_1, \ldots, e_n, e_{-n}, \ldots, e_{-1}$ satisfying $\lrangle{e_i, e_j}=0$ and $\lrangle{e_i, e_{-j}} = \delta_{i, j}$ (Kronecker's delta) for all $1 \leq i,j \leq n$, called a \emph{symplectic basis} of $W$. It gives rise to the Borel pair
\begin{align*}
	B & := \Stab_{\Sp(W)}\left( \text{the flag}\; \lrangle{e_1} \subset \lrangle{e_1, e_2} \subset \cdots \subset \lrangle{e_1, \ldots, e_{-1}} \right), \\
	T & := \left\{ \text{diag}(a_1, \ldots, a_n, a_n^{-1}, \ldots, a_1^{-1}) : (a_1, \ldots, a_n) \in \Gm^n \right\} \; \simeq \Gm^n.
\end{align*}
Write
\[ X := X^*(T), \quad Y := X_*(T). \]
Denote by $\epsilon_i \in X$ the character $\text{diag}(a_1, \ldots, a_1^{-1}) \mapsto a_i$. The $B$-positive roots are
\[ \epsilon_i \pm \epsilon_j \; (n \geq i>j \geq 1), \qquad 2\epsilon_i\; (n \geq i \geq 1) \]
and the simple ones are $\{ \epsilon_i - \epsilon_{i+1}: i=1, \ldots, n-1 \} \sqcup \{2\epsilon_i: i=1, \ldots, n\}$. Writing $\check{\epsilon}_1, \ldots, \check{\epsilon}_n$ for the dual basis, the corresponding coroots are
\[ \check{\epsilon}_i \pm \check{\epsilon}_j, \quad \check{\epsilon}_i. \]
The long positive roots are $2\epsilon_i$ whilst the short positive coroots are $\check{\epsilon}_i$. This is a simply connected root datum of type $\mathbf{C}_n$ ($ :=\mathbf{A}_1$ when $n=1$). The Weyl group attached to $T \subset \Sp(W)$ is
\[ \Omega(G,T) = \{\pm 1\}^n \rtimes \mathfrak{S}_n \]
acting on $Y = \bigoplus_i \Z\check{\epsilon}_i$, where the $\{\pm 1\}^n$ acts by component-wise multiplication and $\mathfrak{S}_n$ permutes. The Weyl-invariant quadratic forms $Y \to \Z$ are readily seen to be integer multiples of
\begin{equation}\label{eqn:Y-Sp}
	Q: (y_1, \ldots, y_n) \mapsto y_1^2 + \cdots + y_n^2.
\end{equation}
Also note that $Q$ takes value $1$ on short coroots. In this case we have
\[ Y = \bigoplus_{\substack{\alpha > 0 \\ \text{simple root}}} \Z\check{\alpha} = \bigoplus_{\substack{\alpha > 0\\ \text{long root}}} \Z\check{\alpha}. \]

Next, we recall the parametrization of regular semisimple conjugacy classes in $\Sp(W)$ following \cite[IV.2]{SS70}, \cite{Wa01} or \cite[\S 3]{Li11}. They are described by the data $(K, K^\sharp, x, c)$ in which \index{$(K, K^\sharp, x, c)$}
\begin{itemize}
	\item $K$ is an étale $F$-algebra of dimension $2n$ equipped with an involution $\tau$;
	\item $K^\sharp$ is the subalgebra $\{ t \in K: \tau(t)=t\}$;
	\item $x \in K^\times$ satisfies $\tau(x) = x^{-1}$ and $K=F[\alpha]$;
	\item $c \in K^\times$ satisfies $\tau(c) = -c$.
\end{itemize}
In the general recipe of \cite[IV.1]{SS70}, one sets $K$ to be the subalgebra of $\End_F(W)$ generated by the element $\gamma$ in question; since $\gamma \in \Sp(W)_\text{reg}$, the étale property of $K$ can be seen over the algebraic closure of $F$. The notation is justified since the extension $K \supset K^\sharp$ determines $\tau$. Indeed, $K^\sharp$ is also an étale $F$-algebra, thus decomposes into $\prod_{i \in I} K_i^\sharp$ where each $K_i^\sharp$ is a field. Accordingly, $K = \prod_{i \in I} K_i$ and for each $i \in I$, either
\begin{compactenum}[(a)]
	\item $K_i$ is a quadratic field extension of $K_i^\sharp$ with $\Gal{K_i/K_i^\sharp} = \{\identity, \tau|_{K_i} \}$, or
	\item $K_i \simeq K_i^\sharp \times K_i^\sharp$ as $K_i^\sharp$-algebras and $\tau(u,v)=(v,u)$.
\end{compactenum}
In both cases $K_i/K_i^\sharp$ determines $\tau|_{K_i}$.

\begin{notation}\label{nota:norm}
	Given an étale $F$-algebra $K$ with involution $\tau$ and fixed subalgebra $K^\sharp$, we define the norm map
	\begin{align*}
		N_{K/K^\sharp}: K & \longrightarrow K^\sharp \\
		t & \longmapsto t\tau(t)
	\end{align*}
	and the norm-one $F$-torus
	\[ K^1 := \left\{ y \in K^\times: N_{K/K^\sharp}(y)=1 \right\}. \]
\end{notation}

We say $(K, K^\sharp, x, c)$ and $(L, L^\sharp, y, d)$ are equivalent if there exists an isomorphism $\varphi: K \rightiso L$ of $F$-algebras such that
\begin{compactitem}
	\item $\varphi$ preserves involutions,
	\item $\varphi(x)=y$,
	\item $\varphi(c)d^{-1} \in N_{L/L^\sharp}(K^\times)$.
\end{compactitem}
To each datum $(K, K^\sharp, x, c)$, we deduce a symplectic form on the $F$-vector space $K$:
\begin{equation}\label{eqn:K-symp-form}
	h(u|v) = \Tr_{K/F}\left( \tau(u)vc \right);
\end{equation}
the trace here is non-degenerate since $K$ is étale over $F$. It can also be described by breaking $K$ into $\prod_{i \in I} K_i$. For every $x \in K^1$, the automorphism $m_x: t \mapsto xt$ preserves $h$. Every symplectic $F$-vector space $(W, \lrangle{\cdot|\cdot})$ of dimension $2n$ is isomorphic to $(K, h)$ via some $\iota: W \rightiso K$: therefore $m_x$ transports to $\delta := \iota^{-1} m_x \iota \in \Sp(W)$. The conjugacy class $\mathcal{O}(K/K^\sharp, x, c)$ of $\delta$ in $\Sp(W)$ is regular semisimple and is independent of the choice of $\iota$. The following result is standard.

\begin{proposition}\label{prop:reg-ss-parameters}
	The mapping $\mathcal{O}$ is a bijection between the equivalence classes of data $(K, K^\sharp, x, c)$ and the regular semisimple conjugacy classes in $\Sp(W)$. Furthermore, by choosing $\iota: W \rightiso K$ as above, there is an isomorphism of $F$-tori
	\begin{align*}
		j: K^1 & \longrightiso Z_{\Sp(W)}(\delta) \\
		y & \longmapsto \iota^{-1} m_y \iota.
	\end{align*}
\end{proposition}
In particular $Z_{\Sp(W)}(\delta)$ is always connected; this is also known from the fact that $\Sp(W)$ is simply connected.

The datum $x$ is not used in constructing $(K,h)$. There is a similar notion of equivalence among data $(K, K^\sharp, c)$. Given such a datum, $j: y \mapsto \iota^{-1} m_y \iota$ embeds $K^1$ as a maximal $F$-torus $j(K^1) \subset \Sp(W)$; its conjugacy class is independent of the choice of $\iota$.

\begin{proposition}\label{prop:parameter-tori}
	The assignment $(K, K^\sharp, c) \mapsto j$ is a bijection between the equivalence classes of data $(K, K^\sharp, c)$ and the conjugacy classes of embeddings of maximal $F$-tori $j: T \hookrightarrow \Sp(W)$.
\end{proposition}
\begin{proof}
	To show surjectivity, given $j$ there exists $t \in T(F)$ such that $j(t) \in \Sp(W)_\text{reg}$ by Zariski density. By Proposition \ref{prop:reg-ss-parameters}, the class of $j(t)$ is parameterized by some $(K, K^\sharp, x, c)$ that comes with an embedding $j_K: K^1 \to \Sp(W)$; upon conjugating $j$, we may assume $j_K(x) = j(t)$. Taking centralizers yields $\Image(j) = \Image(j_K)$, hence we deduce $j_K^{-1} j: T \rightiso K^1$ that implements the equivalence between $j$ and $j_K$.

	As for injectivity, let $j, j'$ be embeddings arising from data $(K, K^\sharp, c)$ and $(K', K'^\sharp, c')$. If $j, j'$ are conjugate, there will exist $x \in K^1$ and $x' \in K'^1$ such that $j(x)$ and $j'(x')$ are conjugate in $\Sp(W)_\text{reg}$, which implies the equivalence between $(K, K^\sharp, x, c)$ and $(K', K'^\sharp, x', c')$ by Proposition \ref{prop:reg-ss-parameters}.
\end{proof}
Note that the classification of conjugacy classes of maximal tori $T \subset \Sp(W)$ is coarser than that embeddings.

\begin{remark}\label{rem:parameter-GSp}
	If $\lrangle{\cdot|\cdot}$ becomes $a\lrangle{\cdot|\cdot}$, then for the $\iota: W \rightiso K$ above to be an isometry, the datum $(K, K^\sharp, x, c)$ must be replaced by $(K, K^\sharp, x, ac)$, i.e. replace $h$ by $ah$. In a similar vein, suppose that the $\iota^{-1} m_y \iota \in \Sp(W)$ parameterized by $y \in K^1$ becomes $g_1 \iota^{-1} m_y \iota g_1^{-1}$, with $g_1 \in \GSp(W)$ having similitude factor $a$; to make $\iota g_1^{-1}: W \to K$ an isometry, we must pass to $(K, K^\sharp, x, ac)$.
\end{remark}

Finally, the action of $\Omega(G,T)(F) := (N_G(T)/T)(F)$ on $T(F) \simeq K^1$ has the algebraic description below.
\begin{proposition}\label{prop:big-Weyl-action}
	The action of $\Omega(G,T)(F)$ on $T$ is the same as that of $\Aut(K, \tau)$ on $K^1$, where $(K, \tau)$ is seen as an étale $F$-algebra with involution.
\end{proposition}
\begin{proof}
	We have $K \otimes_F \bar{F} \simeq (\bar{F} \times \bar{F})^n$, so the description of Weyl groups in \S\ref{sec:Sp} ensures that $\Omega(G,T)(\bar{F}) = \Aut((K, \tau) \otimes_F \bar{F})$. This identification respects $\Gamma_F$-actions, so we conclude by Galois descent.
\end{proof}

\subsection{Stable conjugacy: parameters}\label{sec:Sp-parameters}
Let $G$ be a reductive $F$-group. We say $\delta, \eta \in G_\text{reg}(F)$ are \emph{stably conjugate}, written as $\delta \stackrel{\text{st}}{\sim} \eta$, if they are conjugate in $G(\bar{F})$. Define the variety $\mathcal{T}(\delta, \eta) := \{ g \in G: g\delta g^{-1} = \eta\}$, which admits an $F$-structure. It is nonempty (resp. has an $F$-point) if and only if $\delta$ and $\eta$ are stably conjugate (resp. conjugate).\index{stable conjugacy}

\begin{proposition}\label{prop:canonical-isom}
	Let $\delta, \eta \in G_\mathrm{reg}(F)$. Define the $F$-tori of centralizers
	\[ R := G_\delta, \quad S := G_\eta. \]
	If $\delta$ and $\eta$ are stably conjugate, then there is a canonical isomorphism $\Ad(g): R \rightiso S$ between $F$-tori, where $g \in \mathcal{T}(\delta, \eta)(\bar{F})$ is arbitrary.
\end{proposition}
\begin{proof}
	Indeed, $\mathcal{T}(\delta, \eta)$ is an $(S,R)$-bitorsor, meaning that it is simultaneously a left (resp. right) torsor under $S$ (resp. $R$) satisfying $s(gr) = (sg)r$ for all $g$ in $\mathcal{T}(x,y)$. Therefore $g$ defines an isomorphism $\sigma: S \rightiso R$ such that $sg = g\sigma(r)$, and this is seen to be independent of $g$ as $S,R$ are both commutative. In our case $\sigma = \Ad(g): S \rightiso R$ for any $g \in \mathcal{T}(\delta, \eta)(\bar{F})$, and this isomorphism descends to $F$.
\end{proof}

We say two embeddings of maximal $F$-tori $j,j': S \hookrightarrow G$ are \emph{stably conjugate} if there exists $g \in G(\bar{F})$ inducing $j' = \Ad(g) \circ j$. The foregoing Proposition entails that stable conjugacy of strongly regular semisimple elements is realized by stable conjugacy of embeddings of maximal tori, and the converse is evidently true.

\begin{definition}
	For any subgroup $S \subset G$, define the pointed set
	\[ \mathfrak{D}(S, G; F) := \Ker\left[ H^1(F, S) \to H^1(F, G) \right]. \]
\end{definition}

Consider the situation of stable conjugacy $\eta = g\delta g^{-1}$ as before, where $g \in G(\bar{F})$. Put $T := G_\delta$. The element $g$ induces a $T$-valued $1$-cocycle $c_g: \Gamma_F \ni \tau \mapsto g^{-1}\tau(g)$. It is well-known (eg. \cite[\S 27]{ArIntro}) that the recipe induces a bijection
\begin{equation}\label{eqn:inv} \begin{aligned}
	\text{inv}(\delta, \cdot): \left\{ \eta \in G(F): \delta \stackrel{\text{st}}{\sim} \eta \right\} \bigg/ \text{conj} & \longrightiso \mathfrak{D}(T, G; F) \\
	\eta = g\delta g^{-1} & \longmapsto \text{inv}(\delta, \eta) := [c_g].
\end{aligned}\end{equation}
In fact, $[c_g] \in H^1(F, T)$ parameterizes the right $T$-torsor $\mathcal{T}(\delta, \eta)$. By fixing a stable conjugacy class $\mathcal{O} \subset G(F)$, we may work inside $\varprojlim_\delta H^1(F, G_\delta)$ for $\delta$ ranging over $\mathcal{O}$, at least as a device to simplify expressions. In this setting, one verifies readily that for $\delta, \delta', \delta'' \in \mathcal{O}$
\begin{equation}\label{eqn:st-conj-in-stages}\begin{aligned}
	\text{inv}(\delta, \delta'') & = \text{inv}(\delta, \delta') + \text{inv}(\delta', \delta''), \quad \text{and} \\
	\text{inv}(\delta, \delta') & = -\text{inv}(\delta', \delta).
\end{aligned}\end{equation}
Likewise, we can measure the relative position of stably conjugate embeddings $j,j': S \hookrightarrow G$ by $\text{inv}(j, j') \in H^1(F, S)$: it is still given by $[c_g]$ when $j' = \Ad(g) \circ j$. It has the same properties above. \index{inv@$\mathrm{inv}(\cdot, \cdot)$}

Let us turn to a special instance of stable conjugacy. The adjoint group $G_\text{ad}(F)$ acts on $G(F)$ and $g \delta g^{-1}$ is stably conjugate to $\delta$ whenever $g \in G_\text{ad}(F)$. On the other hand, the short exact sequence $1 \to Z_G \to G \xrightarrow{\pi} G_\text{ad} \to 1$ induces $G_\text{ad}(F)/\pi(G(F)) \rightiso \mathfrak{D}(Z_G, G; F)$.
\begin{proposition}\label{prop:adjoint-rel-pos}
	For all $\delta \in G_\mathrm{reg}(F)$ with $T = G_\delta$ and all $g \in G_\mathrm{ad}(F)$, we have
	\[ \mathrm{inv}(\delta, g\delta g^{-1}) = \text{the image of $g$ under}\; G_\mathrm{ad}(F) \to \mathfrak{D}(Z_G, G; F) \to \mathfrak{D}(T, G; F). \]
\end{proposition}
\begin{proof}
	Unravel the definition of $G_\text{ad}(F) \to H^1(F, Z_G)$.
\end{proof}

Now specialize to $G := \Sp(W)$. The following result is also standard, see for instance \cite[\S 3]{Li11}.
\begin{proposition}\label{prop:D-description}
	In terms of the parameterization of \ref{prop:reg-ss-parameters}, two regular semisimple conjugacy classes $\mathcal{O}(K, K^\sharp, x, c)$ and $\mathcal{O}(L, L^\sharp, y, d)$ in $G = \Sp(W)$ are stably conjugate if and only if there exists an isomorphism of $F$-algebras $K \rightiso L$ that
	\begin{compactitem}
		\item respects the involutions and
		\item maps $x$ to $y$.
	\end{compactitem}
	In other words, passing to stable conjugacy amounts to discarding the datum $c$. Choose $\delta \in \mathcal{O}(K, K^\sharp, x, c)$ and set $T := G_\delta$, decompose $K^\sharp = \prod_{i \in I} K_i^\sharp$ and put $I_0 := \left\{ i \in I: K_i \;\text{is a field} \right\}$. Then there are canonical isomorphisms of groups
	\[ \mathfrak{D}(T, G; F) = H^1(F, T) \rightiso K^{\sharp, \times}/N_{K/K^\sharp}(K^\times) \rightiso \{\pm 1\}^{I_0}. \]
	Finally, forgetting the datum $x$ gives a parameterization of stable conjugacy classes of embeddings of maximal $F$-tori.
\end{proposition}
\begin{proof}
	The algebra $K = F[x]$ depends only on the conjugacy class placed in $\GL(W)$, whilst the involution is determined by $x \mapsto x^{-1}$; the invariance of $(K, K^\sharp, x)$ follows. Conversely, given data $(K, K^\sharp, x, c)$ and $(K, K^\sharp, x, c')$ parameterizing conjugacy classes $\mathcal{O}$ and $\mathcal{O}'$, we define the corresponding symplectic forms $h,h': K \to F$. If there exists $a \in K^\sharp$ such that $c' = cN_{K/K^\sharp}(a)$, then $t \mapsto at$ will define an isomorphism $(K, h) \rightiso (K, h')$ between symplectic $F$-vector spaces. The existence of such an $a$ can be guaranteed upon base-change to a finite extension $E/F$, therefore $\mathcal{O}$, $\mathcal{O}'$ become conjugate in $\Sp(W \otimes_F E)$.
	
	As $H^1(F, G)$ is trivial, $\mathfrak{D}(T, G; F) = H^1(F, T)$ is an abelian group. Define the $K_i^\sharp$-torus $K_i^1$ (Notation \ref{nota:norm}) and use Notation \ref{nota:Weil-restriction} for Weil restriction. Shapiro's lemma \cite[Exp XXIV. Proposition 8.4]{SGA3-3} (which also works for non-separable extensions, see the Remarque 8.5 therein) implies that
	\[ H^1(F,T) \simeq H^1(F, K^1) = \prod_{i \in I} H^1 \left( F, R_{K_i^\sharp/F}(K_i^1) \right) = \prod_{i \in I} H^1(K_i^\sharp, K_i^1). \]
	It remains to show $H^1(F, K^1) = F^\times / N_{K/F}(K^\times)$ when $K$ is a $2$-dimensional étale $F$-algebra. If $K$ is a field, use the short exact sequence
	\[ 1 \to K^1 \to R_{K/F}(\Gmm{K}) \xrightarrow{N_{K/F}} \Gmm{F} \to 1 \]
	of $F$-tori to see $F^\times/N_{K/F}(K^\times) \rightiso H^1(F, K^1)$; the case $K = F \times F$ is trivial. Local class field gives the description of $F^\times/N_{K/F}(K^\times) \simeq \{\pm 1\}$.

	For the final assertion, see the proof of Proposition \ref{prop:parameter-tori}. 
\end{proof}

\begin{definition}\label{def:kappa-minus}\index{$\kappa_-$}
	Let $T \subset G$ be parameterized by $(K, K^\sharp, c)$. Denote by $\kappa_- = \kappa_-^T$ the homomorphism $H^1(F,T) \to \bmu_2$ corresponding to
	\[ \prod_{i \in I_0} K_i^{\sharp, \times} \big/ N_{K_i/K_i^\sharp}(K_i^\times) \ni (t_i)_{i \in I_0} \longmapsto \prod_{i \in I_0} \sgn_{K_i/K_i^\sharp}(t_i). \]
\end{definition}
Pairings of the form $\lrangle{ \kappa_-, \text{inv}(\delta, \delta')}$ will play an important role, where $\delta' \stackrel{\text{st}}{\sim} \delta \in T_\text{reg}(F)$. One can show that $\kappa_-$ is canonically defined, cf. the interpretation via long roots in \S\ref{sec:stable-reduction}, but this will not be needed.

\begin{lemma}\label{prop:parameter-ani}
	Conserve the notation of Proposition \ref{prop:D-description}. Then $T$ is anisotropic if and only if $I_0 = I$.
\end{lemma}
\begin{proof}
	For any finite extension $L/F$ and $L$-torus $S$, the adjunction
	\[ \Hom_{F-\text{grp}}(\Gmm{F}, R_{L/F}(S)) \simeq \Hom_{L-\text{grp}}(\Gmm{L}, S) \]
	implies that $R_{L/F}(S)$ is anisotropic over $F$ if and only if $S$ is over $L$. Our assertion thus reduces to the case $\dim_F W = 2$, which is easy.
\end{proof}

Turn to the rank-one case $\dim_F W = 2$. By choosing a symplectic basis $\{e_1, e_{-1}\}$, we may identify $G$ (resp. $G_\text{ad}$) with $\SL(2)$ (resp. $\PGL(2)$). The $\PGL(2,F)$-action on $\SL(2,F)$ is realizable in $\GL(2,F)$. From $\mathfrak{D}(\mu_2, \SL(2); F) = H^1(F, \mu_2) \leftiso F^\times/F^{\times 2}$, we obtain a canonical homomorphism
\begin{equation}\label{eqn:nu-arises}
	\nu: G_\text{ad}(F)/\Image\left[ G(F) \to G_\text{ad}(F) \right] \simeq F^{\times 2}/F^{\times 2}.
\end{equation}
Let $x \in \SL(2, F)$ be semisimple regular with parameter $(K, F, \ldots)$ for its conjugacy class, so that $T := Z_{\SL(2)}(x) \simeq K^1$. The following has been recorded in \cite[p.728]{LL79}.

\begin{proposition}\label{prop:st-conj-SL2}
	The stable conjugacy class of $\delta \in \SL(2,F)$ consists of elements of the form
	\[ g\delta g^{-1}, \quad g \in \PGL(2,F). \]
	Choose any $g_1 \in \GL(2,F)$ that maps to $g$, then $\mathrm{inv}(\delta, g\delta g^{-1})$ equals
	\begin{gather*}
		\text{image of $g$ under }\; \PGL(2,F) \to \mathfrak{D}(\mu_2, \SL(2); F) \to \mathfrak{D}(T, \SL(2); F) \\
		= \text{image of $\det(g_1)$ or $\nu(g)$ under }\; F^\times/F^{\times 2} \twoheadrightarrow F^\times/N_{K/F}(K^\times) \simeq \mathfrak{D}(T, \SL(2); F).
	\end{gather*}
	Consequently, stable conjugacy in $\SL(2,F)$ is realizable by $\PGL(2,F)$-action.
\end{proposition}
\begin{proof}
	The first equality for $\text{inv}(\delta, g\delta g^{-1})$ is Proposition \ref{prop:adjoint-rel-pos}. The second one amounts to the commutativity of
	\[\begin{tikzcd}
		\PGL(2,F) \arrow{r} \arrow{rd}[swap]{\nu} & \mathfrak{D}(\mu_2, \SL(2); F) \arrow{r} \arrow{d}{\simeq} & \mathfrak{D}(T, \SL(2); F) \arrow{d}{\simeq} \\
		\GL(2,F) \arrow{u} \arrow{r}[swap]{\det} & F^\times/F^{\times 2} \arrow[twoheadrightarrow]{r} & F^\times/N_{K/F}(K^\times)
	\end{tikzcd}\]
	The right square commutes because so does the following diagram.
	\[\begin{tikzcd}
		\mu_2 \arrow{r} \arrow[hookrightarrow]{d} & \Gm \arrow{r}{2} \arrow[hookrightarrow]{d} & \Gm \arrow[-, double equal sign distance]{d} \\
		K^1 \arrow{r} & R_{K/L} \Gmm{K} \arrow{r}[swap]{N_{L/K}} & \Gm.
	\end{tikzcd}\]
	The triangle \begin{tikzpicture}[scale=0.3] \draw (0,0) -- (1,0) -- (1,-1) --cycle; \end{tikzpicture} commutes by the definition of $\nu$. As for the \begin{tikzpicture}[scale=0.3] \draw (0,0) -- (1,-1) -- (0,-1) --cycle; \end{tikzpicture}, note that given $g_1 \mapsto g$, every $\tau \in \Gamma_F$ multiplies $(\det g_1)^{-1/2} g_1 \in \SL(2, \bar{F})$ by a sign $c(\tau)$, and $\tau \mapsto c(\tau)$ represents the image of $g$ in $H^1(F, \mu_2)$. But the same cocycle represents the image of $\det(g_1)$ in $F^\times/F^{\times 2} \rightiso H^1(F, \mu_2)$.
\end{proof}

\subsection{Stable conjugacy: reduction to \texorpdfstring{$\SL(2)$}{SL2}}\label{sec:stable-reduction}
Fix a maximal $F$-torus $T$ in $G := \Sp(W)$. The lattices $X := X^*\left( T_{\bar{F}} \right)$ and $Y := X_*\left( T_{\bar{F}} \right)$ are endowed with $\Gamma_F$-actions. We choose an isomorphism $T_{\bar{F}} \simeq \Gm^n$ as in \S\ref{sec:Sp}, and enumerate the long roots (resp. short coroots) as $\pm 2\epsilon_1, \ldots, \pm 2\epsilon_n$ (resp. $\pm \check{\epsilon}_1, \ldots, \pm \check{\epsilon}_n$). Then $\Gamma_F$ acts on both sets and commutes with the bijection between roots and coroots. If $\mathcal{O}$ is a $\Gamma_F$-orbit, so is $-\mathcal{O}$. Now recall some definitions from \cite[\S 2]{LS1}.

\begin{definition}\label{def:symmetric-orbits}
	Let $\mathcal{O}$ be a $\Gamma_F$-orbit of roots.
	\begin{compactenum}[(i)]
		\item If $\mathcal{O}=-\mathcal{O}$, we say $\mathcal{O}$ is \emph{symmetric}.\index{symmetric root}
		\item If $\mathcal{O} \cap (-\mathcal{O}) = \emptyset$, we say $\mathcal{O}$ is \emph{asymmetric}.
	\end{compactenum}
	The same terminology pertains to $\Gamma_F$-orbits of coroots.
\end{definition}
For any root $\alpha$, set
\begin{equation}\label{eqn:F_alpha}\begin{gathered}
	\Gamma_\alpha := \Stab_{\Gamma_F}(\alpha) \subset \Stab_{\Gamma_F}(\{\pm\alpha\}) =: \Gamma_{\pm\alpha}, \\
	F_\alpha \supset F_{\pm\alpha} \supset F: \quad \text{their fixed fields in}\; \bar{F}.
\end{gathered}\end{equation}\index{F_\alpha@$F_{\alpha}, F_{\pm\alpha}$}
These extensions are all separable since the splitting field of $T$ is.

\begin{lemma}\label{prop:anisotropic-criterion}
	The torus $T$ is anisotropic if and only if every orbit $\mathcal{O}$ of long roots (resp. short coroots) is symmetric.
\end{lemma}
\begin{proof}
	Observe that for any orbit $\mathcal{O}$ of long roots, $\sum_{\alpha \in \mathcal{O}} \alpha \in X^{\Gamma_F}$ equals $0$ if and only if $\mathcal{O}$ is symmetric. Indeed, when $\mathcal{O} = -\mathcal{O}$ we surely have $\sum_{\alpha \in \mathcal{O}} \alpha = 0$. When $\mathcal{O} \cap (-\mathcal{O}) = \emptyset$, we may write $\mathcal{O} = \left\{ (-1)^{s_1} 2\epsilon_{a_1}, \ldots, (-1)^{s_k} 2\epsilon_{a_k} \right\}$ for $k := |\mathcal{O}|$ and uniquely determined $1 \leq a_1 < \ldots < a_k \leq n$, $s_1, \ldots, s_k \in \{0, 1\}$. Then $\mathcal{O}$ is linearly independent in $X \otimes \Q$.

	The existence of asymmetric $\mathcal{O}$ implies $X^{\Gamma_F} \neq 0$, thus $T$ is isotropic. Conversely, assume every $\mathcal{O}$ is symmetric and consider $v = \sum_{i=1}^n a_i\epsilon_i \in X^{\Gamma_F} \smallsetminus \{0\}$. For each $i$, there exists $\sigma \in \Gamma_F$ such that $\sigma(\pm\epsilon_i) = \mp\epsilon_i$ since $2\epsilon_i$ belongs to a symmetric orbit; from $\sigma(v)=v$ we deduce $a_i=0$. This implies $v=0$ so $T$ is anisotropic.
\end{proof}

Hereafter we only consider the orbits of long roots or short coroots.

Select a $\Gamma_F$-orbit $\mathcal{O}$ in $X$ together with a long $\alpha \in \mathcal{O}$. After base-change to $F_{\pm\alpha}$, we see that $\{\alpha,-\alpha\}$ generates a copy of $\SL(2)$ in $G_{F_{\pm\alpha}}$. This $\SL(2)$ contains the subtorus $T_{\pm\alpha}$ of $T_{\bar{F}}$ with $X_*(T_{\pm\alpha, \bar{F}}) = \Z\check{\alpha} \subset Y$; as a shorthand, we say $T_{\pm\alpha}$ is generated by $\alpha$. Hence $T_{\pm\alpha}$, $\SL(2)$ are both defined over $F_{\pm\alpha}$.

\begin{enumerate}
	\item First assume $\mathcal{O}$ is symmetric so that $(\Gamma_{\pm\alpha} : \Gamma_\alpha) = [F_\alpha : F_{\pm\alpha}] = 2$. In this case $T_{\pm\alpha}$ is anisotropic and it splits over $F_\alpha$, and
	\[ T_{\pm\alpha} \simeq F_\alpha^1 := \Ker \left[N_{F_\alpha/F_{\pm\alpha}}: F_\alpha^\times \to F_{\pm\alpha}^\times \right]. \]
	\item Next, assume $\mathcal{O}$ is asymmetric. Then $\Gamma_{\pm\alpha}=\Gamma_\alpha$, $F_\alpha = F_{\pm\alpha}$, and $T_{\pm\alpha}$ is a split.
\end{enumerate}
Identify $\mathcal{O}$ with $\Hom_F(F_\alpha, \bar{F}) = \Gamma_F/\Gamma_{F_\alpha}$. Let $T_{\mathcal{O}}$ be the subtorus of $T_{\bar{F}}$ generated by $\{ T_\alpha : \alpha \in \mathcal{O}\}$, which is now defined over $F$. By the generalities on reductive groups and their Weil restrictions (cf. \eqref{eqn:tensor-split}), we see that there is an embedding (see \S\ref{sec:Sp-parameters} for discussions on $\SL(2)$):
\begin{equation}\label{eqn:O-embedding-1} \begin{tikzcd}[baseline]
	-1 \in \SL(2, F_{\pm\alpha}) & R_{F_{\pm\alpha}/F}(\SL(2)) \arrow[hookrightarrow]{r} & G \\
	& R_{F_{\pm\alpha}/F}(T_{\pm\alpha}) \arrow[hookrightarrow]{u}{\text{max. torus}} & \\
	\quoted{-1} := \prod_{\pm \beta \in \mathcal{O}} \check{\beta}(-1) \arrow[mapsto]{uu} & T_{\mathcal{O}} \arrow[hookrightarrow]{r} \arrow{u}{\simeq} & T \arrow[hookrightarrow]{uu}[swap]{\text{max. torus}} \\
\end{tikzcd}\end{equation}
Denote the image of $R_{F_{\pm\alpha}/F}(\SL(2))$ as $G_{\mathcal{O}}$. It is a canonically defined subgroup of $G$ relative to $\mathcal{O}$. Observe that $(G_\mathcal{O}, T_\mathcal{O}) = (G_{-\mathcal{O}}, T_{-\mathcal{O}})$. Furthermore,
\[ \prod_{\pm\mathcal{O}} T_{\mathcal{O}} \rightiso T. \]

This construction can be described in terms of the parameter $(K, K^\sharp, c)$ of $T \hookrightarrow G$ supplied by Proposition \ref{prop:parameter-tori}. We shall use the familiar decomposition $K = \prod_{i \in I} K_i$.
\begin{itemize}
	\item When $T$ is split, we may assume $I = \{1, \ldots, n\}$, $K_i^\sharp = F$, and $K_i = F \times F$ for all $i$. Identify $T$ with $\left\{  (x_i, y_i)_{i=1}^n \in (F^\times \times F^\times)^n : x_i y_i = 1 \right\}$. Now $X = \bigoplus_{i=1}^n \Z\epsilon_i$ where $\epsilon_i$ (resp. $-\epsilon_i$) corresponds to $(x_i,y_i) \mapsto x_i$ (resp. $(x_i, y_i) \mapsto y_i$). The $\Gamma_F$-orbits of long roots are singletons $\{\pm 2\epsilon_i\}$; they are all asymmetric.
	\item The general case is obtained by a twist as follows. The set $I$ is in bijection with the sets $\{\mathcal{O}, -\mathcal{O}\}$ of $\Gamma_F$-orbits. We have $i \in I_0$ (i.e. $K_i$ is a field) if and only if $\pm\mathcal{O}$ are symmetric. By choosing an homomorphism $K_i \to \bar{F}$ of $F$-algebras, we pick out $\alpha \in \mathcal{O}$ and
	\begin{equation}\label{eqn:K_i-alpha} \begin{gathered}
		K_i^\sharp \simeq F_{\pm\alpha}; \qquad
		K_i \simeq \begin{cases} F_\alpha, & i \in I_0 \\ F_{\pm\alpha} \times F_{\pm\alpha}, & i \notin I_0. \end{cases}
	\end{gathered}\end{equation}
	Decompose $c = (c_i)_{i \in I}$ and define a symplectic form $h^i$ (resp. $h_i$) on the $K_i^\sharp$-vector space (resp. $F$-vector space) $K_i$ à la \eqref{eqn:K-symp-form} as
	\begin{equation}\label{eqn:h-h} \begin{aligned}
		h^i(u|v) & := \Tr_{K_i/K_i^\sharp} \left(\tau(u)v c_i \right), \\
		h_i(u|v) & := \Tr_{K_i/F} \left( \tau(u)v c_i \right) = \Tr_{K_i^\sharp/F} \left( h^i(u|v) \right).
	\end{aligned}\end{equation}
	Under this correspondence, \eqref{eqn:O-embedding-1} coincides with
	\begin{equation}\label{eqn:O-embedding-2} \begin{tikzcd}
		-1 \in \SL(2,K_i^\sharp) & G_{\mathcal{O}} = R_{K_i^\sharp/F}( \overbracket{\Sp(K_i, h^i)}^{= \SL(2)} ) \arrow[hookrightarrow]{r} & \Sp(K_i, h_i) \arrow[hookrightarrow]{r} & \Sp(W) = G \\
		-1 \in K_i^1 \arrow[mapsto]{u} & T_{\mathcal{O}} = K_i^1 \arrow[hookrightarrow]{rr} \arrow[hookrightarrow]{u} & & K^1 = T \arrow[hookrightarrow]{u}
	\end{tikzcd}\end{equation}
	\item The isomorphism $(W, \lrangle{\cdot|\cdot}) \simeq \bigoplus_{i \in I} (K_i, h_i)$ of symplectic $F$-vector spaces gives rise to
	\begin{gather*}
		\prod_{\pm \mathcal{O}} G_{\mathcal{O}} \simeq \prod_{i \in I} R_{K_i^\sharp/F}( \overbracket{\Sp(K_i, h^i)}^{=\SL(2)} ) \subset \prod_{i \in I} \Sp(K_i, h_i) \hookrightarrow \Sp(W) = G, \\
		\prod_{\pm \mathcal{O}} T_{\mathcal{O}} \simeq \prod_{i \in I} R_{K_i^\sharp/F}(K_i^1) = K^1 \simeq T. 
	\end{gather*}
\end{itemize}

\begin{definition}\label{def:G-T}
	Given $T$ as above, we set $G^T := \prod_{\pm\mathcal{O}} G_{\mathcal{O}}$. It is a canonically defined subgroup of $G$ containing $T$, and can be described as $\prod_{i \in I} R_{K_i^\sharp/F}(\SL(2))$ in terms of the parameter $(K, K^\sharp, c)$ for $T \hookrightarrow G$. \index{GT@$G^T$}
\end{definition}

By Shapiro's lemma, $H^1(F, G^T) = \prod_{i \in I} H^1(K_i^\sharp, \SL(2))$ is trivial. By the discussions in \S\ref{sec:Weil-restriction} on Weil restrictions, we see
\begin{compactitem}
	\item $G^T$ is simply connected since $\SL(2)$ is;
	\item $G^T_\text{ad} = \prod_{\pm \mathcal{O}} R_{F_{\pm\alpha}/F}(\PGL(2))$ since Weil restriction commutes with the formation of adjoint groups, as noted in \S\ref{sec:Weil-restriction};
	\item if $\Ad(g): T \rightiso T'$ is a stable conjugation inside $G^T$, then we have $G^T = G^{T'}$. In fact one can take $g \in G^T_\text{ad}(F)$ by Proposition \ref{prop:st-conj-SL2}.
\end{compactitem}

\begin{proposition}\label{prop:stable-reduction-SL2}
	Conserve the notations above and suppose $\delta \in T_\mathrm{reg}(F)$. Then every conjugacy class in the stable conjugacy class of $\delta$ contains an element of the form $g\delta g^{-1}$ where $g \in G^T_{\mathrm{ad}}(F)$, and
	\[ \mathrm{inv}(\delta, g\delta g^{-1}) = \text{image of $g$ under }\; G^T_{\mathrm{ad}}(F) \to \mathfrak{D}(Z_{G^T}, G^T; F) \to \mathfrak{D}(T, G; F). \]
	Furthermore, $\mathfrak{D}(Z_{G^T}, G^T; F) \to \mathfrak{D}(T, G; F)$ is bijective and in terms of parameters, $\mathrm{inv}(x, g\delta g^{-1})$ can be identified with
	\[ \left(\det(g_{i,1}) N_{K_i/K_i^\sharp}(K_i^\sharp) \right)_{i \in I} \; \in K^{\sharp,\times}/N_{K/K^\sharp}(K^\times). \]
	Here $g_{i,1} \in \GL(2, K_i^\sharp)$ is any representative of $g_i \in \PGL(2, K_i^\sharp)$.
\end{proposition}
Note that the indexes $i \notin I_0$ have no contribution, and $I=I_0$ when $T$ is anisotropic by Lemma \ref{prop:parameter-ani}.
\begin{proof}
	The first equality for $\text{inv}(\delta, g\delta g^{-1})$ reduces essentially to Proposition \ref{prop:adjoint-rel-pos}. The next step is to describe the image of $g$ in terms of $(g_{i,1})_{i \in I}$. The map $\mathfrak{D}(Z_{G^T}, G^T; F) \to \mathfrak{D}(T, G; F)$ factors into $\mathfrak{D}(Z_{G^T}, G^T; F) \to \mathfrak{D}(T, G^T; F) \to \mathfrak{D}(T, G; F)$, whilst $\mathfrak{D}(T, G^T; F) = H^1(F, T) = \mathfrak{D}(T, G; F)$. Proposition \ref{prop:D-description} asserts that
	\[ H^1(F,T) \simeq \prod_{i \in I} \dfrac{K_i^{\sharp,\times}}{N_{K_i/K_i^\sharp}(K_i^\times)} = \dfrac{K^{\sharp, \times}}{N_{K/K^\sharp}(K^\times)}. \]
	This is compatible with the decomposition $T = \prod_{i \in I} R_{K^\sharp_i/F}(K_i^1)$, so we may work separately for each $i \in I$. Shapiro's lemma affords the commutative diagram
	\[\begin{tikzcd}
		(R_{K_i^\sharp/F} \PGL(2))(F) \arrow[-, double equal sign distance]{d} \arrow{r} & H^1(F, R_{K_i^\sharp/F}(K^1)) \arrow{d}{\simeq} & \\
		\PGL(2, K_i^\sharp) \arrow{r} & H^1(K_i^\sharp, K^1) \arrow{r}{\simeq} & K_i^\sharp/N_{K_i/K_i^\sharp}(K^\times).
	\end{tikzcd}\]
	We are now reduced to the rank-one case of Proposition \ref{prop:st-conj-SL2} and the description of $\text{inv}(\delta, g\delta g^{-1})$ follows. Finally, this description implies the surjectivity of $G^T_\text{ad}(F) \to \mathfrak{D}(T,G;F)$, therefore every $\eta \stackrel{\text{st}}{\sim} \delta$ is conjugate to some $g\delta g^{-1}$ with $g \in G^T_\text{ad}(F)$.
\end{proof}

Summarizing, we obtain all stable conjugates of elements in $T_\text{reg} \subset G$ (up to ordinary conjugacy) by working inside $G^T$. Modulo Weil restrictions, the description of stable conjugacy boils down to the $\SL(2)$ case. There is also an obvious version for $\mathfrak{g}$.

\section{BD-covers of symplectic groups}\label{sec:BD-covers-Sp}
Except in \S\ref{sec:nr-global}, the assumptions in \S\ref{sec:Sp-gen} on the field $F$ remain in force. For the study of harmonic analysis, we can and do confine ourselves to the BD-covers of $\Sp(2n)$ arising from Matsumoto's central extension; see the Remarks \ref{rem:Matsumoto} and \ref{rem:rescaling-Q}.

\subsection{The covers}\label{sec:BD-Sp}
Let $(W, \lrangle{\cdot | \cdot})$ be a symplectic $F$-vector space of dimension $2n$. Fix a maximal $F$-torus $T$ of $G := \Sp(W)$ and set $Y = X_*(T_{\bar{F}})$. Write $G := \Sp(W)$ and denote by $E_G \to G$ the multiplicative $\shK_2$-torsor constructed by Matsumoto (Remark \ref{rem:Matsumoto}). It corresponds to the quadratic form $Q: Y \to \Z$ in \eqref{eqn:Y-Sp}. We will also write $E_{G,F}$ when the base field is to be stressed. When $\dim W = 2$ with chosen symplectic basis, we adopt the shorthand $E_G = E_{\SL(2)}$.


To $Q$ is associated the symmetric bilinear form on $Y$
\[ B_Q(y, y') := Q(y+y')-Q(y)-Q(y') = 2 \sum_{i=1}^n y_i y'_i. \]
Suppose $m \mid N_F$ as in \S\ref{sec:isogeny}. Since $Y_{Q,m} = \dfrac{m}{\text{gcd}(2,m)} Y \subset Y$, the isogeny $\iota_{Q,m}: T_{Q,m} \to T$, together with its compatibility under $G$-conjugation, can be identified with the endomorphism of $T \subset G$
\begin{equation}\label{eqn:isogeny-Sp} \begin{aligned}
	\iota_{Q,m}: T_{Q,m} = T & \longrightarrow T \\
	t & \longmapsto t^{m/\text{gcd}(2,m)}
\end{aligned}\end{equation}
\textbf{Caveat}: one must be careful when identifying $T_{Q,m}$ and $T$, as they will play different roles in our latter applications.

\begin{lemma}\label{prop:restriction-W_i}
	Suppose that $(W, \lrangle{\cdot|\cdot})$ is the orthogonal direct sum $\bigoplus_{i=1}^r W_i$ of symplectic vector subspaces. Write $G_i := \Sp(W_i) \hookrightarrow G$. There is a natural morphism $\iota_i: E_{G_i} \rightiso E_G|_{G_i}$ for all $1 \leq i \leq r$. They realize the pull-back of $E_G \to G$ to $G_1 \times \cdots \times G_r$ as the contracted product of multiplicative $\shK_2$-torsors $E_{G_1} \utimes{\shK_2} \cdots \utimes{\shK_2} E_{G_r}$.
	
	In particular, elements lying over different components $G_i$ commute in $E_G$.
\end{lemma}
\begin{proof}
	Choose symplectic bases of $W_1, \ldots, W_r$; their union is a symplectic basis of $W$. The corresponding split maximal $F$-torus of $G$ is $T = T_1 \times \cdots \times T_r$; in parallel $(Y, Q) = (Y_1, Q_1) \oplus \cdots \oplus (Y_r, Q_r)$ where $Q_i$ is the quadratic form associated to Matsumoto's $E_{G_i}$, by \eqref{eqn:Y-Sp}.

	By Remark \ref{rem:shared-torus}, $E_G|_{G_1 \times \cdots \times G_r}$ is also classified by the quadratic form $(Y, Q)$. On the other hand, $\bigoplus_{i=1}^r (Y_i, Q_i)$ corresponds to $E_{G_1} \utimes{\shK_2} \cdots \utimes{\shK_2} E_{G_r}$. This gives the required isomorphism. The required $\iota_i$ comes from composing with $E_{G_i} \hookrightarrow E_{G_1} \utimes{\shK_2} \cdots \utimes{\shK_2} E_{G_r}$.
\end{proof}

Now consider the constructions in \S\ref{sec:stable-reduction}. We have a maximal $F$-torus $T \subset G$. Form the canonical subgroup $T \subset G^T \subset G$ of Definition \ref{def:G-T}. By choosing a parameter $(K, K^\sharp, \ldots)$ for $T \hookrightarrow G$ (Proposition \ref{prop:parameter-tori}), we may identify $G^T$ with $\prod_{i \in I} R_{K^\sharp_i/F} \SL(2)$. To reconcile with the notations in \S\ref{sec:Weil-restriction}, denote by $f_i: \Spec(K_i^\sharp) \to \Spec(F)$ the structure morphisms for each $i \in I$.

\begin{theorem}\label{prop:G-T-reduction-K_2}
	The restriction of $E_G$ to $G^T$ is isomorphic to the contracted product of the multiplicative $\shK_2$-torsors $f_{i,*} \left( E_{\SL(2), K_i^\sharp} \right)$.\index{contracted product}
\end{theorem}
\begin{proof}
	Using the notation from \S\ref{sec:stable-reduction}, we restrict $E_G$ in stages
	\[ G^T = \prod_{i \in I} R_{K^\sharp_i/F} \Sp(K_i, h^i) \hookrightarrow \prod_{i \in I} \Sp(K_i, h_i) \hookrightarrow \Sp(W) = G. \]
	With $G_i := \Sp(K_i, h_i)$, Lemma \ref{prop:restriction-W_i} reduces the problem to the case $|I|=1$ and $K^\sharp$ is a field. Write $h = h^i$ and $f = f_i$. Since $R_{K^\sharp/F} \Sp(K, h)$ is simply connected and contains $T$, Remark \ref{rem:shared-torus} implies that $E_G|_{G^T}$ is classified by the $Q: Y \to \Z$ in \eqref{eqn:Y-Sp}.

	Now turn to the quadratic form $Q': Y \to \Z$ associated to $f_* E_{\Sp(K, h)}$. Choose a coroot $\check{\alpha}$ of $K^1 \subset \Sp(K,h)$, which is defined over $F_\alpha$ and recall $K^\sharp = F_{\pm\alpha}$ in the notation of \S\ref{sec:stable-reduction}. Then $Y_{\Sp(K,h)} = \Z\check{\alpha}$ and the $Q_{\Sp(K,h)}$ associated to $E_{\Sp(K,h)}$ is simply $y\check{\alpha} \mapsto y^2$ by \eqref{eqn:Y-Sp}. If $K$ is a field then $K=F_\alpha$ and $\Gal{F_\alpha/F_{\pm\alpha}}$ acts on $Y_{\Sp(K,h)}$ by $\check{\alpha} \mapsto \pm\check{\alpha}$, thus stabilizes $Q_{\Sp(K,h)}$. If $K \simeq K^\sharp \times K^\sharp$ then $T$ splits over $F_{\pm\alpha}$. By the discussions in \S\ref{sec:stable-reduction}, the $\{\iota(\check{\alpha}) : \iota \in \Gamma_F \}$ is precisely the set of short coroots in $Y$ respect to $T \subset G$ (over $\bar{F}$), therefore
	\begin{equation}\label{eqn:SL2-Y-restricted}
		Y = \bigoplus_{\iota \in \Gamma_F/\Gamma_{\pm\alpha}} \Z\iota(\check{\alpha}).
	\end{equation}
	This coincides with the description $\bigoplus_{\iota \in \Hom_F(K^\sharp, \bar{F})} Y_\iota$ of $f_* Y_{\Sp(K,h)}$ in \S\ref{sec:Weil-restriction}. Theorem \ref{prop:BD-restriction} gives the $Q'$ associated to $f_* E_{\Sp(K, h)}$: it is orthogonal direct sum of the forms
	\[ \Z\iota(\check{\alpha}) \to \Z, \quad y\iota(\check{\alpha}) \mapsto y^2, \quad (\iota \in \Gamma_F/\Gamma_{\pm\alpha}). \]
	By \eqref{eqn:SL2-Y-restricted} together with \eqref{eqn:Y-Sp}, we see $Q' = Q$. Therefore $f_* E_{\Sp(K,h)} \simeq E_G|_{G^T}$ in $\cate{CExt}(G^T, \shK_2)$.
\end{proof}

Let $m \mid N_F$. By the construction of \S\ref{sec:local-BD}, to $(E_G, m)$ is attached the topological central extension of locally compact groups
\begin{equation}\label{eqn:G-BD-cext}
	1 \to \mu_m \to \tilde{G} \xrightarrow{\bm{p}} G(F) \to 1.
\end{equation}
When $G = \SL(2,F)$, denote by $\widetilde{\SL}(2, F)$ the topological central extension of $\SL(2,F)$ by $\mu_m$ so obtained.

\begin{theorem}\label{prop:G-T-reduction}
	The restriction of $\tilde{G}$ to $G^T(F)$ is isomorphic to the contracted product of the topological central extensions $\widetilde{\SL}(2, K_i^\sharp)$ by $\mu_m$.
\end{theorem}
When $K_i^\sharp = \CC$, we set $\widetilde{\SL}(2, K_i^\sharp) := \SL(2,\CC) \times \mu_m$, so that the assertion always makes sense.
\begin{proof}
	Consider the push-out to $\mu_m$ of the contracted product of $\left( f_{i,*} E_{\SL(2), K_i^\sharp} \right)(F)$; it is canonically isomorphic to the contracted product of the push-outs of $\left( f_{i,*} E_{\SL(2), K_i^\sharp} \right)(F)$ to $\mu_m$. The latter push-outs are isomorphic to $\widetilde{\SL}(2, K_i^\sharp)$ as topological central extensions of $\SL(2, K_i^\sharp)$ by $\mu_m$, by Proposition \ref{prop:restriction-commutes}. We conclude by applying Theorem \ref{prop:G-T-reduction-K_2}.
\end{proof}

\subsection{Kubota's cover of \texorpdfstring{$\GL(2)$}{GL2}}\label{sec:Kubota}
We review Kubota's description \cite{Ku69} of a covering of $\GL(2,F)$, cf. \cite[0.1]{KP84} and \cite[\S 16.2]{GG}. It is a multiplicative $\shK_2$-torsor $E_\text{Ku} \to \GL(2)$ such that
\begin{compactitem}
	\item $E_\text{Ku}$ restricts to Matsumoto's $E_{\SL(2)} \to \SL(2)$, see \cite[0.1]{KP84};
	\item by \cite[p.41]{KP84}, there is a preferred section $\bm{s}: \GL(2,F) \to E_\text{Ku}(F)$ with $\bm{s}(x)\bm{s}(y) = \bm{c}(x,y) \bm{s}(xy)$ in terms of an explicit $2$-cocycle $\bm{c}$.
\end{compactitem} \index{xbold@$\bm{x}(\cdot), \bm{s}, \bm{c}$}
Using the notations of \S\ref{sec:local-BD}, we describe $\bm{c}$ by
\begin{equation}\label{eqn:Kubota-cocycle}
	\begin{gathered}
		\bm{x}\twobigmatrix{a}{b}{c}{d} := \begin{cases}
		c, & c \neq 0 \\
		d, & c = 0,
	\end{cases} \\
	\bm{c}(g_1, g_2) := - \left\{ \dfrac{ \bm{x}(g_1) }{ \bm{x}(g_1 g_2) }, \; \dfrac{ \det g_1 \cdot \bm{x}(g_2) }{ \bm{x}(g_1 g_2) } \right\}_F \; \in K_2(F).
\end{gathered}\end{equation}

\begin{remark}\label{rem:Kubota-minus}
	We follow \cite[Corollaire 5.12]{Mat69} to take the negative of the usual Kubota's cocycle found in \cite{Ku69, Fl80, KP84}, otherwise $E_{\text{Ku}}|_{\SL(2)}$ would be the negative of Matsumoto's central extension. For the relation between $\bm{s}$ and Steinberg's presentation for $E_{\SL(2)}(F)$, we refer to the discussions preceding \cite[Corollaire 5.12]{Mat69}. After pushing-out to $\mu_2$, the difference disappears.
\end{remark}

\begin{lemma}\label{prop:PGL-action}
	The adjoint action of $\GL(2)$ on $E_{\mathrm{Ku}}$ induces the canonical $\PGL(2)$-action on $E_{\SL(2)}$ given by Proposition \ref{prop:BD-adjoint-action}.
\end{lemma}
\begin{proof}
	The $\GL(2)$-action leaves $\SL(2)$ and $E_{\text{Ku}}|_{\SL(2)} = E_{\SL(2)}$ invariant. The center of $\GL(2)$ acts trivially on $\SL(2)$, thereby giving rise to an automorphism of $E_{\SL(2)} \to \SL(2)$; this action must be trivial as $\SL(2)$ is simply connected. We conclude by the uniqueness part of Proposition \ref{prop:BD-adjoint-action}.
\end{proof}

\begin{lemma}\label{prop:-1-adjoint}
	Let $g \in \PGL(2,F)$ with preimage $g_1 \in \GL(2,F)$. For any preimage $\widetilde{-1} \in E_{\mathrm{Ku}}(F)$ of $-1$, we have $\Ad(g)(\widetilde{-1}) = g_1 (\widetilde{-1}) g_1^{-1} = \xi (\widetilde{-1})$ where
	\begin{align*}
		\xi & = -\left\{ -1, \det g_1 \right\}_F \\
		& = -\left\{ \det g_1, \det g_1 \right\}_F \; \in K_2(F).
	\end{align*}
\end{lemma}
\begin{proof}
	By Lemma \ref{prop:PGL-action} we have $\Ad(g)(\widetilde{-1}) = g_1(\widetilde{-1})g_1^{-1}$. As $-1$ is central in $\GL(2,F)$, we have $\xi = \bm{c}(g_1, -1) - \bm{c}(-1, g_1)$. Write $g_1 = \twomatrix{a}{b}{c}{d}$. First suppose $c \neq 0$, then $\bm{c}(g_1, -1) = -\left\{ -1, (\det g_1) c^{-1} \right\}_F$ whereas $\bm{c}(-1, g_1) = -\left\{ c^{-1}, -1 \right\}_F = -\left\{ -1, c^{-1} \right\}_F$ by the anti-symmetry of Steinberg symbols. Hence $\bm{c}(g_1, -1) - \bm{c}(-1, g_1) = -\left\{ -1, \det g_1 \right\}_F$.
	
	If $c = 0$, replacing $c$ by $d$ in the arguments above gives the same result.
\end{proof}

By fixing $m \mid N_F$ and pushing $E_\text{Ku}(F)$ out via $(\cdot, \cdot)_{F,m}: K_2(F) \to \mu_m$, we obtain a topological central extension
\[ 1 \to \mu_m \to \widetilde{\GL}(2,F) \to \GL(2,F) \to 1. \]
The resulting preferred section and cocycle are still denoted by $\mathbf{s}$ and $\bm{c}$, now with $\{\cdot, \cdot \}_F$ replaced by $(\cdot, \cdot)_{F,m}$. Hence
\begin{compactitem}
	\item $\widetilde{\GL}(2,F)$ restricts to $\widetilde{\SL}(2,F) \to \SL(2,F)$;
	\item by Lemma \ref{prop:PGL-action}, the adjoint action of $\GL(2,F)$ on $\widetilde{\SL}(2,F)$ induces the canonical $\PGL(2,F)$-action on $\widetilde{\SL}(2,F)$ from Proposition \ref{prop:BD-adjoint-action};
	\item the statements in Lemma \ref{prop:-1-adjoint} hold for $\widetilde{\SL}(2,F)$, with $\{\cdot,\cdot\}_F$ replaced by $( \cdot, \cdot)_{F,m}$.
\end{compactitem}

Let $K$ be an étale $F$-algebra of dimension $2$, therefore comes equipped with an involution $\tau \neq \identity$. The $F$-torus $K^\times$ embeds into $\GL(2)$ with $\det|_{K^\times} = N_{K/F}$; it restricts to $K^1 \hookrightarrow \SL(2)$. As the $n=1$ case of Proposition \ref{prop:D-description}, this parameterize stable conjugacy classes of embeddings of maximal tori in $\SL(2)$.

\begin{notation}
	When $K \simeq F \times F$, the elements are expressed as $x = (x_1, x_2)$ and we have $\tau(x_1, x_2) = (x_2, x_1)$. The Hilbert symbols for such $K$ can be conveniently defined as
	\[ (x,y)_{F,m} := (x_1, y_1)_{F,m} (x_2, y_2)_{F,m}. \]
\end{notation}

The result below quantifies the non-commutativity of the preimage of $K^\times$ in $\widetilde{\GL}(2,F)$.
\begin{proposition}[Flicker]\label{prop:Flicker-comm}
	Suppose that $\gamma, g \in \GL(2,F)$ arise from $x, u \in K^\times$. Let $\tilde{\gamma}$ be any preimage of $\gamma$ in $\widetilde{\GL}(2,F)$. The factor $[g,\gamma] \in \mu_m$ in \eqref{eqn:commutator} determined by $g\tilde{\gamma}g^{-1} = [g, \gamma] \tilde{\gamma}$ has the form
	\[ [g,\gamma] = (x, \tau(u))_{K,m}^{-1}. \]
\end{proposition}
\begin{proof}
	This is done in the calculations in \cite[p.128]{Fl80} using the cocycle $-\bm{c}$.
\end{proof}

In the situation above, we define
\begin{align*}
	\iota_{Q,m}: K^1 & \longrightarrow K^1 \\
	x_0 & \longmapsto x_0^{m/\mathrm{gcd}(2,m)}
\end{align*}
following \eqref{eqn:isogeny-Sp}. Given $x = \iota_{Q,m}(x_0)$, there exists $\omega \in K^\times$ with $\omega/\tau(\omega) = x_0$ by Hilbert's theorem 90. Then $N_{K/F}(\omega)$ is well-defined modulo $F^{\times 2}$. When $K \simeq F \times F$ and $x_0 = (a, a^{-1})$ with $a \in F^\times$, we may take $\omega = (a,1)$ to see that $N_{K/F}(\omega)$ represents the class of $a$ inside $F^\times/F^{\times 2}$.

We will need the \emph{projection formula} for the next proof
\begin{equation}\label{eqn:Hilbert-projection-formula}
	(a, b)_{K,m} = (a, N_{K/F}(b))_{F,m}, \quad a \in F^\times, \; b \in K^\times;
\end{equation}
it is standard when $K$ is a field, and the case $K \simeq F \times F$ is straightforward.

\begin{lemma}\label{prop:commutator-GL2-T}
	Suppose $x = \iota_{Q,m}(x_0) \in K^1$ and choose $\omega$ be as above. Then in the setting of Proposition \ref{prop:Flicker-comm} we have
	\begin{align*}
		[g,\gamma] & = \left( \omega, N_{K/F}(u) \right)_{K, \mathrm{gcd}(2,m)} = \left( \omega, \det\gamma \right)_{K, \mathrm{gcd}(2,m)} \\
		& = \left( N_{K/F}(\omega), \det\gamma \right)_{F, \mathrm{gcd}(2,m)}.
	\end{align*}
\end{lemma}
In particular, in this case $[g,\gamma]=1$ whenever $m \notin 2\Z$. It also follows that $\left( N_{K/F}(\omega), \det\gamma \right)_{F, \mathrm{gcd}(2,m)}$ is independent of the choice of $\omega$, although this can also be verified directly.
\begin{proof}
	From \eqref{eqn:norm-residue-d} and the Proposition \ref{prop:Flicker-comm} we infer
	\begin{align*}
		[g, \gamma] & = \left( x_0^{m/\text{gcd}(2,m)}, \tau(u) \right)_{K,m}^{-1} = \left( x_0, \tau(u) \right)_{K, \text{gcd}(2,m)} \\
		& = \left( \omega, \tau(u)\right)_{K, \text{gcd}(2,m)} \left( \tau(\omega)^{-1}, \tau(u) \right)_{K, \text{gcd}(2,m)} \\
		& = \left( \omega, \tau(u)\right)_{K, \text{gcd}(2,m)} \left( \omega, u \right)_{K, \text{gcd}(2,m)} \\
		& = \left( \omega, N_{K/F}(u) \right)_{K, \text{gcd}(2,m)} = \left( \omega, \det\gamma \right)_{K, \text{gcd}(2,m)}.
	\end{align*}
	Now apply \eqref{eqn:Hilbert-projection-formula} to obtain the remaining equality in the assertion.
\end{proof}

We will need further invariance properties for this factor.
\begin{definition-proposition}\label{def:Cali-factor}\index{Cm@$\Cali_m(\nu, \gamma_0)$}
	Take $x_0 = \omega/\tau(\omega) \in K^1$ as before, and let $\nu \in F^\times/F^{\times 2}$. Let $\gamma_0 \in \SL(2,F)$ be associated to $x_0$ via $K^1 \hookrightarrow \SL(2)$ and put
	\[ \Cali_m(\nu, \gamma_0) := \left( N_{K/F}(\omega), \nu \right)_{F, \mathrm{gcd}(2,m)}. \]
	Then $\Cali_m(\nu, \gamma_0)$ depends only on the stable conjugacy class of the element $\gamma_0 \in \SL(2,F)$ associated to $x$. Furthermore, $\Cali_m(\nu, \gamma_0)$ is invariant under any automorphism of the $F$-group $\SL(2)$.
\end{definition-proposition}
\begin{proof}
	Given a stable conjugacy class $\gamma_0$, the datum $(K, x_0)$ is determined up to isomorphisms of étale $F$-algebras, thus $N_{K/F}(\omega)$ is determined up to $F^{\times 2}$. Hence $\Cali_m(\nu, \gamma_0)$ is invariant under $\PGL(2,F)$, and $\PGL(2)$ equals the $F$-scheme of automorphisms of $\SL(2)$ by \cite[Exp XXIV, Théorème 1.3]{SGA3-3}.
\end{proof}

Observe that $\Cali_m(\nu, \gamma_0)$ is bi-multiplicative in $\nu$ and $\gamma_0$. In applications, $\nu$ will arise from the $\nu(g)$ in \eqref{eqn:nu-arises}.

We record the classification by K.\ Hiraga and T.\ Ikeda of good regular semisimple elements in $\widetilde{\SL}(2,F)$. Write $T \subset \SL(2)$ for the maximal $F$-torus parameterized by a two-dimensional étale $F$-algebra $K$.

\begin{theorem}[Hiraga--Ikeda]\label{prop:good-SL2} \index{good element}
	The projection of $Z_{\tilde{T}}$ to $T(F)$ equals $\{\pm 1\} \cdot T(F)^{m/\mathrm{gcd}(2,m)}$. In particular, $\gamma \in T_{\mathrm{reg}}(F)$ is good if and only if $\gamma \in \{\pm 1\} \cdot T(F)^{m/\mathrm{gcd}(2,m)}$.
\end{theorem}
\begin{proof}
	Their original proof is reproduced below. We begin by showing that the preimage of $\{\pm 1\} \cdot T(F)^{m/\mathrm{gcd}(2,m)}$ is central. In view of Proposition \ref{prop:good-in-tori} and \eqref{eqn:isogeny-Sp}, it suffices to show $[\pm 1, \eta] = 1$ for all $\eta \in \SL(2,F)$: this is already true on the $K_2(F)$-level by Proposition \ref{prop:BD-adjoint-action}.
	
	Now suppose $\gamma \in T(F)$ is projected from $Z_{\tilde{T}}$. Set $m_0 := m/\mathrm{gcd}(2,m)$. When $T$ is split, by Proposition \ref{prop:good-in-tori} we have $\gamma \in T(F)^{m_0}$. Thus we can assume $T$ is associated with a quadratic extension of fields $K$ of $F$, and $\gamma$ corresponds to $x \in K^1$. Proposition \ref{prop:Flicker-comm} implies that for all $v \in K^\times$,
	\begin{align*}
		1 & = \left( x, \tau(v)/v \right)_{K,m} \\
		& = ( x, \tau(v) )_{K,m} ( x, v^{-1})_{K,m} = ( \tau(x), v )_{K,m} ( x, v^{-1} )_{K,m} \\
		& = ( x^{-1}, v )_{K,m} ( x, v^{-1} )_{K,m} = \left( x^2, v^{-1} \right)_{K, m}.
	\end{align*}
	This shows $\pm x \in K^1 \cap K^{\times m_0}$ and it remains to show $K^1 \cap K^{\times m_0} = \pm (K^1)^{m_0}$.
	
	We have $-1 \in K^{\times m_0}$, because $(-1)^{m_0} = -1$ when $4 \nmid m$ whilst $\zeta^{m_0} = -1$ when $4 \mid m$ and $\mu_m = \lrangle{\zeta}$. Therefore $K^1 \cap K^{\times m_0} \supset \pm (K^1)^{m_0}$. Next, suppose $y_0 \in K^\times$ satisfies $N_{K/F}(y_0^{m_0}) = 1$, then there exists $\xi \in \mu_m$, $\xi^{m_0} = \pm 1$ such that $N_{K/F}(y_0) = \xi^2$. This leads to $y_1 := \xi^{-1} y_0 \in K^1$ satisfying $(y_1)^{m_0} = \xi^{-m_0} y_0^{m_0} = \pm y_0^{m_0}$; we conclude that $K^1 \cap K^{\times m_0} \subset \pm (K^1)^{m_0}$.
	
	The last assertion results from Corollary \ref{prop:isogeny-good}.
\end{proof}

Below is a supplement to the classification above.
\begin{proposition}\label{prop:iota-kernel}
	Let $T \simeq K^1$ as above, then
	\begin{align*}
		T(F)^{m/\mathrm{gcd}(2,m)} = (-1) \cdot T(F)^{m/\mathrm{gcd}(2,m)}, & \quad  4 \nmid m \quad \text{or}\quad T: \text{split}, \\
		T(F)^{m/\mathrm{gcd}(2,m)} \cap (-1) \cdot T(F)^{m/\mathrm{gcd}(2,m)} = \emptyset, & \quad 4 \mid m \quad\text{and}\quad T: \text{anisotropic}.
	\end{align*}
	On the other hand,
	\begin{equation}
		\Ker(\iota_{Q,m}) = \begin{cases}
			\mu_{m/\mathrm{gcd}(2,m)}, & T: \text{split} \\
			1, & 4 \nmid m, \quad T: \text{anisotropic} \\
			\{\pm 1 \}, & 4 \mid m, \quad T: \text{anisotropic};
	\end{cases}\end{equation}
	in the split case, $K \simeq F \times F$ and we embed $\mu_m$ into $K^1$ via $z \mapsto (z, z^{-1})$.
\end{proposition}
\begin{proof}
	For the first part, note that $(\pm 1) \cdot T(F)^{m/\text{gcd}(2,m)}$ are either identical or disjoint. They are identical if and only if $-1 = \xi^{m/\text{gcd}(2,m)}$ for some $\xi \in K^1$. When $4 \nmid m$, we take $\xi = -1$. When $T$ splits (so $K \simeq F \times F$) and $4 \mid m$, we take $\xi = (\zeta, \zeta^{-1})$ where $\mu_m = \lrangle{\zeta}$; note that $\zeta^{m/2} = -1$. In both cases, these choices of $\xi$ show $T(F)^{m/\mathrm{gcd}(2,m)} = (-1) \cdot T(F)^{m/\mathrm{gcd}(2,m)}$.

	When $4 \mid m$ and $T$ is anisotropic, we know $K$ is a field; by $\xi^m = 1$ and $m \mid N_F$ we obtain $\xi \in F^\times$. Then $N_{K/F}(\xi)=1$ forces $\xi = \pm 1$, but $(-1)^{m/2} = 1$. This shows $T(F)^{m/\mathrm{gcd}(2,m)} \cap (-1) \cdot T(F)^{m/\mathrm{gcd}(2,m)} = \emptyset$.

	For the second part, the case $K = F \times F$ is straightforward. When $K$ is a field, $\iota_{Q,m}(x_0) = 1 \implies x_0 \in F^\times \cap K^1$ since $m \mid N_F$, hence $x_0 = \pm 1$ as before; it remains to observe that $(-1)^{m/\text{gcd}(2,m)} = 1$ if and only if $4 \mid m$.
\end{proof}

\subsection{Stable conjugacy for BD-covers}\label{sec:st-conj-BD}
Revert to the notation of \S\ref{sec:Sp} and \S\ref{sec:BD-Sp}. Fix a maximal $F$-torus $T$. The subgroup $T \subset G^T \subset G$ (Definition \ref{def:G-T}) together with the decomposition $G^T = \prod_{\mathcal{O}/\pm} G_{\mathcal{O}}$ are canonically defined by $T$. Once a parameter $(K, K^\sharp, \ldots)$ for $T$ is chosen, $G^T$ together with its decomposition can be identified with $\prod_{i \in I} R_{K_i^\sharp/F}(\SL(2))$, and $T$ is identified with $\prod_{i \in I} R_{K_i^\sharp/F}(K_i^1)$.

Given $\delta \in T_\text{reg}(F)$ and $\eta \in G_\text{reg}(F)$, recall that $\mathcal{T}(\delta, \eta) := \{g: g\delta g^{-1} = \eta \}$ is nonempty if and only if $\delta \stackrel{\text{st}}{\sim} \eta$, in which case it is a right $T$-torsor. As a first step, we assume $\eta \in G^T(F)$ so that Proposition \ref{prop:stable-reduction-SL2} entails $\mathcal{T}(\delta, \eta) \subset G^T$. In parallel with the decomposition of $G^T$, we have $\mathcal{T}(\delta, \eta) = \prod_{i \in I} \mathcal{T}_i(\delta_i, \eta_i)$, where each $\mathcal{T}_i(\cdots)$ is defined inside $R_{K_i^\sharp/F} (\SL(2))$.

Now suppose $\eta \in G_\text{reg}(F)$ is stably conjugate to $\delta \in T_\text{reg}(F)$. Denote $S := G_\eta$. By Proposition \ref{prop:stable-reduction-SL2}, the canonical isomorphism of pointed tori $\Ad(g): (T, \delta) \rightiso (S, \eta)$ induced by any $g \in \mathcal{T}(\delta, \eta)(\bar{F})$ can be decomposed as
\[\begin{tikzcd}
	(T, \delta) \arrow{r}{\Ad(g')} \arrow[bend right]{rr}[swap]{\Ad(g)} & (T', \delta') \arrow{r}{\Ad(g'')} & (S, \eta)
\end{tikzcd}\]
for some $\delta' \in G^T_\text{reg}(F)$, $g' \in \mathcal{T}^{G^T}(\delta, \delta')(\bar{F})$ and $g'' \in \mathcal{T}(\delta', \eta)(F)$. In particular $\Ad(g')(\delta') = \eta$ is ordinary conjugacy and
\[ \text{inv}(\delta, \eta) = \text{inv}(\delta, \delta'). \]
This equality also determines $\delta'$ up to $G^T(F)$-conjugacy. The goal of this subsection is to adapt these to good elements in $\tilde{G}_\text{reg}$ for the BD-cover $
\tilde{G} \twoheadrightarrow G(F)$.

Fix $m \mid N_F$. Denote by $\widetilde{G^T}$ the pull-back of $\mu_m \hookrightarrow \tilde{G} \twoheadrightarrow G(F)$ to $G^T(F)$. Theorem \ref{prop:G-T-reduction} says that as topological central extensions $\widetilde{G^T}$ is isomorphic to the contracted product of $\mu_m \hookrightarrow \widetilde{\SL}(2, K_i^\sharp) \twoheadrightarrow \SL(2, K_i^\sharp)$. Here it is convenient to identify $\iota_{Q,m}$ with the endomorphism $t_0 \mapsto t_0^{m/\text{gcd}(2,m)}$ of $T$ by \eqref{eqn:isogeny-Sp}; this decomposes into $\iota_{Q_i, m}: K_i^1 \to K_i^1$ for each $i \in I$.

\begin{proposition}\label{prop:good-T}
	The projection of $Z_{\tilde{T}}$ to $T(F)$ equals
	\[ \prod_{i \in I} \{\pm 1 \} \cdot \Image\left( \iota_{Q_i,m} \right). \]
	In particular, its intersection with $T_\mathrm{reg}(F)$ equals the set of good elements in $T_\mathrm{reg}(F)$.
\end{proposition}
\begin{proof}
	By the decomposition of $\widetilde{G^T}$, we reduce immediately to the case $\tilde{G} = \widetilde{G^T} = \widetilde{\SL}(2, F)$ treated in Theorem \ref{prop:good-SL2}, upon passing to a finite extension of $F$.
\end{proof}

\begin{remark}\label{rem:good-pm1}
	Every element $\eta = (\gamma_i)_{i \in I} \in \prod_{i \in I} \{\pm 1\}$ is good with respect to $\tilde{G} \twoheadrightarrow G(F)$. Indeed, Lemma \ref{prop:restriction-W_i} reduces the problem to the case $\eta \in \{\pm 1\}$, and one concludes by Proposition \ref{prop:BD-adjoint-action}.
\end{remark}

\begin{notation}\index{TQm-tilde@$\tilde{T}_{Q,m}, \tilde{T}_{Q,m}^\sigma$}
	Recall that $I$ can be identified with $\{ \text{long roots} \} \big/ \lrangle{\Gamma_F, \pm}$, thus canonical for the given $T$. The embedding $\{\pm 1\}^I \hookrightarrow T(F)$ is also definable in terms of long roots by \eqref{eqn:O-embedding-1}, hence canonical and can be transported under stable conjugacy. Let $\{\pm 1\}^I$ act on $T$ by coordinate-wise multiplication. Consider $\sigma = (\sigma_i)_{i \in I} \in \{\pm 1\}^I$. Pull-back of $\bm{p}: \tilde{T} \twoheadrightarrow T(F)$ along $\sigma \cdot \iota_{Q,m}$ yields
	\begin{equation}\label{eqn:iota-cover}\begin{aligned}
		\tilde{T}^\sigma_{Q,m} & := \left\{ (\tilde{t}, t_0) \in \tilde{T} \times T_{Q,m}(F) : \bm{p}(\tilde{t}) = \sigma \cdot \iota_{Q,m}(t_0) \right\}, \\
		\tilde{T}_{Q,m} & := \tilde{T}^{(+, \ldots, +)}_{Q,m}.
	\end{aligned}\end{equation}
	There are natural maps $\tilde{T} \xleftarrow{\text{pr}_1} \tilde{T}^\sigma_{Q,m} \xrightarrow{\text{pr}_2} T_{Q,m}(F)$. Proposition \ref{prop:good-T} implies that every good regular element in $\tilde{G}_\text{reg}$ lies in some $\text{pr}_1\left(\tilde{T}^\sigma_{Q,m}\right)$; when $4 \nmid m$, it suffices to use $\sigma = (+, \ldots, +)$.
\end{notation}

Note that $\tilde{T}_{Q,m}$ is a group, whilst $\tilde{T}^\sigma_{Q,m}$ is only a $\tilde{T}_{Q,m}$-torsor for general $\sigma$.

Consider a stable conjugation $\Ad(g): T \rightiso S$ between maximal $F$-tori in $G$, then
\begin{itemize}
	\item as explicated above, $\{\pm 1\}^I \hookrightarrow T(F)$ can be transported to $S$ by $\Ad(g)$, thus $\tilde{S}^\sigma_{Q,m}$ makes sense;
	\item under the identification, $\sigma \cdot \Ad(g)(t) = \Ad(g)(\sigma \cdot t)$ for all $t \in T(F)$ and $\sigma \in \{\pm 1\}^I$; 
	\item if moreover $g \in G(F)$, from \eqref{eqn:isogeny-Ad} we have
	\begin{align*}
		\Ad(g): \tilde{T}^\sigma_{Q,m} & \longrightiso \tilde{S}^\sigma_{Q,m} \\
		(\tilde{t}, t_0) & \longmapsto (\Ad(g)(\tilde{t}), \Ad(g)(t_0))
	\end{align*}
	which is a group isomorphism for $\sigma = (+, \ldots, +)$, and is an equivariant map between torsors for general $\sigma$. Proposition \ref{prop:good-T} implies that $\Ad(g) = \identity$ when $g \in T(F)$.
\end{itemize}

Now comes the stable conjugacy in the $G \simeq \SL(2)$ case; the notation above reduces to $\tilde{T}^\pm_{Q,m}$. We will employ systematically the $G_\text{ad}(F)$-action on $\tilde{G}$ from Proposition \ref{prop:BD-adjoint-action}, again denoted by $\Ad$.

\begin{definition-proposition}\label{def:st-conj-SL2}\index{CAd@$\CaliAd$}
	Suppose that $\dim_F W = \dim_F K = 2$, so that $K^\sharp = F$ and $G^T = G$. Let $\Ad(g): T \rightiso S$ be a stable conjugation of maximal $F$-tori; here we may assume $g \in G_\mathrm{ad}(F)$ by Proposition \ref{prop:st-conj-SL2}. Set $\nu := \nu(g) \in F^\times/F^{\times 2}$ by \eqref{eqn:nu-arises}.
	\begin{enumerate}
		\item Define an isomorphism $\CaliAd(g) = \CaliAd^+(g): \tilde{T}^+_{Q,m} \to \tilde{S}^+_{Q,m}$ that fits into the commutative diagram
			\begin{equation*}\begin{gathered} \begin{tikzcd}[row sep=small]
				T(F) \arrow{r}{\Ad(g)} & S(F) \\
				\tilde{T}^+_{Q,m} \arrow{r}{\CaliAd^+(g)} \arrow{u} \arrow{d} & \tilde{S}^+_{Q,m} \arrow{u} \arrow{d} \\
				T_{Q,m}(F) \arrow{r}[swap]{\Ad(g)} \arrow[bend left=50]{uu}{\iota_{Q,m}} & S_{Q,m}(F) \arrow[bend right=50]{uu}[swap]{\iota_{Q,m}}
			\end{tikzcd} \\
				\CaliAd^+(g)(\tilde{t}, t_0) = \left( \Cali_m(\nu, t_0) \cdot \Ad(g)(\tilde{t}), \quad \Ad(g)(t_0) \right).
			\end{gathered}\end{equation*}
		\item When $4 \mid m$, define $\CaliAd^-(g): \tilde{T}^-_{Q,m} \to \tilde{S}^-_{Q,m}$ that fits into the commutative diagram
			\begin{equation*}\begin{gathered} \begin{tikzcd}[row sep=small]
				T(F) \arrow{r}{\Ad(g)} & S(F) \\
				\tilde{T}^-_{Q,m} \arrow{r}{\CaliAd^-(g)} \arrow{u} \arrow{d} & \tilde{S}_{Q,m} \arrow{u} \arrow{d} \\
				T_{Q,m}(F) \arrow{r}[swap]{\Ad(g)} \arrow[bend left=50]{uu}{(-1)\iota_{Q,m}} & S_{Q,m}(F) \arrow[bend right=50]{uu}[swap]{(-1)\iota_{Q,m}}
			\end{tikzcd} \\
				\CaliAd^-(g)\left( \widetilde{-1} \cdot \tilde{t}', t_0 \right) = \left( \Cali_m(\nu, t_0) \cdot \widetilde{-1} \cdot \Ad(g)(\tilde{t}'), \quad \Ad(g)(t_0) \right)
			\end{gathered}\end{equation*}
			for any $\widetilde{-1} \mapsto -1$ and $\tilde{t}' \mapsto \iota_{Q,m}(t_0)$.
	\end{enumerate}
	These constructions are independent of all choices. In particular it is independent of the identification $G \simeq \SL(2)$.
\end{definition-proposition}
\begin{proof}
	We may choose a preimage $g_1 \in \GL(2,F)$ of $g \in \PGL(2,F)$ by identifying $G$ and $\SL(2)$, and then apply the constructions in \S\ref{sec:Kubota}. By Definition--Proposition \ref{def:Cali-factor}, the factor $\Cali_m(\nu, t_0)$ depends only on $g \bmod \Image[G(F) \to G_\text{ad}(F)]$ and on the stable conjugacy class of $t_0$; it is invariant under any re-parameterization or automorphisms of $\SL(2)$.
\end{proof}

The following properties will also hold for stable conjugacy in general. We begin with the $\SL(2)$ case above.
\begin{proposition}\label{prop:CAd-prop}
	For any sign $\sigma$ that is allowed in our situation, the maps $\CaliAd^\sigma(g)$ satisfy the following properties.
	\begin{enumerate}[\bfseries{AD}.1.\;]
		\item $\CaliAd^\sigma(g)(\noyau\tilde{t}) = \noyau \CaliAd^\sigma(g)(\tilde{t})$ whenever $\noyau \in \mu_m$.
		\item $\CaliAd(g): \tilde{T}_{Q,m} \rightiso \tilde{S}_{Q,m}$ is an isomorphism of topological groups. In general, $\CaliAd^\sigma(g)$ is a $\CaliAd(g)$-equivariant map between torsors with respect to $\tilde{T}_{Q,m} \xrightarrow{\CaliAd(g)} \tilde{S}_{Q,m}$.
		\item If $g \in G(F)$, then $\CaliAd^\sigma(g)$ reduces to ordinary conjugation. If $g \in T(\bar{F})$, it equals $\identity$.
		\item Given stable conjugations of maximal $F$-tori $T \xrightarrow{\Ad(h)} S \xrightarrow{\Ad(g)} R$, we have
		\[ \CaliAd^\sigma(g) \circ \CaliAd^\sigma(h) = \CaliAd^\sigma(gh): \tilde{T}^\sigma_{Q,m} \to \tilde{R}^\sigma_{Q,m}. \]
	\end{enumerate}
\end{proposition}
\begin{proof}
	The setting under consideration is $G=\SL(2)$, $g,h \in G_\text{ad}(F)$.
	\begin{asparaenum}[\bfseries{AD}.1.\;]
	\item It follows from the analogous property of the $G_\text{ad}(F)$-action on $\tilde{G}$.
 
	\item Since $\Cali_m(\nu, t_0 t'_0) = \Cali_m(\nu, t_0) \Cali_m(\nu, t'_0)$ and $\Cali_m(\nu, \cdot)$ is locally constant, $\CaliAd(g)$ is a continuous homomorphism. It will follow from \textbf{AD.4} that $\CaliAd(g) \CaliAd(g^{-1}) = \CaliAd(1) = \identity$.
	
	For $(\tilde{t}, t_0) \in \tilde{T}^-_{Q,m}$, \textbf{AD.1} implies that $\CaliAd^-(g)(\tilde{t}, t_0)$ is independent of how we decompose $\tilde{t} = \widetilde{-1} \cdot \tilde{t}'$. Given $(\tilde{s}, s_0) \in \tilde{T}_{Q,m}$, since $\widetilde{-1}$ is central by Proposition \ref{prop:BD-adjoint-action}, $(\tilde{s}, s_0)(\tilde{t}, t_0)$ is mapped to
	\begin{multline*}
		\CaliAd^-(g)\left( \tilde{s} (\widetilde{-1}) \tilde{t}', s_0 t_0 \right) = \CaliAd^-(g)\left( (\widetilde{-1}) \tilde{s}\tilde{t}', s_0 t_0  \right) \\
		= \left( \Cali_m(\nu, t_0) \Cali_m(\nu, s_0) \cdot (\widetilde{-1}) \Ad(g)(\tilde{s}) \Ad(g)(\tilde{t}'), \quad \Ad(g)(s_0) \cdot \Ad(g)(t_0) \right) \\
		= \left( \Cali_m(\nu, s_0) \Ad(g)(\tilde{s}), \; \Ad(g)(s_0)\right) \left( \Cali_m(\nu, t_0) (\widetilde{-1}) \Ad(g)(\tilde{t}'), \; \Ad(g)(t_0) \right) \\
		= \CaliAd(g)(\tilde{s}, s_0) \CaliAd^-(g)(\tilde{t}, t_0);
	\end{multline*}
	the equivariance in the case $\sigma = -$ follows.

	\item If $g$ comes from $G(F)$, then $\nu \in F^{\times 2}$ so $\Cali_m(\nu, \cdot) = 1$; also note that $\Ad(g)(\widetilde{-1}) = \widetilde{-1}$ by Proposition \ref{prop:BD-adjoint-action}. This shows the first assertion.
	
	As for the second assertion, the premise implies that every $g \in G(\bar{F})$ realizing the given $T \rightiso S$ belongs to $T(\bar{F})$; in particular, we may assume that $g \in (T/Z_G)(F)$. The case $\sigma=+$ follows from Lemma \ref{prop:commutator-GL2-T} which says that $\Ad(g)(\tilde{t}) = \Cali_m(\nu, t_0)^{-1} \tilde{t}$. For the same reason, in the case $\sigma=-$ we have $\Ad(g)(\tilde{t}') = \Cali_m(\nu, t_0)^{-1} \tilde{t}'$ and the result follows.

	\item This equality follows from the fact that $\Cali_m(\nu, \gamma_0)$ is multiplicative in $\nu$ and depends on $\gamma_0$ only through its stable class.
	\end{asparaenum}
\end{proof}

Proceed to define stable conjugacy in higher rank inside $\widetilde{G^T}$. The allowable signs in the constructions below will be taken from \index{Sgn@$\mathrm{Sgn}_m(T)$}
\begin{equation}\label{eqn:Sgn} \text{Sgn}_m(T) := \begin{cases}
	\{1\}^I, & 4 \nmid m \\
	\{\pm 1\}^I, & 4 \mid m.
\end{cases}\end{equation}
Write $T = \prod_{i \in I} T_i$ with $T_i = R_{K_i^\sharp/F}(K_i^1)$. Multiplication induces the epimorphism
\begin{equation}\label{eqn:T-Q-m-contracted}
	\prod_{i \in I} (\widetilde{T_i})_{Q,m}^{\sigma_i} \longrightarrow \tilde{T}^\sigma_{Q,m}, \quad \text{kernel} = \left\{ (\noyau_i)_i \in \mu_m^I: \prod_i \noyau_i = 1 \right\}.
\end{equation}

\begin{remark}
	By Proposition \ref{prop:BD-adjoint-action}, $G^T_\text{ad}(F)$ acts on $\widetilde{G^T}$; when restricted to the component corresponding to $i \in I$, this is the same as the $\PGL(2, K_i^\sharp)$-action on $\widetilde{\SL}(2, K_i^\sharp)$ (working over $K_i^\sharp$). One way to see this is to invoke Proposition \ref{prop:lifting-uniqueness}.
\end{remark}

\begin{definition-proposition}\label{def:st-conj-G-T}
	Let $S \subset G^T$ be a maximal $F$-torus stably conjugate to $T$ via $\Ad(g): T \rightiso S$, where $g = (g_i)_{i \in I} \in (G^T)_\mathrm{ad}(F)$. For $\sigma \in \mathrm{Sgn}_m(T)$, define
	\begin{equation*}\begin{gathered} \begin{tikzcd}[row sep=small]
		T(F) \arrow{r}{\Ad(g)} & S(F) \\
		\tilde{T}^\sigma_{Q,m} \arrow{r}{\CaliAd^\sigma(g)} \arrow{u} \arrow{d} & \tilde{S}^\sigma_{Q,m} \arrow{u} \arrow{d} \\
		T_{Q,m}(F) \arrow{r}[swap]{\Ad(g)} \arrow[bend left=50]{uu}{\sigma \cdot \iota_{Q,m}} & S_{Q,m}(F) \arrow[bend right=50]{uu}[swap]{\sigma \cdot \iota_{Q,m}}
	\end{tikzcd} \\
		\CaliAd(g)^\sigma(g)\left( \tilde{t}, t_0 \right) = \prod_{i \in I} \CaliAd^{\sigma_i}(g_i)\left( \tilde{t}_i, t_{0,i} \right)
	\end{gathered}\end{equation*}
	for $((\tilde{t}_i)_i, (t_{0,i})_i) \mapsto (\tilde{t}, t_0)$ under \eqref{eqn:T-Q-m-contracted}. We shall abbreviate $\CaliAd(g) := \CaliAd^{ (+, \ldots, +) }(g)$.

	These constructions are independent of all choices and satisfy the properties in Proposition \ref{prop:CAd-prop}, with $\sigma \in \mathrm{Sgn}_m(T)$.
\end{definition-proposition}
\begin{proof}
	This is just a multi-component version of Definition--Proposition \ref{def:st-conj-SL2} modulo Weil restrictions. It does not depend on the identification $G_{\mathcal{O}} \simeq R_{K_i^\sharp/F}(\SL(2))$. Indeed, the field $K_i^\sharp \simeq F_{\pm\alpha}$ is uniquely determined by $(G,T)$, and all $F$-automorphisms of $R_{K_i^\sharp/F}(\SL(2))$ arise from $K_i^\sharp$-automorphisms of $\SL(2)$ by \cite[Proposition A.5.14]{CGP15}; this does not alter $\CaliAd^{\sigma_i}$ by Definition--Proposition \ref{def:st-conj-SL2}.
\end{proof}

We are ready to state the general recipe.
\begin{definition}\label{def:st-conj}\index{stable conjugacy}
	Let $\Ad(g): T \rightiso S$ be a stable conjugacy of maximal $F$-tori in $G$. Decompose $\Ad(g)$ into
	\[\begin{tikzcd}
		T \arrow{r}{\Ad(g')} & T' \arrow{r}{\Ad(g'')} & S
	\end{tikzcd} \quad g' \in G^T_\mathrm{ad}(F), \; g'' \in G(F). \]
	In particular $T' = \Ad(g')T \subset G^T$. Call such a $(g',g'')$ a \emph{factorization pair}\index{factorization pair} for $\Ad(g)$. Given $\sigma \in \mathrm{Sgn}_m(T)$, define the map
	\begin{gather*}
		\CaliAd^\sigma(g) := \Ad(g'') \circ \CaliAd^\sigma(g'): \tilde{T}^\sigma_{Q,m} \longrightarrow \tilde{S}^\sigma_{Q,m}.
	\end{gather*}
\end{definition}
Recall that for all $\delta \in T_\text{reg}(F)$, we have $\text{inv}(\delta, g\delta g^{-1}) = \text{inv}(\delta, g' \delta g'^{-1})$ in $H^1(F,T)$.

\begin{theorem}
	The map $\CaliAd^\sigma(g)$ in Definition \ref{def:st-conj} is independent of all choices and satisfy the properties in Proposition \ref{prop:CAd-prop}, with $\sigma \in \mathrm{Sgn}_m(T)$. In particular, it depends on the $\Ad(g): T \rightiso S$ but not on the choice of $g$.
\end{theorem}
\begin{proof}
	First we show that $\CaliAd^\sigma(g)$ is independent of the factorization pair. Choose $\delta \in T_\text{reg}(F)$. Let $\delta' := g'\delta g'^{-1} \in G^T(F)$, $T' := g' T g'^{-1} \subset G^T = G^{T'}$, and let $(h', h'')$ be another factorization pair. Then $\text{inv}(\delta, \delta') = \text{inv}(\delta, h' \delta h'^{-1})$. Setting $k := h' g'^{-1} \in G^T_\text{ad}(F)$, the formalism of \eqref{eqn:st-conj-in-stages} yields
	\begin{align*}
		\text{inv}(\delta', k \delta' k^{-1}) & = \text{inv}(\delta', h'\delta h'^{-1}) = \text{inv}(\delta', \delta) + \text{inv}(\delta, h'\delta h'^{-1}) \\
		& = -\text{inv}(\delta, \delta') + \text{inv}(\delta, h'\delta h'^{-1}) = 0.
	\end{align*}
	Hence $\mathcal{T}^{G^T}(\delta', k\delta' k^{-1})$ has an $F$-point $r \in G^T(F)$. Since $r$ also yields an $F$-point of the quotient $\mathcal{T}^{G^T_\text{ad}}(\delta', k\delta' k^{-1})$ by $Z_{G^T}$ which contains $k$, the torsor structure entails $k \in r \cdot (T'/Z_{G^T})(F)$. Property \textbf{AD.3---4} for $G^T$ entail $\CaliAd^\sigma(k) = \Ad(r)$, and
	\begin{gather*}
		\CaliAd^\sigma(h') = \CaliAd^\sigma(k) \CaliAd^\sigma(g') = \Ad(r) \CaliAd^\sigma(g').
	\end{gather*}
	Also, as isomorphisms $T' \rightiso S$ we have
	\[ \Ad(h'')\Ad(r) = \Ad(h'') \Ad(k) = \Ad\left( h'' h' g'^{-1} \right) = \Ad(g''), \quad h'', r, g'' \in G(F); \]
	that is, $h'' r = s g''$ for some $s \in S(F)$. Since $\tilde{S}^\sigma_{Q,m} \to \tilde{S}$ has central image, on $\tilde{S}^\sigma_{Q,m}$ acts $\Ad(s)$ as $\identity$ hence we arrive at
	\begin{align*}
		\Ad(h'') \CaliAd^\sigma(h') &= \Ad(h'') \Ad(r) \CaliAd^\sigma(g') = \Ad(h'' r) \CaliAd^\sigma(g') \\
		& = \Ad(s) \Ad(g'') \CaliAd^\sigma(g') = \Ad(g'') \CaliAd^\sigma(g').
	\end{align*}
	The independence of $\CaliAd^\sigma(g)$ on parameters, identifications with $\SL(2)$, etc. result immediately.

	The properties \textbf{AD.1---2} for $G$ in Proposition \ref{prop:CAd-prop} are inherited from $G^T$. For \textbf{AD.3}, if $g$ comes from $G(F)$ (resp. from $T(\bar{F})$), the factorization pair for $\Ad(g)$ may be taken as $(1, g)$ (resp. $(g,1)$ by adjusting $g$ as in the proof of Proposition \ref{prop:CAd-prop}); the case of $(1,g)$ is ordinary conjugation, whereas the case of $(g,1)$ is handled by \textbf{AD.3} for $G^T$ (Definition--Proposition \ref{def:st-conj-G-T}).
	
	To verify \textbf{AD.4}, suppose that to $\Ad(g)$ (resp. $\Ad(h)$) is associated a factorization pair $(g', g'')$ (resp. $(h', h'')$), and accordingly
	\[\begin{tikzcd}
		T \arrow{r}{\Ad(g')} \arrow[bend right]{rr}[swap]{\Ad(g)} & T' \arrow{r}{\Ad(g'')} & S \arrow{r}{\Ad(h')} \arrow[bend right]{rr}[swap]{\Ad(h)} & S' \arrow{r}{\Ad(h'')} & R.
	\end{tikzcd}\]
	Observe that $G^T = G^{T'}$, $G^S = G^{S'}$. Set $k' := g''^{-1}h' g''$ and $P := g''^{-1} S' g''$, then transport $\sigma$ from $\text{Sgn}_m(S)$ to $\text{Sgn}_m(P)$ via $\Ad(g'')^{-1}$. Since $h' \in G^S_\text{ad}(F)$ (resp. $S' \subset G^S$), we see $k' \in G^{T'}_\text{ad}(F) = G^T_\text{ad}(F)$ (resp. $P \subset G^{T'} = G^T$). The composite above equals $\Ad(h''g'') \Ad(k' g'): T \rightiso P \rightiso R$, therefore $(k'g', h''g'')$ is a factorization pair for $\Ad(hg)$. Now
	\begin{multline*}
		\Ad(h'') \underbracket{ \CaliAd^\sigma(h') }_{\text{in}\; G^S = G^{S'}} \Ad(g'') \underbracket{ \CaliAd^\sigma(g') }_{\text{in}\; G^T = G^{T'}} \\
		= \Ad(h'') \underbracket{\Ad(g'')}_{ \tilde{S}'^\sigma_{Q,m} \leftarrow \tilde{P}^\sigma_{Q,m} } \underbracket{\left(  \Ad(g'')^{-1} \CaliAd^\sigma(h') \Ad(g'') \right)}_{ \tilde{P}^\sigma_{Q,m} \leftarrow \tilde{T}'^\sigma_{Q,m} } \CaliAd^\sigma(g').
	\end{multline*}
	We contend that
	\begin{equation}\label{eqn:AD4-verification}
		 \Ad(g'')^{-1} \CaliAd^\sigma(h') \Ad(g'') = \CaliAd^\sigma\left( k' \right);
	\end{equation}
	if so, we will obtain $\Ad(h''g'') \CaliAd^\sigma(k' g')$ by \textbf{AD.4} inside $G^T = G^{T'}$, which equals $\Ad(hg)$ via the factorization pair $(k'g', h''g'')$. In view of the invariance of $\CaliAd^\sigma$ afforded by Definition--Proposition \ref{def:st-conj-G-T}, the \eqref{eqn:AD4-verification} is a straightforward transport of structure.
\end{proof}

\begin{definition}\label{def:st-conj-elements}
	Let $\delta, \eta \in G_\mathrm{reg}(F)$ and set $T := G_\delta$, $S := G_\eta$. For $\sigma \in \mathrm{Sgn}_m(T)$, we say $(\tilde{\delta}, \delta_0) \in \tilde{T}^\sigma_{Q,m}$ and $(\tilde{\eta}, \eta_0) \in \tilde{S}^\sigma_{Q,m}$ are \emph{stably conjugate} if
	\begin{itemize}
		\item there exists $g \in G(\bar{F})$ such that $g\delta g^{-1} = \eta$;
		\item $\CaliAd^\sigma(g)(\tilde{\delta}, \delta_0) = (\tilde{\eta}, \eta_0)$.
	\end{itemize}
\end{definition}
The reference to $\delta_0, \eta_0$ can be dropped when $4 \nmid m$: see Corollary \ref{prop:st-conj-elements-canonical}.


\subsection{Further properties and stability}\label{sec:further-properties}
Let $T \subset G$ be any maximal $F$-torus, parameterized by the datum $(K, K^\sharp, \ldots)$ with $K = \prod_{i \in K} K_i$, etc.

\begin{proposition}\label{prop:CaliAd-simple}
	Let $\Ad(g): T \rightiso S$ be a stable conjugacy of maximal $F$-tori with factorization pair $(g', g'')$ (Definition \ref{def:st-conj}). Suppose either
	\begin{inparaenum}[(a)]
		\item $F$ is archimedean, or
		\item $m \notin 2\Z$.
	\end{inparaenum}
	Then $\CaliAd^\sigma(g)(\tilde{\delta}, \delta_0) = (\Ad(g'') \Ad(g')\tilde{\delta}, \; \Ad(g)(\delta_0))$ for all $(\tilde{\delta}, \delta_0) \in \tilde{T}^\sigma_{Q,m}$.
\end{proposition}
\begin{proof}
	It suffices to treat the $\SL(2)$ case. This boils down to show that the $\Cali_m(\nu, \delta_0)$ from Definition--Proposition \ref{def:Cali-factor} is trivial. When $m \notin 2\Z$ this is evident. When $F=\R$ and $T$ splits, we may take $g' = 1$ in factorization pair since $H^1(F,T)=0$, accordingly $\nu = 1$. When $F=\R$ and $T$ is anisotropic, this follows from $N_{\CC/\R}(\CC^\times) = \R_{>0}$.
\end{proof}

Below are some useful results for the basic building block: the $\SL(2)$ case. 

\begin{proposition}\label{prop:CAd-minus-1}
	Assume $G = \SL(2)$. Choose any preimage $\widetilde{-1} \in \bm{p}^{-1}(-1)$. Let $g \in G_\mathrm{ad}(F)$ with $\nu = \nu(g) \in F^\times/F^{\times 2}$ via \eqref{eqn:nu-arises}. Suppose $(\tilde{\delta}, \delta_0) \in \tilde{T}_{Q,m}$.
	\begin{enumerate}
		\item When $m \notin 2\Z$, we have $\CaliAd(g)\left( \widetilde{-1} \cdot \tilde{\delta}, -\delta_0 \right) = \left( \widetilde{-1}, -1\right) \cdot \CaliAd(g)\left(\tilde{\delta}, \delta_0 \right)$.
		\item When $m \equiv 2 \pmod 4$,
			\[ \CaliAd(g)\left( \widetilde{-1} \cdot \tilde{\delta}, -\delta_0 \right) = \sgn_{K/F}(\nu) \cdot \left( \widetilde{-1}, -1 \right) \cdot \CaliAd(g)(\tilde{\delta}, \delta_0) \]
			with $\sgn_{K/F}(\nu) \in \mu_2 \subset \mu_m = \Ker(\bm{p})$. Note that $\sgn_{K/F}(\nu) = \lrangle{\kappa_-, \mathrm{inv}(\delta, \Ad(g)\delta)}$ (Definition \ref{def:kappa-minus}).
	\end{enumerate}
\end{proposition}
\begin{proof}	
	In the case $m \notin 2\Z$, the factor $\Cali_m(\cdots) = 1$ by Proposition \ref{prop:CaliAd-simple}. It remains to show that $\widetilde{-1}$ is central in $\widetilde{\GL}(2,F)$. By Lemma \ref{prop:-1-adjoint}, this amounts to $(-1, x)_{F,m} = 1$ for all $x \in F^\times$. Indeed, $(-1)^2 = 1$ implies that $(-1, x)_{F,m} \in \mu_2 \cap \mu_m$, hence trivial.
	
	In the case $m \equiv 2 \pmod 4$, suppose $\delta_0$ is parameterized by $x_0 = \omega/\tau(\omega) \in K^1$ for some $\omega \in K^\times$. Then $-x = (-x_0)^{m/2}$ and
	\[ -x_0 = \dfrac{c\omega}{\tau(c\omega)}, \quad c :=
	\begin{cases}
		\sqrt{D}, & K = F(\sqrt{D}): \;\text{field} \\
		(1,-1), & K = F \times F.
	\end{cases}\]
	By Lemma \ref{prop:-1-adjoint},
	\begin{equation}\label{eqn:CAd-flip-derivation} \begin{aligned}
		\CaliAd(g)\left( \widetilde{-1} \cdot \tilde{\delta}, -\delta_0 \right) & = (N_{K/F}(c\omega), \nu)_{F,2} \cdot \left( \Ad(g)\left( \widetilde{-1} \cdot \tilde{\delta} \right), \; -\Ad(g)(\delta_0) \right) \\
		& = (N_{K/F}(c), \nu)_{F,2} (N_{K/F}(\omega), \nu)_{F,2} (-1, \nu)_{F,m} \cdot \left( \widetilde{-1}, -1\right) \cdot \\
		& \quad \cdot \left( \Ad(g)(\tilde{\delta}),\; \Ad(g)(\delta_0) \right) \\
		& = (N_{K/F}(c), \nu)_{F,2} (-1, \nu)_{F,m} \cdot \left( \widetilde{-1}, -1 \right) \cdot \CaliAd(g)(\tilde{\delta}, \delta_0).
	\end{aligned}\end{equation}
	Notice that $(-1, \nu)_{F,m} = ((-1)^{m/2}, \nu)_{F,m} = (-1, \nu)_{F,2}$ by \eqref{eqn:norm-residue-d}. Suppose $K = F(\sqrt{D})$, then
	\[ (N_{K/F}(c), \nu)_{F,2} = (-D, \nu)_{F,2}, \]
	hence \eqref{eqn:CAd-flip-derivation} is $(\widetilde{-1}, -1) \cdot \CaliAd(g)(\tilde{\delta}, \delta_0)$ times $(D, \nu)_{F,2} = \sgn_{K/F}(\nu)$. Next, suppose $K = F \times F$ so that $\sgn_{K/F}(\cdot)=1$. Then $(N_{K/F}(c), \nu)_{F,2} = (-1, \nu)_{F,2}$ and \eqref{eqn:CAd-flip-derivation} reduces to $(\widetilde{-1}, -1) \cdot \CaliAd(g)(\tilde{\delta}, \delta_0)$.
\end{proof}

\begin{proposition}\label{prop:dependence-on-delta_0}
	Assume $G = \SL(2)$ and $\eta_0 \in \Ker(\iota_{Q,m})$ corresponds to $y_0 \in K^1$. Let $g \in G_\mathrm{ad}(F)$ with $\nu = \nu(g) \in F^\times/F^{\times 2}$ via \eqref{eqn:nu-arises}. For all $(\tilde{\delta}, \delta_0) \in \tilde{T}_{Q,m}$ and $\sigma \in \mathrm{Sgn}_m(T)$,
	\begin{enumerate}
		\item when $T$ splits, we have $y_0 \in \mu_{m/\mathrm{gcd}(2,m)}$ and the diagram
			\[\begin{tikzcd}
				\tilde{T}^\sigma_{Q,m} \arrow{d}[swap]{\CaliAd^\sigma(g)} \arrow{rr}{\cdot (1, \eta_0)} & & \tilde{T}^\sigma_{Q,m} \arrow{d}{\CaliAd^\sigma(g)} \\
				\tilde{T}^\sigma_{Q,m} \arrow{rr}[swap]{{\cdot \left(1, \Ad(g)\eta_0 \right)}} & & \tilde{T}^\sigma_{Q,m}
			\end{tikzcd}\]
			commutes: both composites send $(\tilde{\delta}, \delta_0)$ to $\left( \Ad(g)(\tilde{\delta}), \Ad(g)(\eta_0 \delta_0)\right)$;
		\item when $T$ is anisotropic, only for $4 \mid m$ can $\eta_0$ be nontrivial, in which case $\eta_0 = -1$ and
			\[\begin{tikzcd}
				\tilde{T}^\sigma_{Q,m} \arrow{d}[swap]{\CaliAd^\sigma(g)} \arrow{rr}{\cdot (1, \eta_0)} & & \tilde{T}^\sigma_{Q,m} \arrow{d}{\CaliAd^\sigma(g)} \\
				\tilde{T}^\sigma_{Q,m} \arrow{r}[swap, inner sep=1em]{{\cdot \left(1, \eta_0 \right)}} & \tilde{T}^\sigma_{Q,m} \arrow{r}[swap, inner sep=1em]{\cdot \sgn_{K/F}(\nu) } & \tilde{T}^\sigma_{Q,m}
			\end{tikzcd}\]
			commutes, where $\sgn_{K/F}(\nu)$ is viewed as an element of $\mu_2 \subset \mu_m$.
	\end{enumerate}
\end{proposition}
\begin{proof}
	The relevant descriptions of $\Ker(\iota_{Q,m})$ are already in Proposition \ref{prop:iota-kernel}. When $T$ splits, $H^1(F,T)=0$ thus $\Ad(g): T \rightiso gTg^{-1}$ can be realized by ordinary conjugacy. The equalities follow from \textbf{AD.3} of Proposition \ref{prop:CAd-prop}.

	When $T$ is anisotropic, $4 \mid m$ and $\eta_0 = -1 \in Z_G(F)$, one may write $K = F(\sqrt{D})$ and take $\omega := \sqrt{D}$. From the definition of $\CaliAd^\sigma(g)$ we see
	\begin{align*}
		\CaliAd^\sigma(g)(\tilde{\delta}, -\delta_0) & = (N_{K/F}(\omega), \nu)_{F,2} \cdot (1, -1) \cdot \CaliAd^\sigma(\tilde{\delta}, \delta_0) \\
		& = (-D, \nu)_{F,2} \cdot (1, -1) \cdot \CaliAd^\sigma(\tilde{\delta}, \delta_0).
	\end{align*}
	It follows from $4 \mid m$ that $(-D, \nu)_{F,2} = (D, \nu)_{F,2} = \sgn_{K/F}(\nu)$.
\end{proof}

Now we switch back to $G$ of general rank.
\begin{corollary}\label{prop:st-conj-elements-canonical}
	When $4 \nmid m$, the Definition \ref{def:st-conj-elements} works on the level of $\tilde{G}_\mathrm{reg}$: one can define two good elements $\tilde{\delta}, \tilde{\eta} \in \tilde{G}_\mathrm{reg}$ to be stably conjugate if
	\begin{compactitem}
		\item there exists $g \in G(\bar{F})$ such that $g \delta g^{-1} = \eta$;
		\item $\CaliAd(g)(\tilde{\delta}, \delta_0) = (\tilde{\eta}, \Ad(g)(\delta_0))$, where $\delta_0 \in \iota_{Q,m}^{-1}(\delta)$.
	\end{compactitem}
	This notion is independent of the choice of $\delta_0$.
\end{corollary}
\begin{proof}
	By construction this reduces to the $\SL(2)$ case. Using Proposition \ref{prop:dependence-on-delta_0}, we may assume $T$ split and infer that
	\[ \CaliAd(g)(\tilde{\delta}, \eta_0 \delta_0) = \left( \tilde{\eta}, \Ad(g)(\eta_0 \delta_0) \right) \]
	for all $\eta_0 \in \Ker(\iota_{Q,m})$, as required.
\end{proof}

\begin{definition}[Stable distributions]\label{def:stability} \index{stable distribution}
	Let $\Xi$ be a distribution on $\tilde{G}$ represented by a genuine, $G(F)$-invariant locally integrable function that is smooth over $\tilde{G}_{\mathrm{reg}}$. We say that $\Xi$ is \emph{stable} if the following requirements are met.
	\begin{itemize}
		\item When $4 \nmid m$, we require that $\Xi(\tilde{\delta}) = \Xi(\tilde{\eta})$ for any two stably conjugate good elements $\tilde{\delta}, \tilde{\eta} \in \tilde{G}_\mathrm{reg}$.
		\item When $4 \mid m$, consider any maximal torus $T \subset G$ and $\sigma \in \mathrm{Sgn}_m(T)$ (see \eqref{eqn:Sgn}). For every good $\tilde{\delta} \in \tilde{T}_{\mathrm{reg}}$, we require the existence of a $\delta_0 \in T_{Q,m}(F)$ such that $(\tilde{\delta}, \delta_0) \in \tilde{T}^\sigma_{Q,m}$ and
		\[ \CaliAd^\sigma(g)(\tilde{\delta}, \delta_0) = \left( \tilde{\eta}, \Ad(g)(\delta_0)\right) \implies \Xi(\tilde{\delta}) = \Xi(\tilde{\eta}) \]
		where $\Ad(g): \delta \mapsto \eta$ is a stable conjugacy in $G$.
	\end{itemize}
	Note that the $G(F)$-invariance forces $\Xi(\tilde{\delta}) = 0$ if $\tilde{\delta} \in \tilde{G}_{\mathrm{reg}}$ is not good.
\end{definition}
The definition for $4 \mid m$ might seem unnatural. Nonetheless, Proposition \ref{prop:dependence-on-delta_0} shows that even when $n=1$, the $\tilde{\eta}$ depends on how we choose $\delta_0$ to define $\CaliAd^\sigma(g)$ when $T$ is anisotropic.

\subsection{The unramified and global settings}\label{sec:nr-global}
First, we consider the unramified case by supposing that
\begin{compactitem}
	\item $F$ is a non-archimedean local field with residual characteristic $p \neq 2$;
	\item $G = \Sp(W)$ is endowed with a structure of smooth connected $\mathfrak{o}_F$-group, in particular $G(\mathfrak{o}_F)$ is a hyperspecial subgroup of $G(F)$;
	\item by the theory in \S\ref{sec:BD-classification}, $E_G \to G$ is also defined over $\mathfrak{o}_F$;
	\item $m \mid N_F$ is coprime to $p$.
\end{compactitem}

By \cite[10.7]{BD01} and the discussions in \S\ref{sec:torsors-generalities}, with $S := \Spec(\mathfrak{o}_F)$, the pull-back of $E_G(F)$ to $G(\mathfrak{o}_F)$ factors through the central extension $K_2(\mathfrak{o}_F) \hookrightarrow E_G(\mathfrak{o}_F) \twoheadrightarrow G(\mathfrak{o}_v)$. Standard results on tame symbols show that the composite $K_2(\mathfrak{o}_F) \to K_2(F) \to \mu_m$ is trivial. This trivializes the restriction of the BD-cover \eqref{eqn:G-BD-cext} to $G(\mathfrak{o}_F)$. We will view $K$ as a subgroup of $\tilde{G}$ in what follows.

Now put $K := G(\mathfrak{o}_F)$. Let $\delta \in K \cap G_\text{reg}(F)$ and $T := G_\delta$. We recall a result of Kottwitz.
\begin{theorem}[{\cite[Proposition 7.1]{Ko86}}]\label{prop:Kottwitz}
	Suppose that for every root $\alpha$ of $T \dtimes{F} \bar{F}$ in $G \dtimes{F} \bar{F}$, either $\alpha(\delta) = 1$ or $v(\alpha(\delta)-1)=0$. If $\eta \in G(\mathfrak{o}_F)$ is stably conjugate to $\delta$, then $\mathcal{T}(\delta, \eta)$ contains a point in $G(\mathfrak{o}_F)$; in particular $\eta$ is conjugate to $\delta$.
\end{theorem}

\begin{proposition}\label{prop:CAd-unramified}
	In the circumstance above, write $\Ad(g): T \rightiso S := G_\eta$ for the corresponding stable conjugacy between maximal tori, and suppose that
	\[ (\delta, \delta_0) \in \tilde{T}^\sigma_{Q,m}, \quad \sigma \in \mathrm{Sgn}_m(T). \]
	Then we can take $g \in G(\mathfrak{o}_F) \cap \mathcal{T}(\delta, \eta)(F)$ and $\CaliAd^\sigma(g)(\delta, \delta_0) = (\Ad(g) \delta, \Ad(g)\delta_0)$, i.e. the usual conjugacy.
\end{proposition}
\begin{proof}
	One may take $g \in G(\mathfrak{o}_F) \cap \mathcal{T}(\delta, \eta)(F)$ by Theorem \ref{prop:Kottwitz}. It remains to apply \textbf{AD.3} in Proposition \ref{prop:CAd-prop}.
\end{proof}

Next, let $F$ be a global field with characteristic $\neq 2$ and consider Matsumoto's $E_G \twoheadrightarrow G$ over $F$. Fix $m \mid N_F$ and take a large finite subset $S$ of places of $F$, verifying
\begin{itemize}
	\item $S \supset \{v: v \mid \infty \} \sqcup \{v: v \nmid \infty \; \wedge \; \text{res.char}(F_v) \nmid m \}$;
	\item $G$ is defined over the ring $\mathfrak{o}_S$ of $S$-integers as a connected smooth group scheme, and $E_G \to G$ is defined over $\mathfrak{o}_S$ as well;
	\item the earlier conditions in the unramified case hold at every $v \notin S$.
\end{itemize}

At each place $v$, we construct the BD-cover $\mu_m \hookrightarrow \tilde{G}_v \twoheadrightarrow G(F_v)$, except that for complex places we set $\tilde{G}_v = G(F_v)$. Following \cite[10.4]{BD01}, the adélic BD-cover $\tilde{G}$ is the $\varinjlim_V$ of
\begin{equation}\label{eqn:adelic-BD-cover}\begin{aligned}
	\tilde{G}_V  & := \prod_{v \in V} \tilde{G}_v \big/ \mathbf{N}_V \xrightarrow{\bm{p}_V} \prod_{v \in V} G(F_v), \\
	\mathbf{N}_V & := \left\{ (\noyau_v)_v \in \prod_{\substack{v \in V \\ \text{non-complex}}} \mu_m : \prod_v \noyau_v = 1 \right\}.
\end{aligned}\end{equation}
over finite subsets of places $V \supsetneq S$, i.e. the $\varinjlim_V$ of the contracted product $\bm{p}_V$ of the local BD-covers $\tilde{G}_v \xrightarrow{\bm{p}_v} G(F_v)$. Using the aforementioned section $G(\mathfrak{o}_v) \hookrightarrow \tilde{G}_v$, the transition map for $V' \supset V$ is
\[ \tilde{G}_V \hookrightarrow \tilde{G}_V \times \prod_{v \in V' \smallsetminus V} G(\mathfrak{o}_v) \subset \tilde{G}_{V'}. \]

By \cite[10.4.3]{BD01}, we obtain a central extension of locally compact groups equipped with a section $s$ over $G(F)$
\[\begin{tikzcd}
	1 \arrow{r} & \mu_m \arrow{r} & \tilde{G} \arrow{r}{\bm{p}} & G(\mathbb{A}_F) \arrow{r} & 1 \\
	& & & G(F) \arrow[hookrightarrow]{u} \arrow[hookrightarrow]{lu}{\exists s} &
\end{tikzcd}\]

\begin{remark}\label{rem:adelic-Weil-restriction}
	The same construction applies to all multiplicative $\shK_2$-torsors over a reductive $F$-group, but the $s$ here is unique since $G(F)$ equals its own commutator subgroup.
	
	Furthermore, the formation of adélic BD-covers is compatible with Weil restriction by Proposition \ref{prop:restriction-commutes}: if $E_G \to G$ is over a separable extension $L/F$, then the adélic BD-cover obtained from to $R_{L/K}(E_G) \to R_{L/K}(G)$ is the same as the one from $E_G \to G$.
\end{remark}

\begin{definition}\index{good element}
	Call an element $(\delta_v)_v$ of $G(\mathbb{A}_F)$ \emph{good} in $\tilde{G}$ if $\delta_v$ is good in $\tilde{G}_v$ for all $v$. Call $\delta \in G(F)$ \emph{good} in $\tilde{G}$ if $s(\delta) \in \tilde{G}$ is.
\end{definition}

We consider only the case of $\delta \in G_\text{reg}(F)$. The local classification of good elements in Proposition \ref{prop:good-T} can be adapted to the present setting. Notice that the regular semisimple classes and maximal tori in $G$ can still be parameterized by étale $F$-algebras with involution, together with other data. The construction $T \leadsto G^T \subset G$ (Definition \ref{def:G-T}) also works here.

For any closed subvariety $H \subset G$, denote $\tilde{H} := \bm{p}^{-1}(H(\A_F)) \subset \tilde{G}$.

\begin{proposition}\label{prop:good-T-global}
	Let $T \subset G$ be a maximal $F$-torus parameterized by $(K, K^\sharp, \ldots)$, $K = \prod_{i \in I} K_i$ as usual. An element $\delta \in T(F)$ has central preimages in $\tilde{T}$ if and only if
	\[ \delta \in \prod_{i \in I} \{\pm 1 \} \cdot \Image\left[ \iota_{Q_i,m}: K_i^1 \to K_i^1 \right]. \]
	In particular, $\delta \in T_{\mathrm{reg}}(F)$ is good if and only if the property above holds.
\end{proposition}
\begin{proof}
	By Theorem \ref{prop:G-T-reduction} applied to each place $v$, together with Remark \ref{rem:adelic-Weil-restriction}, this reduces immediately to the case $\dim_F K = 2$ and $K^\sharp = F$. Let $x \in K^1$ be an element corresponding to the class of $\delta$, and choose any $\tilde{\delta} \in \bm{p}^{-1}(\delta)$.
	\begin{itemize}
		\item If $m \notin 2\Z$, then $-1$ clearly belongs to $\Image(\iota_{Q,m})$. The goal is to show $\tilde{\delta} \in Z_{\tilde{T}} \iff x \in (K^1)^m$.
		\item If $m \in 2\Z$, the goal is to show $\tilde{\delta} \in Z_{\tilde{T}} \iff x^2 \in (K^1)^m$, as the $2$-torsion subgroup of $K^1$ is clearly $\{\pm 1\}$.
	\end{itemize}
	The same recipe applies to each non-complex place $v$. Set $K_v := K \otimes_F F_v$. By Theorem \ref{prop:good-SL2}, we see
	\[ \tilde{\delta} \in Z_{\tilde{T}} \iff \begin{cases}
		\forall v, \; x_v \in (K^1_v)^m, & m \notin 2\Z \\
		\forall v, \; (x^2)_v \in (K^1_v)^m, &  m \in 2\Z
	\end{cases} \]
	where $v$ ranges over the non-complex places; this is immaterial since the conditions are clearly satisfied when $F_v = \CC$. Hence we are reduced to show that $y \in K^1$ is an $m$-th power in $K^1$ if and only if it is so locally everywhere.
	
	Write $Z := \Ker[K^1 \xrightarrow{m} K^1]$ and set $\Ker^1(F, Z) := \Ker\left[ H^1(F, Z) \to \prod_v H^1(F_v, Z)\right]$. The required local-global principle amounts to $\Ker^1(F,Z) = 0$. When $K \simeq F \times F$, we have $K^1 \simeq F^\times$, $Z \simeq \mu_m$ and this is covered by the Grunwald--Wang theorem since $m \mid N_F$. When $K$ is a field, the vanishing of $\Ker^1(F,Z)$ has been shown in \cite[Proposition 2.1]{Mi87T} (taking $S=\emptyset$, $2^s \| m$), which also works over function fields.
\end{proof}

Next, recall that $\tilde{T}$ denotes the preimage of $T(\mathbb{A}_F)$ in $\tilde{G}$. Given $\sigma \in \text{Sgn}_m(T)$, define
\begin{align*}
	T^\sigma_{Q,m} & := \left\{ (t, t_0) \in T \times T_{Q,m}: t = \sigma \cdot \iota_{Q,m}(t_0) \right\} \quad \text{(fibered product of varieties)}, \\
	\tilde{T}^\sigma_{Q,m} & := \left\{ (\tilde{t}, t_0) \in \tilde{T} \times T_{Q,m}(\mathbb{A}_F): \bm{p}(\tilde{t}) = \sigma \cdot \iota_{Q,m}(t_0) \right\}.
\end{align*}
Here the action of $\sigma$, etc. are defined in the same manner as in \S\ref{sec:st-conj-BD}. As in \eqref{eqn:adelic-BD-cover}, observe that
\[ \tilde{T}^\sigma_{Q,m} = \varinjlim_V \left( \prod_{v \in V} \tilde{T}^\sigma_{Q,m,v} \big/ \mathbf{N}_V \right), \]
the transition maps are again defined using integral models of BD-covers off $V$, for $V$ sufficiently large. Using the section $s$, one embeds $T^\sigma_{Q,m}(F)$ into $\tilde{T}^\sigma_{Q,m}$.

Given Proposition \ref{prop:good-T-global}, it is natural to study the effect of adélic stable conjugacy on $\tilde{T}^\sigma_{Q,m}$ and $T^\sigma_{Q,m}(F)$. This is based on the local avatars $\CaliAd^\sigma(g)_v$ at each place $v$; it reduces to the usual one when $F_v = \CC$.

\begin{theorem}\label{prop:adelic-conj}
	Let $\Ad(g): T \rightiso S$ be stable conjugacy between maximal $F$-tori of $G$. Then $\CaliAd^\sigma(g) := \prod_v \CaliAd^\sigma(g)_v$ defines a map $\tilde{T}^\sigma_{Q,m} \to \tilde{S}^\sigma_{Q,m}$. It satisfies the properties enunciated in Proposition \ref{prop:CAd-prop}, and restricts to a map $T^\sigma_{Q,m}(F) \to S^\sigma_{Q,m}(F)$; when $g \in G^T_\mathrm{ad}(F)$, this restriction comes from the usual $\Ad(g)$.
\end{theorem}
\begin{proof}
	Pick $\delta \in T_\text{reg}(F)$ and $\eta = \Ad(g)\delta \in S_\text{reg}(F)$. There is a large finite set $S$ of places such that for all $v \notin S$, the Theorem \ref{prop:Kottwitz} applies to $\Ad(g): \delta_v \mapsto \eta_v$. At such places, Proposition \ref{prop:CAd-unramified} implies that $\CaliAd^\sigma(g)$ reduces to ordinary conjugacy by $G(\mathfrak{o}_v)$. Together with \textbf{AD.1} of Proposition \ref{prop:CAd-prop}, this implies $\prod_v \CaliAd^\sigma(g)_v$ is well-defined. The properties in Proposition \ref{prop:CAd-prop} are inherited from $\CaliAd^\sigma(g)_v$, for all non-complex place $v$.

	Now move to the restriction of $\CaliAd^\sigma(g)$ to $T^\sigma_{Q,m}(F)$. As in the local setting, we may choose a factorization pair $(g', g'')$ for $\Ad(g)$ over $F$, with $g' \in G^T_\text{ad}(F)$ and $g'' \in G(F)$. Hence it suffices to consider the case $G \simeq \SL(2)$ and $g \in G_\text{ad}(F)$ (see Remark \ref{rem:adelic-Weil-restriction}). The relevant factors $\Cali_m(\cdots)$ in Definition--Proposition \ref{def:Cali-factor} are trivial for almost all $v$, and cancel out by the product formula for Hilbert symbols.	It remains to show $\widetilde{\Ad}(g) := \prod_v \Ad(g_v): \tilde{G} \to \tilde{G}$ leaves $s(G(F))$ invariant, noting that $\Ad(g_v)$ is realized by $G(\mathfrak{o}_v)$-conjugacy for almost all $v$, by Theorem \ref{prop:Kottwitz}. We contend that
	\[\begin{tikzcd}
		\tilde{G} \arrow{r}{\widetilde{\Ad}(g)} & \tilde{G} \\
		G(F) \arrow[hookrightarrow]{u}{s} \arrow{r}[swap]{\Ad(g)} & G(F) \arrow[hookrightarrow]{u}[swap]{s}
	\end{tikzcd}\]
	commutes. Indeed, $\widetilde{\Ad}(g) s  \Ad(g)^{-1} = s$ by the uniqueness of the section.
\end{proof}

\section{\texorpdfstring{$L$}{L}-groups after Weissman}\label{sec:Weissman}
Except in \S\ref{sec:rescaling}, we consider $G = \Sp(W)$ over a local field $F$ with $\text{char}(F) \neq 2$, $m \mid N_F$, and $\tilde{G} \twoheadrightarrow G(F)$ as in \S\ref{sec:BD-Sp}; also fix $\epsilon: \mu_m \hookrightarrow \CC^\times$. Note that the constructions in \cite{Weis18} can also be adapted to global or the unramified integral cases.

\subsection{Definitions}\label{sec:L-group}
For later use, the data in \cite[\S 2.2]{Weis18} derived from the Brylinski--Deligne data $(Q,\mathcal{D},\varphi)$ are tabulated below in the notation of \S\ref{sec:Sp}. 
\begin{gather*}
	\beta_Q = \frac{1}{m} B_Q = \frac{2}{m}\sum_{i=1}^n \epsilon_i \otimes \epsilon_i: Y \otimes Y \to \frac{1}{m}\Z, \\
	Y_{Q,m} = \frac{m}{\text{gcd}(2,m)} Y, \quad X_{Q,m} = \frac{\text{gcd}(2,m)}{m}  X, \\
	\quad n_\alpha := \frac{m}{\text{gcd}(m, Q(\check{\alpha}))} = \begin{cases}
		m, & \alpha: \text{long root} \\
		\frac{m}{\text{gcd}(2,m)}, & \alpha: \text{short root}
	\end{cases}, \quad \tilde{\alpha} := n_\alpha^{-1} \alpha, \quad \tilde{\alpha}^\vee := n_\alpha \alpha^\vee, \\
	Y_{Q,m}^\text{sc} = \sum_\alpha \Z\tilde{\alpha}^\vee = \begin{cases}
	\frac{m}{2} Y_0, \quad Y_0 := \left\{ \sum a_i \check{\epsilon}_i \in Y: \sum_i a_i \in 2\Z \right\}, & m \in 2\Z \\
	Y_{Q,m}, & m \notin 2\Z.
	\end{cases}
\end{gather*} \index{Gvee@$\tilde{G}^\vee$}\index{YQmsc@$Y_{Q,m}^{\mathrm{sc}}$}\index{Y0@$Y_0$}
We call $\tilde{\alpha}$ and $\tilde{\alpha}^\vee$ the \emph{modified roots and coroots}. Let $\tilde{\Delta} \subset \tilde{\Phi} \subset X_{Q,m}$ be the sets of modified roots and that of the simple ones; similarly $\tilde{\Delta}^\vee \subset \tilde{\Phi}^\vee \subset Y_{Q,m}$. In \textit{loc. cit.}, the dual group $\tilde{G}^\vee$ is defined as the pinned $\CC$-group with based root datum $(Y_{Q,m}, \tilde{\Delta}^\vee, X_{Q,m}, \tilde{\Delta})$, here with trivial $\Gamma_F$-action. Thus $\tilde{G}^\vee$ comes with a maximal torus $\tilde{T}^\vee$ with
\[ X^*(\tilde{T}^\vee) = Y_{Q,m}, \quad Z_{\tilde{G}^\vee} = \Hom(Y_{Q,m}/Y_{Q,m}^\text{sc}, \CC^\times) \subset \tilde{T}^\vee. \]
There is a homomorphism $\tau_{Q,m}: \mu_2 \to Z_{\tilde{G}^\vee}$ that is Cartier-dual to
\[ Y_{Q,m}/Y_{Q,m}^\text{sc} \twoheadrightarrow Y_{Q,m}\big/(Y_{Q,m}^\text{sc} + mY_{Q,m}) \xrightarrow{y \mapsto m^{-1}Q(y)} \frac{1}{2}\Z \big/ \Z. \]
As the roots/coroots are modified by rescaling, $\tilde{G}^\vee$ and $G$ share the same Weyl group. In fact
\[ \tilde{G}^\vee = \begin{cases}
	\SO(2n+1, \CC), & m \notin 2\Z \\
	\Sp(2n, \CC), & m \in 2\Z.
\end{cases}\]
Quoting \cite[\S 2.7.4]{Weis18}, we have $\tau_{Q,m}(-1) \neq 1$ if and only if $m \equiv 2 \pmod 4$. \index{tauQm@$\tau_{Q,m}$}

To define the Galois form of the $L$-group of $\tilde{G}$, consider the following two extensions of groups.
\begin{enumerate}
	\item In \cite[\S 4.1]{Weis18} the \emph{metaGalois group} is defined as the central extension
	\begin{equation}\label{eqn:metaGalois}
		1 \to \mu_2 \to \widetilde{\Gamma_F} \to \Gamma_F \to 1.
	\end{equation}
	It is just $\Gamma_F \times \mu_2$ with multiplication given by the cocycle $(\tau_1, \tau_2) \mapsto (\text{rec}_F(\tau_1), \text{rec}_F(\tau_2))_{F,2}$, where the reciprocity homomorphism $\text{rec}_F: \Gamma_F \to F^\times$ is normalized to send a geometric Frobenius to a uniformizer in $\mathfrak{p}_F$. We obtain the push-out $Z_{\tilde{G}^\vee} \hookrightarrow \tau_{Q,m,*}(\widetilde{\Gamma_F}) \twoheadrightarrow \Gamma_F$.
	
	\item The constructions in \cite[\S 3.2]{Weis18} yield a gerbe $\mathsf{E}_\epsilon(\tilde{G})$ over $\Spec(F)_{\text{ét}}$ banded by $Z_{\tilde{G}^\vee}$. Its fundamental group sits in an extension
	\begin{equation}\label{eqn:gerbe-pi1}
		1 \to Z_{\tilde{G}^\vee} \to \pi_1^{\text{ét}}(\mathsf{E}_\epsilon(\tilde{G})) \to \Gamma_F \to 1.
	\end{equation}
	The relevant definitions will be recalled in \S\S\ref{sec:second-twist}---\ref{sec:rescaling} whenever needed.
\end{enumerate}
Their Baer sum is an extension $Z_{\tilde{G}^\vee} \hookrightarrow \Lgrp{Z} \twoheadrightarrow \Gamma_F$; a further $\Gamma_F$-equivariant push-out via $Z_{\tilde{G}^\vee} \hookrightarrow \tilde{G}^\vee$ yields $\Lgrp{\tilde{G}}$. To obtain a continuous section $\Gamma_F \to \Lgrp{\tilde{G}}$, we inspect the metaGalois group first.\index{GL@$\Lgrp{G}$}

When $m \not\equiv 2 \pmod 4$, the central extension $\tau_{Q,m,*}(\widetilde{\Gamma_F})$ is already trivial. Now suppose $m \equiv 2 \pmod 4$. Fix an additive character $\psi$. According to \cite[Proposition 4.5]{Weis18}, upon enlarging $\widetilde{\Gamma_F}$ to $\widetilde{\Gamma_F}^{(4)}$ by $\mu_2 \hookrightarrow \bmu_4$, there is a splitting $s(\psi): \Gamma_F \to \widetilde{\Gamma_F}^{(4)}$ given by
\begin{gather*}
	s(\psi)(\tau) := \left( \tau, \; \dfrac{ \gamma_\psi(\text{rec}_F(\tau)) }{\gamma_\psi(1)} \right) \; \in \Gamma_F \times \mu_4.
\end{gather*}
This is based on the standard identity \cite[Proposition 1.3.3]{Per81} for Weil's constants
\begin{equation}\label{eqn:Weil-Hilbert}
	(a,b)_{F,2} = \dfrac{\gamma_\psi(ab) \gamma_\psi(1)}{\gamma_\psi(a) \gamma_\psi(b)}.
\end{equation}
Notice that the enlargement to $\bmu_4$ is realizable inside $\tilde{G}^\vee = \Sp(2n, \CC)$: consider the subgroup
\begin{equation}\label{eqn:C}
	C := \left\{ \text{diag}(\zeta, \ldots, \zeta, \zeta^{-1}, \ldots, \zeta^{-1}) \in \Sp(2n, \CC): \zeta \in \bmu_4. \right\}
\end{equation}
We have $\tau_{Q,m}(-1) = -1 \in C$, and such a subgroup $C \simeq \bmu_4$ is unique up to conjugacy.

\begin{lemma}[{\cite[Proposition 4.5]{Weis18}}]\label{prop:variance-metaGalois}
	If $\psi$ is replaced by $x \mapsto \psi(cx)$ where $c \in F^\times$, then $s(\psi)$ will be twisted by the character
	\begin{equation}\label{eqn:metaGalois-section-twist}
		\chi_c: \Gamma_F \to \mu_2, \quad \chi_c(\tau) = (\mathrm{rec}_F(\tau), c)_{F,2}.
	\end{equation}
\end{lemma}
Clearly, $\chi_{c_1 c_2} = \chi_{c_1} \chi_{c_2}$ for all $c_1, c_2 \in F^\times$.

Later on, by using the form $\lrangle{\cdot|\cdot}$ or the associated $F$-pinning \cite[Exp XXIII]{SGA3-3} of $G$,
\begin{compactitem}
	\item it will be shown in Lemma \ref{prop:gerbe-splitting} that \eqref{eqn:gerbe-pi1} also splits after a push-out via $Z_{\tilde{G}^\vee} \hookrightarrow C$;
	\item in Lemma \ref{prop:variance-gerbe}, it will be shown that changing $\lrangle{\cdot|\cdot}$ or the pinning will twist that splitting by the $\chi_c$ in \eqref{eqn:metaGalois-section-twist}.
\end{compactitem}

With these auxiliary data, we can switch to the Weil form of $\Lgrp{G}$ and identify
\[\begin{tikzcd}[row sep=tiny, column sep=small]
	\Lgrp{\tilde{G}} \arrow{rr}{\sim} \arrow{rd} & & \tilde{G}^\vee \times \Weil_F \arrow{ld} \\
	& \Weil_F &
\end{tikzcd}\]
\begin{definition}\index{L-parameter}
	An $L$-parameter for $\tilde{G}$ is a continuous homomorphism $\phi: \Weil_F \times \SU(2) \to \Lgrp{\tilde{G}}$ commuting with the projections to $\Weil_F$, such that the $\tilde{G}^\vee$-component of $\phi(w)$ is semi-simple for all $w \in \Weil_F$. Equivalence between $L$-parameters is given by $\tilde{G}^\vee$-conjugacy.
\end{definition}

By the foregoing discussions, the choice of splitting $\Lgrp{G} \simeq \tilde{G}^\vee \times \Weil_F$ does not affect the definitions above. The $L$-parameters that we will encounter are all trivial on $\SU(2)$.

\subsection{On the second twist}\label{sec:second-twist}
For the next result, we recall from \cite[Remark 19.8]{Weis18} that $\pi_1^{\text{ét}}(\mathsf{E}_\epsilon(\tilde{G}))$ is realized as $\varinjlim_{\bar{z}} \pi_1^{\text{ét}}(\mathsf{E}_\epsilon(\tilde{G}), \bar{z})$ with respect to unique transition isomorphisms, where the ``geometric basepoint'' $\bar{z}$ ranges over objects of $\mathsf{E}_\epsilon(\tilde{G})(\bar{F})$. We will employ the concrete description of \eqref{eqn:gerbe-pi1} in \cite{Weis18b} or \cite[\S 5.2]{GG} for pinned split groups, called the \emph{second twist} in \cite{GG}.

Take a symplectic basis, which gives rise to a standard $F$-pinning for $G$; note that the symplectic bases form a single $G(F)$-orbit. Let $T \subset G$ be the corresponding split maximal torus adapted to the chosen symplectic basis and consider the $\mathcal{D}$ from \eqref{eqn:D-extension}. Taking pull-backs, we obtain the objects $\mathcal{D}_{Q,m}$, $\mathcal{D}^\text{sc}_{Q,m}$ in $\cate{CExt}(Y_{Q,m}, \Gm)$, $\cate{CExt}(Y_{Q,m}^\text{sc}, \Gm)$ respectively. Using these data, \cite[\S 11.7]{BD01} gives the map for each simple coroot $\alpha^\vee$: \index{DQm@$\mathcal{D}_{Q,m}, \mathcal{D}_{Q,m}^{\mathrm{sc}}$}
\begin{gather*}
	k\alpha^\vee \longmapsto s(k\alpha^\vee) = \Res\left( s[n](k\alpha^\vee(\mathbf{t})) \right) \in \mathcal{D}, \quad k \in \Z.
\end{gather*}
Here we work in the $K_2(F(\!(\mathbf{t})\!))$-torsor over $T(F(\!(\mathbf{t})\!))$, and $K_2(F(\!(\mathbf{t})\!)) \xrightarrow{\Res} F$ is the tame symbol \cite[(11.1.8)]{BD01}; $s[n]$ is defined before \cite[(11.1.3)]{BD01} by working in the $\SL(2)$ generated by $\alpha$ and using the pinning. We have $s(k\alpha^\vee) \mapsto k\alpha^\vee$. From this we see $\mathcal{D}_{Q,m}(F) \to Y_{Q,m}$ (taking $F$-points yields a central extension by $F^\times$ by Hilbert's Satz 90) admits a splitting $s_0$ over $Y_{Q,m}^\text{sc}$, by sending each $\tilde{\alpha}^\vee \in \tilde{\Delta}^\vee$ to $s(n_\alpha \alpha^\vee)$.

Now assume $m \in 2\Z$ so that \cite[Assumption 3.1]{Weis18} is in force. Choose a \emph{convenient basepoint} $\bar{z}_0$ as in \cite[\S 1.1]{Weis18b} for the pinning, which arises from a splitting $\hat{s}_0: Y_{Q,m} \to \mathcal{D}_{Q,m}(\bar{F})$ over $\bar{F}$ extending $s_0$. Define the group
\[ E_0 := \left\{ \chi \in \Hom(\mathcal{D}_{Q,m}(F), \CC^\times): \chi|_{F^{\times m}} = \chi|_{\Image(s_0)} = 1 \right\}. \]
Then $E_0 \twoheadrightarrow \Hom(F^\times/F^{\times m}, \bmu_m)$ with kernel $Z_{\tilde{G}^\vee}$. Consider the surjective homomorphism
\begin{align*}
	q_m: \Gamma_F & \longrightarrow \Hom(F^\times/F^{\times m}, \bmu_m) \\
	\gamma & \longmapsto \left[ u \mapsto \epsilon\left(\frac{\gamma^{-1}u^{1/m}}{u^{1/m}}\right)\right].
\end{align*}
By \cite[\S\S 2.1--2.3]{Weis18b}, one can describe $\pi_1^{\text{ét}}(\mathsf{E}_\epsilon(\tilde{G}), \bar{z}_0)$ via the comparison isomorphism of central extensions in \textit{loc. cit.}
\[ C_0: \pi_1^{\text{ét}}(\mathsf{E}_\epsilon(\tilde{G}), \bar{z}_0) \rightiso (q_m)^* E_0 \quad (\text{pull-back by}\; q_m). \]
Such isomorphisms are shown to be compatible with the transition isomorphisms when $\bar{z}_0$ varies.

\begin{lemma}\label{prop:gerbe-splitting}
	When $m \notin 2\Z$, the central extension \eqref{eqn:gerbe-pi1} is trivial. When $4 \mid m$, \eqref{eqn:gerbe-pi1} splits; when $m \equiv 2 \pmod 4$, it splits after a push-out $Z_{\tilde{G}^\vee} \hookrightarrow C$, see \eqref{eqn:C}. These splittings for $m \in 2\Z$ are canonical for the chosen $(W, \lrangle{\cdot|\cdot})$.
\end{lemma}
\begin{proof}
	When $m \notin 2\Z$ we have $Z_{\tilde{G}^\vee} = \{1\}$, so assume $m \in 2\Z$ in what follows.
	
	Let us take the quotient of $\mathcal{D}_{Q,m}(F)$ by $s_0(Y_{Q,m}^\text{sc})$, which gives an object $\mathcal{D}_1$ of $\cate{CExt}(Y_{Q,m}/Y_{Q,m}^\text{sc}, F^\times)$. A further push-out yields an object $\mathcal{D}_2$ of $\cate{CExt}(Y_{Q,m}/Y_{Q,m}^\text{sc}, F^\times/F^{\times m})$. The discussions above imply that \eqref{eqn:gerbe-pi1} is obtained from $q_m^* \Hom(\mathcal{D}_2, \CC^\times)$, the multiplication in $\Hom(\mathcal{D}_2, \CC^\times)$ being pointwise.
	
	To find splittings, we first reduce $\mathcal{D}_1$ to a $\mu_2 \hookrightarrow \mathcal{D}_1^\diamond \twoheadrightarrow Y_{Q,m}/Y_{Q,m}^\text{sc}$. As $\frac{m}{2}\check{\epsilon}_1$ generates $Y_{Q,m}/Y_{Q,m}^\text{sc}$, using the pinning, the coset of $s(m\check{\epsilon}_1/2)$ affords a canonical preimage in $\mathcal{D}_1$. By \cite[(11.1.4)---(11.1.5)]{BD01} and $\Res\{\mathbf{t}, \mathbf{t}\} = -1$, we have
	\begin{align*}
		s\left( \frac{m\check{\epsilon}_1}{2} \right)^2 & = \Res\{ \mathbf{t}^{m/2}, \mathbf{t}^{m/2} \}^{Q(\check{\epsilon}_1)} \cdot \overbracket{ s(m\check{\epsilon}_1) }^{=1 \in \mathcal{D}_1} \\
		& = (-1)^{m/2} \quad \in \mu_2.
	\end{align*}
	This furnishes the $\mathcal{D}_1^\diamond$. Accordingly $\Hom(\mathcal{D}_2, \CC^\times)$ is pulled back from $\Hom(\mathcal{D}_1^\diamond, \CC^\times)$ by $\Hom(\mu_2, \CC^\times) \xrightarrow{\text{res}} \Hom(F^\times/F^{\times m}, \CC^\times)$, and both splits canonically when $4 \mid m$.
	
	Hereafter suppose $m \equiv 2 \pmod 4$. Take $\zeta \in \CC$ with $\zeta^2 = -1$. Prescribe $\chi_\zeta \in \Hom(\mathcal{D}_1^\diamond, \CC^\times)$ lying over $\epsilon|_{\mu_2} \in \Hom(\mu_2, \CC^\times)$ by
	\[ \chi_\zeta|_{\mu_2} = \epsilon|_{\mu_2}, \quad \chi_\zeta(s(m\check{\epsilon}_1/2)) = \zeta, \]
	which is legitimate by the foregoing computation. Then $\chi_\zeta^2 \in \Hom(Y_{Q,m}/Y_{Q,m}^\text{sc}, \CC^\times) = Z_{\tilde{G}^\vee}$ is nontrivial, so it is not a section for $\Hom(\mathcal{D}_1^\diamond, \CC^\times) \twoheadrightarrow \Hom(\mu_2, \CC^\times)$. However, upon pushing out by $Z_{\tilde{G}^\vee} \hookrightarrow C \simeq \bmu_4$, we have $(\zeta^{-1} \wedge \chi_\zeta)^2 = 1$ and this does afford a splitting.
	
	Finally, $\chi_{\zeta^{-1}}/\chi_\zeta$ is trivial on $\mu_2 \subset \mathcal{D}_1^\diamond$ and maps $s(m\check{\epsilon}_1/2)$ to $-1$; this cancels with the ratio $\zeta/\zeta^{-1} = -1$ in the push-out to $C$. Hence the splitting does not depend on the choice of $\zeta$.
\end{proof}

The upshot is that both \eqref{eqn:metaGalois} and \eqref{eqn:gerbe-pi1} split after pushed out to $C$, therefore the Galois form of $\Lgrp{\tilde{G}}$ splits as well. These splittings depend on $\psi$ and $(W, \lrangle{\cdot|\cdot})$ (or the $F$-pinning). Our ultimate goal is an ambiguity-free formalism of $L$-packets. The dependence on $\psi$ is quantified by Lemma \ref{prop:variance-metaGalois} so we shall concentrate on $\lrangle{\cdot|\cdot}$. Set $G_1 := \GSp(W)$ and note that
\[\begin{tikzcd}[column sep=small]
	\dfrac{\left\{ F-\text{pinnings}\right\}}{G(F)-\text{conj}} & \dfrac{ G_\text{ad}(F) }{\Image[G(F) \to G_\text{ad}(F)]} \arrow{d}[swap]{\simeq} \arrow[bend right=40]{l}[midway, swap]{\text{torsor}} & \dfrac{G_1(F)}{Z_{G_1}(F) G(F)} \arrow{ld}{\text{similitude}}[swap]{\simeq} \arrow{l}[swap]{\sim} \\
	& F^{\times}/F^{\times 2} &
\end{tikzcd} \qquad \text{commutes.} \]
Hence $G_\text{ad}(F)$ has three effects:
\begin{inparaenum}[(a)]
	\item dilation of $\lrangle{\cdot|\cdot}$ up to $F^{\times 2}$,
	\item change of $F$-pinnings,
	\item action on $E_G$ via Proposition \ref{prop:BD-adjoint-action}.
\end{inparaenum}

Define the groups $\hat{T} := \Hom(Y_{Q,m}, \bar{F}^\times)$, $\hat{T}^\text{sc} := \Hom(Y_{Q,m}^\text{sc}, \bar{F}^\times)$ and $\hat{Z} := \Hom(Y_{Q,m}/Y_{Q,m}^\text{sc}, \bar{F}^\times)$.

\begin{lemma}\label{prop:variance-gerbe}
	Suppose that $m \in 2\Z$ and let $g_1 \in G_1(F)$ be of similitude factor $c \in F^\times$. Then changing pinnings by $\Ad(g_1)$ induces a twist by the $\chi_c$ of \eqref{eqn:metaGalois-section-twist} on the splitting constructed in Lemma \ref{prop:gerbe-splitting}.
\end{lemma}
\begin{proof}
	Let $T_1 := Z_{G_1}(T)$. Upon a translation by $G(F)$, we may assume $g_1 \in T_1(F)$. In fact, in the symplectic basis we take $g_1 = t := \text{diag}(c, \ldots, c, 1, \ldots, 1)$.
	
	For any two choices of $s_0, \hat{s}_0$ and $s_1, \hat{s}_1$ as before, which give rise convenient base points $\bar{z}_0, \bar{z}_1$, we have comparison isomorphisms
	\[ C_i: \pi_1^{\text{ét}}(\mathsf{E}_\epsilon(\tilde{G}), \bar{z}_i) \rightiso (q_m)^* E_i, \quad i = 0, 1. \]
	The prior constructions are based on the pinning and canonical splittings of multiplicative $K_2$-torsors over unipotent groups; therefore they are transported by the $\Ad(t): E_G \to E_G$ of Proposition \ref{prop:BD-adjoint-action}. Now regard $s_0, \hat{s}_0$ as given and set
	\[ s_1 := \Ad(t) \circ s_0, \quad \hat{s}_1 := \Ad(t) \circ \hat{s}_0. \]
	Define $\varphi_t: E_0 \rightiso E_1$ by $\chi \mapsto \chi \circ \Ad(t)^{-1}$ and observe that $\varphi_t$ transports the splittings (after a push-out) afforded by Lemma \ref{prop:gerbe-splitting} for similar reasons. It suffices to show that the commutativity up to $\chi_c$-twist of
	\begin{equation}\label{eqn:comparison-twist} \begin{tikzcd}
		\pi_1^{\text{ét}}(\mathsf{E}_\epsilon(\tilde{G}), \bar{z}_0) \arrow{r}{\sim}[swap]{\iota} \arrow{d}[swap]{C_0} & \pi_1^{\text{ét}}(\mathsf{E}_\epsilon(\tilde{G}), \bar{z}_1) \arrow{d}{C_1} \\
		(q_m)^* E_0 \arrow{r}{\sim}[swap]{\varphi_t} & (q_m)^* E_1
	\end{tikzcd} \quad \iota: \text{transition isomorphism}. \end{equation}
	We shall make free use of the computations in \cite{Weis18b} below. Let $b \in \hat{T}$ with $b^m = \hat{s}_1/\hat{s}_0$. Represent the elements of $\pi_1^{\text{ét}}(\mathsf{E}_\epsilon(\tilde{G}), \bar{z}_0)$ as $(\tau, \zeta) \in \hat{T} \times Z_{\tilde{G}^\vee}$. To simplify matters, we fix $\gamma \in \Gamma_F$ and look only at the fibers in $E_0, E_1$ over $q_m(\gamma)$; it is shown in \cite[\S\S 2.2--2.3]{Weis18b} that
	\begin{gather*}
		\iota(\tau, \zeta) = \left( \tau \cdot \frac{b}{\gamma^{-1} b}, \zeta\right), \quad C_i(\tau, \zeta) = \left[ \hat{s}_i(y)u \mapsto \epsilon\left( \frac{\gamma^{-1} u^{1/m}}{u^{1/m}} \cdot \tau(y) \right) \zeta(y) \right] \quad (i=0,1)
	\end{gather*}
	where $(y, u) \in Y_{Q,m} \times \bar{F}^\times$ satisfies $\hat{s}_i(y)u \in \mathcal{D}_{Q,m}(\bar{F})^{\Gamma_F} = \mathcal{D}_{Q,m}(F)$. Note that in \S 2.3 of \textit{loc. cit.} one has $s_0 = s_1$ in describing $\iota$; we cannot assume this, so $b \notin \hat{Z}$ in our case.
	
	Suppose $\zeta=1$ for simplicity. Then $\varphi_t \circ C_0(\tau, 1)$ is the map
	\[ \hat{s}_1(y)u \xmapsto{\Ad(t)^{-1}} \hat{s}_0(y) u \mapsto \epsilon\left( \frac{\gamma^{-1} u^{1/m}}{u^{1/m}} \cdot \tau(y) \right). \]
	On the other hand, $C_1  \circ\iota(\tau, 1)$ maps $\hat{s}_1(y) u \in \mathcal{D}_{Q,m}(F)$ to
	\[ \epsilon\left( \frac{\gamma^{-1} u^{1/m}}{u^{1/m}} \cdot \frac{ (\tau b) (y)}{\gamma^{-1} b(y)} \right) = \epsilon\left( \frac{\gamma^{-1} u^{1/m}}{u^{1/m}} \cdot \tau(y) \right) \epsilon\left( \frac{b(y)}{\gamma^{-1} b(y)} \right). \]
	So the twist to make \eqref{eqn:comparison-twist} commutative is $y \mapsto \epsilon(b(y)/\gamma^{-1}b(y))$.

	Note that $\Ad(t)$ induces an automorphism of the central extension $\mathcal{D}_{Q,m}$, i.e. an element of $\Hom(Y_{Q,m}, \Gm)$. By Lemma \ref{prop:restriction-W_i} and our choice of $T, t$, it suffices to study $\Ad(t)$ inside copies of $E_{\SL(2)} \to \SL(2)$, one for each $\check{\epsilon}_i$. For any preimage $\tilde{\delta} \in E_{\SL(2)}(F(\!(\mathbf{t})\!))$ of $\check{\epsilon}_1(\mathbf{t})$, Kubota's cocycle \eqref{eqn:Kubota-cocycle} yields $\Ad\twomatrix{c}{}{}{1}(\tilde{\delta}) = (-\{\mathbf{t}, c\}_F) \cdot \tilde{\delta}$. From $\Res\{\mathbf{t}, c\}_F = c$ we deduce
	\begin{align*}
		b^m = \frac{\Ad(t) \circ \hat{s}_0}{\hat{s}_0}: Y_{Q,m} & \longrightarrow F^\times \\
		y = \sum_{i=1}^n a_i \check{\epsilon}_i & \longmapsto c^{-(a_1 + \cdots a_n)}.
	\end{align*}
	Recall $Y_{Q,m} = \frac{m}{2}Y$; thus the definition of Hilbert symbols and \eqref{eqn:norm-residue-d} imply
	\[ \epsilon\left( \frac{b(y)}{\gamma^{-1}b(y)}\right) = \epsilon \left( \text{rec}_F(\gamma), c^{a_1 + \cdots + a_n} \right)_{F, m}^{-1} = \epsilon\left( \text{rec}_F(\gamma), c^{2(a_1 + \cdots + a_n)/m} \right)_{F,2}^{-1}. \]
	As $Y_{Q,m}^\text{sc} = \frac{m}{2} Y_0$, it yields a character of $Y_{Q,m}/Y_{Q,m}^\text{sc}$ mapping any generator to $\epsilon(\text{rec}_F(\gamma), c)_{F,2} = \chi_c(\gamma)$.
\end{proof}

See also \cite[\S 7]{GG} for a different construction of splittings.

\subsection{Rescaling}\label{sec:rescaling}
In this subsection we work with a pinned quasisplit $F$-group $G$ and $E \twoheadrightarrow G$ in $\cate{CExt}(G, \shK_2)$, classified by some $(Q, \mathcal{D}, \varphi)$. We only retain the datum $Q$ in the notation; for instance $kQ$ really means the $k$-fold Baer sum of $E$ or the corresponding triple. Consider
\[ m \in \Z_{\geq 1}, \quad k \in \Z \smallsetminus \{0\}, \quad d := \text{gcd}(k,m), \quad m = dm', \; k = dk'. \]
Fix $\epsilon: \mu_m \hookrightarrow \CC^\times$ and set $\epsilon' := \epsilon^{k'}: \mu_{m'} \hookrightarrow \CC^\times$. Denote by $\tilde{G}[kQ, m]$, $\tilde{G}[Q, m']$ the BD-covers attached to these data, and affix $\epsilon, \epsilon'$ to denote their push-outs; the same formalism also applies to other constructions. Our goal is to identify the $L$-groups $\Lgrp{\tilde{G}}_\epsilon[kQ,m]$ and $\Lgrp{\tilde{G}}_{\epsilon'}[Q, m']$, which should reflect Remark \ref{rem:rescaling-Q}. Unexplained notations can be found in \cite[\S 3]{Weis18}.

Note that $B_{kQ} = kB_Q$. We claim that
\begin{align*}
	Y_{kQ,m} & := \left\{ y \in Y: kB_Q(y,Y) \subset m\Z \right\} \\
	& = \left\{ y \in Y: k' B_Q(y,Y) \subset m'\Z \right\} = Y_{Q,m'}.
\end{align*}
Only the last equality is nontrivial, for which $\supset$ is clear. As for $\subset$, take $a,b \in \Z$ such that $k' a + m' b = 1$, then $y \in Y_{kQ,m}$ implies that $B_Q(y,Y) = (k'a + m'b) B_Q(y,Y) \subset m'\Z$, thus $y \in Y_{Q,m'}$.

Accordingly, $X_{kQ, m} := \left\{x \in X \otimes \Q: \lrangle{x, Y_{kQ,m}} \subset \Z \right\}$ equals $X_{Q,m'}$. Next, for every root $\phi$ we have
\[ n_\phi[kQ, m] := \dfrac{m}{\text{gcd}(m, kQ(\check{\phi}))} = \dfrac{m'}{\text{gcd}(m', k'Q(\check{\phi}))} = n_\phi[Q, m']. \]
Hence the modified roots/coroots for $kQ, m$ are the same as those for $Q, m'$. It also follows that $Y_{kQ, m}^\text{sc} = Y_{Q, m'}^\text{sc}$. Write $\tilde{Z}^\vee := \Hom(Y_{Q,m'}/Y^\text{sc}_{Q,m'}, \CC^\times)$.

\begin{lemma}\label{prop:rescale-dual}
	The dual groups $\tilde{G}^\vee[kQ, m]$ and $\tilde{G}^\vee[Q, m']$ are defined by the same based root datum with $\Gamma_F$-action; the corresponding $\tau_{kQ, m}, \tau_{Q, m'}: \bmu_2 \to \tilde{Z}^\vee$ are also equal.
\end{lemma}
\begin{proof}
	We only need to show $\tau_{kQ,m} = \tau_{Q,m'}$. By construction, they are dual to the rows of the commutative diagram below (see \cite[\S 2.2]{Weis18}).
	\[\begin{tikzcd}[row sep=small]
		Y_{kQ,m}/Y_{kQ,m}^\text{sc} \arrow[twoheadrightarrow]{r} \arrow[-, double equal sign distance]{d} & \left( Y_{kQ,m}/Y_{kQ,m}^\text{sc}\right)\big/m(\cdots) \arrow{r}{m^{-1}kQ} & \frac{1}{2}\Z \big/ \Z \\
		Y_{Q,m'}/Y_{Q,m'}^\text{sc} \arrow[twoheadrightarrow]{r} & \left( Y_{Q,m'}/Y_{Q,m'}^\text{sc}\right)\big/m'(\cdots) \arrow[twoheadrightarrow]{u} \arrow{r}[swap]{m'^{-1}Q} & \frac{1}{2}\Z \big/ \Z \arrow{u}[swap]{k'}
	\end{tikzcd}\] \index{tauQm@$\tau_{Q,m}$}
	We contend that the two arrows $Y_{\cdots}/Y_{\cdots}^\text{sc} \to \frac{1}{2}\Z / \Z$ are equal. We may assume $k'$ even; $Y_{kQ,m}/Y_{kQ,m}^\text{sc} \to \frac{1}{2}\Z / \Z$ is then trivial. In this case $m'$ must be odd, so $Y_{Q,m'}/Y_{Q,m'}^\text{sc} \to \frac{1}{2}\Z / \Z$ is trivial as well since it factorizes through an $m'$-torsion group.
\end{proof}
As a consequence, the extensions $\tau_{kQ, m,*}(\widetilde{\Gamma_F})$ and $\tau_{Q, m', *}(\widetilde{\Gamma_F})$ of $\Gamma_F$ by $\tilde{Z}^\vee$ (see \eqref{eqn:metaGalois}) are also equal. The next step is to compare the gerbes $\mathsf{E}_\epsilon(\tilde{G}[kQ, m])$ and $\mathsf{E}_{\epsilon'}(\tilde{G}[Q, m'])$. To $E$ and $kE$ are associated $\mathcal{D}[Q]$, $\mathcal{D}[kQ]$ from \eqref{eqn:D-extension}, respectively; their pull-backs to $Y_{Q,m'} = Y_{kQ, m}$ yield $\mathcal{D}_{Q, m'}$ and $\mathcal{D}_{kQ, m}$. The following is a quick replacement of \cite[Assumption 3.1]{Weis18}. It enables us to use the constructions in \cite[\S\S 3.1---3.3]{Weis18}.

\begin{hypothesis}\label{hyp:D-comm}
	The group $\mathcal{D}_{Q,m'}$ is commutative.
\end{hypothesis}
By construction, $\mathcal{D}[kQ]$ (resp. $\mathcal{D}_{kQ,m}$) is the $k$-fold Baer sum of $\mathcal{D}[Q]$ (resp. of $\mathcal{D}_{Q,m'}$), thus is commutative as well.

\begin{example}
	When $G = \Sp(W)$, we know $B_Q$ is even-valued on $Y$: it suffices to consider the $Q$ in \eqref{eqn:Y-Sp}. Hence the commutator formula of \cite[Proposition 3.11]{BD01} implies the commutativity of $\mathcal{D}[Q]$. 
\end{example}

We tabulate some functorial operations from \cite[\S 19.3]{Weis18}. In what follows, $A$, $B$, etc. stand for sheaves of commutative groups over $S_{\text{ét}}$ where $S$ is some scheme; more generally one can work in a $1$-topos. For all $m \in \Z \smallsetminus \{0\}$ in view, we assume $A \xrightarrow{m} A$ is an epimorphism and set $A_{[m]} := \Ker(m)$. Torsors will be manipulated ``set-theoretically'' below, as justified by the usual formal apparatus, eg. \cite[Exp VII, I.1.2.1]{SGA7-1}. Likewise, we shall forget topology when talking about gerbes (generally in a $(2,1)$-topos) and treat them as groupoids.

\begin{itemize}
	\item Let $\varphi: A \to B$ be a homomorphism of sheaves of groups. The push-out of an $A$-torsor $\mathcal{P}$ under $\varphi$ will be written as $\varphi_* \mathcal{P} = B \wedge^{A, \varphi} \mathcal{P}$, whose elements are expressed as $b \wedge p$, with $b \wedge ap = b \varphi(a) \wedge p$ as usual. Let $\mathcal{P}$ be a $A$-torsor and $\mathcal{Q}$ be a $B$-torsor; an equivariant morphism $\mathcal{P} \xrightarrow{f} \mathcal{Q}$ relative to $\varphi$ is the same as a morphism $B \wedge^{A,\varphi} \mathcal{P} \to \mathcal{Q}$ of $B$-torsors (covering $\identity_B$), which sends $b \wedge p \mapsto b f(p)$.
	\item Let $m \in \Z_{\geq 1}$. The gerbe $\sqrt[m]{\mathcal{P}}$ banded by $A_{[m]}$ is defined as follows. Its objects are morphisms $\mathcal{H} \xrightarrow{f} \mathcal{P}$ of torsors, equivariant relative to $m: A \to A$. Its morphisms are commutative diagrams
	$\begin{tikzcd}[column sep=small]
		\mathcal{H} \arrow{r}{g} \arrow[bend right]{rr}[swap]{f} & \mathcal{H}' \arrow{r}{f'} & \mathcal{P}
	\end{tikzcd}$ where $g$ is a morphism of $A$-torsors, in particular $\Aut(\mathcal{H}, f) = A_{[m]}$.

	\item Consider a commutative central extension $C \hookrightarrow \mathcal{E} \twoheadrightarrow Z$ over $S_{\text{ét}}$. Take $A := \Hom(Z, C)$ (internal Hom). The splittings of $\mathcal{E} \to Z$ form an $A$-torsor, denoted by $\text{Spl}(\mathcal{E})$. Let $k \in \Z$. Every splitting $s$ for $\mathcal{E}$ gives rise to a splitting $1 \wedge s$ for $k_* \mathcal{E}$. Thus we obtain an isomorphism of $A$-torsors
\[ k_* \text{Spl}(\mathcal{E}) \longrightarrow \text{Spl}(k_* \mathcal{E}), \quad s \mapsto 1 \wedge s. \]
The $k_*$ on the left-hand side is required to obtain a morphism covering $\identity_A$.

	\item Suppose $m = dm'$. There is a functor of gerbes $\sqrt[m']{\mathcal{P}} \to \sqrt[m]{d_* \mathcal{P}}$ lying over $A_{[m']} \hookrightarrow A_{[m]}$, which maps an object $\mathcal{H} \xrightarrow{f} \mathcal{P}$ to $\mathcal{H} \xrightarrow{1 \wedge f} d_* \mathcal{P}$. Here $1 \wedge f: h \mapsto 1 \wedge f(h)$ is $m$-equivariant since $1 \wedge f(th) = 1 \wedge (t^{m'} f(h)) = t^m \wedge f(h)$ for all $t \in A$, $h \in \mathcal{H}$. The definition on morphisms is clear, and it is readily seen to be $A_{[m']} \hookrightarrow A_{[m]}$ on automorphisms.
	
	\item Suppose $k', m'$ are coprime, $m' \in \Z_{\geq 1}$. There is a functor between gerbes $\sqrt[m']{\mathcal{P}} \to \sqrt[m']{k'_* \mathcal{P}}$ lying over $k':A_{[m']} \rightiso A_{[m']}$. Indeed, on objects we send $\mathcal{H} \xrightarrow{f} \mathcal{P}$ to $k'_* \mathcal{H} \xrightarrow{k'_* f} k'_* \mathcal{P}$, where $k'_*f(t \wedge h) = t^{m'} \wedge f(h)$. One readily verifies that $k'_* f$ is well-defined and $m'$-equivariant. The definition on morphisms is clear, and it sends every automorphism $t \in A_{[m']}$ to $t^{k'}$.
\end{itemize}

Reverting to our original problem, define
\begin{gather*}
	\hat{\mathcal{T}} := \Hom(Y_{Q,m'}, \Gm) \to \Hom(Y_{Q,m'}^\text{sc}, \Gm) =: \hat{\mathcal{T}}_\text{sc} \quad (F-\text{tori}), \\
	\tilde{\mathcal{T}}^\vee := \Hom(Y_{Q,m'}, \CC^\times) \to \Hom(Y^\text{sc}_{Q,m'}, \CC^\times) =: \tilde{\mathcal{T}}^\vee_\text{sc}, \quad (\CC-\text{tori}).
\end{gather*}
By taking $A = \hat{\mathcal{T}}$, $C = \Gm$ and $Z = Y_{Q,m'}$, we deduce $k_* \text{Spl}(\mathcal{D}_{Q,m'}) \rightiso \text{Spl}(\mathcal{D}_{kQ, m})$ and obtain functors of gerbes
\[\begin{tikzcd}[row sep=small]
	 \sqrt[m']{ \text{Spl}(\mathcal{D}_{Q,m'}) } \arrow{r} & \sqrt[m']{ k'_* \text{Spl}(\mathcal{D}_{Q,m'}) } \arrow{r} & \sqrt[m]{k_* \text{Spl}(\mathcal{D}_{Q,m'}) } \arrow{r} & \sqrt[m]{\text{Spl}(\mathcal{D}_{kQ, m})} \\
	 \hat{\mathcal{T}}_{[m']} \arrow{r}{k'}[swap]{\sim} & \hat{\mathcal{T}}_{[m']} \arrow[hookrightarrow]{r} & \hat{\mathcal{T}}_{[m]} \arrow{r}{\identity} & \hat{\mathcal{T}}_{[m]}.
\end{tikzcd}\]
In parallel, $k_* \text{Spl}(\mathcal{D}^\text{sc}_{Q,m'}) \rightiso \text{Spl}(\mathcal{D}^\text{sc}_{kQ, m})$ and there is a functor $\sqrt[m']{ \text{Spl}(\mathcal{D}^\text{sc}_{Q,m'}) } \to \sqrt[m]{\text{Spl}(\mathcal{D}^\text{sc}_{kQ, m})}$ lying over $\hat{\mathcal{T}}_{\text{sc}, [m']} \xrightarrow[\sim]{k'} \hat{\mathcal{T}}_{\text{sc}, [m']} \hookrightarrow \hat{\mathcal{T}}_{\text{sc}, [m]}$.

In \cite[\S 19.3.2]{Weis18}, the push-out of a gerbe is defined by pushing out the $\Hom$-torsors, thus we obtain $k'_* \sqrt[m']{ \text{Spl}(\mathcal{D}_{Q,m'}) } \to \sqrt[m]{\text{Spl}(\mathcal{D}_{kQ, m})}$ lying over $\hat{\mathcal{T}}_{[m']} \hookrightarrow \hat{\mathcal{T}}_{[m]}$; same for the sc-case. Pushing out by $\epsilon_*: \hat{\mathcal{T}}_{[m]} \rightiso \tilde{\mathcal{T}}^\vee_{[m]}$ and its sc-variant yield a diagram of gerbes and functors (cf. \cite[\S 3.1]{Weis18})
\begin{equation}\label{eqn:E-functors} \begin{tikzcd}[column sep=small]
	\mathsf{E}_{\epsilon'}(\tilde{T}[Q, m']) \arrow{r} \arrow{d}[swap]{\pi} & \mathsf{E}_\epsilon(\tilde{T}[kQ, m]) \arrow{d}{\pi} \\
	\mathsf{E}_{\epsilon'}^\text{sc}(\tilde{T}[Q, m']) \arrow{r} & \mathsf{E}^\text{sc}_\epsilon(\tilde{T}[kQ, m])
\end{tikzcd} \;\text{lying over}\; \begin{tikzcd}[column sep=small]
	\tilde{\mathcal{T}}^\vee_{[m']} \arrow{d} \arrow[hookrightarrow]{r} & \tilde{\mathcal{T}}^\vee_{[m]} \arrow{d} \\
	\tilde{\mathcal{T}}^\vee_{\text{sc}, [m']} \arrow[hookrightarrow]{r} & \tilde{\mathcal{T}}^\vee_{\text{sc}, [m]}
\end{tikzcd}\end{equation}
which is $2$-commutative. Here $\pi$ is induced by the natural arrows $\text{Spl}(\mathcal{D}_{Q,m'}) \to \text{Spl}(\mathcal{D}^\text{sc}_{Q, m'})$ and $\text{Spl}(\mathcal{D}_{kQ,m}) \to \text{Spl}(\mathcal{D}^\text{sc}_{kQ, m})$.

\begin{lemma}\label{prop:rescale-gerbe}
	Under Hypothesis \ref{hyp:D-comm}, the extensions \eqref{eqn:gerbe-pi1} attached to $(kQ,m,\epsilon)$ and $(Q,m',\epsilon')$ are canonically isomorphic.
\end{lemma}
\begin{proof}
	Every functor $\mathsf{E}_1 \to \mathsf{E}_2$ of gerbes induces a morphism $\pi_1^{\text{ét}}(\mathsf{E}_1) \to \pi_1^{\text{ét}}(\mathsf{E}_2)$ of group extensions by \cite[\S 19.4]{Weis18}. By definition, $\mathsf{E}_{\epsilon'}(\tilde{G}[Q, m'])$ is the ``gerbe of liftings'' $\pi^{-1}(\bm{w}[Q, m'])$ (see \cite[\S 19.3.3]{Weis18}) for the ``Whittaker object'' $\bm{w}[Q, m']$ in $\mathsf{E}_{\epsilon'}^\text{sc}(\tilde{T}[Q, m'])$. Similarly for $\mathsf{E}_\epsilon(\tilde{G}[kQ,m])$ and $\bm{w}[kQ, m]$. Both gerbes are banded by $\tilde{Z}^\vee$ and we want a functor in between. For reasons of functoriality, it suffices to show that $\bm{w}[Q, m'] \mapsto \bm{w}[kQ, m]$ under \eqref{eqn:E-functors}.
	
	By the identification made in Lemma \ref{prop:rescale-dual}, the $\hat{\mathcal{T}}_\text{sc}$-torsor $\text{Whit}$ of generic characters in \cite[\S 3.3]{Weis18} is the same for both $(kQ,m)$ and $(Q, m')$. Given each $\eta$ in $\text{Whit}$ (living possibly in some étale covering), we have the $\omega[Q, m'](\eta) \in \text{Spl}(\mathcal{D}_{Q, m'}^\text{sc})$ defined using the pinning, as discussed in \textit{loc.\ cit.}, \cite[\S 1.1]{Weis18b} and partly reviewed in \S\ref{sec:second-twist}. By construction, $\omega[kQ, m](\eta)$ is nothing but its image $1 \overset{k}{\wedge} \omega[Q, m'](\eta)$ in $\text{Spl}(\mathcal{D}_{kQ, m}^\text{sc}) \simeq k_* \text{Spl}(\mathcal{D}_{Q, m'}^\text{sc})$. Next, for every root $\phi$,
	\[ m_\phi[kQ, m] := \dfrac{kQ(\check{\phi})}{\text{gcd}(m, kQ(\check{\phi}))} = \dfrac{k'Q(\check{\phi})}{\text{gcd}(m', Q(\check{\phi}))} = k' m_\phi[Q, m']. \]
	Hence the endomorphism $\mu[kQ, m]$ of $\hat{\mathcal{T}}_\text{sc}$ in \cite[\S 3.3]{Weis18} equals $\mu[Q, m']^{k'}$. It is shown in \textit{loc. cit.} that $1 \wedge \eta \mapsto \omega[Q, m'](\eta)$ yields an $m'$-equivariant morphism $\mu[Q, m']_* \text{Whit} \to \text{Spl}(\mathcal{D}^\text{sc}_{Q,m'})$ of $\hat{\mathcal{T}}_\text{sc}$-torsors. This furnishes the Whittaker object $\bm{w}[Q, m']$ of $\sqrt[m']{\text{Spl}(\mathcal{D}^\text{sc}_{Q,m'})}$; the same holds for $\bm{w}[kQ, m]$. Let us study the image of $\bm{w}[Q, m']$ under the lower row of \eqref{eqn:E-functors} in stages.
	\[\begin{tikzcd}[row sep=small]
		\mu[Q,m']_* \text{Whit} \arrow{r}{m'} & \text{Spl}(\mathcal{D}^\text{sc}_{Q, m'}) & \omega[Q, m'] \arrow[mapsto]{d} \\
		\underbracket{k'_* \mu[Q,m']_*}_{= \mu[kQ, m]_*} \text{Whit} \arrow{r}{m'} & k'_* \text{Spl}(\mathcal{D}^\text{sc}_{Q, m'}) & 1 \wedge \eta \xmapsto{k'_* \omega[Q, m']} 1 \wedge^{k'} \omega[Q, m'](\eta) \arrow[mapsto]{d} \\
		\mu[kQ, m]_* \text{Whit} \arrow{r}{m'd} & \underbracket{d_* k'_*}_{\simeq k_*} \text{Spl}(\mathcal{D}^\text{sc}_{Q, m'}) \arrow{d}{\simeq} & 1 \wedge \eta \xmapsto{1 \overset{d}{\wedge} (k'_* \omega[Q, m'])} 1 \wedge^{d} 1 \wedge^{k'} \omega[Q, m'](\eta) \arrow[mapsto]{d}{\simeq} \\
		\mu[kQ, m]_* \text{Whit} \arrow{r}{m} & \text{Spl}(\mathcal{D}^\text{sc}_{kQ, m}) &  \left[ \eta \xmapsto{1 \wedge^k \omega[Q, m']} 1 \wedge^k \omega[Q, m'](\eta) \right] \simeq \omega[kQ, m]
	\end{tikzcd}\]
	Here $m'$, etc. over the arrows between torsors record the equivariance. This completes the proof.
\end{proof}

\begin{theorem}\label{prop:rescale-L}
	Under Hypothesis \ref{hyp:D-comm}, the $L$-groups $\Lgrp{\tilde{G}}[kQ, m]$ and $\Lgrp{\tilde{G}}[Q, m']$ are canonically isomorphic as extensions of $\Gamma_F$ by $\tilde{G}^\vee$.
\end{theorem}
\begin{proof}
	Combine Lemma \ref{prop:rescale-dual} and \ref{prop:rescale-gerbe}.
\end{proof}

\section{Construction of epipelagic supercuspidals for \texorpdfstring{$\widetilde{\Sp}(W)$}{tilde Sp(W)}} \label{sec:construction-rep}
Throughout this section, $F$ stands for a non-archimedean local field of residual characteristic $p$. From \S\ref{sec:compact-induction} onwards, we assume $p \neq 2$ and work with $G = \Sp(W)$; this excludes the bad and torsion primes \cite[I.4]{SS70} for $G$. The BD-covers $\tilde{G} \twoheadrightarrow G(F)$ will have degree $m$ with $p \nmid m$, which amounts to the \emph{tameness} of the cover, see \cite[\S 4]{GG}.

In \S\ref{sec:Adler-Spice}, stronger constraints on $F$ and $p$ will be imposed (Hypothesis \ref{hyp:p-large}).

\subsection{Generalities}\label{sec:MP-generalities}
For any reductive $F$-group $G$, denote by $\mathcal{B}(G, F)$ (resp. $\mathcal{B}^\text{red}(G, F)$) the enlarged Bruhat--Tits building (resp. the reduced one) of $G$. In what follows we assume $G$ splits over a tamely ramified extension. \index{Building@$\mathcal{B}(G, F)$}

Given $x \in \mathcal{B}^\text{red}(G,F)$, the Moy--Prasad filtration on $G(F)$ (resp. $\mathfrak{g}(F)$) is an increasing family of open compact subgroups (resp. of $\mathfrak{o}_F$-lattices), written as \index{G(F)_xr@$G(F)_x, G(F)_{x,r}$} \index{g(F)_xr@$\mathfrak{g}(F)_{x,r}$}
\begin{gather*}
	G(F)_{x, r}, \quad G(F)_{x, r+} := \bigcup_{s > r} G(F)_{x, s} \quad (r \geq 0), \\
	\mathfrak{g}(F)_{x,r}, \quad \mathfrak{g}(F)_{x, r+} := \bigcup_{s > r} \mathfrak{g}(F)_{x, s} \quad (r \in \R),
\end{gather*}
respectively. Conjugation by $g \in G(F)$ carries $G(F)_{x,r}$, $\mathfrak{g}(F)_{x,r}$ to $G(F)_{gx,r}$, $\mathfrak{g}(F)_{gx,r}$. Also set
\[ G(F)_{x, r:s} := G(F)_{x,r}/G(F)_{x, s}, \quad \mathfrak{g}(F)_{x, r:s} := \mathfrak{g}(F)_{x, r}/\mathfrak{g}(F)_{x,s} \]
whenever $s > r$. The meanings of $G(F)_{x, r:r+}$ and $\mathfrak{g}(F)_{x, r:r+}$ are clear. Also recall that $G(F)_{x,0}$ is the parahoric subgroup of $G(F)$ determined by $x$, and $G(F)_{x,0+}$ is its pro-unipotent radical; they are both contained in the stabilizer $G(F)_x := \Stab_{G(F)}(x)$. On the dual side, define the filtration
\begin{align*}
	\mathfrak{g}^*(F)_{x,-r} & := \left\{ X \in \mathfrak{g}^*(F) : \forall Y \in \mathfrak{g}(F)_{x,r+}, \;  v(\lrangle{X,Y}) > 0 \right\}, \\
	\mathfrak{g}^*(F)_{x, (-r)+} & := \bigcup_{s < r} \mathfrak{g}^*(F)_{x, -s}.
\end{align*} \index{g*F_xr@$\mathfrak{g}^*(F)_{x,r}$}
When $r > 0$, we have the $G(F)_{x,0:0+}$-invariant non-degenerate pairing $\mathfrak{g}^*(F)_{x, -r:(-r)+} \times \mathfrak{g}(F)_{x, r:r+} \to \kappa_F$, as well as the \emph{Moy--Prasad isomorphism} \cite{MP96} $\mathrm{MP}: G(F)_{x,r:r+} \simeq \mathfrak{g}(F)_{x, r:r+}$ between commutative groups; $\mathrm{MP}$ is $G(F)_x$-equivariant and identifies Pontryagin duals. One obtains the open subsets
\[ G(F)_r := \bigcup_x G(F)_{x,r}, \quad \mathfrak{g}(F)_r := \bigcup_x \mathfrak{g}(F)_{x,r}. \]

Let $S$ be a $F$-torus. Denote by $S(F)_b$ (resp. $S(F)_0$) the maximal bounded subgroup (resp. the parahoric subgroup) of $S(F)$. The Moy--Prasad filtration for $S$ reduces to
\begin{align*}
	S(F)_r & := \left\{ \gamma \in S(F)_0 : \forall \chi \in X^*(S_{\bar{F}}), \; v(\chi(\gamma)-1) \geq r \right\} \quad (r \geq 0), \\
	\mathfrak{s}(F)_r & := \left\{ X \in \mathfrak{s}(F) : \forall \chi \in X^*(S_{\bar{F}}), \; v(\dd\chi(X)) \geq r \right\} \quad (r \in \R).
\end{align*}
Here $v$ stands for the unique extension to $\bar{F}$ of the normalized valuation of $F$. Note that $v(\chi(\gamma)-1) \geq 0$ holds for all $\gamma \in S(F)_b$ by boundedness.

Denote the splitting field extension of $S$ by $L/F$ and let $e := e(L/F)$. We say $S$ is \emph{tame} if $L$ is, i.e. $p \nmid e$. We say $S$ is \emph{inertially anisotropic} if $X_*(S_{\bar{F}})^{I_F} = 0$, or equivalently if $S_{F^{\text{nr}}}$ is anisotropic. \index{inertially anisotropic}

Define the subgroups of $S(F)$
\begin{align*}
	S(F)_{p'} & := \left\{ x \in S(F)_b: \exists n, \; x^n = 1, \; p \nmid n \right\}, \\
	S(F)_\text{tu} & := \left\{ x \in S(F)_b: \lim_{k \to +\infty} x^{p^k} = 1 \right\}.
\end{align*}
Here $S(F)_\text{tu}$ consists of the \emph{topologically unipotent} elements in $S(F)$. By \cite[5.2]{Wa08}, we have the topological Jordan decomposition for bounded (or compact) elements \index{topological Jordan decomposition}
\begin{equation}\label{eqn:S-Jordan}
	S(F)_b = S(F)_{p'} \times S(F)_\text{tu}.
\end{equation}

Assume henceforth that $S$ is tame. It admits the ft-Néron model satisfying $S(\mathfrak{o}_F) = S(F)_b$. This Néron model coincides with all the other versions when $S$ is inertially anisotropic, see the explanations in \cite[p.50]{Kal13}; in that case
\[ S(F)_0 = S^\circ(\mathfrak{o}_F). \]
In what follows, $S$ will be a maximal $F$-torus of $G$ that is anisotropic modulo $Z_G$. It determines a point $x \in \mathcal{B}^\text{red}(G,F)$ by \cite[Remark 3]{Pra01}; this point being fixed by $N_G(S)(F)$, we deduce $S(F) \subset G(F)_x$. From \cite[Proposition 1.9.1]{Adl98} we have $\mathfrak{s}(F)_r = \mathfrak{g}(F)_{x,r} \cap \mathfrak{s}(F)$; same for the dual avatars $\mathfrak{s}^*(F)_r$ and $\mathfrak{g}^*(F)_{x,r}$.

Let $\Omega(G, S)(\bar{F})$ act on $Y := X_*(S_{\bar{F}})$ and $X := X^*(S_{\bar{F}})$. Following Springer \cite{Spr74}, we say an element $w \in \Omega(G,S)(\bar{F})$ is
\begin{itemize}
	\item \emph{regular}, if some eigenvector of $w$ acting on $Y \dotimes{\Z} \CC$ has trivial stabilizer under $\Omega(G,S)(\bar{F})$;
	\item \emph{elliptic}, if $\det( w-1 | Y \otimes \R) \neq 0$.
\end{itemize}
The torus $S$ gives rise to continuous $\Gamma_F$-actions on the $\Z$-modules $Y$ and $X$.

\begin{definition}[{\cite[Conditions 3.3]{Kal15}}]\label{def:type-ER} \index{type (ER)}
	Assume $Z_G^\circ$ is anisotropic. We say that a maximal $F$-torus $S \subset G$ is of \emph{type (ER)} if
	\begin{enumerate}[(a)]
		\item $S$ is tame, and
		\item the $I_F$-action on $Y$ is generated by an elliptic regular element of $\Omega(G,S)(\bar{F})$.
	\end{enumerate}
	It follows that $S$ is inertially anisotropic, and the foregoing discussions are applicable. When the $I_F$-action is generated by a Coxeter element, such $S$ is said to be of type (C) in \cite{Kal13}.
\end{definition}

\begin{lemma}\label{prop:S_0}
	Suppose $S$ is an inertially anisotropic $F$-torus. We have
	\[ S(F)_0 = S(F)_{0+} = S(F)_{1/e} = S(F)_\mathrm{tu}. \]
	Consequently, \eqref{eqn:S-Jordan} becomes $S(F) = S(F)_{p'} \times S(F)_{1/e}$ and $S(F)_{1/e}$ is the pro-$p$ part of $S(F)$.
\end{lemma}
\begin{proof}
	Clearly we have $S(F)_0 \supset S(F)_{0+} \supset S(F)_{1/e}$. If $\gamma \in S(F)_\text{tu}$ then $v(\chi(\gamma)-1) \geq \frac{1}{e}$ for all $\chi$, as one sees over the splitting field, thus by \cite[(3.1.1)]{Kal19} we have $S(F)_{1/e} \supset S(F)_\text{tu}$ as well. It remains to show that $S(F)_0 \subset S(F)_{\text{tu}}$.
	
	By \cite[Lemma 3.1.3]{Kal19} we have $S(F)_0 = S(F)_{0+}$, since $Y^{I_F} = \{0\}$. Hence every element in $S(F)_0$ is pro-$p$ and must belong to $S(F)_\text{tu}$ as desired.
\end{proof}

\subsection{Tori in covers}\label{sec:covers-tori}
Assume $\text{char}(F) \neq 2$, $n \in \Z_{\geq 1}$ and specialize the formalism in \S\ref{sec:MP-generalities} to $G := \Sp(W)$ where $\dim_F W = 2n$. Let $S \subset G$ be a maximal $F$-torus and label the short coroots of $S_{\bar{F}} \subset G_{\bar{F}}$ by $\pm\check{\epsilon}_1, \ldots \pm\check{\epsilon}_n \in Y$ as in \S\ref{sec:Sp}. Given $w \in \Omega(G, S)(\bar{F})$, we decompose $\{ \pm\check{\epsilon}_i : 1 \leq i \leq n\}$ into $\lrangle{w}$-orbits. In parallel with Definition \ref{def:symmetric-orbits}, we say a $\lrangle{w}$-orbit $\mathcal{O}$ of short coroots is \emph{symmetric} if $\mathcal{O} = -\mathcal{O}$, otherwise \emph{asymmetric}.

\begin{lemma}\label{prop:2-torsion}
	Assume that $S$ is a tame and inertially anisotropic, for example when $S$ is of type (ER) in $G$. Then $x^2=1$ for all $x \in S(F)_{p'}$.
\end{lemma}
\begin{proof}
	It follows from Lemma \ref{prop:S_0} that $S(F)_{p'} \rightiso S(F)/S(F)_0$. It suffices to show that every continuous homomorphism $\theta: S(F) \to \CC^\times$ that is trivial on $S(F)_0$ satisfies $\theta^2=1$. Let $\check{S}$ be the $\CC$-torus dual to $S$. By \cite[Lemma 3.1.5]{Kal19}, the $L$-parameter $\phi \in H^1(\Weil_F, \check{S})$ of $\theta$ is inflated from $H^1(\Weil_F/I_F, \check{S}^{I_F})$. Hence it suffices to show that $\check{S}^{I_F}$ is a $2$-torsion group.
	
	We contend that the action of $I_F$ on $\check{S}$ contains an elliptic element $w$ from $\Omega(G, S)(\bar{F})$. Indeed, by writing $S_{\bar{F}} = g T_{\bar{F}} g^{-1}$ with $g \in G(\bar{F})$, where $T$ is a split maximal $F$-torus, we see that $\Gamma_F$ acts through $\Omega(G, S) = g\Omega(G,T) g^{-1}$. Since $\check{S} = \check{S}^{P_F}$ and $I_F/P_F$ is pro-cyclic, the action of $I_F$ is generated by some $w \in \Omega(G, S)(\bar{F})$, which must be elliptic since $S$ is inertially anisotropic.
	
	It remains to show that $\check{S}^{w=1}$ is a $2$-torsion group. By decomposing the short coroots into $\lrangle{w}$-orbits, ellipticity is equivalent to that that every orbit is symmetric; cf. the proof of Lemma \ref{prop:anisotropic-criterion}. It is now clear that $\check{S}^{w=1} = \bmu_2^{|\text{orbits}|}$.
\end{proof}

\begin{theorem}\label{prop:S-p'}
	Assume $p \neq 2$. Choose a parameter $(K, K^\sharp, \ldots)$ for a maximal $F$-torus $S \subset G=\Sp(2n)$ of type (ER) (see Definition \ref{def:type-ER}), $K = \prod_{i\in I} K_i$. Identify $S$ with $\prod_{i \in I} R_{K_i^\sharp /F}(K_i^1)$. Then
	\[ S(F)_{p'} = \prod_{i \in I} \{\pm 1\} \subset \prod_{i \in I} K_i^1. \]
	Furthermore, each $K_i/K_i^\sharp$ is a tamely ramified quadratic extension of fields.
\end{theorem}
\begin{proof}
	Since $p \neq 2$, we have $\{\pm 1\}^I \subset S(F)_{p'}$. The reverse inclusion follows from Lemma \ref{prop:2-torsion} since the $2$-torsion part of each $K_i^1$ is $\{\pm 1\}$.
	
	Each $K_i$ is a field since $S$ is anisotropic (Lemma \ref{prop:parameter-ani}). Let $q_i$ be the cardinality of the residue field of $K_i^\sharp$. If $K_i$ is unramified over $K_i^\sharp$, there will be an embedding $\Z/(q_i - 1)\Z \hookrightarrow K_i^1$ with image in $S(F)_{p'}$ since $p \mid q_i$. This contradicts the first part.
\end{proof}

In what follows, we fix $m \mid N_F$ with $p \nmid m$. Consider the BD-cover constructed in \S\ref{sec:BD-Sp}
\[ 1 \to \mu_m \to \tilde{G} \xrightarrow{\bm{p}} G(F) \to 1. \]
We fix an embedding $\epsilon: \mu_m \rightiso \bmu_m \subset \CC^\times$ and identify $\tilde{G}$ with a topological central extension of $G(F)$ by $\bmu_m$.

Let $S \subset G$ be a maximal $F$-torus of type (ER). Fix a parameter $(K, K^\sharp, \ldots)$ for $S$, with the usual decomposition $K_i = \prod_{i \in I} K_i$, etc. Write $\bm{p}: \tilde{S} \to S(F)$ for the pull-back of $\bm{p}$ to $S(F)$. Recall from Lemma \ref{prop:S_0} that $S(F) = S(F)_{p'} \times S(F)_{\text{tu}}$, and $S(F)_{\text{tu}} = S(F)_0 = S(F)_{0+}$.

\begin{lemma}\label{prop:splitting-pro-p}
	The covering $\tilde{S} \to S(F)$ splits uniquely over $S(F)_0$, and all elements from $S(F)_0$ are good in $\tilde{S}$.
\end{lemma}
\begin{proof}
	This follows from Lemma \ref{prop:pro-p-splitting} and $p \nmid m$, since $S(F)_0 = S(F)_{0+}$ is a pro-$p$-group.
\end{proof}
It remains to study the covering over $S(F)_{p'}$, whose structure is described by Theorem \ref{prop:S-p'}.

\begin{proposition}\label{prop:S-splitting-odd}
	When $m \notin 2\Z$, the covering $\tilde{S} \to S(F)$ splits uniquely.  
\end{proposition}
\begin{proof}
	In view of Lemma \ref{prop:splitting-pro-p}, the assertion follows from that $S(F)_{p'}$ is of $2$-torsion whilst $2 \nmid m$. An explicit section of $\bm{p}$ over $S(F)_{p'}$ is given by $\gamma \mapsto \tilde{\eta}^m$, where we take the $\eta \in S(F)_{p'}$ with $\eta^m = \gamma$, and $\tilde{\eta} \in \bm{p}^{-1}(\eta)$ is arbitrary.
\end{proof}

The case $m \in 2\Z$ requires more sophisticated constructions. As a first step, we prove the commutativity of $\tilde{S}$. Refined results will be given in \S\ref{sec:splittings-S}.
\begin{proposition}\label{prop:splitting-even-coarse}
	The group $\tilde{S}$ is commutative. In particular, all elements from $S_\mathrm{reg}(F)$ are good in $\tilde{G}$.
\end{proposition}
\begin{proof}
	Decompose $S = \prod_{i \in I} R_{K_i^\sharp / F} (K_i^1)$ using parameters. Recall from Theorem \ref{prop:G-T-reduction} that $S \subset G^S$ and $\bm{p}: \widetilde{G^S} \to G^S(F)$ is isomorphic to the contracted product of topological central extensions $\bmu_m \hookrightarrow \widetilde{\SL}(2, K_i^\sharp) \twoheadrightarrow \SL(2, K_i^\sharp)$. We are thus reduced the case $G = \SL(2)$, $F = K^\sharp$. By Lemma \ref{prop:splitting-pro-p}, it suffices to show that the preimage of $S(F)_{p'} = \{\pm 1\}$ is commutative.
	
	To show this, one may push-out the covering by $\bmu_m \hookrightarrow \CC^\times$. The result follows since any topological central extension $1 \to \CC^\times \to C \to \{\pm 1\} \to 1$ splits.
\end{proof}

We conclude that the irreducible $\epsilon$-genuine representations of $\tilde{S}$ are continuous characters. As $\epsilon$ is fixed throughout, we shall abbreviate $\epsilon$-genuine as genuine.

\subsection{Compact induction}\label{sec:compact-induction}
Assume $p \neq 2$. Fix an additive character $\Lambda: \kappa_F \to \CC^\times$. Let $G = \Sp(W)$ and $\bm{p}: \tilde{G} \twoheadrightarrow G(F)$ be as in \S\ref{sec:covers-tori}; thus $p \nmid m$. We begin by reviewing \cite[\S 3.2]{Kal15}.

Let $S \subset G$ be a tame anisotropic maximal $F$-torus, and let $x \in \mathcal{B}^\text{red}(G,F)$ be the point its determines. Let $L$ be the splitting field of $S_{F^\text{nr}}$, with residue field $\kappa_L$.

Take $r \in \R_{> 0}$. To a continuous character $\theta: S(F) \to \CC^\times$ such that $\theta|_{S(F)_{r+}} = 1$, we attach a $\kappa_L$-line $\ell_\theta \subset \mathfrak{g}^*(L)_{x,0:0+}$ as follows. The $\mathrm{MP}: S(F)_{r:r+} \simeq \mathfrak{s}_{r:r+}(F)$ and $\theta|_{S(F)_r}$ gives rise to a character $\theta_r$ of the $\kappa_F$-vector space $\mathfrak{s}(F)_{r:r+}$. Identify $\theta_r$ with the $Y^* \in \mathfrak{s}^*(F)_{-r:(-r)+}$ such that $\Lambda(\lrangle{Y^*, -}) = \theta_r$. As $S_L$ splits, there exists $z \in L^\times$ such that $zY^* \in \mathfrak{s}^*(L)_{0:0+}$, and the $\kappa_L$-line obtained from by $zY^*$ depends only on $\theta_r$. The line $\ell_\theta$ is obtained via the inclusion $\mathfrak{s}^* = (\mathfrak{g}^*)^S \hookrightarrow \mathfrak{g}^*$, which induces $\mathfrak{s}^*(L)_{0:0+} \hookrightarrow \mathfrak{g}^*(L)_{x, 0:0+}$.

Following \cite{Kal15}, call a continuous character $\theta: S(F) \to \CC^\times$ \emph{generic of depth $r$} if
\begin{enumerate}
	\item the restriction of $\theta$ to $S(F)_{r+}$ (resp. to $S(F)_r$) is trivial (resp. non-trivial);
	\item the $\kappa_L$-line $\ell_\theta \subset \mathfrak{g}^*(L)_{x,0:0+}$ is strongly regular semisimple, i.e. its stabilizer in $G(L)_{x, 0:0+}$ is a maximal $\kappa_L$-torus.
\end{enumerate}
This notion depends solely on $\theta|_{S(F)_r}$, and $r$ must be a jump for the filtration on $S(F)$. We proceed to adapt it to the cover $\bmu_m \hookrightarrow \tilde{G} \xrightarrow{\bm{p}} G(F)$ and $\tilde{S} \xrightarrow{\bm{p}} S(F)$, where $S \subset G$ is of type (ER). By Lemma \ref{prop:splitting-pro-p}, the covering splits uniquely over $S(F)_0 = S(F)_{1/e}$, where $e$ stands for the ramification degree of the splitting extension of $S$.

\begin{definition}\label{def:epipelagic-character}\index{epipelagic genuine character}
	Let $S \subset G$ be a maximal $F$-torus of type (ER). Call a continuous genuine character $\theta: \tilde{S} \to \CC^\times$ \emph{generic} of depth $r > 0$ if
	\begin{enumerate}
		\item the restriction of $\theta$ to $S(F)_{r+}$ (resp. to $S(F)_r$) via the unique splitting over $S(F)_0$ is trivial (resp. non-trivial);
		\item the $\kappa_L$-line $\ell_\theta$ attached to $\theta|_{S(F)_r}$ is strongly regular semisimple.
	\end{enumerate}
	Call a continuous genuine character $\theta: \tilde{S} \to \CC^\times$ \emph{epipelagic} if it is generic of depth $1/e$.
\end{definition}
By Lemma \ref{prop:S_0}, the first jump of the Moy--Prasad filtration on $S(F)$ occurs at $r=1/e$. Hence the minimal possible depth is attained by epipelagic characters. Assume $\theta: \tilde{S} \to \CC^\times$ to be genuine epipelagic. We proceed to construct supercuspidals as follows.

\begin{asparaenum}[1.]
	\item Decompose $\mathfrak{g} = \mathfrak{s} \oplus \mathfrak{n}$, where $\mathfrak{n}$ is the direct sum of nontrivial isotypic components under $S$. Then for all $r \geq 0$ we have
		\[ \mathfrak{g}(F)_{x,r} = \mathfrak{s}(F)_r \oplus \mathfrak{n}(F)_{x,r} \]
		where $\mathfrak{n}(F)_{x,r} := \mathfrak{g}(F)_{x,r} \cap \mathfrak{n}(F)$. Cf. \cite[Proposition 1.9.3]{Adl98}. From now onwards $r := 1/e$. The isomorphisms $\mathrm{MP}: \mathfrak{s}(F)_{r:r+} \simeq S(F)_{r:r+}$ and $\mathfrak{g}(F)_{x, r:r+} \simeq G(F)_{x, r:r+}$ allow us to extend $\theta|_{S(F)_r}$ to a character $\widehat{\theta_0}$ of the group $G(F)_{x, r:r+}$.
	\item Let $\mathsf{V}_{x,r} := \mathfrak{g}(F^\text{nr})_{x, r:r+}$, which is a $\overline{\kappa_F}$-vector space with a descent datum to $\kappa_F$ such that $\mathsf{V}_{x,r}(\kappa_F) = \mathfrak{g}(F)_{x, r:r+}$. Let $\lambda \in \check{\mathsf{V}}_{x,r}$ be the linear functional determined by
		\[ \Lambda(\lrangle{\lambda, -}) = \text{the composite}\; \mathsf{V}_{x,r}(\kappa_F) \xrightarrow{\text{MP}} G(F)_{x, r:r+} \xrightarrow{\widehat{\theta_0}} \CC^\times. \]
		Recall that $S(F) \subset G(F)_x$. By \cite[Proposition 3.4]{Kal15},
		\begin{equation}\label{eqn:lambda-stab}\begin{gathered}
			\Stab_{G(F)_x}(\lambda) = S(F) G(F)_{x, r}, \quad \text{and} \\
			\lambda \in \check{\mathsf{V}}_{x,r} \; \text{is a stable vector for the}\; G(F^\text{nr})_{x,0:0+} \text{-action}.
		\end{gathered}\end{equation}
		Stability here is understood in the sense of geometric invariant theory.
	\item By Lemma \ref{prop:pro-p-splitting}, $\bm{p}$ splits uniquely over $G(F)_{x,r}$. As $S(F)$ normalizes $G(F)_{x,r}$, one can define the genuine continuous character
		\begin{align*}
			\widehat{\theta}: \tilde{S} \cdot G(F)_{x,r} = \bm{p}^{-1}\left( \Stab_{G(F)_x}(\lambda) \right) & \longrightarrow \CC^\times \\
			\tilde{\gamma} \cdot g & \longmapsto \theta(\tilde{\gamma}) \widehat{\theta_0}(g)
		\end{align*}
		where $\tilde{\gamma} \in \tilde{S}$, $g \in G(F)_{x,r}$. By Lemma \ref{prop:S_0} we may also write $\widehat{\theta}$ as
		\[ \widehat{\theta}: \widetilde{S(F)_{p'}} \ltimes G(F)_{x,r} \longrightarrow \CC^\times \]
		where $\widetilde{S(F)_{p'}} := \bm{p}^{-1}(S(F)_{p'})$. Notice that $\tilde{S} \cdot G(F)_{x,r}$ is a compact open subgroup of $\tilde{G}$.
	\item Now we can build a smooth representation from $(\tilde{S}, \theta)$ via compact induction
		\begin{equation}\label{eqn:c-Ind}
			\pi_{\tilde{S}, \theta} := \cInd^{\tilde{G}}_{\tilde{S} \cdot G(F)_{x, 1/e}} (\widehat{\theta}).
		\end{equation}
		It is clear that $\pi_{\tilde{S}, \theta}^{G(F)_{x, 1/e+}}$ contains the unrefined minimal $\mathsf{K}$-type $(G(F)_{x,r}, \lambda)$; see \cite[3.4]{MP96} for this notion.
\end{asparaenum}

The construction of $\pi_{\tilde{S}, \theta}$ is a covering version of \cite{RY14}, therefore can be baptized as a \emph{genuine epipelagic supercuspidal representation} for the following reason. \index{representation!epipelagic supercuspidal}
\begin{theorem}\label{prop:epipelagic-supercuspidal}
	Let $\theta: \tilde{S} \to \CC^\times$ be an epipelagic genuine character. The representation $\pi := \pi_{\tilde{S}, \theta}$ in \eqref{eqn:c-Ind} is irreducible and supercuspidal. When $S$ and $\theta|_{S(F)_{1/e}}$ are fixed, different choices of $\theta|_{\widetilde{S(F)_{p'}}}$ lead to non-isomorphic $\pi$.
\end{theorem}
\begin{proof}
	Since $\widehat{\theta}$ is genuine, so is $\pi$. Irreducibility and the determination of $\theta|_{\widetilde{S(F)_{p'}}}$ follow by plugging these data into \cite[\S 2.1]{RY14} with $H := \widetilde{G(F)_x}$, $J := G(F)_{x, 1/e}$ and $A_\theta \simeq \widetilde{S(F)_{p'}}$; see also the Remark 1 in \textit{loc. cit.} More precisely, let $\theta_\lambda$ be the character of $S(F)_{1/e}$ attached to $\lambda \in \check{\mathsf{V}}_{x,r}$; the key ingredient \cite[(2.6)]{RY14} is the property that for all $g \in G(F)$,
	\[ \theta_\lambda|_{J \cap g^{-1} Jg} = \theta_\lambda \circ \Ad(g) |_{J \cap g^{-1}Jg} \implies g \in G(F)_x, \]
	which has nothing to do with coverings, since the splittings of $\bm{p}$ over subgroups of $J$ are unique.

	The cuspidality of $\pi$ is well-known, see for instance \cite[11.4]{BH06}.
\end{proof}

\begin{remark}\label{rem:Yu-data}
	The construction of $\pi = \pi_{\tilde{S}, \theta}$ is also a special case of Yu's construction \cite{Yu01} of tame supercuspidals, adapted to the case of coverings with $p \nmid m$, attached to the following quintuple (with length $d=1$)
	\begin{align*}
		\vec{G} & := (G^0 = S, \; G^1 = G), \; \text{or rather their preimages in}\; \tilde{G}, \quad \text{(i.e. a toral datum)} \\
		x & \in \mathcal{B}(S,F) \hookrightarrow \mathcal{B}^\text{red}(G,F), \;\text{cf. \cite{Pra01}}, \\
		\vec{r} & := \left( r_0 = \frac{1}{e}, \; r_1 = \frac{1}{e} \right), \\
		\rho & := \theta|_{\widetilde{S(F)_{p'}}}, \; \text{inflated to a genuine character of}\; \tilde{S} = \widetilde{S(F)_{p'}} \times S(F)_{1/e}, \\
		\vec{\phi} & := (\phi_0, \; \phi_1 = 1), \quad \phi_0 := \theta|_{S(F)_{1/e}}, \; \text{inflated to a character of}\; S(F) = S(F)_{p'} \times S(F)_{1/e}.
	\end{align*}
	See \cite[\S 3]{Yu01} as well as the explanations in \cite[\S 3.1]{HM08}; in the notation of the latter reference we have $\pi_{-1} = \rho$ and $\pi_0 = \rho\phi_0 = \theta$. The choice of $\vec{\phi}$ and $\rho$ is not unique, however.
\end{remark}

\begin{remark}\label{rem:refactorization}
	We say $(\tilde{S}_1, \theta_1)$ and $(\tilde{S}_2, \theta_2)$ are conjugate if there exists $g \in G(F)$ such that $\tilde{S}_2 = g \tilde{S}_1 g^{-1}$ and $\theta_2 \circ \Ad(g) = \theta_1$, where $(\tilde{S}_i, \theta_i)$ stands for maximal tori and epipelagic genuine characters ($i=1,2$). In view of the case of reductive groups \cite[Fact 3.8]{Kal15}, one expects that $\pi_{\tilde{S}_1, \theta_1} \simeq \pi_{\tilde{S}_2, \theta_2}$ if and only if $(\tilde{S}_1, \theta_1)$ and $(\tilde{S}_2, \theta_2)$ are conjugate. The ``if'' direction is evident. To establish the other direction, one has to address the issue of uniqueness of Yu's data (Remark \ref{rem:Yu-data}) for $\pi$.
	I
	For reductive groups, this uniqueness follows from Hakim--Murnaghan theory \cite[Corollary 6.10]{HM08}. Their results should carry over to $\tilde{G}$ and its maximal tori of type (ER), and the ``only if'' part would follow. Cf. the discussions in \cite[\S 3.5]{LM18}, especially for the case $m=2$. This is expected to follow from an ongoing project of Ju-Lee Kim and Wee Teck Gan.
\end{remark}

\subsection{The Adler--Spice character formula}\label{sec:Adler-Spice}
Keep the formalism of \S\ref{sec:compact-induction}. We also fix an additive character $\xi: F \to \CC^\times$ which restricts to an additive character $\Lambda: \kappa_F \to \CC^\times$. The assumptions below are to be imposed.
\begin{hypothesis}\label{hyp:p-large}
	Assume that
	\begin{enumerate}[\bfseries P.1]
		\item $\text{char}(F)=0$.
		\item $p \geq (2+e(F/\Q_p)) n(G)$, where $n(G)$ is the dimension of some faithful $F$-rational representation of $G$; in particular $p > 2$.
		\item $p \nmid m$ as usual, where $\mu_m = \Ker(\bm{p}: \tilde{G} \to G(F))$.
	\end{enumerate}
	We remark that \textbf{P.2} is required for the exponential and logarithm maps in the character formula, see \cite[Appendix B]{DR09}. The assumptions above make sense for any group $G$.
\end{hypothesis}

Let $\pi = \pi_{\tilde{S}, \theta}$ be the genuine irreducible supercuspidal representation constructed in Theorem \ref{prop:epipelagic-supercuspidal}. Denote its character by $\Theta_\pi$. This is a $G(F)$-invariant distribution $f \mapsto \Theta_\pi(f) = \Tr(\pi(f))$ on $\tilde{G}$. As explained in \cite{Li12b}, $\Theta_\pi$ is represented by a locally integrable genuine function on $\tilde{G}$, which is locally constant over $\tilde{G}_{\text{reg}}$ and is independent of Haar measure.

\begin{notation}\index{topological Jordan decomposition}
	Suppose $\tilde{\gamma} \in \tilde{G}$ is a compact element, i.e. contained in a compact subgroup. The topological Jordan decomposition (see the discussions in \S\ref{sec:MP-generalities}) also applies to $\tilde{G}$: it yields a unique expression
	\[ \tilde{\gamma} = \tilde{\gamma}_0 \tilde{\gamma}_{>0} = \tilde{\gamma}_{>0} \tilde{\gamma}_0 \]
	with $\tilde{\gamma}_0$ of finite order prime to $p$, and $\tilde{\gamma}_{>0}$ is topologically unipotent. Taking images by $\bm{p}$ yields the topological Jordan decomposition $\gamma = \gamma_0 \gamma_{>0} = \gamma_{>0} \gamma_0$ in $G(F)$. The following easy result says that we can lift $\gamma_{>0}$ uniquely to $\tilde{G}$, which is exactly $\tilde{\gamma}_{>0}$. Hence we can simply write $\tilde{\gamma} = \tilde{\gamma} \gamma_{>0} = \gamma_{>0} \tilde{\gamma}$ for the decomposition of $\tilde{\gamma}$.
\end{notation}

\begin{lemma}
	Let $\tilde{G} \xrightarrow{\bm{p}} G(F)$ be any topological covering with $\Ker(p) = \bmu_m$ and $p \nmid m$. Given a topologically unipotent element $\gamma_{>0}$ of $G(F)$, there exists a unique $\tilde{\gamma}_{>0} \in \bm{p}^{-1}(\gamma_{>0})$ that is topologically unipotent.
\end{lemma}
\begin{proof}
	For the uniqueness, note that if $\tilde{\gamma}_{>0}$ and $\noyau\tilde{\gamma}_{>0}$ are both topologically unipotent, where $\noyau \in \bmu_m$, then $\noyau^{p^k} \to 1$, which implies $\noyau=1$ since $p \nmid m$.
	
	To show existence, take the topological Jordan decomposition of any preimage of $\gamma_{>0}$, say $\tilde{\eta} \tilde{\gamma}_{>0} = \tilde{\gamma}_{>0} \tilde{\eta}$. Since its image is topologically unipotent, we must have $\bm{p}(\tilde{\eta})=1$ and $\tilde{\gamma}_{>0}$ is the required lifting of $\gamma_{>0}$.
\end{proof}
Also note that $\gamma_0$ is semi-simple and $\gamma_{>0} \in J_{\text{reg}}(F)$, where $J := G_{\gamma_0}$.

It is shown in \cite{D76} that the character of any supercuspidal representation of $\tilde{G}$ is supported on compact elements. Moreover, $\Theta_\pi|_{\tilde{G}_{\text{reg}}}$ vanishes off the good locus since it is an invariant genuine function. Let us state the Adler--Spice character formula \cite[Theorem 7.1]{AS09} for $\pi$, rephrased as in \cite[(6.1)]{Kal15}; the only difference is that we work with covering groups. Several definitions are in order.

\begin{itemize}
	\item Condition \textbf{P.2} implies that we have the exponential map $\mathfrak{g}(F^\text{nr})_{0+} \to G(F^\text{nr})_{0+}$ that is a $G(F^\text{nr})$ and $\Frob$-equivariant homeomorphism, whose inverse we denote by $\log$; see \cite[p.57]{Kal13} or \cite[Appendix B]{DR09}.
	\item Now take $Y \in \mathfrak{s}^*(F)_{-1/e}$ satisfying $\theta \circ \exp = \xi(\lrangle{Y, \cdot}): \mathfrak{s}(F)_{1/e} \to \CC^\times$. Only the coset $Y + \mathfrak{s}^*(F)_0$ is well-defined.
	\item It is sometimes useful to identify $\mathfrak{g}$, $\mathfrak{g}^*$ and $\mathfrak{s}$, $\mathfrak{s}^*$. To achieve this, we use a non-degenerate invariant bilinear form $B_{\mathfrak{g}}$ on $\mathfrak{g}$ such that for some (thus for all) maximal $F$-torus $T$ and every $H_\alpha := \dd\check{\alpha}(1)$ over $\bar{F}$, we have $v\left( B_{\mathfrak{g}}(H_\alpha, H_\alpha)\right) = 0$. For classical groups $G$ inside $\GL(W)$, such as $G = \Sp(W)$, we follow \cite[\S 2.1.1]{LM18} to take
	\begin{equation}\label{eqn:B_g-LMS}
		\mathbb{B}(X_1, X_2) := \frac{1}{2} \Tr\left( X_1 \cdot X_2 | W \right), \quad X_1, X_2 \in \mathfrak{g}.
	\end{equation} \index{Btr@$\mathbb{B}$}
	This corrects the earlier \cite[Definition 6.1.1]{LMS16} by a sign, as kindly explained to the author by H.\ Y.\ Loke (private communication).

	Using $\mathbb{B}$, one can identify the dual of $\mathfrak{g}(F)_{x, t:t+}$ with $\mathfrak{g}(F)_{x, (-t):(-t)+}$ for all $x$ and $t \in \R$, same for $\mathfrak{s}$; see \cite[Lemma A.1.1]{DR09}. In particular, we may view $Y$ as an element of $\mathfrak{s}(F)_{-1/e}$. The following result guarantees the regularity of $Y$ in $\mathfrak{g}$.
\end{itemize}

\begin{proposition}\label{prop:stable-good}
	Let $v_{\bar{F}}$ be the valuation of $\bar{F}$ extending $v_F$. For $Y$ as above, we have $v_{\bar{F}}(\dd\alpha(Y)) = -\frac{1}{e}$ for any root $\alpha$ of $S_{\bar{F}} \subset G_{\bar{F}}$. In particular $Y \in \mathfrak{s}(F)_{-1/e} \cap \mathfrak{s}_{\mathrm{reg}}(F)$, and each eigenvalue $\lambda \in \bar{F}$ of $Y$ (as an element of $\syp(W) \subset \gl(W)$) satisfies $v_{\bar{F}}(\lambda) = -\frac{1}{e}$.
\end{proposition}
\begin{proof}
	In view of \eqref{eqn:lambda-stab}, this is just \cite[Lemma 7.3.1]{LMS16} or \cite[Lemma 3.2]{Kal15}. To deduce the assertion on $v_{\bar{F}}(\lambda)$, consider the long roots $\alpha$.
\end{proof}

\begin{notation}
	For any reductive $F$-group $J$, a maximal torus $S \subset J$ together with $Z \in \mathfrak{s}^*_\text{reg}(F)$ (i.e. $J_Z = S$), we define the unnormalized orbital integral
	\[ \mu^J_Z: f^* \longmapsto \int_{S(F) \backslash J(F)} f^*(g^{-1} Z g) \dd g, \quad f^* \in C^\infty_c(\mathfrak{j}^*(F)) \]
	and its Fourier transform
	\[ \widehat{\mu^J_Z}(f) = \mu^J_Z(\hat{f}), \quad \hat{f}(X) = \int_{\mathfrak{j}(F)} f(X') \xi\lrangle{X', X} \dd X' \]
	for all $f \in C^\infty_c(\mathfrak{j}(F))$; here we adopt the Haar measures in \cite[\S 4.2]{Kal19}. Set
	\begin{align*}
		D^J(Z) & := \prod_{\substack{\alpha: \text{root} \\ \dd\check{\alpha}(Z) \neq 0 }} \dd\check{\alpha}(Z), \\
		\hat{\iota}^J(Z,X) & := |D^J(Z)|^{\frac{1}{2}} \cdot |D^J(X)|^{\frac{1}{2}} \cdot \widehat{\mu^J_Z}(X), \quad (Z, X) \in \mathfrak{j}^*_\text{reg}(F) \times \mathfrak{j}_\text{reg}(F).
	\end{align*}
	This is well-defined since $\widehat{\mu^J_Z}$ is representable by a locally integrable function over $\mathfrak{j}(F)$, smooth over $\mathfrak{j}_\text{reg}(F)$ by \cite{HC99}. It is $J(F)$-invariant in both variables.
\end{notation}

\begin{theorem}[Cf. {\cite[(6.1)]{Kal15}} or {\cite[\S 4.3]{Kal19}} ]\label{prop:Adler-Spice}
	Let $\tilde{\gamma} \in \tilde{G}_{\mathrm{reg}}$ be a compact element with topological Jordan decomposition $\tilde{\gamma} = \tilde{\gamma}_0 \gamma_{>0}$. Write $\gamma, \gamma_0$ for their images in $G(F)$ and put $J := G_{\gamma_0}$. Then
	\[ |D^G(\gamma)|^{\frac{1}{2}} \Theta_\pi(\tilde{\gamma}) = \sum_{\substack{ g \in J(F) \backslash G(F) / S(F) \\ g^{-1}\gamma_0 g \in S(F) }} \theta(g^{-1} \tilde{\gamma}_0 g) \hat{\iota}^J \left(gYg^{-1}, \log\gamma_{>0} \right). \]
\end{theorem}
\begin{proof}
	It is unrealistic to reproduce \cite{AS09} here, so we only indicate the following ingredients.
	\begin{enumerate}
		\item Since $S$ is a maximal torus,
			\[ g^{-1} \gamma_0 g \in S(F) \iff \Ad(g)(S) \subset G_{\gamma_0} = J. \]
			Thus $gYg^{-1} \in \mathfrak{j}^*_\text{reg}(F)$ and $\hat{\iota}^J(gYg^{-1}, \cdot)$ makes sense.
		\item There are no Weil representations in Yu's construction \cite{Yu01} for the epipelagic case (cf. Remark \ref{rem:Yu-data}), which simplifies enormously the character formula. In fact, the complex fourth roots of unity (summarized in \cite[\S 4.3]{Kal19}) in the formula of Adler--Spice disappear here.
		\item Since we are in the case of positive depth $r$, by Lemma \ref{prop:pro-p-splitting} $G(F)_{x,r}$ splits uniquely in $\tilde{G}$.
		\item In view of the description of Yu's data in Remark \ref{rem:Yu-data}, the normal approximation in \cite{AS09, Kal19} reduces to the topological Jordan decomposition, which works for $\tilde{G}$ as explained above.
		\item Basic results in harmonic analysis on coverings such as
		\begin{inparaenum}[(a)]
			\item semisimple descent,
			\item local integrability, smoothness and local expansions of $\Theta_\pi$, and
			\item properties of orbital integrals,
		\end{inparaenum}
		have been established in \cite{Li12b}.
		\item All elements from $S(F)_{p'}$ are good by Theorem \ref{prop:S-p'} + Remark \ref{rem:good-pm1}.
	\end{enumerate}
	The arguments should actually be simpler since the mock-exponential maps in \cite{AS09} are avoided by \textbf{P.2}.
\end{proof}

\section{The \texorpdfstring{$L$}{L}-packets}\label{sec:packets}
We shall work with a local field $F$ with $\text{char}(F) \neq 2$, and a symplectic $F$-vector space $(W, \lrangle{\cdot|\cdot})$ with $\dim_F W = 2n$. Set $G = \Sp(W)$ and let $\bm{p}: \tilde{G} \to G(F)$ be the BD-cover defined in \S\ref{sec:BD-Sp} with kernel $\mu_m$, where $m \mid N_F$. We also fix $\epsilon: \mu_m \rightiso \bmu_m \subset \CC^\times$ in order to talk about genuine representations of $\tilde{G}$, etc. When $m \equiv 2 \pmod 4$, we also fix an additive character $\psi: F \to \CC^\times$. The class of multiplicative $\shK_2$-torsors over $G$ under consideration is the same as in \S\ref{sec:BD-covers-Sp}. For the dual side, this limitation is justified by Theorem \ref{prop:rescale-L}.

From \S\ref{sec:inducing-data} onwards, $F$ will be non-archimedean of residual characteristic $p \nmid 2m$. In \S\ref{sec:stability}, the conditions from \S\ref{sec:Adler-Spice} will be in force.

\subsection{Convenient splittings}\label{sec:splittings-S}
The first task is to prescribe a preimage of $-1$ in the $K_2(F)$-torsor $E_G(F) \twoheadrightarrow G(F)$ using the symplectic form $\lrangle{\cdot|\cdot}$. Pick a symplectic basis $e_1, \ldots, e_n, e_{-n}, \ldots, e_{-1}$ of $W$ and define $B = TU$ as in \S\ref{sec:Sp}, with opposite $B^- = TU^-$. For every root $\alpha$ of $T$, let $U_\alpha$ be the root subgroup. The symplectic basis gives rise to a standard pinning for $G$: for each simple $\alpha$ we are given an isomorphism $x_\alpha: \Ga \rightiso U_\alpha$; concurrently we have $x_{-\alpha}: \Ga \rightiso U_{-\alpha}$, see \cite[Exp XXIII, 1.2]{SGA3-3}. They correspond to $X_{\pm\alpha} := \dd x_{\pm\alpha}(1)$ and $x_{-\alpha}$ is characterized by
\[ [X_\alpha, X_{-\alpha}] = H_\alpha := \dd\check{\alpha}(1); \]
see \cite[Exp XX, Corollaire 2.11]{SGA3-3}. Now set
\[ w_\alpha(t) := x_\alpha(t) x_{-\alpha}(-t^{-1}) x_\alpha(t), \quad t \in \Gm; \]
this gives a representative of the root reflection with respect to $\alpha$. Recall that $E_G \to G$ splits canonically over $U$ and $U^-$ by \cite[Proposition 3.2]{BD01}. Hence one may view $w_\alpha(t)$ as elements of $E_G(F)$ for all $t \in F$. So it makes sense on the level of $E_G$ to define
\begin{equation}\label{eqn:h_alpha}
	h_\alpha(t) := w_\alpha(t) w_\alpha(-1).
\end{equation}

We will only need to deal with positive long roots $\alpha = 2\epsilon_i$. Since it involves only $e_{\pm i}$ in the symplectic basis, one can calculate inside $\SL(2)$ to see that $H_\alpha = \twomatrix{1}{}{}{-1}$ and
\begin{align*}
	x_\alpha(t) = \twobigmatrix{1}{t}{}{1}, & \quad x_{-\alpha}(t) = \twobigmatrix{1}{}{t}{1}, \\
	w_\alpha(t) = \twobigmatrix{}{t}{-t^{-1}}{}, & \quad h_\alpha(t) = \twobigmatrix{t}{}{}{t^{-1}}.
\end{align*}
Note that if we pass to the basis $e_{-1}, \ldots, e_{-n}, e_n, \ldots, e_1$ by conjugation in $G(F)$, namely by $\twomatrix{}{-1}{1}{}$ in copies of $\SL(2)$, then $x_\alpha(t) = \twomatrix{1}{t}{}{1}$ will become $x_{-\alpha}(-t) = \twomatrix{1}{}{-t}{1}$.

\begin{lemma}\label{prop:minus-1-refined}
	Denote by $\alpha$ a positive long root of $T$.
	\begin{enumerate}[(i)]
		\item We have $w_\alpha(1) = w_{-\alpha}(-1)$, $w_\alpha(1) w_\alpha(-1) = 1$ and $h_\alpha(-1) = w_\alpha(-1)^2 = w_\alpha(1)^{-2}$ in $E_G(F)$.
		\item When $\alpha$ ranges over the positive long roots of $T$, the elements $h_\alpha(-1)$ form a commuting family in $E_G(F)$. Their product
			\[ \overline{-1} := \prod_{\substack{\alpha: \text{long root} \\ \alpha > 0 }} h_\alpha(-1) \]
			is a preimage of $-1 \in G(F)$.
		\item The element $\overline{-1}$ depends only on $(W, \lrangle{\cdot|\cdot})$, not on the choice of symplectic bases.
	\end{enumerate}
\end{lemma}
\begin{proof}
	For (i), notice that the elements $e^+ := x_\alpha(1)$, $e^- := x_{-\alpha}(-1)$ and $\nu := e^+ e^- e^+$ of $G(F)$ are in Tits trijection in the sense of \cite[11.1]{BD01}. It is shown in \cite[p.73]{BD01} that $e^+ e^- e^+ = e^- e^+ e^-$ holds on the level of $E_G(F)$, which amounts to $w_\alpha(1) = w_{-\alpha}(-1)$. The equation $w_\alpha(1)w_\alpha(-1)=1$ results from the definitions. The third equality in (i) follows immediately.
	
	For (ii), apply Lemma \ref{prop:restriction-W_i} to $W = \bigoplus_{i=1}^n \lrangle{e_i, e_{-i}}$ to deduce commutativity. It follows that $\overline{-1}$ is well-defined and maps to $-1 \in G(F)$.

	The symplectic bases of $W$ form a single $G(F)$-orbit; changing them amounts to replacing $\overline{-1}$ by a $G(F)$-conjugate. Proposition \ref{prop:BD-adjoint-action} asserts that $\overline{-1}$ is central in $E_G(F)$. This proves (iii).
\end{proof}

\begin{lemma}
	The element $\overline{-1}$ in $E_G(F)$ constructed in Lemma \ref{prop:minus-1-refined} satisfies
	\[ (\overline{-1})^2 = \{ -1, -1 \}_F^{n}. \]
\end{lemma}
\begin{proof}
	In view of Lemma \ref{prop:minus-1-refined}, it suffices to deal with the case $n=1$ and calculate inside $\SL(2)$. By the construction of $E_G(F)$ in \cite[\S 5]{Mat69}, we know that $h_\alpha(t)h_\alpha(t') = \{t, t'\}_F \cdot h_\alpha(t t')$ for all $t,t' \in F^\times$.
\end{proof}

Henceforth we consider the BD-cover $\bmu_m \hookrightarrow \tilde{G} \twoheadrightarrow G(F)$ deduced from $E_G(F)$ and $\epsilon$. Denote again by $\overline{-1}$ the image of the element in Lemma \ref{prop:minus-1-refined} in the push-out $\tilde{G}$. It satisfies $(\overline{-1})^2 = (-1, -1)_{F,m}^n$ (omitting $\epsilon$), which is $\pm 1$ by what follows. \index{$\overline{-1}, \widetilde{-1}$}

\begin{lemma}\label{prop:1-c}
	For any $c \in F^\times$, we have
	\[ (-1, c)_{F,m} = \begin{cases}
			1, & m \notin 2\Z \\
			(-1,c)_{F,2}, & m \equiv 2 \pmod 4 \\
			(\zeta, c)_{F,2}, & 4 \mid m
	\end{cases}\]
	where $\zeta$ is any generator of $\mu_m$. It belongs to $\mu_2$ and depends solely on $cF^{\times 2}$.
\end{lemma}
\begin{proof}
	Bi-multiplicativity implies $(-1, c)_{F,m} \in \mu_2 \cap \mu_m$, hence $(-1, c)_{F,m} = 1$ for $m \notin 2\Z$. When $m \equiv 2 \pmod 4$ (resp. $4 \mid m$), use $(-1)^{m/2} = -1$ (resp. $\zeta^{m/2} = -1$) and \eqref{eqn:Hilbert-projection-formula} to reduce to $(\cdot, c)_{F,2}$.
\end{proof}

\begin{definition}\label{def:lifting-minus-1}
	Set $\tilde{G}^\natural \twoheadrightarrow G(F)$ to be the push-out of $\tilde{G}$ via $\bmu_m \hookrightarrow \bmu_{\text{lcm}(4,m)}$ if $m \in 2\Z$, otherwise set $\tilde{G}^\natural = \tilde{G}$. The next step is to define a preimage $\widetilde{-1} \in \tilde{G}^\natural$ of $-1$ such that $(\widetilde{-1})^2 = 1$, using Lemma \ref{prop:1-c}.
	\begin{itemize}
		\item \textbf{The case $m \not\equiv 2 \pmod 4$}. Take $\widetilde{-1} := \overline{-1}$. This works when $m \notin 2\Z$ by Lemma \ref{prop:1-c}. When $4 \mid m$, the same Lemma yields $(-1, -1)_{F,m} = (\zeta, -1)_{F,2} = (\zeta, \zeta)_{F,2}^{m/2}$, which still equals $1$.

		\item \textbf{The case $m \equiv 2 \pmod 4$}. We have a quadratic character $(-1, \cdot)_{F,m}: F^\times \to \mu_2$. The local root number $\epsilon\left( \frac{1}{2}, (-1, \cdot)_{F,m}; \psi \right)$ lies in $\bmu_4$ and satisfies
		\[ \epsilon\left( \frac{1}{2}, (-1, \cdot)_{F,m} ; \psi \right)^2 = (-1, -1)_{F,m} = (-1, -1)_{F,2}. \]
		Thus we can take
		\[ \widetilde{-1} := \epsilon\left( \frac{1}{2}, (-1, \cdot)_{F,m}; \psi\right)^{-n} \cdot \overline{-1} \; \in \tilde{G}^\natural. \]
	\end{itemize}
\end{definition}

\begin{lemma}\label{prop:splitting-variance}
	The formation of $\widetilde{-1}$ depends on $(W, \lrangle{\cdot|\cdot})$ and $\psi$ in the following manner. Suppose
	\begin{compactenum}[(i)]
		\item all data are acted upon by $\Ad(g)$ with $g \in G_\mathrm{ad}(F)$ (Proposition \ref{prop:BD-adjoint-action}), and $g$ comes from $g_1 \in \GSp(W)$ with similitude factor $N(g_1) = c$;
		\item $\lrangle{\cdot|\cdot}$ is replaced by $c\lrangle{\cdot|\cdot}$;
		\item $m \equiv 2 \pmod 4$ and $\psi$ is replaced by $\psi_c$
	\end{compactenum}
	where $c \in F^\times$. Then $\widetilde{-1}$ will be replaced by $(-1, c)_{F,m}^n \cdot \widetilde{-1}$ in each case.
\end{lemma}
\begin{proof}
	The case (iii) follows from the dependence of local root numbers on $\psi$. (i) and (ii) are equivalent, so it suffices to address (i).

	Decompose $W$ as $\bigoplus_{i=1}^n W_i$ with $W_i := Fe_i \oplus Fe_{-i}$. Note that if $h_i \in \GSp(W_i)$ satisfies $N(h_i)=c$, then $g_1 := \text{diag}(h_1, \ldots, h_n) \in \GSp(W)$ satisfies $N(g_1) = c$ as well. It remains to show $\Ad(g)(-1) = (-1,c)_{F,m}$ in the case $n=1$, which is just Lemma \ref{prop:-1-adjoint} joint with Lemma \ref{prop:1-c}.
\end{proof}

Finally, consider an orthogonal decomposition of symplectic $F$-vector spaces
\[ W = \bigoplus_{i=1}^r W_i \]
and let $A$ be the commutative group $\prod_{i=1}^r \{\pm 1\} \subset \prod_{i=1}^r \Sp(W_i)$. Let $\tilde{G}^\natural_i$ be the analogues for the group $G_i := \Sp(W_i)$. In each $\tilde{G}^\natural_i$ we have the element $\widetilde{-1}_i$ from Definition \ref{def:lifting-minus-1}.

\begin{proposition}\label{prop:lifting-A}
	The topological central extension $\tilde{G}^\natural \to G(F)$ admits a section over $A$ given by
	\[ (1, \ldots, \underbracket{-1}_{i-\text{th slot}}, \ldots, 1) \longmapsto \iota_i\left( \widetilde{-1}_i \right) \]
	for every $1 \leq i \leq r$, where $\iota_i: \tilde{G}^\natural_i \to \tilde{G}$ is the natural homomorphism furnished by Lemma \ref{prop:restriction-W_i}.
\end{proposition}
\begin{proof}
	Using Lemma \ref{prop:restriction-W_i}, the problem is reduced to the case $r=1$. But to say $-1 \mapsto \widetilde{-1}$ gives a section is the same as requiring $(\widetilde{-1})^2 = 1$.
\end{proof}

\begin{remark}\label{rem:splitting-invariance}
	In the situation above, if we transport everything by applying $g \in G(F)$, the resulting splitting over $gAg^{-1}$ only differs by $\Ad(g): \tilde{G} \to \tilde{G}$.
\end{remark}

\begin{example}\label{eg:section-S}
	Consider the setting of \S\ref{sec:covers-tori}, assume that $S \subset G$ is of type (ER) and $p \neq 2$. We have $S(F) = S(F)_{p'} \times S(F)_{\text{tu}}$ by \eqref{eqn:S-Jordan}, and Theorem \ref{prop:S-p'} asserts $S(F)_{p'} = \{\pm 1\}^I$. The parameters $K = \prod_{i \in I} K_i$ for $S \subset G$ decompose $W$ into $\bigoplus_{i \in I} W_i$, and accordingly $A = S(F)_{p'} = \{\pm 1\}^I$ in Proposition \ref{prop:lifting-A}. In fact $W = \bigoplus_i W_i$ is exactly the decomposition into joint eigenspaces under $S(F)_{p'}$. In this way we obtain a section over $S(F)_{p'}$. In view of Lemma \ref{prop:splitting-pro-p}, it extends uniquely to a section
	\[\begin{tikzcd}
		\tilde{S}^\natural \arrow[twoheadrightarrow]{r}[swap]{\bm{p}} & S(F) \arrow[bend right]{l}[swap]{\sigma}
	\end{tikzcd}\]
	where $\tilde{S}^\natural$ is the preimage of $S(F)$ in $\tilde{G}^\natural$. Writing $\Pi_-(\cdot)$ (resp. $\Pi(\cdot)$) to denote the set of equivalence classes of genuine representations (resp. usual representations), we obtain bijections
	\begin{equation}\label{eqn:splitting-genuine} \begin{tikzcd}
		\Pi_-(\widetilde{S}) \arrow[leftrightarrow]{r}{1:1}[swap, inner sep=1em]{\text{well-known}} & \Pi_-(\widetilde{S}^\natural) \arrow[leftrightarrow]{r}{1:1}[swap, inner sep=1em]{\text{via}\;\sigma} & \Pi(S(F)).
	\end{tikzcd}\end{equation}
	Their dependence on $(W, \lrangle{\cdot|\cdot})$ and $\psi$ (when $m \equiv 2 \pmod 4$) is described by Lemma \ref{prop:splitting-variance}.
\end{example}

\subsection{Toral invariants}\label{sec:toral-invariants}
Consider a maximal $F$-torus $S$ in a reductive $F$-group $G$, and let $\alpha$ be an absolute root with $[F_{\pm\alpha} : F_\alpha] = 2$, i.e. $\alpha$ is a \emph{symmetric root}, cf. \eqref{eqn:F_alpha}. We shall denote by $R(G,S)(\bar{F})$ the set of absolute roots of $S$, and by $R(G,S)(\bar{F})_\text{sym}$ the subset of symmetric roots; both are $\Gamma_F$-stable. Below is a review of the \emph{toral invariant} of $\alpha \in R(G,S)(\bar{F})_\text{sym}$ in \cite[\S 4.1]{Kal15}, which will enter into our construction of $L$-packets. Define
\[ f_{(G,S)}(\alpha) := \sgn_{F_\alpha/F_{\pm\alpha}}\left( \dfrac{[X_\alpha, \tau X_\alpha]}{H_\alpha} \right), \]
where $X_\alpha \in \mathfrak{g}_\alpha(F_\alpha) \smallsetminus \{0\}$ is arbitrary, and $H_\alpha := \dd\check{\alpha}(1)$ is the infinitesimal coroot. \index{f_GS@$f_{(G,S)}(\alpha)$}
\begin{compactitem}
	\item By \cite[Fact 4.1]{Kal15}, $\alpha \mapsto f_{(G,S)}(\alpha)$ is well-defined and is $\Gamma_F$-invariant. It is even $N_G(S)(F)$-invariant, cf. the explanation of \cite[Fact 4.7.5]{Kal19}.
	\item Due to the infinitesimal nature of this definition, one can also compute $f_{(G,S)}(\alpha)$ in $G_\text{der}$ or $G_\text{ad}$.
\end{compactitem}
We compute two cases below.

\textbf{The case $G=\Sp(W)$}. Here $W$ is a $2n$-dimensional symplectic $F$-vector space. Take a finite separable extension $L/F$ to split $S$, such that $S_L$ is associated to a symplectic basis $e_{\pm 1}, \ldots, e_{\pm n}$ of $W \otimes_F L$.
\begin{asparaenum}
	\item Suppose $\alpha$ is a long root. Reasoning as in \S\ref{sec:stable-reduction}, one may calculate inside $\SL(2)$ over $F_{\pm\alpha}$. We deduce $f_{(G,S)}(\alpha) = 1$ by \cite[Lemma 7.3]{Kal15}.
	
	\item Suppose $\alpha$ is a short root. Without loss of generality, assume $\alpha = \epsilon_1 - \epsilon_2$ in this basis, so that $\check{\alpha} = \check{\epsilon}_1 - \check{\epsilon}_2$. Take the data
	\begin{gather*}
		X_\alpha: \begin{array}{l} e_2 \mapsto e_1 \\ e_{-1} \mapsto -e_{-2} \end{array}, \quad
		H_\alpha: \begin{array}{l}
		e_1 \mapsto e_1 \\
		e_2 \mapsto -e_2 \\
		e_{-2} \mapsto e_{-2} \\
		e_{-1} \mapsto -e_{-1}
	\end{array} \quad \text{(the other basis vectors $\mapsto 0$)}.	\end{gather*}
	Recall that $\Gal{L/F}$ acts on $\{\pm\epsilon_1, \ldots, \pm\epsilon_n\} \subset X^*(S_L)$ by permutation since it does so on long roots. Let $\sigma$ be the nontrivial element of $\Gal{F_\alpha/F_{\pm\alpha}}$.
	\begin{enumerate}[(i)]
		\item Suppose that $\epsilon_1, \epsilon_2$ are in the same $\Gal{L/F}$-orbit, so that $\sigma: \epsilon_1 \leftrightarrow \epsilon_2$. Calculate the Lie bracket as
		\[ [X_\alpha, \sigma X_\alpha] = \begin{array}{l}
		e_1 \mapsto e_1 \\
		e_2 \mapsto -e_2 \\
		e_{-2} \mapsto e_{-2} \\
		e_{-1} \mapsto -e_{-1}
		\end{array} = H_\alpha. \]
		Hence $f_{(G,S)}(\alpha) = 1$ in this case. Alternatively, we may also calculate $f_{(G,S)}(\alpha)$ in a twisted Levi subgroup $R_{F_{\pm\alpha}/F}(\GL(2)) \times \cdots$, and argue as in the case of long roots that $f_{(G,S)}(\alpha) = 1$.
		\item Suppose that $\epsilon_1, \epsilon_2$ are not in the same $\Gal{L/F}$-orbit, then $\sigma: \epsilon_i \leftrightarrow -\epsilon_i$ ($i=1,2$). This time
		\[ [X_\alpha, \sigma X_\alpha] = \begin{array}{l}
			e_1 \mapsto -e_1 \\
			e_2 \mapsto e_2 \\
			e_{-2} \mapsto -e_{-2} \\
			e_{-1} \mapsto e_{-1}
		\end{array} = -H_\alpha. \]
		Hence $f_{(G,S)}(\alpha) = \sgn_{F_\alpha/F_{\pm\alpha}}(-1)$ in this case.
	\end{enumerate}
\end{asparaenum}

\begin{definition}\label{def:hyperbolic-basis}
	By stipulation, if a quadratic $F$-vector space $(L,h)$ of dimension $2n$ can be written as the orthogonal direct sum of non-degenerate subspaces $Fe_i \oplus Fe_{-i}$, where $1 \leq i \leq n$, such that
	\[ h(e_i|e_{-i}) \neq 0, \quad h(e_i|e_i) = 0 = h(e_{-i}|e_{-i}), \quad 1 \leq i \leq n \]
	then we say $\{e_{\pm i}\}_{i=1}^n$ is a \emph{hyperbolic basis} for $(L,h)$.
\end{definition}
Suppose that $(L,h)$ admits a hyperbolic basis, there is then a split maximal torus $T \subset \SO(L,h)$ consisting of matrices in the basis $e_1, \ldots, e_n, e_{-n}, \ldots, e_{-1}$:
\[ \gamma = \text{diag}(a_1, \ldots, a_n, a_n^{-1}, \ldots, a_1^{-1}), \quad a_1, \ldots, a_n \in \Gm. \]
Therefore $X^*(T) = \bigoplus_{i=1}^n \Z\epsilon_i$ by setting $\epsilon_i(\gamma) = a_i$. Note that we do not require $h(e_i|e_{-i}) = 1$.

\textbf{The case $G=\SO(V,q)$}. Here $(V,q)$ is a $(2n+1)$-dimensional quadratic $F$-vector space. Take a finite separable extension $L/F$ to split $S$, such that $S_L$ is the maximal torus associated to a hyperbolic basis $\{e_{\pm i}\}_{i=1}^n$ of a $2n$-dimensional non-degenerate $L$-subspace of $V \otimes_F L$; define $\{\epsilon_i\}_{i=1}^n$ as above. Its orthogonal complement is an anisotropic line $\ell$ which is the weight-zero subspace of $V \otimes_F L$ under $S_L$, hence $\ell$ is generated by an anisotropic $F$-vector $e_0$.
\begin{asparaenum}
	\item Suppose $\alpha$ is a short root. Without loss of generality, assume $\alpha = \epsilon_1$. Consider the subspace of $S$-weights $\{\alpha,0,-\alpha\}$ of $V \otimes_F F_\alpha$, namely
	\[ U := F_\alpha e_1 \oplus F_\alpha e_0 \oplus F_\alpha e_{-1}. \]
	The restriction $q_U := (q \otimes_F F_\alpha)|_U$ is non-degenerate, and $f_{(G,S)}(\alpha)$ can be calculated inside $\SO(U, q_U)$. Since $\Gamma_{\pm\alpha}$ permutes $\{\alpha, 0, -\alpha\}$, we see that $(U, q_L)$ descends to a quadratic $F_{\pm\alpha}$-vector space $(U_0, q_0)$ of dimension $3$.
	\begin{enumerate}[(i)]
		\item If $\SO(U_0, q_0)$ is split, it will be isomorphic to $\PGL(2)$ and we have $f_{(G,S)}(\alpha) = 1$ by \cite[Lemma 7.13]{Kal15} as before.
		\item If $\SO(U_0, q_0)$ is not split, it will be an inner form of $\PGL(2)$ of Kottwitz sign $-1$. By \cite[Proposition 4.3]{Kal15} we have $f_{(G,S)}(\alpha) = -1$ in this case.
	\end{enumerate}
	Note that $\SO(U_0, q_0)$ is split if and only if $(U_0, q_0)$ is isotropic, hence by \cite[12.7 + 14.3]{Sch85}
	\begin{equation}\label{eqn:SO-split}
		f_{(G,S)}(\alpha) = \epsilon(U_0, q_0) (-1, d^\pm(U_0, q_0))_{F_{\pm\alpha}, 2}.
	\end{equation}
	\item Suppose $\alpha$ is a long root. We may assume $\alpha = \epsilon_1 - \epsilon_2$. The formulas for $X_\alpha$ and $H_\alpha$ are exactly the same as the case for $\Sp(W)$, the same calculation thus leads to
	\[ f_{(G,S)}(\alpha) = \begin{cases}
		1, & \epsilon_1, \epsilon_2 \in \text{the same $\Gal{L/F}$-orbit} \\
		\sgn_{F_\alpha/F_{\pm\alpha}}(-1), & \epsilon_1, \epsilon_2 \notin \text{the same $\Gal{L/F}$-orbit}.
	\end{cases} \]
\end{asparaenum}

Henceforth assume $F$ non-archimedean with residual characteristic $\neq 2$. For $S \subset G$ as above, Kaletha constructed in \cite[\S 4.6]{Kal15} a character $\epsilon_S: S(F) \to \{\pm 1\}$ using various $f_{(G,S)}(\alpha)$. By construction, $\epsilon_T$ is $N_G(T)(F)$-invariant since $f_{G,T}(\alpha)$ is, cf. \cite[Fact 4.7.5]{Kal19}. \index{epsilonS@$\epsilon_S$}
\begin{lemma}\label{prop:epsilon-invariance}
	Let $G=\Sp(W)$ where $(W, \lrangle{\cdot|\cdot})$ is a $2n$-dimensional symplectic $F$-vector space. Let $j,j': S \hookrightarrow G$ be two embeddings of maximal $F$-tori, related by stable conjugacy $\Ad(g): jS \rightiso j'S$. For any $\gamma \in S(F)$ with topological Jordan decomposition $\gamma = \gamma_0 \gamma_{>0}$, we have
	\[ \epsilon_j(j\gamma) = \epsilon_j(j\gamma_0) = \epsilon_{j'}(j'\gamma_0) = \epsilon_{j'}(j'\gamma). \]
\end{lemma}
\begin{proof}
	The first and the last equalities are immediate consequences of \cite[Lemma 4.12]{Kal15}. Let us turn to the middle one. By \cite[Lemma 4.12]{Kal15}
	\[ \epsilon_j(\gamma_0) = \prod_{\substack{\alpha \in R(G, jS)_\text{sym}(\bar{F}) /\Gamma_F \\ \alpha(\gamma_0) \neq 1 }} f_{(G, jS)}(\alpha), \]
	and similarly for $\epsilon_{j'}(\gamma'_0)$. Now choose a Galois extension $L/F$ to split $S$ and a symplectic basis $\{ e_{\pm i}\}_{i=1}^n$ for $W \otimes_F L$ to calculate $f_{(G, jS)}(\cdot)$ as before. Transport it to $e'_{\pm i} := g e_{\pm i}$. Define $\epsilon_{\pm i} \in X^*(jS_L)$ (resp. $\epsilon'_{\pm i} \in X^*(j'S_L)$) using the basis $\{e_{\pm i}\}_{i=1}^n$ (resp. $\{e'_{\pm i}\}_{i=1}^n$); these data will be used to calculate  $f_{(G, jS)}(\cdot)$ and $f_{(G, j'S)}(\cdot)$.


	From $\Ad(g)$ we deduce a $\Gamma_F$-equivariant isomorphism $X^*(jS_L) \rightiso X^*(j'S_L)$. It restricts to a bijection $R(G, jS)(\bar{F}) \rightiso R(G, j'S)(\bar{F})$, written as $\alpha \mapsto \alpha' := \alpha \circ \Ad(g)^{-1}$. Concurrently, $\epsilon_{\pm i}$ are mapped to $\epsilon'_{\pm i} = \epsilon_{\pm i} \circ \Ad(g)^{-1}$. We also have $\Ad(g)\gamma_0 = \gamma'_0$, hence $\alpha'(\gamma'_0) \neq 1 \iff \alpha(\gamma_0) \neq 1$.
	
	It follows from the previous calculations for $\Sp(W)$ that $f_{(G, jS)}(\alpha) = f_{(G, j'S)}(\alpha')$, because these invariants depend solely on the $\Gamma_F$-action on $\{ \epsilon_{\pm i}, \epsilon'_{\pm i} \}_{i=1}^n$.
\end{proof}
Note that the argument fails for $\SO(V,q)$: already in the case $n=1$ and $S$ anisotropic, we saw that $f_{(G,S)}(\alpha)$ depends on finer invariants of quadratic forms.

\subsection{Inducing data}\label{sec:inducing-data}
In this subsection, $F$ is non-archimedean of residual characteristic $p \neq 2$, and for the BD-covers we assume $p \nmid m$. The characters in question are all continuous, and $\Hom(\cdots)$ denotes the continuous $\Hom$.

Fix a nonempty stable conjugacy class $\mathcal{E}$ of embeddings of tame maximal $F$-tori $j: S \rightiso jS \subset G$ of type (ER) (Definition \ref{def:type-ER}). Here we leave $S$ fixed and vary $j$. As usual, $e$ will stand for the ramification degree of the splitting field extension of $S$. Stable conjugacy preserves Moy--Prasad filtrations. These embeddings together with stable conjugacy form a groupoid with trivial automorphism groups. Furthermore,
\begin{compactitem}
	\item for any $j, j' \in \mathcal{E}$, stable conjugacy furnishes a canonical $F$-isomorphism $\Omega(G, jS) \rightiso \Omega(G, jS')$;
	\item in view of the absence of nontrivial automorphisms, one may form the universal absolute root system $R(G, S) := \varprojlim_j R(G, jS)$ living on $X^*(S_{\bar{F}})$, and similarly for $\Omega(G, S)$;
	\item regard $R(G,S)$, $R(G, jS)$ and $\Omega(G,S)$, $\Omega(G,jS)$ as sheaves over $\Spec(F)_{\text{ét}}$, so that $\Omega(G,S)(F)$ acts on $S$ and on $R(G,S)$;
	\item since the formation of $\iota_{Q,m}: (jS)_{Q,m} \to jS$ from \S\ref{sec:isogeny} respects conjugacy, we have the universal isogeny $ \iota_{Q,m}: S_{Q,m} \to S$.
\end{compactitem} \index{$\theta^\flat$}
Next, consider a character $\theta^\flat: S_{Q,m}(F) \to \CC^\times$, or equivalently a family of characters $\theta^\flat_j: jS_{Q,m}(F) \to \CC^\times$ compatibly with the groupoid. Our requirement is that any character $\theta: S(F) \to \CC^\times$ with $\theta \circ \iota_{Q,m} = \theta^\flat$ is epipelagic (Definition \ref{def:epipelagic-character}). This is a condition on $\theta^\flat|_{S_{Q,m}(F)_{0+}}$ since $\iota_{Q,m}$ is an isomorphism on the pro-$p$ parts, as one infers from \eqref{eqn:isogeny-Sp}.

\begin{definition}\label{def:inducing-data} \index{$(\mathcal{E}, S, \theta^\flat)$}
	Abbreviate the data above as $(\mathcal{E}, S, \theta^\flat)$. An isomorphism $(\mathcal{E}, S, \theta^\flat) \rightiso (\mathcal{E}_1, S_1, \theta^\flat_1)$ consists of a commutative diagram of $F$-tori
	\[\begin{tikzcd}[row sep=small, column sep=small]
		S_{Q,m} \arrow{r}{\varphi_{Q,m}} \arrow{d} & S_{1, Q,m} \arrow{d} \\
		S \arrow{r}[swap]{\varphi} & S_1
	\end{tikzcd}\]
	such that $\varphi$, $\varphi_{Q,m}$ are isomorphisms and
	\begin{compactitem}
		\item $\theta^\flat_1 = \theta^\flat \circ \varphi_{Q,m}$,
		\item $j \mapsto j \circ \varphi^{-1}$ is a bijection from $\mathcal{E}$ onto $\mathcal{E}_1$.
	\end{compactitem}
\end{definition}
Note that $\varphi_{Q,m}$ determines $\varphi$; in view of \eqref{eqn:isogeny-Sp} we can even write $\varphi = \varphi_{Q,m}$. Using automorphisms we may modify $\theta^\flat$ by $\Omega(G,S)$.

\begin{lemma}\label{prop:Fiber}
	We have $\Hom(S(F)_{p'}, \CC) = \Hom(S(F)/S(F)_{0+}, \CC^\times) \simeq (\Z/2\Z)^I$. Given $\theta^\flat$, the set
	\[ \mathrm{Fiber}(S, \theta^\flat) := \left\{\theta \in \Hom(S(F), \CC^\times) : \theta \circ \iota_{Q,m} = \theta^\flat \right\} \]
	is a singleton when $4 \nmid m$. When $4 \mid m$, we have
	\[ \mathrm{Fiber}(S, \theta^\flat) = \begin{cases}
		\text{a torsor under} \; \Hom(S(F)_{p'}, \CC), & \theta^\flat|_{S_{Q,m}(F)_{p'}} = 1 \\
		\emptyset, & \text{otherwise}.
	\end{cases} \]
	Furthermore, all the $\theta \in \mathrm{Fiber}(S, \theta^\flat)$ share the same pro-$p$ part $\theta|_{S(F)_{0+}}$.
\end{lemma}
\begin{proof}
	Theorem \ref{prop:S-p'} gives the structure of $\Hom(S(F)_{p'}, \CC)$. Observe from this and \eqref{eqn:isogeny-Sp} that
	\begin{compactitem}
		\item $\iota_{Q,m}: S_{Q,m}(F)_{0+} \to S(F)_{0+}$ is an isomorphism,
		\item $\iota_{Q,m}$ induces $S_{Q,m}(F)_{p'} \rightiso S(F)_{p'}$ when $4 \nmid m$, and $\iota_{Q,m}|_{S_{Q,m}(F)_{p'}} = 1$ when $4 \mid m$.
	\end{compactitem}
	For $4 \mid m$ we have complete freedom to choose $\theta|_{S(F)_{p'}}$, whence the torsor structure.
\end{proof}

Given $(\mathcal{E}, S, \theta^\flat)$, the next step is to prescribe genuine characters $\theta_j: \widetilde{jS} \to \CC^\times$ to each $j \in \mathcal{E}$. The isogeny $\iota_{Q,m}: jS_{Q,m} \to jS$ pulls back to
\[\begin{tikzcd}
	\widetilde{jS}_{Q,m} \arrow{d} \arrow{r}{\tilde{\iota}_{Q,m}} & \widetilde{jS} \arrow{d}{\bm{p}} \arrow[hookrightarrow]{r} & \tilde{G} \arrow{d}{\bm{p}} \\
	jS_{Q,m}(F) \arrow{r}{\iota_{Q,m}} & jS(F) \arrow[hookrightarrow]{r} & G(F)
\end{tikzcd}\]
Cf. \eqref{eqn:iota-cover}. Using \eqref{eqn:splitting-genuine}, it suffices to prescribe a character $\theta_j$ of $jS(F)$, and it can be further transported to $S(F)$ via $j$. The precise recipe will depend on $m \bmod 4$.

\begin{description}\index{$\theta^\circ_j, \theta^\dagger_j$}
	\item[The case $4 \mid m$.] We take all the possible $\theta_j \in \text{Fiber}(S, \theta^\flat)$ described by Lemma \ref{prop:Fiber}, viewed as a genuine character of $\widetilde{jS}$. By Lemma \ref{prop:Fiber} they have the same pro-$p$ part. This construction is non-vacuous only when $\theta^\flat|_{S_{Q,m}(F)_{p'}} = 1$. Cf. Remark \ref{rem:comparison-GG-tori}.
	\item[The case $4 \nmid m$.] Lemma \ref{prop:Fiber} asserts that $\text{Fiber}(S, \theta^\flat)$ is a singleton; the unique element therein transports to a genuine character $\theta^\circ_j: \widetilde{jS} \to \CC^\times$ by \eqref{eqn:splitting-genuine} for all $j$. Further modifications on the prime-to-$p$ part of $\theta^\circ_j$ are needed to achieve canonicity and stability, at least when $m \equiv 2 \pmod 4$. The required modification is encapsulated into the following axioms.
\end{description}

\begin{definition}\label{def:stable-system}\index{stable system}
	Suppose $4 \nmid m$. A \emph{stable system} is a rule assigning a family of genuine characters $\theta_j: \widetilde{jS} \to \CC^\times$ to a triple $(\mathcal{E}, S, \theta^\flat)$, relative to any given $(W, \lrangle{\cdot|\cdot})$ (and $\psi$ when $m \equiv 2 \pmod 4$). It must satisfy the following properties.
	\begin{enumerate}[\bfseries SS.1]
		\item We require that $\theta_j = \theta_j^\circ \theta_j^\dagger$, where $\theta_j^\dagger$ is a character of $j(S(F)/S(F)_{0+})$. When $m \equiv 2 \pmod 4$, we require further that if $c \in F^\times$ and
		\begin{compactitem}
			\item $\psi$ is replaced by $\psi_c$, or
			\item $\lrangle{\cdot|\cdot}$ is replaced by $c\lrangle{\cdot|\cdot}$
		\end{compactitem}
		then $\theta_j^\dagger$ will be twisted by the quadratic character of $jS(F)_{p'}$ that maps
		\[ \left( -1 \in K_i^1 \hookrightarrow \prod_{i' \in I} K_{i'}^1 \simeq S(F) \right) \mapsto \sgn_{K_i/K_i^\sharp}(c) \]
		(notation of Theorem \ref{prop:S-p'}) for all $i \in I$.
		\item For any stable conjugacy $j' = \Ad(g) \circ j$ in $\mathcal{E}$, we require that
		\begin{gather*}
			\theta_{j'}\left( \CaliAd(g)(\tilde{\gamma}) \right) = \theta_j(\tilde{\gamma}), \quad \tilde{\gamma} \in \widetilde{jS} 
		\end{gather*}
		where $\CaliAd(g)$ is as in Definition \ref{def:st-conj}.
		\item The character $\theta_j^\dagger$ depends only on the maximal torus $jS \subset G$ and the $\theta^\flat|_{S_{Q,m}(F)_{0+}}$ transported to $jS_{Q,m}(F)_{0+}$. It follows that if $(\varphi, \varphi_{Q,m}): (\mathcal{E}, S, \theta^\flat) \rightiso (\mathcal{E}_1, S_1, \theta^\flat_1)$ is an isomorphism, then $\theta_j^\dagger = \theta_{j \varphi^{-1}}^\dagger$; indeed, $\Image(j) = \Image(j\varphi^{-1}) =: R$ and the transportations to $R_{Q,m}(F)$ of $\theta^\flat$ and $\theta_1^\flat$ coincide since $\theta^\flat = \theta_1^\flat \varphi_{Q,m}$.
	\end{enumerate}
\end{definition}

Several quick observations are in order.
\begin{compactitem}
	\item \textbf{SS.1} and Lemma \ref{prop:Fiber} determine the pro-$p$ component of $\theta_j$; the pro-$p$ version of all the other conditions follow.
	\item The dependence of $\theta_j^\dagger$ on $(\psi, \lrangle{\cdot|\cdot})$ in \textbf{SS.1} is imposed by the canonicity of $L$-packets for BD-covers, see Theorem \ref{prop:packet-independence}.
	\item In view of \textbf{AD.3} of Proposition \ref{prop:CAd-prop}, taking $g \in G(F)$ in \textbf{SS.2} implies that $\theta_j$ is compatible with $G(F)$-conjugacy of $j$.
	\item For each $j \in \mathcal{E}$, the genuine character $\theta_j$ determines $\theta^\flat$ as follows: $\theta_j|_{jS(F)_{0+}}$ determines $\theta^\flat|_{S(F)_{0+}}$. By \textbf{SS.3}, this in turn determines $\theta_j^\dagger$ as well as $\theta_j^\circ = \theta_j \cdot  (\theta_j^\dagger)^{-1}$; we conclude that $\theta^\flat$ is also determined.
\end{compactitem}

\begin{lemma}
	In each case, $\theta_j$ is an epipelagic genuine character of $\widetilde{jS}$ in the sense of Definition \ref{def:epipelagic-character}.
\end{lemma}
\begin{proof}
	This concerns only the pro-$p$ part of $\theta_j$, thus unaffected by $\theta_j^\dagger$. The required property is thus built into the definition of triples $(\mathcal{E}, S, \theta^\flat)$.
\end{proof}

\begin{proposition}[Standard stable system for odd $m$]\label{prop:std-stable-system}
	Suppose that $m \notin 2\Z$. There is a standard stable system given by $\theta_j := \theta_j^\circ$ for all $j$.
\end{proposition}
\begin{proof}
	The properties \textbf{SS.1} and \textbf{SS.3} are automatic. Since $m \notin 2\Z$, by Proposition \ref{prop:CaliAd-simple} $\CaliAd(g)$ is defined by the natural actions of $G^T_\text{ad}(F)$ and $G(F)$. On the other hand, \eqref{eqn:splitting-genuine} is realized by the unique splitting over $S(F)_{p'}$ (Proposition \ref{prop:S-splitting-odd}). This entails \textbf{SS.2}.
\end{proof}

\begin{definition}\label{def:packet-0}\index{$\Pi(S, \theta^\flat)$}
	Suppose we are given
	\begin{itemize}
		\item a triple $(\mathcal{E}, S, \theta^\flat)$ as above;
		\item a stable system (Definition \ref{def:stable-system}) when $4 \nmid m$.
	\end{itemize}
	Write $\epsilon_j := \epsilon_{jS}$ for the character of $jS(F)$ constructed from toral invariants, as reviewed in \S\ref{sec:toral-invariants}. Define the set
	\[ \Pi(S, \theta^\flat) := \left\{ \pi_{(\widetilde{jS}, \epsilon_j \theta_j)} \right\}_{j, \theta_j} \]
	where
	\begin{itemize}
		\item $j$ ranges over all $G(F)$-conjugacy classes of embeddings $S \hookrightarrow G$ in $\mathcal{E}$;
		\item for each $j$, let $\theta_j$ be the genuine epipelagic character(s) of $\widetilde{jS}$ specified as follows:
		\begin{description}
			\item[($4 \mid m$):] $\theta_j$ ranges over all the elements of $\text{Fiber}(S, \theta^\flat)$ transported to $\widetilde{jS}$, which could be empty,
			\item[($4 \nmid m$):] $\theta_j$ is uniquely specified by the stable system and $(\mathcal{E}, S, \theta^\flat)$;
		\end{description}
	\item $\pi_{(\widetilde{jS}), \epsilon_j \theta_j}$ is the representation of $\tilde{G}$ constructed in Theorem \ref{prop:epipelagic-supercuspidal}.
	\end{itemize}
\end{definition}

\begin{lemma}
	The set $\Pi(S, \theta^\flat)$ depends only on the isomorphism class of $(\mathcal{E}, S, \theta^\flat)$.
\end{lemma}
\begin{proof}
	Indeed, given $(\varphi, \varphi_{Q,m}): (\mathcal{E}, S, \theta^\flat) \rightiso (\mathcal{E}_1, S_1, \theta^\flat_1)$, pull-back by $\varphi^{-1}$ induces bijections
	\[ \text{Fiber}(S, \theta^\flat) \to \text{Fiber}(S_1, \theta^\flat_1) \quad \text{and} \quad \mathcal{E} \to \mathcal{E}_1. \]
	Suppose $j \mapsto j_1$, the characters of $jS(F) = j_1 S_1(F)$ specified from $\theta^\flat$ and $\theta^\flat_1$ are therefore equal. Since the identification \eqref{eqn:splitting-genuine} depends only on $jS = j_1 S_1$, the inducing genuine characters in the case $4 \mid m$ also coincide. For the case $4 \nmid m$, we invoke \textbf{SS.3}.
\end{proof}

\begin{proposition}\label{prop:packet-prop-0}
	The set $\Pi(S, \theta^\flat)$ consists of genuine epipelagic supercuspidal representations of $\tilde{G}$. Conversely every genuine epipelagic supercuspidal representation belongs to some $\Pi(S, \theta^\flat)$.
	
	Granting the basic properties of Yu's construction for $\tilde{G}$ in Remark \ref{rem:refactorization}, we have
	\begin{itemize}
		\item when $4 \nmid m$, $\Pi(S, \theta^\flat)$ is a torsor under $H^1(F, S)$, of cardinality $2^{|I|}$, where $I$ is as in Theorem \ref{prop:S-p'};
		\item suppose $4 \mid m$, $\Pi(S, \theta^\flat)$ is
		\begin{itemize}
			\item a torsor under $\Hom(S(F)_{p'}, \CC) \times H^1(F, S)$ of cardinality $2^{2|I|}$, if $\theta^\flat|_{S_{Q,m}(F)_{p'}}=1$,
			\item empty if $\theta^\flat|_{S_{Q,m}(F)_{p'}} \neq 1$.
		\end{itemize}
	\end{itemize}
	Under the same premises, any two $\Pi(S_1, \theta^\flat_1)$ and $\Pi(S_2, \theta^\flat_2)$ are either disjoint or equal, and the latter case occurs if and only if both are empty, or $(\mathcal{E}_1, S_1, \theta^\flat_1)$ is isomorphic to $(\mathcal{E}_2, S_2, \theta^\flat_2)$.
\end{proposition}
\begin{proof}
	The first part follows from the construction and Theorem \ref{prop:epipelagic-supercuspidal}. For the second part, in view of \eqref{eqn:lambda-stab} and Remark \ref{rem:refactorization}, different $G(F)$-conjugacy classes of $j$ give non-isomorphic supercuspidals of $\tilde{G}$. The group $H^1(F, S)$ alters the embeddings $j$ via $H^1(F, S) \simeq H^1(F, jS) = \mathfrak{D}(j S, G ;F)$ and makes a torsor. Proposition \ref{prop:D-description} implies $|H^1(F, S)| = 2^{|I|}$. This settles the case $4 \nmid m$, and for $4 \mid m$ we appeal to Lemma \ref{prop:Fiber}.

	The last part is clear by Lemma \ref{prop:Fiber} when $4 \mid m$ (only the pro-$p$ part matters). Suppose $4 \nmid m$ and that there exist $j_i \in \mathcal{E}_i'$ with $(\widetilde{j_i S_i}, \theta_{i, j_i})$ conjugate for $i=1,2$; upon modifying $(\mathcal{E}_1, S_1, \theta^\flat_1)$ by an isomorphism $(\varphi, \varphi_{Q,m})$ we may assume $S_1 = S_2$, $\mathcal{E}_1 = \mathcal{E}_2$ and $j_1 S_1 = j_2 S_2$. Then $\theta_{1, j_1} = \theta_{2, j_2} \Ad(w)$ for some $w \in \Omega(G, S_2)(F)$; by a further modification we may assume $\theta_{1, j_1} = \theta_{2, j_2}$. As remarked below Definition \ref{def:stable-system}, this implies $\theta_1^\flat = \theta_2^\flat$ by using \textbf{SS.3}.
\end{proof}

\subsection{Epipelagic \texorpdfstring{$L$}{L}-parameters}\label{sec:epipelagic-parameters}
We continue the thread of \ref{sec:inducing-data}. The Weil form of the $L$-group $\Lgrp{\tilde{G}}$ for $\tilde{G}$ has been introduced in \S\ref{sec:L-group}. Given the symplectic $F$-vector space $(W, \lrangle{\cdot|\cdot})$, together with $\psi: F \to \CC^\times$ when $m \equiv 2 \pmod 4$, we deduce an identification $\Lgrp{\tilde{G}} = \tilde{G}^\vee \times \Weil_F$ that respects the projections to $\Weil_F$.

\begin{definition}[{\cite[Conditions 5.1]{Kal15}}]\label{def:epipelagic-parameter} \index{L-parameter!epipelagic}
	An $L$-parameter $\phi: \Weil_F \to \Lgrp{\tilde{G}}$ is called \emph{epipelagic} if
	\begin{compactitem}
		\item $Z_{\tilde{G}^\vee}(\phi(P_F))$ is a maximal torus $\tilde{T}^\vee$ that fits into a $\Gamma_F$-pinning of $\tilde{G}^\vee$;
		\item the image of $\phi(I_F)$ in $\Omega(\tilde{G}^\vee, T^\vee)$ is generated by a regular elliptic element $t$;
		\item let $o(t)$ be the order of $t$, then $w \in \Gamma_F^{\frac{1}{o(t)}+} \implies \phi(w) = (1, w)$.
	\end{compactitem}
	Upon conjugating by $\tilde{G}^\vee$, we may assume that $\tilde{T}^\vee$ is the maximal torus in the given pinning for $\tilde{G}^\vee$. As in \textit{loc. cit.}, $\tilde{T}^\vee$ can be endowed with the continuous $\Gamma_F$-action induced from
	\[ \Weil_F \xrightarrow{\phi} N_{\tilde{G}^\vee}(T^\vee) \rtimes \Weil_F \twoheadrightarrow \Omega(\tilde{G}^\vee, T^\vee) \rtimes \Weil_F, \]
	which factors through a finite quotient. Name the resulting torus with Galois action as $S_{Q,m}^\vee$; recall that $X^*(S_{Q,m}^\vee) = Y_{Q,m}$.
\end{definition}
This is a special case of \emph{toral supercuspidal $L$-parameters} defined in \cite[Definition 6.1.1]{Kal19}\index{L-parameter!toral supercuspidal}. These conditions are independent of the splitting for $\Lgrp{\tilde{G}}$, since different splittings differ by an element from $H^1(\Weil_F, Z_{\tilde{G}^\vee})$.

The stable conjugacy classes of maximal $F$-tori in $G = \Sp(W)$ are parameterized by $H^1(F, \Omega)$, where $\Omega$ is the absolute Weyl group of $\Sp(2n)$. Now comes the \emph{type map} for maximal tori, let us choose a Borel subgroup $B$ with a reductive quotient $T$, thereby obtaining the Weyl group $\Omega$. For any maximal $F$-torus $S \subset G$ there exists $g \in G(\bar{F})$ such that $S_{\bar{F}} = g T_{\bar{F}} g^{-1}$. Then $\sigma \mapsto g^{-1}\sigma(g)$ is the required $1$-cocycle, whose class in $H^1(F, \Omega)$ is independent of the choice of $g$. Facts:
\begin{inparaenum}[(a)]
	\item Two maximal tori are stably conjugate if and only if they have the same class in $H^1(F, \Omega)$.
	\item Every class in $H^1(F, \Omega)$ comes from some $S \subset G$; see \cite[\S 3.2]{Kal19} or \cite[Theorem 1.1]{Rag04}.
\end{inparaenum}

Recall from \S\ref{sec:L-group} that $\Omega(\tilde{G}^\vee, \tilde{T}^\vee)$ and $\Omega(G,T)$ can be identified: $Y_{Q,m}$ and $Y$ span the same $\Q$-vector space on which $\Omega$ acts by reflections. Hence the $\Gamma_F$-action on $S_{Q,m}^\vee$ gives rise to
\begin{inparaenum}[(i)]
	\item a class $c \in H^1(F, \Omega)$;
	\item whence a stable class $\mathcal{E}$ of embeddings $j: S \hookrightarrow G$;
	\item by construction, $\Gamma_F$ acts on $X_*(S_{\bar{F}})$ via a $1$-cocycle in the class $c$.
\end{inparaenum}
We may assume that $\Gamma_F$ acts on $X_*(S_{\bar{F}})$ and $X^*(S_{Q,m}^\vee)$ by the same cocycle. Let $S_{Q,m}$ be the $F$-torus dual to $S_{Q,m}^\vee$, so we deduce $\Gamma_F$-equivariant isomorphisms
\[ X_*(S_{\bar{F}}) \simeq Y \hookleftarrow Y_{Q,m} \simeq X^*(S_{Q,m}^\vee) = X_*((S_{Q,m})_{\bar{F}}). \]
The naming is thus justified: the isomorphisms above induces $S_{Q,m} \to S$ that is exactly the $\iota_{Q,m}$ in \S\ref{sec:isogeny} for the BD-cover $\tilde{G}$.

To proceed, we follow \cite[\S 5.2]{Kal15} to construct an $L$-embedding $\Lgrp{j}: \Lgrp{S_{Q,m}} \to \Lgrp{\tilde{G}}$ up to $\tilde{G}^\vee$-conjugacy, an $L$-parameter $\phi_{S_{Q,m}, \Lgrp{j}}$ together with factorization
\[\begin{tikzcd}[row sep=small, column sep=small]
	\Weil_F \arrow{rr}{\phi} \arrow{rd}[swap]{\phi_{S_{Q,m}, \Lgrp{j}} } & & \Lgrp{\tilde{G}} \\
	& \Lgrp{S_{Q,m}} \arrow{ru}[swap]{\Lgrp{j}} &
\end{tikzcd} \qquad
	\left( \Lgrp{j}|_{S_{Q,m}^\vee}: S_{Q,m}^\vee \rightiso \tilde{T}^\vee \right) = \text{the given one.}
\]
The resulting $L$-parameter $\phi_{S_{Q,m}, \Lgrp{j}}$ will be canonical up to $\Omega(G,S)(F)$-action by \cite[Lemma 5.3]{Kal15}. In \textit{loc. cit.} such an $\Lgrp{j}$ comes from a carefully chosen $\chi$-datum of $(G,S)$ in the sense of Langlands--Shelstad \cite[(2.5), (2.6)]{LS1}. A $\chi$-datum consists of characters $\chi_\alpha: F_\alpha^\times \to \CC^\times$ (see \eqref{eqn:F_alpha}) where $\alpha \in R(G,S)(\bar{F})$, such that \index{$\chi$-data}
\begin{gather*}
	\chi_{-\alpha} = \chi_\alpha^{-1}, \\
	\sigma \in \Gamma_F \implies \chi_{\sigma\alpha} = \chi_\alpha \circ \sigma^{-1}, \\
	[F_\alpha : F_{\pm\alpha}] = 2 \implies \chi_\alpha|_{F_{\pm\alpha}} = \sgn_{F_\alpha/F_{\pm\alpha}}.
\end{gather*}
Rigorously speaking, here we must work with the $S_{Q,m}$ embedded into the split $F$-group $G_{Q,m}$ dual to $\tilde{G}^\vee$, since the roots/coroots are rescaled in constructing $\tilde{G}^\vee$. However, the $\Gamma_F$-action on the roots are unaffected by such rescaling. All in all, the construction of $\Lgrp{j}$ carries over to our setting, and \cite[Lemma 5.4]{Kal15} implies that
\begin{itemize}
	\item $jS$ is a maximal torus of type (ER) in $G$, for every $j \in \mathcal{E}$;
	\item local Langlands correspondence for $S_{Q,m}$ yields a character $\theta^\flat: S_{Q,m}(F) \to \CC^\times$, such that $(\mathcal{E}, S, \theta^\flat)$ satisfies the requirements of Definition \ref{def:inducing-data}.
\end{itemize}
Indeed, both conditions can be phrased in terms of Weyl group actions, thus the arguments in \textit{loc. cit.} carry over verbatim. The triple $(\mathcal{E}, S, \theta^\flat)$ is well-defined only up to isomorphism: recall the ambiguity by $\Omega(G,S)(F)$.

\begin{remark}
	Different choices of $\chi$-data lead to different assignments $\phi \leadsto (\mathcal{E}, S, \theta^\flat)$. This choice will not be used in the proof of stability, but it will intervene in our later comparison with $\Theta$-correspondence in \S\ref{sec:theta}. Also note that the $\chi$-data is not fixed in the approach of \cite{Kal19}.
\end{remark}

Recall that the identification $\Lgrp{\tilde{G}} = \tilde{G}^\vee \times \Weil_F$ can be twisted by $H^1(\Weil_F, Z_{\tilde{G}^\vee})$. All elements therein take the form $\chi_c$ as in \eqref{eqn:metaGalois-section-twist}. Its effect on $(\mathcal{E}, S, \theta^\flat)$ is to twist $\theta^\flat$ by the quadratic character associated to the image of $\chi_c$ in $H^1(\Weil_F, S_{Q,m}^\vee)$.
\begin{lemma}\label{prop:pre-twist}
	Suppose $4 \mid m$. The $H^1(\Weil_F, Z_{\tilde{G}^\vee})$-orbit of $(\mathcal{E}, S, \theta^\flat)$ contains at most one element $\theta^\flat_1$ satisfying $\theta^\flat_1|_{S_{Q,m}(F)_{p'}} = 1$.
\end{lemma}
\begin{proof}
	Recall that $S_{Q,m}(F) = S_{Q,m}(F)_{p'} \times S_{Q,m}(F)_{0+}$, as $S$ satisfies the same property and $S \simeq S_{Q,m}$. Since $p \neq 2$, the quadratic twists do not affect the pro-$p$ part of $\theta^\flat$.
\end{proof}

\begin{definition}[$L$-packets]\label{def:epipelagic-packet} \index{$\Pi_\phi, \Pi_{[\phi]}$}
	Choose a stable system (Definition \ref{def:stable-system}) when $4 \nmid m$. Let $\phi: \Weil_F \to \Lgrp{\tilde{G}}$ be an epipelagic $L$-parameter and let $(\mathcal{E}, S, \theta^\flat)$ be the resulting isomorphism class of inducing data. Using the notation from Lemma \ref{prop:pre-twist} for $4 \mid m$, define
	\[ \Pi_\phi := \begin{cases}
		\Pi(S, \theta^\flat), & 4 \nmid m \\
		\Pi(S, \theta^\flat_1), & 4 \mid m, \; \exists \theta^\flat_1 \\
		\emptyset, & 4 \mid m, \; \nexists \theta^\flat_1
	\end{cases}\]
	where $\Pi(S, \cdot)$ is that of Definition \ref{def:packet-0}. Call it the \emph{$L$-packet} (resp. \emph{pre-$L$-packet}) attached to $\phi$ when $4 \nmid m$ (resp. to the $H^1(\Weil_F, Z_{\tilde{G}^\vee})$-orbit $[\phi]$ of $\phi$ when $4 \mid m$). We also define the centralizer group
	\[ C_\phi := Z_{\tilde{G}^\vee}\left( \Image(\phi) \right). \]
\end{definition}

\begin{remark}\label{rem:pre-packets}
	Assume $4 \mid m$. Then $\Pi_\phi$ and $C_\phi$ depend only on the $H^1(\Weil_F, Z_{\tilde{G}^\vee})$-orbit $[\phi]$, hence we may write $\Pi_{[\phi]}$, $C_{[\phi]}$ instead. Note that among the twists of $\theta^\flat$, Lemma \ref{prop:pre-twist} specified the unique one (if exists) such that $\Pi(S, \theta^\flat_1) \neq \emptyset$, by Proposition \ref{prop:packet-prop-0}.
\end{remark}	

\begin{lemma}\label{prop:isocrystal}
	The $\CC$-group $C_\phi$ is finite and diagonalizable. There is a isomorphism
	\[ (S_{Q,m}^\vee)^{\Gamma_F} \simeq C_\phi. \]
\end{lemma}
\begin{proof}
	This is just \cite[(5.2)]{Kal15}. 
\end{proof}

Let $\check{\iota}_{Q,m}: S^\vee \to S_{Q,m}^\vee$ be the dual of $\iota_{Q,m}$. As $\iota_{Q,m}$ is identifiable with the endomorphism $t_0 \mapsto t_0^{m/\text{gcd}(2,m)}$ of $S$, from Theorem \ref{prop:S-p'} and Proposition \ref{prop:D-description} we infer that
\begin{gather*}
	H^1(\Weil_F, S^\vee) \xrightarrow{(\check{\iota}_{Q,m})_*} H^1(\Weil_F, S_{Q,m}^\vee): \quad \Ker \simeq \Coker \simeq
	\begin{cases}
		0, & 4 \nmid m \\
		\mu_2^I, & 4 \mid m;
	\end{cases} \\
	H^1(F, S_{Q,m}) \xrightarrow{\iota_{Q,m,*}} H^1(F, S): \quad
	\begin{cases}
		\text{bijective}, & 4 \nmid m \\
		\text{trivial}, & 4 \mid m.
	\end{cases}
\end{gather*}

\begin{theorem}\label{prop:L-packet-prop}
	Every genuine epipelagic irreducible representation of $\tilde{G}$ belongs to some $\Pi_\phi$. Granting the premises in Proposition \ref{prop:packet-prop-0}, we have the following properties.
	\begin{itemize}
		\item When $4 \nmid m$, all $\Pi_\phi$ are nonempty. When $4 \mid m$, denote by $\left[ \phi_{S_{Q,m}, \Lgrp{j}} \right]$ the $H^1(\Weil_F, Z_{\tilde{G}^\vee})$-orbit of $\phi_{S_{Q,m}, \Lgrp{j}}$; we have
			\[ \Pi_{[\phi]} = \emptyset \iff \left[ \phi_{S_{Q,m}, \Lgrp{j}} \right] \cap \Image\left( (\check{\iota}_{Q,m})_* \right) = \emptyset. \]
		\item When $\Pi_\phi \neq \emptyset$, it is canonically a torsor under $\Ker\left[ (\check{\iota}_{Q,m})_* \right] \times H^1(F, S)$; for $4 \nmid m$, it is canonically a torsor under $\pi_0(C_\phi, 1)^D$ where $(\cdots)^D$ means the Pontryagin dual.
		\item Any two $\Pi_{\phi}$, $\Pi_{\phi'}$ are either disjoint or equal, and the latter case occurs exactly when
		\begin{compactitem}
			\item both are empty,
			\item $4 \nmid m$ and $\phi$ is equivalent to $\phi'$, or
			\item $4 \mid m$ and $[\phi] = [\phi']$.
		\end{compactitem}
	\end{itemize}
\end{theorem}
\begin{proof}
	The descriptions involving $\check{\iota}_{Q,m}$ and orbits just reinterpret Proposition \ref{prop:packet-prop-0} in terms of local Langlands correspondence for tori.
	
	In view of Lemma \ref{prop:isocrystal} and the discussion on $\iota_{Q,m,*}$, the refinement for $4 \nmid m$ is reduced to the existence of a natural $H^1(F,S_{Q,m}) \rightiso \pi_0(C_\phi, 1)^D$. This follows from Kottwitz's isomorphism $H^1(F, S_{Q,m}) \rightiso \pi_0((S_{Q,m}^\vee)^{\Gamma_F}, 1)^D$.
	

	It remains to show that non-equivalent $L$-parameters $\phi_1, \phi_2$ correspond to non-isomorphic $(\mathcal{E}_1, S_1, \phi_1)$, $(\mathcal{E}_2, S_2, \phi_2)$, provided that both $L$-packets are non-empty. This follows from the broader framework in \cite[Proposition 5.2.4]{Kal19},  since the $\chi$-data here are prescribed.
\end{proof}

\begin{remark}\label{rem:comparison-GG-tori}
	The recipe here can probably be compared with the local Langlands for metaplectic tori discussed in \cite[\S\S 8.1--8.3]{GG}.
\end{remark}

\begin{remark}
	The description of $\Pi_\phi$ here is modeled upon \cite{Kal15}. The \emph{extended pure inner forms} in \textit{loc. cit.} reduce to $G$ or $\tilde{G}$ itself in our situation, since the set of basic $G$-isocrystals $\mathbf{B}(G)_\text{bas}$ is trivial. This also amounts to taking $Z = \{1\}$ in the general framework of \cite{Kal19}.
\end{remark} 

We compare the central characters $\omega_\pi$ for $\pi \in \Pi_\phi$ next. Proposition \ref{prop:BD-adjoint-action} ensures $Z_{\tilde{G}} = \bm{p}^{-1}(\{\pm 1\})$ and $\omega_\pi$ is genuine. Recall that when $4 \nmid m$, the members of $\Pi_\phi$ are parameterized by conjugacy classes in a given stable class $\mathcal{E}$.

\begin{theorem}\label{prop:central-character}\index{central character}
	Assume $4 \nmid m$. Let $j,j': S \hookrightarrow G$ correspond to $\pi, \pi' \in \Pi_\phi$.
	\begin{compactitem}
		\item If $m \notin 2\Z$, then $\omega_\pi = \omega_{\pi'}$.
		\item If $m \equiv 2 \pmod 4$, then $\omega_\pi = \omega_{\pi'} \otimes \eta$ where $\eta: \{\pm 1\} \to \CC^\times$ maps $-1$ to $\lrangle{ \kappa^S_-, \mathrm{inv}(j,j')}$.
	\end{compactitem}
	Here $\kappa^S_-$ is as in Definition \ref{def:kappa-minus}.
\end{theorem}
\begin{proof}
	We have $\omega_\pi = \epsilon_j \theta_j|_{Z_{\tilde{G}}}$ and $\omega_{\pi'} = \epsilon_{j'} \theta_{j'}|_{Z_{\tilde{G}}}$. Apply \textbf{SS.2} of Definition \ref{def:stable-system} with Proposition \ref{prop:CAd-minus-1} (resp. Lemma \ref{prop:epsilon-invariance}) to compare the restrictions to $Z_{\tilde{G}}$ of $\theta_j, \theta_{j'}$ (resp. of $\epsilon_j, \epsilon_{j'}$).
\end{proof}
This is in clear contrast with the case of reductive groups and conforms to the prediction in \cite[\S 12.1]{GG}.

\subsection{Independencies}\label{sec:independence}
In what follows, $\phi$ always denote an epipelagic $L$-parameter into $\Lgrp{\tilde{G}}$. Recall from \S\ref{sec:L-group} that when $m \notin 2\Z$, the identification $\Lgrp{\tilde{G}} = \tilde{G}^\vee \times \Weil_F$ is canonical, and so is our construction of $\Pi_\phi$.

When $m \in 2\Z$, the identification $\Lgrp{\tilde{G}} = \tilde{G}^\vee \times \Weil_F$ depends on the choices of $(W, \lrangle{\cdot|\cdot})$, as well as $\psi: F \to \CC^\times$ when $m \equiv 2 \pmod 4$; the same choices enter in \eqref{eqn:splitting-genuine}. The next result should be compared with \cite[Proposition 11.1]{GG}.

\begin{theorem}\label{prop:packet-independence}
	Assume $m \equiv 2 \pmod 4$. For $\phi: \Weil_F \to \Lgrp{\tilde{G}}$, the $L$-packet $\Pi_\phi$ is independent of the choice of $\psi$; it is also invariant under dilation of $\lrangle{\cdot|\cdot}$.
\end{theorem}
\begin{proof}
	First consider the effect of changing $\psi$. Let $c \in F^\times$.
	\begin{compactenum}
		\item Changing $\psi$ to $\psi_c$ amounts to twist $\phi: \Weil_F \to \tilde{G}^\vee \times \Weil_F$ by the $\chi_c$ in \eqref{eqn:metaGalois-section-twist}. Accordingly, $\theta^\flat$ in the inducing datum will be twisted by the $\chi_{S_{Q,m}}: S_{Q,m}(F) \to \bmu_2$ corresponding to the image of $\chi_c$ under $H^1(\Weil_F, Z_{\tilde{G}^\vee}) \to H^1(\Weil_F, S_{Q,m}^\vee)$.
		\item The genuine character $\theta_j^\circ$ in Definition \ref{def:stable-system} is also controlled by the splitting over $S(F)_{p'}$; Lemma \ref{prop:splitting-variance} describes its dependence on $\psi$.
		\item The \textbf{SS.1} of Definition \ref{def:stable-system} describes the dependence of $\theta_j^\dagger$ on $\psi$.
	\end{compactenum}
	
	By construction (cf. \S\ref{sec:L-group}), there are $\Gamma_F$-equivariant commutative diagrams
	\[\begin{tikzcd}
		Y \arrow{r}{\frac{m}{2}}[swap]{\simeq} \arrow{rd}[swap]{\frac{m}{2}} & Y_{Q,m} \arrow[twoheadrightarrow]{r}  \arrow[hookrightarrow]{d} & Y_{Q,m}/Y_{Q,m}^\text{sc} \arrow{d}{\simeq} \\
		& Y \arrow[twoheadrightarrow]{r}[swap]{\check{\epsilon}_i \mapsto 1+2\Z } & \Z/2\Z
	\end{tikzcd} \quad \begin{tikzcd}
		S^\vee & S_{Q,m}^\vee \arrow{l}[swap]{\simeq} & Z_{\tilde{G}^\vee} \arrow[hookrightarrow]{l} \\
		& S^\vee \arrow{u}[swap]{\iota_{Q,m}^\vee} \arrow{lu}{\frac{m}{2}} & \bmu_2  \arrow{u}[swap]{\simeq} \arrow[hookrightarrow]{l}{i}
	\end{tikzcd} \]
	related by $\Hom(-, \CC^\times)$. We se see that $\chi_{S_{Q,m}} = \chi_S \circ \iota_{Q,m}$ where $\chi_S: S(F) \to \CC^\times$ corresponds to the composite of $i \circ \chi_c$ under the Langlands correspondence. If we identify $S^\vee$ with $(\CC^\times)^n$ through the usual basis $\{\check{\epsilon}_i\}_{i=1}^n$ of $Y$, then $i(-1) = (-1, \ldots, -1)$.

	To simplify notations, let us assume $S \simeq R_{K^\sharp/F} (K^1)$ in terms of parameters in \S\ref{sec:Sp-parameters}, where $K, K^\sharp$ are fields and $[K^\sharp:F]=n$. We must show
	\[ \overbracket{ \sgn_{K/K^\sharp}(c) }^{\text{for}\; \theta_j^\dagger} \overbracket{ (-1, c)_{F,m}^n }^{\text{for splitting}} = \overbracket{ \chi_S(-1) }^{\text{for}\; \theta^\flat}, \qquad -1 \in K^1. \]
	By Lemma \ref{prop:1-c} and \eqref{eqn:Hilbert-projection-formula}, this is equivalent to $\sgn_{K/K^\sharp}(c) (-1, c)_{K^\sharp, 2} = \chi_S(-1)$.

	Express $S$ as the quotient of $S_1 := R_{K/F}(\Gmm{K})$ via $\omega \mapsto \omega/\tau(\omega)$, where $\Gal{K/K^\sharp} = \{\identity, \tau \}$; dually $i_1: S^\vee \hookrightarrow S_1^\vee$ is the anti-diagonal embedding into $(\CC^\times \times \CC^\times)^n = (\CC^\times)^{2n}$ (cf. \S\ref{sec:Weil-restriction}). The composite $i_1 \circ i \circ \chi_c: \Gamma_F \to S_1^\vee$ corresponds, by Shapiro's isomorphism (``restriction composed with evaluation at 1''), to $\chi_c|_{\Gamma_K}: \Gamma_K \to \bmu_2$; it corresponds via local class field theory to the character
	\begin{align*}
		\chi_{S_1}: K^\times & \longrightarrow \mu_2 \\
		\omega & \longmapsto (c, N_{K/F}(\omega))_{F,2} = (c, N_{K/K^\sharp}(\omega))_{K^\sharp, 2} \quad \because\text{\text{\eqref{eqn:Hilbert-projection-formula}}}.
	\end{align*}
	Take $\omega = D \in K^{\sharp, \times}$ such that $K = K^\sharp(\sqrt{D})$, we conclude that
	\begin{align*}
		\chi_S(-1) & = \chi_{S_1}(\sqrt{D}) = (c, -D)_{K^\sharp, 2} = (c, D)_{K^\sharp, 2} (c, -1)_{K^\sharp, 2} \\
		& = \sgn_{K/K^\sharp}(c) (c, -1)_{K^\sharp, 2}.
	\end{align*}
	
	Now keep $\psi$ fixed and replace $\lrangle{\cdot|\cdot}$ by $c\lrangle{\cdot|\cdot}$. Upon replacing Lemma \ref{prop:variance-metaGalois} by Lemma \ref{prop:variance-gerbe}, the argument here is verbatim.
\end{proof}

\begin{remark}
	As a consequence, $\Pi_\phi$ depends only on $\psi \circ \lrangle{\cdot|\cdot}: W \to \CC^\times$ when $m \equiv 2 \pmod 4$. This principle is familiar in the case $m=2$.
\end{remark}

Next, assume that $4 \mid m$. The first observation is that the splitting over $S(F)_{p'}$ in \S\ref{sec:splittings-S} is irrelevant. Indeed, when $\Pi_{[\phi]} \neq \emptyset$, the inducing datum $\theta_j: \widetilde{jS} \to \CC^\times$ ranges over all genuine characters with prescribed pro-$p$ component. On the other hand, there is still an ambiguity by $H^1(\Weil_F, Z_{\tilde{G}^\vee})$ for identifying $\Lgrp{\tilde{G}}$ and $\tilde{G}^\vee \times \Weil_F$.

\begin{theorem}
	The $L$-packet $\Pi_{[\phi]}$ is unaltered by $H^1(\Weil_F, Z_{\tilde{G}^\vee})$-twists, therefore invariant under dilation of $\lrangle{\cdot|\cdot}$.
\end{theorem}
\begin{proof}
	This is built into the construction, since we worked with the $H^1(\Weil_F, Z_{\tilde{G}^\vee})$-orbit $[\phi]$.
\end{proof}

\subsection{Proof of stability}\label{sec:stability}
Let $F$ and the additive characters $\xi$, $\Lambda$ be as in \S\ref{sec:Adler-Spice}. Choose a stable system (Definition \ref{def:stable-system}) when $4 \nmid m$. Given any epipelagic $L$-parameter $\phi: \Weil_F \to \Lgrp{\tilde{G}}$, in \S\ref{sec:epipelagic-parameters} we have obtained
\begin{compactitem}
	\item the triple $(\mathcal{E}, S, \theta^\flat)$ up to isomorphism;
	\item the $L$-packet $\Pi_\phi$ (or the pre-$L$-packet $\Pi_{[\phi]}$ when $4 \mid m$).
\end{compactitem}

\begin{definition}\index{STheta@$S\Theta_\phi, S\Theta_{[\phi]}$}
	The \emph{stable character} associated to $\phi$ is
	\[ S\Theta_\phi := \sum_{\pi \in \Pi_\phi} \Theta_\pi \]
	where $\Theta_\pi$ is as in \S\ref{sec:Adler-Spice}. It is an invariant distribution represented by a locally integrable genuine function on $\tilde{G}$, smooth over $\tilde{G}_\text{reg}$. As in Remark \ref{rem:pre-packets}, when $4 \mid m$ it is reasonable to write $S\Theta_{[\phi]}$ instead.
\end{definition}

In view of the prescription $\theta^\flat \leadsto \theta_j$ on the pro-$p$ part and the Adler--Spice character formula, let us take $Y = Y_\xi \in \mathfrak{s}^*(F)_{-1/e}$ satisfying
\begin{equation}\label{eqn:Y-theta}
	\theta \circ \iota_{Q,m} = \theta^\flat \implies \theta \circ \exp = \xi( \lrangle{Y, \cdot}): \; \mathfrak{s}(F)_{1/e} \to \CC^\times,
\end{equation}
where $j: S \hookrightarrow G$ is any embedding in $\mathcal{E}$. We may also view $Y$ as in $\mathfrak{s}(F)_{-1/e}$ using \eqref{eqn:B_g-LMS}, i.e. $\theta \circ \exp = \xi(B_{\mathfrak{g}}(jY, j(\cdot)))$ for any $j \in \mathcal{E}$. Proposition \ref{prop:stable-good} asserts that $Y$ is regular.

In what follows, $\tilde{\gamma} \in \tilde{G}_{\mathrm{reg}}$ will stand for a compact element with topological Jordan decomposition $\tilde{\gamma} = \tilde{\gamma}_0 \gamma_{>0}$, and we write $\gamma, \gamma_0$ for their images in $G(F)$. Also put $J := G_{\gamma_0} = Z_G(\gamma_0)$. \index{topological Jordan decomposition}

\begin{lemma}\label{prop:STheta-formula}
	For $\tilde{\gamma}$ as above, we have
	\[ |D^G(\gamma)|^{\frac{1}{2}} S\Theta_\phi(\tilde{\gamma}) = \sum_{[j]: S \hookrightarrow J} \sum_{\substack{k: S \hookrightarrow J \\ k \in [j]}} \sum_{\theta_k} (\theta_k \epsilon_k)(\tilde{\gamma}_0) \hat{\iota}^J(kY, \log \gamma_{>0}), \]
	where
	\begin{compactitem}
		\item $[j]$ ranges over the stable $J$-conjugacy classes of embeddings $j: S \hookrightarrow J$, whose composite with $J \hookrightarrow G$ lies in $\mathcal{E}$;
		\item $k$ ranges over the conjugacy classes of embeddings $S \hookrightarrow J$ within $[j]$;
		\item $\theta_k$ ranges over the genuine characters $\widetilde{kS} \to \CC^\times$ prescribed in \S\ref{sec:inducing-data}, which is a singleton unless $4 \mid m$.
	\end{compactitem}
	On the other hand, $S\Theta_\phi$ vanishes at non-compact elements in $\tilde{G}_{\mathrm{reg}}$.
\end{lemma}
Note that $\gamma_0 \in Z_J(F)$ implies that $\gamma _0 \in kS(F)$ for all $k \in [j]$.
\begin{proof}
	The vanishing at non-compact $\tilde{\gamma}$ follows from the same property for each $\Theta_\pi$. Apply Theorem \ref{prop:Adler-Spice} to express $|D^G(\gamma)|^{1/2} S\Theta_\phi(\tilde{\gamma})$ as the sum of
	\[ \sum_{\substack{ g \in J(F) \backslash G(F) / jS(F) \\ g^{-1}\gamma_0 g \in jS(F) }} (\epsilon_j \theta_j) \left(g^{-1} \tilde{\gamma}_0 g\right) \hat{\iota}^J\left( g j(Y) g^{-1}, \log \gamma_{>0} \right) \]
	over conjugacy classes of $j: S \hookrightarrow G$ and those $\theta_j$ prescribed in \S\ref{sec:inducing-data}. As remarked in the proof of Theorem \ref{prop:Adler-Spice},
	\[ g^{-1} \gamma_0 g \in jS(F) \iff \Ad(g)(jS) \subset G_{\gamma_0} = J. \]
	One verifies readily the bijection
	\[\begin{tikzcd}[column sep=small]
		\left\{ (j,g) : \begin{array}{l}
			j: S \hookrightarrow G, \; j \in \mathcal{E} \\
			g \in J(F) \backslash G(F) / jS(F) \\
			\text{s.t.}\; \gamma_0 \in \Ad(g)jS(F)
		\end{array} \right\} \bigg/
		\begin{array}{l}
			(j,g) \sim (\Ad(h)j, gh^{-1} ) \\
			\forall h \in G(F)
		\end{array}
		\arrow[leftrightarrow]{d}{1:1} & (j,g) \arrow[mapsto]{d} & (k,1) \\
		\left\{ \begin{array}{l}
			k: S \hookrightarrow J \\
			\text{s.t.}\; (S \to J \hookrightarrow G) \in \mathcal{E} 
		\end{array}\right\} \bigg/ J(F)-\text{conj} & k := \Ad(g) \circ j & k \arrow[mapsto]{u}
	\end{tikzcd}\]
	By \textbf{SS.2} in Definition \ref{def:stable-system} and the naturality of toral invariants, for $k = \Ad(g) \circ j$ as above we have
	\[ \theta_k = \theta_j \circ \Ad(g^{-1}) , \quad \epsilon_k = \epsilon_j \circ \Ad(g^{-1}). \]
	All in all, the formula for $|D^G(\gamma)|^{1/2} S\Theta_\phi(\tilde{\gamma})$ can be written as a sum over $k$ modulo $J(F)$-conjugacy, followed by a sum over $\theta_k$. Furthermore, the $k$-sum can be partitioned according to the stable $J$-conjugacy classes $[j]$ of embeddings $S \hookrightarrow J$, whose composite with $J \hookrightarrow G$ lies in $\mathcal{E}$. This leads to the required formula $\sum_{[j]} \sum_k \sum_{\theta_k} (\theta_k \epsilon_k)(\tilde{\gamma}_0) \hat{\iota}^J(kY, \log \gamma_{>0})$.
\end{proof}

We will show the stability of $S\Theta_\phi$ (Definition \ref{def:stability}) by analyzing the formula of Lemma \ref{prop:STheta-formula}. Consider compact elements $\tilde{\gamma} = \tilde{\gamma}_0 \gamma_{>0} \in \tilde{G}_\text{reg}$ as before, with $\gamma_0 = \bm{p}(\tilde{\gamma}_0)$, etc. For the sum $\sum_k$ to be non-vacuous, we can further assume that $\gamma_0$ lies in some $kS$. It follows from Theorem \ref{prop:S-p'} that $\gamma_0^2 = 1$.

Some notational preparations are in order. Let $X \in \mathfrak{g}_\text{reg}(F)$, whose conjugacy class is parameterized by the datum $(L, L^\sharp, y, d)$ with $L = \prod_{h \in H} L_h$, $L^\sharp = \prod_{h \in H} L^\sharp_h$ as in \S\ref{sec:Sp-parameters} on the Lie algebra level, cf. the end of \cite[\S 3.1]{Li11}. Write $H_0 := \{h \in H: L_h \;\text{is a field} \}$ and $T := G_X$. As in \S\ref{sec:Sp-parameters} stable conjugacy $\Ad(g): X \mapsto X'$ modulo conjugacy is parameterized by $\text{inv}(\Ad(g)) = \text{inv}(X, X') \in H^1(F, T)$. Similarly, the stable conjugacy $\Ad(h): kS \rightiso k'S$ between embeddings $k,k': S \hookrightarrow G$ will be parameterized by $\text{inv}(\Ad(h)) = \text{inv}(k,k') \in H^1(F, S)$

\begin{theorem}\label{prop:stability-1}\index{stable distribution}
	When $4 \nmid m$, the distribution $S\Theta_\phi$ on $\tilde{G}$ is stable.
\end{theorem}
\begin{proof}
	Recall that stable conjugacy operates on the level of $\tilde{G}_\text{reg}$ since $4 \nmid m$. Consider a stable conjugacy $\Ad(g): \gamma \mapsto \gamma'$ of compact elements of $G_\text{reg}(F)$, which lifts to $\CaliAd(g): \tilde{\gamma} \mapsto \tilde{\gamma}'$. Uniqueness of topological Jordan decompositions implies $\CaliAd(g)(\tilde{\gamma}_0) = \tilde{\gamma}'_0$, since $\CaliAd(g): \widetilde{G_\gamma} \to \widetilde{G_{\gamma'}}$ is a group isomorphism (Proposition \ref{prop:CAd-prop}, \textbf{AD.2}).
	
	Observe that upon adjusting $g$ by $G(F)$, we may assume $\gamma_0 = \gamma'_0$. Indeed, for $2$-torsion elements of $G(F)$, conjugacy in $G(\bar{F})$ is the same as ordinary conjugacy since $H^1(F, \Sp(W_+) \times \Sp(W_-)) = 0$ where $W_\pm$ are the $\pm 1$-eigenspaces of $\gamma_0$; it remains to apply \textbf{AD.3--4} of Proposition \ref{prop:CAd-prop}. After this adjustment, we have $J := Z_G(\gamma_0) = Z_G(\gamma'_0)$ and $g \in J(\bar{F})$.

	Fix $[j]$. It remains to establish the stability of the piece
	\[ \sum_{\substack{k: S \hookrightarrow J \\ k \in [j] \\ / J(F)-\text{conj}. }} (\epsilon_k \theta_k)(\tilde{\gamma}_0) \hat{\iota}^J(kY, \log \gamma_{>0}) \]
	under $\CaliAd(g)$ where $g \in J(\bar{F})$. By Lemma \ref{prop:epsilon-invariance}, the character $\epsilon_k$ on $kS(F)$ is not affected by stable conjugacy. Therefore it suffices to look at
	\[ \sum_k \theta_k (\tilde{\gamma}_0) \hat{\iota}^J(kY, \log \gamma_{>0}). \]

	By writing $W = W_+ \oplus W_-$ according to the eigenvalues of $\gamma_0$, we may write $S = S_+ \times S_-$ and decompose $\theta_k = \vartheta_{k_+} \times \vartheta_{k_-}$ in parallel. The sum breaks into
	\begin{multline*}
		\sum_{k_+: S_+ \hookrightarrow \Sp(W_+)} \; \sum_{k_-: S_- \hookrightarrow \Sp(W_-)} \vartheta_{k_+}(\tilde{\gamma}_{0,+}) \hat{\iota}^{\Sp(W_+)}\left( k_+ Y_+, \log\gamma_{>0,+} \right) \\
		\vartheta_{k_-}(\tilde{\gamma}_{0,-}) \hat{\iota}^{\Sp(W_-)}\left( k_- Y_-, \log\gamma_{>0,-} \right).
	\end{multline*}
	Caution: the characters $\vartheta_{k_\pm}$ are not necessarily the ones associated to some stable system for $\bm{p}^{-1}(\Sp(W_\pm))$, whence the different notation. Nonetheless, $g = (g_+, g_-) \in J(F)$ and $\CaliAd(g) = \CaliAd((1, g_-)) \CaliAd((g_+, 1))$ operates separately in $\bm{p}^{-1}(\Sp(W_\pm))$ as $\CaliAd(g_\pm)$, a property that can be traced back to Theorem \ref{prop:G-T-reduction}. Hence both of $\vartheta_{k_\pm}$ inherit the property \textbf{SS.2} of Definition \ref{def:stable-system} under $\CaliAd(g_\pm)$.
	
	Our problem is thus reduced to the case $\gamma_0 = \pm 1$ and $J = G$, at the cost of replacing $\theta_k$ by some genuine character $\vartheta_k$ satisfying only \textbf{SS.2}. Given the \textbf{AD.1} of Proposition \ref{prop:CAd-prop}, we may and do assume that $\tilde{\gamma}_0 = \pm 1 = \gamma_0$ via the splitting in Definition \ref{def:lifting-minus-1}. Put
	\[ X := \log \gamma_{>0}, \quad X' := \log \gamma'_{>0} = \Ad(g)X, \quad X,X' \in \mathfrak{g}_\text{reg}(F). \]
	We have to show the constancy of $\sum_k \vartheta_k(\tilde{\gamma}'_0) \hat{\iota}^G(kY, X')$ when $g$ (thus $\tilde{\gamma}'_0$, $X'$) varies in $G(F) \backslash (G/G_\gamma)(F)$. Observe that the sum over $k$ also varies in a stable conjugacy class.
	
	In what follows, we shall regard $Y$ as an element of $\mathfrak{s}_\text{reg}(F)$, and $\hat{\iota}^G$ as a function on $\mathfrak{g}_\text{reg}(F) \times \mathfrak{g}_\text{reg}(F)$, by using the $\mathbb{B}$ from \eqref{eqn:B_g-LMS}.
	
	\textbf{Case A: $\gamma_0 = 1$}. We also have $\tilde{\gamma}'_0 = \CaliAd(g)(1) = 1$, therefore $\vartheta_k(\tilde{\gamma}_0) = 1 = \vartheta_k(\tilde{\gamma}'_0)$. The required stability amounts to
	\begin{equation}\label{eqn:stability-easy}
		\sum_k \hat{\iota}^G(kY, X) = \sum_k \hat{\iota}^G(kY, X').
	\end{equation}
	This is assured by Waldspurger's result \cite[1.6 Corollaire]{Wa97} with $G = \Sp(W) = H$, or its non-standard version \cite[\S 1.8]{Wa08}.
	
	\textbf{Case B: $\gamma_0 = -1$}. Treat the easier case $m \notin 2\Z$ first. Using the Proposition \ref{prop:CAd-minus-1} with $\delta_0=1$, we see that $\tilde{\gamma}'_0 = \CaliAd(g)(-1) = -1$ for all $g$. Moreover, when $k$ gets replaced by a stable conjugate $k' = \Ad(h)k$, \textbf{SS.2} of Definition \ref{def:stable-system} and the previous step imply
	\[ \vartheta_{k'}(-1) = \vartheta_{k'}(\CaliAd(h)(-1)) = \vartheta_k(-1). \]
	Hence $\vartheta_k(\tilde{\gamma}'_0)$ is independent of $(k,g)$. The stability can thus be established as in Case A.
	
	Henceforth assume $m \equiv 2 \pmod 4$ in Case B. Recall the description of endoscopic data in \cite[Chapitre X]{Wa01}. Choose an endoscopic datum $(\SO(V_1, q_1), \ldots)$ of $G$ (regarding only the endoscopic group), where $\dim_F V_1 = 2n$, such that there is a matching $Y \leftrightarrow Y_{\SO} \in \so(V_1, q_1)_\text{reg}$ between stable conjugacy classes. Such an endoscopic datum is necessarily elliptic since $Y \in \mathfrak{s}(F)_\text{reg}$ and $S/Z_G$ is anisotropic. Pick a transfer factor $\Delta^1$ on Lie algebras for this datum, which is canonical up to $\CC^\times$.

	Claim: there exists a function $\Delta^2: \left\{Z \in \mathfrak{g}_\text{reg}(F) : Z \stackrel{\text{st}}{\sim} X \right\} \to \CC^\times$ such that
	\begin{gather*}
		 Z_1, Z_2 \stackrel{\text{st}}{\sim} X \implies \Delta^2(Z_2) = \lrangle{\kappa_-, \text{inv}(Z_1, Z_2)} \Delta^2(Z_1), \\
		\dfrac{\vartheta_k(\tilde{\gamma}'_0) \Delta^2(X') }{ \Delta^1(Y_{\SO}, kY) } = \text{constant}, \quad \text{when $k,g$ vary;}
	\end{gather*}
	recall Definition \ref{def:kappa-minus} for $\kappa_-$. By picking basepoints for $X'$ and $k$, this reduces to the observations below.
	\begin{compactitem}
		\item The variation of $g$ can be realized in various $\SL(2, K_h^\sharp)$ with $h \in H_0$, by the Lie algebra version of Proposition \ref{prop:stable-reduction-SL2}. Since $\CaliAd(g)$ is also realized in this manner, Proposition \ref{prop:CAd-minus-1} implies that $\tilde{\gamma}'_0 = \CaliAd(g)(-1)$ equals $\lrangle{\kappa_-, \Ad(g)} \cdot (-1)$. The same holds for $\vartheta_k(\tilde{\gamma}'_0)$ because $\vartheta_k$ is genuine.
		\item when $k$ is replaced by a stable conjugate $\Ad(h) k$, \textbf{SS.2} and Proposition \ref{prop:CAd-minus-1} imply
		\[ \lrangle{ \kappa_-, \text{inv}(\Ad(h)) } \cdot \vartheta_{\Ad(h)k}(\tilde{\gamma}'_0) = \vartheta_{\Ad(h)k}(\CaliAd(h)(\tilde{\gamma}'_0)) = \vartheta_k(\tilde{\gamma}'_0). \]
	\end{compactitem}
	These match the behavior of $\Delta^1(Y_{\SO}, \cdot)$ (resp. $\Delta^2$) when $k$ (resp. $g$) varies; for $\Delta^1$ we invoke the description in \cite[X.8]{Wa01}. This proves the claim and we are reduced to show that
	\begin{equation}\label{eqn:Delta-stability}
		\Delta^2(X')^{-1} \sum_k \Delta^1(Y_{\SO}, kY) \hat{\iota}^G(kY, X')
	\end{equation}
	is independent of $g$ or of the conjugacy class of $X'$. Set $\hat{i}^J(Z,X) := |D^J(X)|^{-1} \hat{\iota}(Z,X)$. Then
	\begin{multline*}
		\Delta^2(X')^{-1} \sum_k \Delta^1(Y_{\SO}, kY) \hat{i}^G(kY, X') = \gamma_\xi(\mathfrak{g})^{-1} \gamma_\xi(\so(V_1, q_1)) \\
		\times \sum_{Y_1 \stackrel{\text{st}}{\sim} Y_{\SO}} \sum_{Z_1 \leftrightarrow X'}  w(Z_1)^{-1} \Delta^2(X')^{-1} \Delta^1(Z_1, X')  \hat{i}^{\SO(V_1, q_1)}(Y_1, Z_1)
	\end{multline*}
	by virtue of \cite[p.155]{Wa97}, where
	\begin{compactitem}
		\item $\gamma_\xi(\mathfrak{g})$, $\gamma_\xi(\so(V_1, q_1))$ are as in \cite[p.154]{Wa97},
		\item $Y_1, Z_1$ range over the conjugacy classes in $\so(V_1, q_1)_\text{reg}$,
		\item $w(Z_1)$ is the number of conjugacy classes in the stable class of $Z_1$.
	\end{compactitem}
	Next, note that $Z_1 \leftrightarrow X' \iff Z_1 \leftrightarrow X$ since $X \stackrel{\text{st}}{\sim} X'$ in $\mathfrak{g}$. Assume that such a $Z_1$ exists, otherwise \eqref{eqn:Delta-stability} reduces to $0$ for all $X' \stackrel{\text{st}}{\sim} X$, and there is nothing to prove.
	
	All in all, we are reduced to show the constancy of $\Delta^2(X')^{-1} \Delta^1(Z_1, X')$ where $Z_1$ is kept fixed. Again, this is because both factors undergo a sign change $\lrangle{\kappa_-, \text{inv}(X'', X') }$ when $X'$ is replaced by a stable conjugate $X''$; for $\Delta^1(Z_1, \cdot)$ this is again a consequence of \cite[X.8]{Wa01}.
\end{proof}

The case $4 \mid m$ requires different arguments.
\begin{lemma}\label{prop:four-vanishing}
	Suppose $4 \mid m$. For every $\tilde{\gamma} \in \tilde{G}_{\mathrm{reg}}$, we have $S\Theta_\phi(\tilde{\gamma}) = 0$ unless the image $\gamma \in G_\mathrm{reg}(F)$ of $\tilde{\gamma}$ is topologically unipotent, in which case the formula of Lemma \ref{prop:STheta-formula} reduces to
	\[ |D^G(\gamma)|^{\frac{1}{2}} S\Theta_{[\phi]}(\tilde{\gamma}) = |S(F)_{p'}| \underbracket{\tilde{\gamma}_0}_{\in \bmu_m} \sum_{\substack{k: S \hookrightarrow G \\ k \in \mathcal{E}}} \hat{\iota}^G(kY, \log \gamma_{>0}). \]
\end{lemma}
\begin{proof}
	We may assume $\tilde{\gamma}$ to be a compact element. The formula in Lemma \ref{prop:STheta-formula} contains a sum
	\[ \left( \sum_{\theta_k} \theta_k \right) (\tilde{\gamma}_0) \]
	where $\theta_k$ ranges over all genuine characters of $\widetilde{jS}$ that has a prescribed pro-$p$ component. Since $\widetilde{jS} = \widetilde{jS(F)_{p'}} \times S(F)_{0+}$, it remains to apply Fourier inversion.
\end{proof}

\begin{theorem}\label{prop:stability-2}\index{stable distribution}
	When $4 \mid m$, the distribution $S\Theta_{[\phi]}$ on $\tilde{G}$ is stable.
\end{theorem}
\begin{proof}
	Following the paradigm of Definition \ref{def:stability}, we consider a maximal $F$-torus $T$, take $\sigma \in \text{Sgn}_m(T)$ as in \eqref{eqn:Sgn} and form the homomorphism $\tilde{T}^\sigma_{Q,m} \to \tilde{T} \subset \tilde{G}$ of \eqref{eqn:iota-cover}. By Lemma \ref{prop:four-vanishing}, it suffices to consider $\tilde{\delta} \in \tilde{T}_\text{reg}$ of topologically unipotent image, or equivalently $\tilde{\delta}_0 \in \bmu_m$.

	Consider an element $(\tilde{\delta}, \delta'_{Q,m}) \in \tilde{T}^\sigma_{Q,m}$; we have to adjust $\delta'_{Q,m}$ to some $\delta_{Q,m}$ to verify the requirement of Definition \ref{def:stability}. For this purpose, we may translate $\tilde{\delta}$ by $\bmu_m$ so that $\tilde{\delta} = \delta_{>0}$.
	\begin{enumerate}
		\item Parameterize $T$ by a datum $(L, L^\sharp, \ldots)$ as usual, with $L = \prod_{h \in H} L_h$, etc. Recall that $\sigma \in \text{Sgn}_m(T) = \{\pm 1\}^H$. Decompose $\delta$ into $(\delta_h)_{h \in H}$; each $\delta_h \in L^1_h$ is still topologically unipotent. The decomposition also applies to $\sigma$, $\delta_{Q,m}$ and it respects $\iota_{Q,m}$. Hereafter we fix $h \in H$ and work inside $\widetilde{\SL}(2, L_h^\sharp)$. In other words, we reduce to the case $n=1$ modulo Weil restriction.
		\item Assume $n=1$. By the topological unipotence of $\delta$ and $p \nmid m$, there exists a topologically unipotent $\mu \in T_{Q,m}(F)$ such that $\delta = \iota_{Q,m}(\mu)$. We contend that $\sigma=1$ when $T$ is anisotropic: otherwise we would have $\delta \in \Image(\iota_{Q,m}) \cap ((-1) \cdot \Image(\iota_{Q,m}))$ that contradicts Proposition \ref{prop:iota-kernel}.
		\begin{compactitem}
			\item If $T$ is split, we take $\delta_{Q,m} = \delta'_{Q,m}$.
			\item If $T$ is anisotropic, then $\sigma=1$ and we take $\delta_{Q,m} = \mu$.
		\end{compactitem}
		\item Reassembling these rank-one pieces, we obtain $(\tilde{\delta}, \delta_{Q,m}) \in \tilde{T}^\sigma_{Q,m}$.
	\end{enumerate}

	Using Lemma \ref{prop:four-vanishing}, the stability of $S\Theta_{[\phi]}$ amounts to
	\begin{multline*}
		\CaliAd^\sigma(g)(\delta_{>0}, \delta_{Q,m}) = \left(\tilde{\eta}, \Ad(g)(\delta_{Q,m})\right) \implies \\
		\sum_{\substack{k: S \hookrightarrow G \\ k \in \mathcal{E}}} \hat{\iota}^G(kY, \log \gamma_{>0}) = \underbracket{ \tilde{\eta}_0 }_{\in \bmu_m} \sum_{\substack{k: S \hookrightarrow G \\ k \in \mathcal{E}}} \hat{\iota}^G (kY, \log \eta_{>0}))
	\end{multline*}
	for any stable conjugation $\Ad(g)(\delta) = \eta$, which also implies $\Ad(\delta_{>0}) = \eta_{>0}$. As seen in \eqref{eqn:stability-easy}, the two sums $\sum_k$ are equal, thus it suffices to show $\tilde{\eta}_0 = 1$. Now recall that $\CaliAd^\sigma(g)$ is built upon
	\begin{compactitem}
		\item stable conjugacy in the case of $\widetilde{\SL}(2, L_h^\sharp)$ (Definition \ref{def:st-conj-SL2}), which uses $g \in G_\text{ad}(F)$ and incorporates a factor $\Cali_m(\nu, t_0)$;
		\item the harmless $G(F)$-conjugacy.
	\end{compactitem}
	Hence the calculation of $\tilde{\eta}_0$ reduces to the case $n=1$, which is dealt with as follows.
	\begin{itemize}
		\item When $T$ is split, we have $H^1(F,T)=0$ so $\CaliAd^\sigma(g)$ reduces to ordinary conjugacy by \textbf{AD.3} of Proposition \ref{prop:CAd-prop}. Thus $\tilde{\eta} = \Ad(g)(\delta_{>0}) = \eta_{>0}$.
		\item When $T$ is anisotropic, $\sigma=1$, we may assume $g \in G_\text{ad}(F)$ and $\tilde{\eta} = \Cali_m(\nu, t_0) \Ad(g)(\delta_{>0})$, where $t_0 \in L^1$ corresponds to $\delta_{Q,m}$. The factor $\Cali_m(\nu, \cdot)$ is multiplicative and $\mu_2$-valued. On the other hand $p \neq 2$ as $p \nmid m$, and $t_0$ is topologically unipotent since $\delta_0$ is. Hence $\Cali_m(\nu, t_0) = 1$ and $\tilde{\eta} = \eta_{>0}$.
	\end{itemize}
	Reassembling matters, we conclude that $\tilde{\eta}_0 = 1$ as desired.
\end{proof}

\section{A stable system for \texorpdfstring{$m \equiv 2 \bmod 4$}{m congruent to 2 mod 4}}\label{sec:stable-system}
For the definition of stable systems, see Definition \ref{def:stable-system}. The case of $m \notin 2\Z$ has been discussed in Proposition \ref{prop:std-stable-system}. Now we address the case of $m \equiv 2 \pmod 4$. Note that the only external evidence comes from the $m=2$ case, see Theorem \ref{prop:Theta-compatibility}.

\subsection{Moment maps}\label{sec:MM}
Let $F$ be a field with $\text{char}(F) \neq 2$. Let $(V,q)$ be a quadratic $F$-vector space with $\dim_F V = 2n+1$, $d^\pm(V,q) = 1$. The corresponding special orthogonal group is $\SO(V,q)$.
	
For any given $F$-linear map $T: W \to V$, define its adjoint ${}^\star T$ by
\[ \lrangle{ {}^\star T v | w } = q(v|Tw), \quad v \in V, \; w \in W. \]
Note that this differs from \cite[\S 6.1]{LMS16} by a sign. We say $Y \in \syp(W)$ corresponds to $Y' \in \so(V,q)$ if they are related by a $T \in \Hom_F(W,V)$ by the diagram below.\index{moment map}
\[\begin{tikzcd}
	& \Hom_F(W, V) \arrow{ld}{M_W}[swap]{{}^\star T \cdot T \mapsfrom T } \arrow{rd}{T \mapsto T \cdot {}^\star T}[swap]{M_V} & \\
	\syp(W) & & \so(V,q)
	\end{tikzcd}\]
The arrows $M_W, M_V$ are the \emph{moment maps}. Note that $\Sp(W) \times \Or(V,q)$ acts on the left of $\Hom_F(W,V)$ as
\begin{equation}\label{eqn:moment-action}
	T \longmapsto (g,h)T = hTg^{-1}, \quad (g,h) \in \Sp(W) \times \Or(V,q).
\end{equation}
It is routine to verify that ${}^\star((g,h)T) = g \cdot {}^\star T \cdot h^{-1}$. Hence
\[ M_W((g,h)T) = g M_W(T) g^{-1}, \quad M_V((g,h)T) = h M_V(T) h^{-1}, \quad (g,h) \in \Sp(W) \times \Or(V,q). \]

\begin{theorem}\label{prop:mm}
	Suppose $F$ is a local field. The correspondence above yields a bijection
	\[ \dfrac{\syp(W)_{\mathrm{reg}}}{\text{conj}} \xleftrightarrow{1:1} \bigsqcup_{(V,q)} \dfrac{\so(V,q)_{\mathrm{reg}} }{\text{conj}}, \]
	where $(V,q)$ ranges over isomorphism classes of quadratic $F$-vector spaces with $\dim_F V = 2n+1$ and $d^\pm(V,q) = 1$. Two elements match if and only if they have the same nonzero eigenvalues (counting multiplicities).
\end{theorem}
\begin{proof}
	The archimedean case is \cite[Proposition 2.5, Lemma 2.8]{Ad98}, and the arguments therein work in general.
\end{proof}

\begin{remark}\label{rem:explicit-moment-map}\index{qY@$q\lrangle{Y}$}
	Suppose that $Y \in \syp(W)_\text{reg}$ lands in $\so(V,q)$. In the proof cited above, the quadratic $F$-space $(V,q)$ is obtained as follows. Define the quadratic form
	\[ q\lrangle{Y}: w \mapsto \lrangle{Yw|w}, \quad w \in W. \]
	There exists a class $a \in F^\times/F^{\times 2}$ such that $(V,q) := q\lrangle{Y} \oplus \lrangle{a}$ satisfies $d^\pm(V,q)=1$. To determine $a$, take $d^\pm$ on both sides to conclude that $(-1)^n a \det Y = 1$ modulo $F^{\times 2}$, i.e. $a = (-1)^n \det Y \bmod F^{\times 2}$.

	Write
	$\begin{tikzcd}
	V = W \oplus F \arrow[yshift=2pt]{r}{\text{pr}} & W \arrow[yshift=-2pt]{l}{\iota}
	\end{tikzcd}$ for the evident projection and inclusion. We can construct $T \in M_W^{-1}(Y)$ by taking $T = \iota$. Indeed, by definition ${}^\star T = Y \circ \text{pr}$, hence ${}^\star T \cdot T = Y$ and $Y' := T \cdot {}^\star T = \iota \circ Y \circ \text{pr} \in \so(V,q)$ corresponds to $Y$.
	
	Consider the maximal $F$-torus $S := Z_{\Sp(W)}(Y)$. It naturally sits in $\SO(W, q\lrangle{Y}) \subset \SO(V,q)$ as a maximal torus, since for all $g \in S(F)$ we have
	\[ q\lrangle{Y}(gw) = \lrangle{Ygw|gw} = \lrangle{gYw|gw} = q\lrangle{Y}(w), \quad w \in W. \]
\end{remark}

\begin{remark}\label{rem:moment-map-basis}
	In the explicit construction of Remark \ref{rem:explicit-moment-map}, assume that $S$ is the split maximal torus associated to a symplectic basis $\{ e_{\pm i} \}_{i=1}^n$ of $W$, thus $Y$ is diagonalizable. Then $\{e_{\pm i}\}_{i=1}^n$ will become a hyperbolic basis (Definition \ref{def:hyperbolic-basis}) for the quadratic $F$-vector space $(W, q\lrangle{Y}) \subset (V, q)$, as a quick computation in $\SL(2)$ shows. Consequently, $Y'$ belongs to the split maximal torus of $\SO(V,q)$ associated to the latter basis. This also implies that the Weyl groups of $S$ inside $\Sp(W)$ and $\SO(V,q)$ are naturally identified; the same holds for the roots. Caution: long roots of $\Sp(W)$ go to short roots of $\SO(V,q)$.
	
	The correspondence via moment maps is stable under change of base fields, thus the conclusions above extend to non-split $S$: the absolute Weyl groups and roots are in bijection in a $\Gamma_F$-equivariant manner.
\end{remark}

Let $F$ be local with residual characteristic $\neq 2$. To prove the next result, we shall adopt the language of lattice functions to describe the Bruhat--Tits buildings for classical groups, as summarized in \cite[\S 4]{LMS16}.

\begin{definition}
	For an $F$-vector space of finite dimension, a \emph{lattice function} is an assignment $s \mapsto \mathcal{L}_s$, where $s \in \R$ and $\mathcal{L}_s$ are $\mathfrak{o}_F$-lattices in $V$, satisfying
	\begin{compactitem}
		\item $s' \geq s \implies \mathcal{L}_{s'} \subset \mathcal{L}_s$;
		\item $\mathcal{L}_{s + v_F(\varpi_F)} = \mathfrak{p}_F \mathcal{L}_s$;
		\item $\mathcal{L}_s = \bigcap_{s'' < s} \mathcal{L}_{s''}$.
	\end{compactitem}
	Denote by $\text{Latt}_V$ the set of lattice functions. For any $\mathcal{L} \in \text{Latt}_V$ set $\mathcal{L}[r]: s \mapsto \mathcal{L}_{s+r}$. We also set $\mathcal{L}_{s+} := \bigcup_{s' > s} \mathcal{L}_{s'}$.	Next, let $h$ be a quadratic or symplectic form on $V$. For every $\mathcal{L} \in \text{Latt}_V$, set
	\[ \mathcal{L}^\sharp: s \mapsto \left\{ v \in V: h\left( v | \mathcal{L}_{(-s)+} \right) \subset \mathfrak{p}_F \right\}, \quad \mathcal{L}^\sharp \in \text{Latt}_V. \]
	We say $\mathcal{L}$ is \emph{self-dual} if $\mathcal{L} = \mathcal{L}^\sharp$; the self-dual lattice functions form a subset $\text{Latt}_V^h \subset \text{Latt}_V$.
\end{definition}

It is known that there exists an equivariant affine bijection $\text{Latt}_V^h \rightiso \mathcal{B}(U(V,h), F)$, where $U(V,h)$ stands for the isometry group of $(V,h)$, possibly disconnected. Also, $\text{Latt}_V$ is in equivariant affine bijection with $\mathcal{B}(\GL(V), F)$; furthermore $\gl(V)_{\mathcal{L}, r} = \left\{ A \in \gl(V) : \forall s, \; A\mathcal{L}_s \subset \mathcal{L}_{s+r} \right\}$ under this identification.

Now consider the circumstance of Remark \ref{rem:explicit-moment-map}, write $G := \Sp(W)$, $H := \SO(V,q)$ and assume that
\begin{itemize}
	\item $S := Z_G(Y)$ is a tame anisotropic maximal $F$-torus;
	\item $v_F(\dd\alpha(Y)) = r$ for all absolute roots $\alpha$ of $S \subset G$, for some $r \in \Q$.
\end{itemize}
Denote by $G \xleftarrow{j} S \xrightarrow{k} H$ the natural embeddings, and let $x \in \mathcal{B}(G, F)$, $x' \in \mathcal{B}(H, F)$ be determined by $j,k$ as in \cite{Pra01}. Denote by $\mathcal{L} \in \mathrm{Latt}^{\lrangle{\cdot|\cdot}}_W$ and $\mathcal{L}' \in \mathrm{Latt}^q_V$ the self-dual lattice functions corresponding to $x$ and $x'$, respectively.

\begin{lemma}\label{prop:Y-translate}
	For $\mathcal{L}$ as above, we have $Y\mathcal{L}_s = \mathcal{L}_{s + r}$ for all $s \in \R$.
\end{lemma}
\begin{proof}
	The goal is to identify $s \mapsto Y\mathcal{L}_s$ and $\mathcal{L}[r]$ in $\text{Latt}_W$. Take a tamely ramified finite Galois extension $L/F$ to split $S$, and consider the base change map $\mathrm{BC}: \mathcal{B}(\GL(W), F) \to \mathcal{B}(\GL(W), L)^{\Gal{L/F}}$. Note that $\text{BC}(\mathcal{L}[r]) = \text{BC}(\mathcal{L})[r]$, provided that the simplicial structures are defined by the valuation $v$ of $L$ extending $v_F$. Since $\text{BC}$ is $\GL(W)$-equivariant, $\text{BC}(Y\mathcal{L}) = Y \cdot \text{BC}(\mathcal{L})$. Put $\mathcal{L}^\natural := \text{BC}(\mathcal{L})$. By tamely ramified descent on Bruhat--Tits buildings, $\text{BC}$ is bijective and it suffices to check $\mathcal{L}^\natural [r] = Y \mathcal{L}^\natural$.
	
	What remains is easy: by the recipe of \cite[Remark 3]{Pra01} there exists a symplectic basis of $W \otimes_F L$ that diagonalizes $S$, such that $\mathcal{L}^\natural$ lies in the corresponding apartment (equivalently, $\mathcal{L}^\natural$ corresponds to a splittable norm under this basis, see \cite[Definition 4.1.2]{LMS16}). Since all eigenvalues of $Y$ satisfy $v(\cdot)=r$, an easy computation in the split case yields $\mathcal{L}^\natural [r] = Y \mathcal{L}^\natural$ as asserted.
\end{proof}

\begin{proposition}\label{prop:T-depth}
	In the circumstance above, the $T = \iota \in \Hom_F(W, V)$ in Remark \ref{rem:explicit-moment-map} satisfies $T\mathcal{L}_s \subset \mathcal{L}'_{s + \frac{r}{2}}$ for all $s \in \R$.
\end{proposition}
\begin{proof}
	Recall that $(V,q) = (W, q\lrangle{Y}) \oplus \lrangle{(-1)^n \det Y}$. Take any self-dual lattice function $\mathcal{M}$ for $\lrangle{(-1)^n \det Y}$. Take
	\[ \mathcal{L}'' := \mathcal{L}\left[ \frac{-r}{2} \right] \oplus \mathcal{M} \; \subset \text{Latt}_V. \]
	We contend that $\mathcal{L}'' \in \text{Latt}_V^q$. Since the $\sharp$-operator is readily seen to commute with $\oplus$, it suffices to show $\mathcal{L}[-\frac{r}{2}]$ is self-dual relative to $q\lrangle{Y}$. Lemma \ref{prop:Y-translate} implies that
	\begin{align*}
		\left( \mathcal{L}\left[ \frac{-r}{2} \right]^\sharp\right)_s & = \left\{ w \in W : \lrangle{Yw | \mathcal{L}_{(-s - \frac{r}{2})+}} \subset \mathfrak{p}_F \right\} \\
		& = \left\{ w \in W : \lrangle{w | \mathcal{L}_{( \frac{r}{2} - s )+} } \subset \mathfrak{p}_F \right\} \\
		& = (\mathcal{L}^\sharp)_{\frac{-r}{2} + s} = \mathcal{L}_{ \frac{-r}{2} + s},
	\end{align*}
	thus $\mathcal{L}[\frac{-r}{2}]$ is self-dual. Lemma \ref{prop:Y-translate} also implies $Y\mathcal{L}_s = \mathcal{L}''_{s + \frac{r}{2}}$. Hence it remains to show that $\mathcal{L}'' = \mathcal{L}'$.
	
	Make a tame base change to a Galois extension $L/F$ that splits $S$. As seen in the proof of Lemma \ref{prop:Y-translate}, the fact that $\mathcal{L}$ is determined by $j$ means that $\mathcal{L}$ lies in the apartment $\mathcal{A}_j \subset \mathcal{B}(G, L)$ associated to a symplectic basis $\{e_{\pm i}\}_i$ diagonalizing $jS_L$. By Remark \ref{rem:moment-map-basis}, $\{e_{\pm i}\}_i$ is also a hyperbolic basis for $(W, q\lrangle{W}) \otimes_F L$ diagonalizing $kS_L$. We conclude that $\mathcal{L}'' = \mathcal{L}[\frac{-r}{2}] \oplus \mathcal{M}$ corresponds to a point of $\mathcal{B}(H, L)$ lying in the apartment determined by $kS_L$, after base change.

	The aforementioned point is $\Gal{L/F}$-fixed as it is defined over $F$. By the characterization \cite{Pra01} of $x'$ via tamely ramified descent, we infer that $x'$ is the point corresponding to $\mathcal{L}''$. See also \cite[\S 5.1]{LMS16}.
\end{proof}

\subsection{The construction}\label{sec:ss-construction}
The assumptions on $F$, $\tilde{G}$, etc. from \S\ref{sec:inducing-data} are in force. Assume $m \equiv 2 \pmod 4$ and choose an additive character $\psi$ of $F$. According to Definition \ref{def:stable-system}, given
\begin{compactitem}
	\item an inducing datum $(\mathcal{E}, S, \theta^\flat)$ (Definition \ref{def:inducing-data}),
	\item a $G(F)$-conjugacy class $j: S \hookrightarrow G$ in $\mathcal{E}$,
\end{compactitem}
we have to define a character $\theta_j^\dagger: jS(F)_{p'} \to \{\pm 1\}$. To begin with, define
\begin{equation}\label{eqn:a}
	a := a(\psi) = \min\left\{ b \in \Z: \psi|_{\mathfrak{p}_F^b} \; \text{non-trivial} \right\}.
\end{equation}
Pick $Y = Y_\psi \in \mathfrak{s}(F)$ such that for any $j \in \mathcal{E}$,
\begin{equation}\label{eqn:Y-theta-2}
	\theta \circ \iota_{Q,m} = \theta^\flat \implies \theta \circ \exp = \psi(\mathbb{B}(jY_\psi , j(\cdot))): \mathfrak{s}(F)_{1/e} \to \CC^\times.
\end{equation}
This is reminiscent of \eqref{eqn:Y-theta}, but here we pass to $\mathfrak{s}$ using $\mathbb{B}$ and make no assumption on $a(\psi)$.

\begin{lemma}\label{prop:stable-good-2}
	We have $Y_\psi \in \mathfrak{s}(F)_{a(\psi) - \frac{1}{e}}$; only the coset $Y_\psi + \mathfrak{s}(F)_{a(\psi)}$ is canonical. Furthermore, $Y_\psi \in \mathfrak{s}_{\mathrm{reg}}(F)$ and all eigenvalues $\lambda \in \bar{F}$ of $Y_\psi$ satisfies $v_{\bar{F}}(\lambda) = a(\psi) - \frac{1}{e}$
\end{lemma}
\begin{proof}
	Set $\xi := \psi_{\varpi_F^{a(\psi)}}$. Then $Y_\psi = \varpi_F^{a(\psi)} Y_\xi$ and $a(\xi)=0$. Our problem is thus reduced to the case $a(\psi)=0$, which has been addressed in Proposition \ref{prop:stable-good}.
\end{proof}

Using the recipe of \S\ref{sec:Sp}, we parameterize the conjugacy class $jY$ by the datum $(K, K^\sharp, \vec{y}, \vec{c})$, where $K = \prod_{i=1}^n K_i$, $K^\sharp = \prod_{i=1}^n K_i^\sharp$, $\vec{y} = (y_i)_{i \in I}$ and $\vec{c} = (c_i)_{i \in I} \in K^\times$, with $\tau(\vec{y}) = -\vec{y}$, $\tau(\vec{c}) = -\vec{c}$ as usual. All $K_i$ are fields since $S$ is anisotropic. We caution the reader that the parametrization depends on $\lrangle{\cdot|\cdot}$, which will be rescaled later on. Forgetting $Y$ gives rise to the datum $(K, K^\sharp, \vec{y})$ that parameterizes the conjugacy class of $j$.

\begin{definition}\label{def:dagger}
	For each $i \in I$, define a quadratic $K_i^\sharp$-vector space $(V_{i,Y}, q_{i,Y})$ as follows
	\begin{align*}
		\left( V_{i,Y}, q_{i,Y} \right) & := (K_i, q_{i,Y}^0) \oplus \lrangle{ (-1)^n \det Y }, \quad \text{where} \\
		K_i: & \quad \text{viewed as a $K_i^\sharp$-vector space}, \\
		q^0_{i,Y}(u|v) & := -\Tr_{K_i/K_i^\sharp} \left( \tau(u)v y_i c_i \right) \\
		& = \Tr_{K_i/K_i^\sharp} \left( \tau(y_i u)v c_i \right), \quad u,v \in K_i.
	\end{align*}
	Let $\theta_j^\dagger: jS(F)_{p'} \to \{ \pm 1\}$ be the homomorphism that maps $-1 \in K_i^1 \hookrightarrow jS(F)_{p'}$ to
	\[ \epsilon(V_{i,Y}, q_{i,Y}) \left(-1, d^\pm(V_{i,Y}, q_{i,Y})\right)_{K_i^\sharp, 2} , \qquad \text{for each}\; i \in I. \]
\end{definition}
Notice that $\det Y = N_{K/F}(\vec{y})$, thus the definition is phrased entirely in terms of the algebraic data $(K, K^\sharp, \vec{y}, \vec{c})$. If we modify everything by an isomorphism of étale $F$-algebras with involution, then $\theta_j^\dagger$ remains unaltered.

\begin{lemma}\label{prop:dagger-Weil}
	Let $\psi_i$ be any additive character of $K_i^\sharp$, for a given $i \in I$. The value of $\theta_j^\dagger$ at $-1 \in K_i^1$ equals
	\[ \gamma_{\psi_i} (q_{i,Y}^0) \gamma_{\psi_i}((-1)^n \det Y) \gamma_{\psi_i}(1)^{-2} \gamma_{\psi_i}\left( \mathfrak{d}_i (-1)^n \det Y \right)^{-1} \cdot \left(-1, \mathfrak{d}_i (-1)^n \det Y \right)_{K_i^\sharp, 2} \]
	where $\mathfrak{d}_i$ is the class in $F^\times/F^{\times 2}$ represented by $D_i$, if $K_i = K_i^\sharp(\sqrt{D_i})$.
\end{lemma}
\begin{proof}
	The relation between Weil's constant, discriminant and Hasse invariant over $K_i^\sharp$ (see eg. \cite[1.3.4]{Per81}) says that
	\[ \epsilon(V_{i,Y}, q_{i,Y}) = \gamma_{\psi_i}( q_{i,Y} ) \gamma_{\psi_i}(1)^{-2} \gamma_{\psi_i}\left(d^\pm(V_{i,Y}, q_{i,Y})\right). \]
	We have $\gamma_{\psi_i}(q_{i,Y}) = \gamma_{\psi_i}(q_{i,Y}^0) \gamma_{\psi_i}((-1)^n \det Y)$ by the additivity of $\gamma_{\psi_i}$. It remains to calculate
	\[ d^\pm(V_{i,Y}, q_{i,Y}) = d^\pm(K_i, q_{i,Y}^0) (-1)^n \det Y. \]
	Note that $q_{i,Y}^0(x|x) = -2y_i c_i N_{K_i/K_i^\sharp}(x)$, therefore its discriminant is the same as that of the norm form $N_{K_i/K_i^\sharp}(\cdot)$, which equals $\mathfrak{d}_i \bmod (K_i^{\sharp, \times})^2$.
\end{proof}

\begin{lemma}\label{prop:dagger-variance}
	Suppose that the datum $\vec{y}$ or $\vec{c}$ is multiplied by $\vec{d} = (d_i)_i \in (K^\sharp)^\times$, then the value of $\theta_j^\dagger$ at $-1 \in K_i^1$ is multiplied by $\sgn_{K_i/K_i^\sharp}(d_i)$.
\end{lemma}
\begin{proof}
	Note that multiplying $\vec{y}$ by $\vec{d}$ will multiply $\det Y$ by $N_{K/F}(\vec{d}) = N_{K^\sharp/F}(\vec{d})^2$. Given Lemma \ref{prop:dagger-Weil}, it suffices to apply \cite[Lemma 4.13]{Li11} to $K_i^\sharp$ and $\psi_i$.
\end{proof}
By taking $\vec{d} \in N_{K/K^\sharp}(K^\times)$, we infer that $\theta_j^\dagger$ depends only on the equivalence class of $(K, K^\sharp, \vec{y}, \vec{c})$. This is not enough: we need a homomorphism that depends only on the coset $Y + \mathfrak{s}(F)_0$.

\begin{lemma}\label{prop:dagger-indep}
	The character $\theta_j^\dagger$ depends only on the coset of $Y + \mathfrak{s}(F)_0$. Moreover, $\theta^\flat$ intervenes only through its pro-$p$ component.
\end{lemma}
\begin{proof}
	Recall that prescribing $Y + \mathfrak{s}(F)_0$ amounts to prescribing the pro-$p$ component of $\theta^\flat$. In view of Lemma \ref{prop:dagger-Weil}, it suffices to fix $i \in I$ and argue that
	\[ \gamma_{\psi_i}((-1)^n \det Y), \quad \gamma_{\psi_i}\left( \mathfrak{d}_i (-1)^n \det Y \right), \quad (-1, \mathfrak{d}_i (-1)^n \det Y )_{K_i^\sharp, 2} , \quad \gamma_{\psi_i}(q_{i,Y}^0) \]
	are all determined by the coset. The following standard fact will be used: consider
	\begin{compactitem}
		\item $M$: a non-archimedean local field of residual characteristic $\neq 2$,
		\item $\eta: M \to \CC^\times$ is an additive character with $a(\eta)$ defined as in \eqref{eqn:a},
		\item $t \in M^\times$,
		\item $\mathcal{L} := \mathfrak{p}_M^{\lceil r \rceil} \subset \mathfrak{p}_M^{\lfloor r \rfloor} =: \mathcal{L}'$, where $r := \dfrac{a(\eta) + 1 - v_M(t)}{2}$.
	\end{compactitem}
	Then $\gamma_\eta(t) = g_\eta(t, \mathcal{L}') \big/ |g_\eta(t, \mathcal{L}')|$, with
	\[ g_\eta(t, \mathcal{L}') := \sum_{x \in \mathcal{L}'/\mathcal{L}} \eta(tx^2). \]
	Here $\eta(tx^2)$ depends only on $x + \mathcal{L} \subset \mathcal{L}'$. Indeed, using $\lfloor r \rfloor + \lceil r \rceil = a(\eta) + 1 -v_M(t)$, it is routine to check that $f(x) := \eta(tx^2)$ equals $1$ on $\mathcal{L}$, and $\mathcal{L}' = \{y \in M: \eta(ty\mathcal{L})=1 \}$. Thus the required formula follows from \cite[\S 16 and \S 27]{Weil64}.
	
	As a consequence, suppose $t_1 \in M^\times$ satisfies $v_M(t_1 - t) > v_M(t)$, then the $r$ associated to $t,t_1$ are the same, whilst for all $x \in \mathcal{L}'$,
	\[ v_M\left((t-t_1)x^2\right) \geq v_M(t) + 1 + 2\lfloor r \rfloor \geq a(\eta) + 1 \implies \eta(tx^2) = \eta(t_1 x^2). \]
	Hence $\gamma_\eta(t) = \gamma_\eta(t_1)$.
	
	Now apply this to $M = K_i^\sharp$, $\eta = \psi_i$ and $t = (-1)^n \det Y$, $t_1 = (-1)^n \det(Y+Z)$, where $Z \in \mathfrak{s}(F)_{a(\psi)}$. By Proposition \ref{prop:stable-good-2}, the eigenvalues $\lambda \in \bar{F}$ of $Y$ satisfy $v_{\bar{F}}(\lambda) = a(\psi) - 1/e$. On the other hand, the eigenvalues of $Z$ satisfy $v_{\bar{F}}(\lambda) \geq a(\psi)$; this is a consequence of the concrete description of Moy--Prasad filtrations inside $\gl(W)$, see \cite[\S 4.1]{LMS16}. Since $[Y,Z]=0$, we infer that $v_{\bar{F}}(\det(Y+Z) - \det Y) > v_{\bar{F}}(\det Y)$, and same for $v_M(\cdots)$. Therefore the previous result implies $\gamma_{\psi_i}((-1)^n \det Y) = \gamma_{\psi_i}((-1)^n \det(Y+Z))$. Multiplying $t, t_1$ by $d := d(K_i/K_i^\sharp)$, the same argument gives
	\[ \gamma_{\psi_i}( \mathfrak{d}_i (-1)^n \det Y) = \gamma_{\psi_i}(\mathfrak{d}_i (-1)^n \det(Y+Z)). \]
	
	Next, we contend that $\gamma_{\psi_i}(q_{i,Y}^0)$ and $(-1, \mathfrak{d}_i (-1)^n \det Y)_{K_i^\sharp, 2}$ are both unaltered under $Y \leadsto Y+Z$. Parameterize $Y$, $Z$ by $\vec{y}, \vec{z} \in K^\times$ as usual. The $y_i, z_i$ are actually eigenvalues of $Y, Z$, thus $\frac{y_i + z_i}{y_i} \in 1 + \mathfrak{p}_{K_i^\sharp}$ by the foregoing discussions. Since the group $1 + \mathfrak{p}_{K_i^\sharp}$ is pro-$p$, we have
	\[ \left( -1, \frac{\det(Y+Z)}{\det Y} \right)_{K_i^\sharp, 2} = 1, \quad \sgn_{K_i/K_i^\sharp}\left( \dfrac{(y_i + z_i)c_i}{y_i c_i} \right) = 1. \]
	The latter implies $\gamma_{\psi_i}(q_{i,Y}^0) = \gamma_{\psi_i}(q_{i,Y+Z}^0)$ by Lemma \ref{prop:dagger-variance}.
\end{proof}

One can also regard $\theta^\dagger_j$ as a character of $j(S(F)/S(F)_{0+})$. All the foregoing constructions hinge on $\psi$ and $\lrangle{\cdot|\cdot}$.
\begin{theorem}\label{prop:dagger-SS}\index{stable system}
	The characters $\theta^\dagger_j$ in Definition \ref{def:dagger} depend only on $(\mathcal{E}, S, \theta^\flat)$ and $j$, and they form a stable system in the sense of Definition \ref{def:stable-system} by putting $\theta_j = \theta_j^\circ \theta_j^\dagger$.
\end{theorem}
\begin{proof}
	The remarks after Lemma \ref{prop:dagger-variance} entail that $\theta^\dagger_j$ is $G(F)$-invariant in $j$; the same invariance holds for $\theta^\circ_j$ by Remark \ref{rem:splitting-invariance}. Lemma \ref{prop:dagger-indep} shows that only the pro-$p$ part of $\theta^\flat$ matters for $\theta^\dagger_j$.
	
	\begin{asparadesc}
		\item[SS.1]\quad The first part follows by construction. For the second part, Remark \ref{rem:parameter-GSp} asserts that multiplying $\lrangle{\cdot|\cdot}$ by $a \in F^\times$ amounts to replacing $(K, K^\sharp, \vec{c})$ by $(K, K^\sharp, a\vec{c})$. Similarly, replacing $\psi$ by $\psi_a$ is the same as replacing $Y_\psi$ by $Y_{\psi_a} = a^{-1} Y_\psi$ by \eqref{eqn:Y-theta-2}; this is in turn equivalent to replacing $\vec{y}$ by $a^{-1}\vec{y}$. The required behavior of $\theta^\dagger_j$ is then ensured by Lemma \ref{prop:dagger-variance}.
		\item[SS.2]\quad The invariance under $G(F)$-conjugacy has been observed above. In general, consider $j' = \Ad(g) \circ j$ and $\tilde{\gamma} \in \widetilde{jS}$ as in Definition \ref{def:stable-system}. Take the topological Jordan decomposition $\tilde{\gamma} = \tilde{\gamma}_0 \gamma_{>0}$. The covering splits uniquely over pro-$p$ subgroups, thus
		\[ \CaliAd(g)(\tilde{\gamma}) = \CaliAd(g)(\tilde{\gamma}_0) \CaliAd(g)(\gamma_{>0}) = \CaliAd(g)(\tilde{\gamma}_0) \Ad(g)(\gamma_{>0}). \]
		On the other hand, the pro-$p$ parts of $\theta_j$, $\theta_{j'}$ coincides with those of $\theta^\circ_j$, $\theta^\circ_{j'}$ by \textbf{SS.1}, which are respected by $\Ad(g)$. Therefore it remains to show that
		\[ \theta_{j'}(\CaliAd(g)(\tilde{\gamma}_0)) = \theta_j(\tilde{\gamma}_0). \]
		Recall that $\CaliAd(g) = \CaliAd(g'')\CaliAd(g')$ where $(g', g'')$ is a factorization pair for $\Ad(g)$ (Definition \ref{def:st-conj}). The verification thus reduces to the case $g \in G^{jS}_\text{ad}(F)$, and then to $\dim_F W = 2$ upon Weil restriction; recall \S\ref{sec:st-conj-BD} for this procedure. By Theorem \ref{prop:S-p'} we have $\gamma_0 := \bm{p}(\tilde{\gamma}_0) \in \{\pm 1\}$, thus it suffices to deal with the case $\gamma_0 = -1$. Choose a parameter $(K, K^\sharp = F, \ldots)$ for $j$. Proposition \ref{prop:CAd-minus-1} asserts that $\CaliAd(g)(\tilde{\gamma}_0) = \sgn_{K/F}(\nu(g)) \tilde{\gamma}_0$, with $\nu(g)$ coming from \eqref{eqn:nu-arises}.
		
		On the other hand, $\theta^\circ_j(\tilde{\gamma}_0) = \theta^\circ_{j'}(\tilde{\gamma}_0)$ since they rely on $\theta^\flat$ and the $\widetilde{-1}$ prescribed in \S\ref{sec:splittings-S}, which is insensitive to $j, j'$ when $\dim_F W = 2$. Compare $\theta^\dagger_{j'}(\gamma_0)$ and $\theta^\dagger_j(\gamma_0)$ next. Let $(K, F, c)$ and $(K, F, c')$ be the parameters of $j$ and $j'$. Lemma \ref{prop:dagger-variance} implies $\theta^\dagger_{j'}(\gamma_0) = \sgn_{K/F}(c/c') \theta^\dagger_j(\gamma_0)$. On the other hand, Proposition \ref{prop:st-conj-SL2} implies $\sgn_{K/F}(c/c') = \sgn_{K/F}(\nu(g))$: both are identifiable with $\text{inv}(j, j')$.
		
		\item[SS.3]\quad First fix $j \in \mathcal{E}$. Lemma \ref{prop:dagger-indep} asserts that $\theta_j^\dagger$ is determined by the pro-$p$ part of $\theta^\flat$. Next, fix $\theta^\flat$ and suppose that $jS = j'S$ where $j, j' \in \mathcal{E}$; denote this common image by $R$. Then $j' = \Ad(w) j$ for some $w \in \Omega(G,R)(F)$. By Proposition \ref{prop:big-Weyl-action}, $\Ad(w)$ corresponds to the action by some $\varphi \in \Aut(K,\tau)$ on parameters. An automorphism $\varphi$ induces a permutation $f$ on $I$ characterized by $\varphi|_{K_i}: K_i \rightiso K_{f(i)}$, thus the identification $\mu_2^I \simeq R(F)_{p'}$ changes by $f$. As remarked after Definition \ref{def:dagger}, $q_{i,Y} \simeq q_{f(i), Y}$ as quadratic vector spaces over $K_i^\sharp \simeq K_{f(i)}^{\sharp}$. It follows that $\theta_j^\dagger = \theta_{\Ad(w)j}^\dagger$ as functions on $R(F)_{p'}$.
	\end{asparadesc}
\end{proof}

\subsection{Interplay}\label{sec:interplay}
Keep the assumptions of \S\ref{sec:ss-construction}. Let $(\mathcal{E}, S, \theta^\flat)$ be as in Definition \ref{def:inducing-data} and consider a $j \in \mathcal{E}$. Define $Y = Y_\psi$ by \eqref{eqn:Y-theta-2}. By Theorem \ref{prop:mm}, there exists a unique $(V,q)$ with $\dim_F V = 2n+1$ and $d^\pm(V,q)=1$ such that $jY$ corresponds to some $Y' \in \mathfrak{h}_\text{reg}(F)$, where $H := \SO(V,q)$; such a $Y'$ is unique up to conjugacy. Furthermore, Remark \ref{rem:explicit-moment-map} says that we can take
\[ (V,q) := q\lrangle{Y} \oplus \lrangle{(-1)^n \det Y} \]
and there is a natural embedding $k: S \hookrightarrow H$ determined by $j$.

Also recall Kaletha's quadratic character $\epsilon_{jS}$ (resp. $\epsilon_{kS}$) of $jS(F) \subset G$ (resp. $kS(F) \subset H$) from \S\ref{sec:toral-invariants}. Denote their pull-backs to $S(F)$ as $\epsilon_j$ and $\epsilon_k$, respectively.\index{epsilonS@$\epsilon_S$}

\begin{theorem}\label{prop:interplay}
	For every $\gamma \in S(F)$, let $\gamma = \gamma_0 \gamma_{>0}$ be its topological Jordan decomposition. Then
	\[ \dfrac{\epsilon_k(\gamma)}{\epsilon_j(\gamma)} = \theta^\dagger_j(\gamma_0) \]
	where $\theta^\dagger_j$ is as in Definition \ref{def:dagger}, pulled back to $S(F)_{p'}$
\end{theorem}
\begin{proof}
	First observe that $kS \subset H$ satisfies Definition \ref{def:type-ER}, since $G$ and $H$ share the same Weyl group whose actions on $jS$ and $kS$ match. This follows immediately from Remark \ref{rem:moment-map-basis}.
	
	By \cite[Lemma 4.12]{Kal15} we have $\epsilon_j(\gamma) = \epsilon_j(\gamma_0)$ and $\epsilon_k(\gamma) = \epsilon_k(\gamma_0)$. It also implies
	\begin{gather*}
		\dfrac{\epsilon_j(\gamma_0)}{\epsilon_k(\gamma_0)} = \dfrac{ \displaystyle\prod_{\substack{ \alpha \in R(G, jS)_\text{sym} / \Gamma_F \\ \alpha(j\gamma_0) \neq 1 }} f_{(G,jS)}(\alpha) }{ \displaystyle\prod_{\substack{ \beta \in R(H, kS)_\text{sym} / \Gamma_F \\ \beta(k\gamma_0) \neq 1 }} f_{(H,kS)}(\beta) } \quad \in \{\pm 1\}.
	\end{gather*}
	We shall use the $\Gamma_F$-equivariant bijection $\alpha \leftrightarrow \beta$ of roots of Remark \ref{rem:moment-map-basis}. It preserves symmetric roots by equivariance. Recall that $\gamma_0^2 = 1$ (Theorem \ref{prop:S-p'}) and consider the following cases.
	\begin{itemize}
		\item $\alpha$ is a short symmetric root, $\alpha(j\gamma_0) = -1$, then $\beta(k\gamma_0) = -1$ as well: in fact both roots take the form $\epsilon_r \pm \epsilon_s$ or its negative, where $1 \leq r,s \leq n$. Since $F_\alpha = F_\beta$, $F_{\pm\alpha} = F_{\pm\beta}$, the calculations in \S\ref{sec:toral-invariants} lead to $f_{G,jS}(\alpha) = f_{(H,kS)}(\beta)$.
		\item $\alpha$ is a long root, say of the form $\pm 2\epsilon_r$ with $1 \leq r \leq n$, then $\beta = \pm\epsilon_r$ and $\alpha(j\gamma_0) = 1$ always holds. In this case $\alpha, \beta$ are both symmetric as $S$ is anisotropic (Lemma \ref{prop:anisotropic-criterion}). Let $(U, q_U)$ be the $3$-dimensional $F_\alpha$-vector subspace with $kS$-weights $\{\epsilon_r, 0, -\epsilon_r \}$. The calculations in \S\ref{sec:toral-invariants} show that $(U, q_U)$ descends to a quadratic $F_{\pm\alpha}$-vector space $(U_\beta, q_\beta)$, and by \eqref{eqn:SO-split}
		\[ f_{(H, kS)}(\beta) = \epsilon(\SO(U_\beta, q_\beta)) (-1, d^\pm(U_\beta, q_\beta))_{F_{\pm\beta}, 2}. \]
	\end{itemize}

	We shall determine $(U_\beta, q_\beta)$ in steps. First, we parameterize the conjugacy class of $jY$ by a datum $(K, K^\sharp, \vec{y}, \vec{c})$ as in Definition \ref{def:dagger}, and identify $W$ with $K$ to simplify notation. Let $i \in I$.
	\begin{asparaenum}[\bfseries Step 1.]
		\item By the general recipe \eqref{eqn:tensor-split},
			\begin{equation}\label{eqn:K_i-decomp}
				K_i \dotimes{F} \bar{F} = K_i \dotimes{K_i^\sharp} \left( K_i^\sharp \dotimes{F} \bar{F} \right) \rightiso K_i \dotimes{K_i^\sharp} \bar{F}^{\oplus \Hom_F(K_i^\sharp, \bar{F}) };
			\end{equation}
			the $\Gamma_F$-action on the right-hand side can be described by \eqref{eqn:MF-Gal-action}: it operates only on the second slot, and permutes the summands transitively.
			
			Take $\iota \in \Hom_F(K_i^\sharp, \bar{F})$; there are exactly two $\iota', \iota'\tau: K_i \to \bar{F}$ extending $\iota$. Identifying $\Hom_F(K_i, \bar{F})$ and $\Hom_{\bar{F}}\left( K_i \dotimes{F} \bar{F}, \bar{F}\right)$, we obtain $\pm\epsilon_\iota \in X^*(S_{\bar{F}})$ where $\epsilon_\iota := \iota'|_{(K_i \otimes_F \bar{F})^1}$. They are the $jS$-weights of the $\iota$-th component of $K_i \dotimes{F} \bar{F}$ in the decomposition above; denote this space as $K_i[\iota]$. Also note that $\Stab_{\Gamma_F}\left( \{\pm\epsilon_\iota\}\right) = \Stab_{\Gamma_F}(\iota)$ (resp. $\Stab_{\Gamma_F}(\epsilon_\iota) = \Stab_{\Gamma_F}(\iota')$) corresponds to the intermediate field $\iota(K_i^\sharp) \simeq K_i^\sharp$ (resp. $\iota'(K_i) \simeq K_i$).
		\item Define $h_i, h^i$ as in \eqref{eqn:h-h}. The involution on $K_i$ transports to the right-hand side of \eqref{eqn:K_i-decomp}, acting only on $K_i$. Hence $h_i \dotimes{F} \bar{F}$ equals $(h^i \dotimes{K_i^\sharp} \bar{F})^{\oplus \Hom_F(K_i^\sharp, \bar{F}) }$.
		
		Let $h_i[\iota]$ denote the $\iota$-th component of $h_i \dotimes{F} \bar{F}$; it lives on the subspace $K_i[\iota] \simeq K_i \otimes_{K_i^\sharp} \bar{F}$ with $jS$-weights $\{\pm\epsilon_\iota\}$. Hence $h_i[\iota]$ descends to the symplectic form $h^i$ on the $K_i^\sharp$-vector space $K_i$. The same descent works if we consider the symmetric forms $(u, v) \mapsto h_i(y_i u|v)$, $h^i(y_i u|v)$ instead, which yield the quadratic form $q_{i,Y}^0$ on $K_i$.
		\item Now we may choose a symplectic basis for $h^i \otimes_{K_i^\sharp} K_i$ with associated characters $\pm\epsilon_\iota \in X^*(S_{K_i})$ in the notation above; this is easily done by reducing to $n=1$. By varying $(i, \iota)$ and $i \in I$, we obtain a symplectic basis $\{e_{\pm r}\}_{r=1}^n$ for $K \otimes_F \bar{F}$, as well as the adapted characters $\pm\epsilon_1, \ldots, \pm\epsilon_n$. The procedure in Remark \ref{rem:explicit-moment-map}, \ref{rem:moment-map-basis} renders $\{e_{\pm r}\}_{r=1}^n$ into a hyperbolic basis for $q\lrangle{Y} \dotimes{F} \bar{F}$.
		
		Now consider $\beta \in R(H, kS)(\bar{F})$ and let $R(G, jS)(\bar{F}) \ni \alpha \leftrightarrow \beta$, so that $\alpha = 2\epsilon_r$ for some $r$ as in Step 3.  As remarked in Step 1 (cf. \eqref{eqn:K_i-alpha}),
		\[ F_\beta = F_\alpha \simeq K_i, \quad F_{\pm\beta} = F_{\pm\alpha} \simeq K_i^\sharp, \]
		so that $\beta \in R(H, kS)(\bar{F})_\text{sym}$. Comparing the step 2 with \S\ref{sec:toral-invariants} yields
		\[ (U_\beta, q_\beta) \simeq \underbracket{(K_i, q_{i,Y}^0)}_{ kS-\text{weight}=\pm\beta } \oplus \underbracket{\lrangle{(-1)^n \det Y}}_{kS-\text{weight}=0} = (V_{i,Y}, q_{i,Y}). \]
	\end{asparaenum}

	Thus for $\beta$ as above, $f_{(H, kS)}(\beta)$ equals the value of $\theta^\dagger_j$ at $-1 \in K_i^1$. Recall from the first part of our proof that
	\[ \dfrac{\epsilon_j(\gamma)}{\epsilon_k(\gamma)} = \prod_{\substack{\beta \in R(H, kS)/\Gamma_F \\ \text{short} \\ \beta(\gamma_0) = -1 }} f_{(H,kS)}(\beta); \]
	the product can be equivalently taken over $i \in I$ by the construction of $\{\epsilon_{\pm r}\}_{r=1}^n$. If we write $\gamma_0 = (\gamma_{0,i})_{i \in I}$ with $\gamma_{0,i} \in \{\pm 1\} \subset K_i^1$, then $\beta(\gamma_0) = -1 \iff \gamma_{0,i} = -1$ and the product is exactly $\theta^\dagger_j(\gamma_0)$.
\end{proof}

\section{Compatibilities}\label{sec:compatibilities}
Fix a local field $F$ with $\text{char}(F) \neq 2$, and consider a symplectic $F$-vector space $(W, \lrangle{\cdot|\cdot})$ of dimension $2n$. Set $G := \Sp(W)$. To rule out trivial cases, we assume $F \neq \CC$.

\subsection{Review of Adams' stable conjugacy}\label{sec:Adams}
Fix an additive character $\psi$ of $F$. Let $H(W)$ be the Heisenberg group of $(W, \lrangle{\cdot|\cdot})$, which has a smooth irreducible representation $(\rho_\psi, S_\psi)$ with central character $\psi$, unique up to isomorphisms. Weil's \emph{metaplectic group} is a topological central extension
\begin{gather*}
	1 \to \CC^\times \to \overline{G}_\psi \xrightarrow{\bm{p}} G(F) \to 1, \\
	\overline{G}_\psi := \left\{ (x, M_x) \in G(F) \times \Aut_{\CC}(S_\psi) : \rho_\psi \xrightarrow[\sim]{M_x} \rho_\psi^x \right\}
\end{gather*}
where $\rho_\psi^x(h) = \rho_\psi(xh)$ for all $h \in H(W)$, and $(x, M_x)(y, M_y) = (xy, M_x M_y)$. Specifically, we choose a Lagrangian $\ell \subset W$ and follow \cite[\S 2.4.1]{Li11} to construct $(\rho_\psi, S_\psi)$ using the \emph{Schrödinger model} attached to $\ell$; there is then a set-theoretic section $x \mapsto (x, M_\ell[x])$ of $\overline{G}_\psi \twoheadrightarrow G(F)$. The multiplication is described by the \emph{Maslov cocycle} \cite[Théorème 2.6]{Li11}:
\begin{equation}\label{eqn:Maslov-cocycle}
	M_\ell[x] M_\ell[y] = \gamma_\psi\left( \tau(\ell, y\ell, xy\ell) \right) M_\ell[xy].
\end{equation}
Here, for any Lagrangians $\ell_1, \ell_2, \ell_3$ in $W$, one has the quadratic $F$-vector space $\tau(\ell_1, \ell_2, \ell_3)$ canonically constructed by T.\ Thomas \cite{Th06}. This reduces $\overline{G}_\psi$ to $\bmu_8 \hookrightarrow \overline{G}_\psi^{(8)} \twoheadrightarrow G(F)$, by setting $\overline{G}_\psi^{(8)} = \{(x, zM_\ell[x]): z \in \bmu_8 \}$. To get a Lagrangian-independent definition, one may use the canonical intertwining operators between Schrödinger models; see \cite[\S 2.1]{Per81}. \index{Gpsi@$\overline{G}_\psi, \overline{G}_\psi^{(2)}, \overline{G}_\psi^{(8)}$}

Furthermore, taking the commutator subgroup yields a further reduction $\bmu_2 \hookrightarrow \overline{G}_\psi^{(2)} \twoheadrightarrow G(F)$. A direct description of $\overline{G}_\psi^{(2)}$ is given by Lion and Perrin \cite{Per81}
\begin{gather*}
	\overline{G}_\psi^{(2)} = \left\{ (x, \pm m(x\ell, \ell) M_\ell[x]) \in \overline{G}_\psi : x \in G(F) \right\}, \\
	m(\ell_1, \ell_2) := \gamma_\psi(1)^{n - \dim \ell_1 \cap \ell_2 - 1 } \gamma_\psi(A_{\ell_1, \ell_2}).
\end{gather*}
The notation is explained as follows.
\begin{itemize}
	\item We take $\ell_1, \ell_2$ to be Lagrangians in $W$ endowed with orientations $o_1, o_2$. An orientation on a finite-dimensional $F$-vector space $V$ means a nonzero element from $\topwedge V$ taken up to $F^{\times 2}$, with the convention $\topwedge \{0\} = F$.
	\item Put $\ell'_i := \ell_i/\ell_1 \cap \ell_2$ for $i=1,2$. The restriction of $\lrangle{\cdot|\cdot}$ to $\ell'_1 \times \ell'_2$ is non-degenerate and induces a pairing
	\[ \alpha: \topwedge \ell'_1 \dotimes{F} \topwedge \ell'_2 \to F. \]
	Specifically, if $v^{(1)}_1, v^{(1)}_2, \ldots$ and $v^{(2)}_1, v^{(2)}_2, \ldots$ are dual bases, then $\alpha\left( v^{(1)}_1 \wedge \cdots, v^{(2)}_1 \wedge \cdots \right) = 1$. Now set $A_{\ell_1, \ell_2} := \alpha(o'_1, o'_2)$ by writing $o_i = o'_i \otimes c$ where $c$ is any orientation on $\ell_1 \cap \ell_2$.
	\item In $m(x\ell, \ell)$ we choose any orientation $o$ on $\ell$ and transport it to $x\ell$ by $\topwedge(x)$.
\end{itemize}
It turns out that $\overline{G}_\psi^{(2)} \twoheadrightarrow G(F)$ is isomorphic to the BD-cover $\tilde{G} \twoheadrightarrow G(F)$ in \S\ref{sec:BD-Sp} with $m=2$. Denote by $\sigma_\text{LP}$ the set-theoretic section $x \mapsto (x, m(x\ell, \ell)M_\ell[x])$. Multiplication on $\overline{G}_\psi^{(2)}$ is given by Rao's cocycle in terms of $\sigma_\text{LP}$. When $n=1$, we get Kubota's cocycle $\bm{c}$ in \eqref{eqn:Kubota-cocycle} by identifying $\sigma_{\text{LP}}$ with $\bm{s}$; see \cite[Chapitre 3, I.3]{MVW87} or \cite[2.4.2]{Per81}.

Let $P_\ell := \Stab_G(\ell)$. It is a Siegel parabolic with Levi component $\GL(\ell)$ once transversal Lagrangians $W = \ell \oplus \ell'$ are chosen. The Schrödinger model furnishes a splitting $\sigma_\ell: x \mapsto (x, M_\ell[x])$ of $\overline{G}_\psi^{(8)}$ over $P_\ell(F)$; see \cite[Chapitre 2, II.6]{MVW87}. The following result will be needed in \S\ref{sec:theta}. \index{sigma_ell@$\sigma_\ell, \sigma_{\mathrm{LP}}$}
\begin{proposition}\label{prop:convenient-Schrodinger}
	We have $\sigma_\ell(-1) = \frac{\gamma_\psi(1)}{\gamma_\psi((-1)^n)} \sigma_{\mathrm{LP}}(-1)$. Furthermore, the $\widetilde{-1}$ in Definition \ref{def:lifting-minus-1} coincides with $\sigma_\ell(-1)$ via $\tilde{G}^\natural \hookrightarrow \overline{G}^{(8)}_\psi$.
\end{proposition}
\begin{proof}
	For all $x \in \GL(\ell)$ we have $\sigma_\ell(x) = (1, m(x\ell, \ell))^{-1} \sigma_\text{LP}(x)$. Put $x=-1$ in the definition of $m(x\ell, \ell)$ to deduce the first assertion.

	Take a basis of $\ell$ and extend it to a symplectic basis of $W = \ell \oplus \ell'$, so that $\sigma_\ell(-1) = \prod_\alpha \sigma_\ell(\check{\alpha}(-1))$ where $\alpha$ ranges over the positive long roots. One can calculate $\sigma_\ell(\check{\alpha}(-1))$ inside rank-one pieces, see \cite[Chapitre 2, II.6]{MVW87}. Since $\widetilde{-1}$ and $\overline{-1}$ also decompose in this manner, we can assume $n=1$.

	Let $t \in F^\times$. By the discussions preceding \cite[Corollaire 5.12]{Mat69} or a direct computation via Kubota's cocycle,
	\[ \sigma_\text{LP}(\check{\alpha}(t)) = \bm{s}\twobigmatrix{t}{}{}{t^{-1}} = (t,t)_{F,2}^{-1} h_\alpha(t) = (t,-1)_{F,2} h_\alpha(t)  \quad \in \widetilde{\SL}(2,F). \]
	Using \eqref{eqn:Weil-Hilbert}, we see $(-1,-1)_{F,2} = \gamma_\psi(1)^4$ so that $\sigma_\text{LP}(-1) = \gamma_\psi(1)^4 \cdot \overline{-1}$. It remains to prove
	\[ \gamma_\psi(1)^6 = \epsilon\left( \frac{1}{2}, (-1, \cdot)_{F,2}; \psi\right)^{-1}, \quad \text{i.e.}\; \gamma_\psi(\lrangle{1,1}) = \epsilon\left( \frac{1}{2}, (-1, \cdot)_{F,2}; \psi\right). \]
	As $\lrangle{1,1}$ is the norm form of the $F$-algebra $F[X]/(X^2+1)$, from \cite[Lemma 1.2]{JL70} we deduce $\gamma_\psi(\lrangle{1,1}) = \epsilon(\frac{1}{2}, (-1, \cdot)_{F,2}; \psi_{1/2})$, noting the different normalization of $\gamma_\psi$ therein. As $(-1, \frac{1}{2})_{F,2} = (-1, 1-(-1))_{F,2} = 1$, we may replace $\psi_{1/2}$ by $\psi$.
\end{proof}

The group $\overline{G}_\psi$ and its reductions carry the \emph{Weil representation} $\omega_\psi = \omega_\psi^+ \oplus \omega_\psi^-$ on $S_\psi$, which is genuine and canonically defined with respect to $\psi \circ \lrangle{\cdot|\cdot}$. Here $\omega_\psi^\pm$ are irreducible genuine admissible representations. There is also a canonical element $-1 \in \overline{G}_\psi$ lying over $-1 \in G(F)$, described as $(-1, M_\ell[-1])$ in the Schrödinger model, such that $\omega_\psi^\pm(-1) = \pm \identity$ and $(-1)^2 = 1$. Write $\Theta_\psi = \Theta_\psi^+ + \Theta_\psi^-$ for the corresponding characters. We are in a position to state the notion of stable conjugacy after Adams.

\begin{definition}[J. Adams]\label{def:Adams-conj}\index{stable conjugacy}
	Call two elements $\tilde{\gamma}, \tilde{\delta} \in \overline{G}_{\psi, \mathrm{reg}}$ stably conjugate, if
	\begin{itemize}
		\item the images $\gamma, \delta \in G_{\mathrm{reg}}(F)$ are stably conjugate, and
		\item $(\Theta^+_\psi - \Theta^-_\psi)(\tilde{\delta}) = (\Theta^+_\psi - \Theta^-_\psi)(\tilde{\gamma})$.
	\end{itemize}
\end{definition}
This definition does not rely on the choice of Lagrangians. For details of these constructions, we refer to \cite{Li11} and the bibliography therein. We record the key property below. For any $x \in G(F)$, set $\Gamma_x := \left\{ (w, xw) : w \in W \right\}$; it is a Lagrangian in $W^- \oplus W$ where $W^- := (W, -\lrangle{\cdot|\cdot})$. Define the genuine function $\overline{G}_\psi \to \CC$ \index{nabla@$\nabla$}
\begin{equation}\label{eqn:nabla}
	\nabla: (x, zM_\ell[x]) \mapsto z \gamma_\psi\left( \tau(\Gamma_{-x}, \Gamma_1, \ell \oplus \ell) \right).
\end{equation}

\begin{theorem}\label{prop:nabla}
	The function $\nabla|_{\overline{G}_{\psi, \mathrm{reg}}}$ is invariant under $G(F)$-conjugation. Two elements $\tilde{\gamma}, \tilde{\delta} \in \overline{G}_{\psi, \mathrm{reg}}$ are stably conjugate if and only if their images $\gamma, \delta$ are stably conjugate in $G_{\mathrm{reg}}(F)$ and $\nabla(\tilde{\gamma}) = \nabla(\tilde{\delta})$.
	
	If $\tilde{\delta} \in \overline{G}_{\psi, \mathrm{reg}}^{(2)}$ is stably conjugate to $\tilde{\delta} \in \overline{G}_{\psi, \mathrm{reg}}$, then $\delta \in \overline{G}_\psi^{(2)}$ as well.
\end{theorem}
\begin{proof}
	From \eqref{eqn:Maslov-cocycle} we infer that $M_\ell[-1] M_\ell[x] = M_\ell[-x]$ since $\dim \tau(\ell, x\ell, x\ell) = 0$ by \cite[Proposition 2.5]{Li11}. Therefore when $\det(x+1|W) \neq 0$, Maktouf's character formula \cite[Corollaire 4.4]{Li11} becomes
	\begin{align*}
		\left(\Theta^+_\psi - \Theta^-_\psi\right) \left( (x, zM_\ell[x]) \right) & = \left(\Theta^+_\psi + \Theta^-_\psi\right)\left( (-1, M_\ell[-1]) (x, z M_\ell[x]) \right) \\
		& = \Theta_\psi\left( (-x, z M_\ell[-x]) \right) \\
		& = |\det(x+1)|^{-\frac{1}{2}} z \gamma_\psi\left( \Gamma_{-x}, \Gamma_1, \ell \oplus \ell \right).
	\end{align*}
	Hence $\nabla|_{\overline{G}_{\psi, \mathrm{reg}}}$ is $G(F)$-invariant, and the first part follows. The second part is a consequence of \cite[Théorème 4.2 (iii)]{Li11}.
\end{proof}

Next, let $T \subset G$ be a maximal $F$-torus. Construct $T \subset G^T \subset G$ following Definition \ref{def:G-T}. By choosing a parameter $(K, K^\sharp, c)$ for $T \hookrightarrow G$ via Proposition \ref{prop:parameter-tori}, with $K = \prod_{i \in I} K_i$ etc., we may identify $G^T(F)$ with $\prod_{i \in I} \SL(2, K_i^\sharp)$. We may regard each $\SL(2, K_i^\sharp)$ as $G_i := \Sp(W_i)$ for a symplectic $K_i^\sharp$-vector space $W_i$ of dimension $2$ with suitably chosen Lagrangian $\ell_i$, which will be specified anon. To each $i \in I$ we construct Weil's metaplectic group
\[ 1 \to \CC^\times \to \overline{G}_{\psi_i, i} \xrightarrow{\bm{p}_i} G_i(F) \to 1, \quad \psi_i := \psi \circ \Tr_{K_i^\sharp/F}. \]
On each $\overline{G}_{\psi_i, i}$ we define the genuine function $\nabla_i$ using \eqref{eqn:nabla} for $W_i$, $\psi_i$.

\begin{lemma}\label{prop:metaplectic-reduction}\index{contracted product}
	The pull-back of $\overline{G}_\psi$ (resp. of $\overline{G}_\psi^{(2)}$) to $G^T(F)$ is isomorphic to the contracted product of $\CC^\times \hookrightarrow \overline{G}_{\psi_i, i} \twoheadrightarrow G_i(F)$ (resp. $\bmu_2 \hookrightarrow \overline{G}^{(2)}_{\psi_i, i} \twoheadrightarrow G_i(F)$), for $i \in I$. Under such an isomorphism, the restriction of $\nabla$ to $\bm{p}^{-1}(G^T(F))$ coincides with $\bigotimes_{i \in I} \nabla_i$.
\end{lemma}
Note that the asserted isomorphism must be unique, since $G_i(F)$ is a perfect group.
\begin{proof}
	Identify $(W, \lrangle{\cdot|\cdot})$ with $\bigoplus_{i \in I} (K_i, h_i)$ using the parameter, where $h_i = \Tr_{K_i^\sharp/F} \circ h^i$ are defined in \eqref{eqn:h-h}. Select a Lagrangian $\ell_i \subset K_i$ relative to $h^i$. Then $\ell_i$ is also a Lagrangian for $(K_i, h_i)$ since it is totally isotropic for $h_i$ and has the right $F$-dimension. We use the Lagrangian $\ell := \bigoplus_{i \in I} \ell_i$ (resp. $\ell_i$) to realize $\overline{G}_\psi$ (resp. $\overline{G}_{\psi_i, i}$).
	
	To prove the first assertion for $\overline{G}_\psi$, denote by $\overline{G}^T_\psi$ the contracted product in question. Represent the elements of $G^T(F)$ as $x = (x_i)_{i \in I} \in G^T(F)$ with $x_i \in \Sp(K_i, h^i) \simeq \SL(2, K_i^\sharp)$. Consider the bijection
	\begin{align*}
		\text{Bij}: \bm{p}^{-1}(G^T(F)) & \longrightarrow \overline{G}^T_\psi \\
		(x, zM_\ell[x]) & \longmapsto z \cdot \prod_{i \in I} (x_i, M_{\ell_i}[x_i])_{/K_i^\sharp};
	\end{align*}
	the final subscript indicates that we work over symplectic $K_i^\sharp$-vector spaces. Note that $x \mapsto (x, M_\ell[x])$ is a continuous section over the open cell $\{x: x\ell \cap \ell = 0 \}$; the same applies to each $\ell_i$ as well. Therefore it suffices to show that $\text{Bij}$ is a homomorphism, which amounts to matching the Maslov cocycles from both sides.
	
	The symplectic additivity of $\tau(\cdots)$ (see \cite[p.532]{Li11}) implies
	\[ \tau(\ell, y\ell, xy\ell) = \bigoplus_{i \in I} \tau(\ell_i, y_i\ell_i, x_i y_i\ell_i). \]
	Since $x_i, y_i, \ell_i$ all come from ``upstairs'' by forgetting $K_i^\sharp$-structures, which is an exact functor, the construction of $\tau(\cdots)$ in \cite[\S 2.2.3]{Th06} immediately leads to
	\[ \tau(\ell_i, y_i \ell_i, x_i y_i \ell_i) = (\Tr_{K_i^\sharp/F})_* \left( \tau(\ell_i, y_i \ell_i, x_i y_i \ell_i)_{/K_i^\sharp} \right). \]
	Now invoke the additivity of $\gamma_\psi$ to obtain
	\[ \gamma_\psi\left( \tau(\ell, y\ell, xy\ell) \right) = \prod_{i \in I} \gamma_{\psi_i}\left( \tau(\ell_i, y_i \ell_i, x_i y_i \ell_i )_{/K_i^\sharp} \right). \]
	The right-hand side matches the Maslov cocycle for $\overline{G}^T_\psi$. This proves the case of $\overline{G}_\psi$. Now notice that the pull-back of $\overline{G}_\psi^{(2)}$ to $G^T(F)$ is closed under commutators, thus contains the contracted product of $\overline{G}^{(2)}_{\psi_i, i}$, $i \in I$. Both are twofold coverings of $G^T(F)$, hence they coincide and the case of $\overline{G}_\psi^{(2)}$ follows.
	
	As for the second assertion, apply the same reasoning to \eqref{eqn:nabla} for the symplectic $K_i^\sharp$-vector spaces $W_i^- \oplus W_i$.
\end{proof}

Last but not least, the adjoint $G(F)$-action on $\overline{G}_\psi^{(2)}$ extends uniquely to $\GSp(W)$ by \cite[Chapitre 4, I.8]{MVW87}.

\subsection{Stable conjugacy for \texorpdfstring{$m=2$}{m=2}}
We begin with the case $n=1$. Fix a symplectic basis $e_1, e_{-1}$ for $W$ and take $\ell := Fe_1$. Use the standard orientation generated by $e_1$ of $\ell$ in the Lion--Perrin construction. Identifying $W$ with $F^2$ in this basis, we have
\[ \lrangle{r,s | r',s'} = rs' - r's, \quad \ell = \bigl(\begin{smallmatrix} * \\ 0 \end{smallmatrix}\bigr), \quad \Sp(W) = \SL(2), \quad \GSp(W) = \GL(2). \]
The similitude factor becomes determinant. As seen in \S\ref{sec:Adams}, $\widetilde{\mathrm{SL}}(2, F) \simeq \overline{G}_\psi^{(2)}$ by a unique isomorphism such that $\sigma_\text{LP}$ matches Kubota's $\bm{s}$ in \S\ref{sec:Kubota}.

\begin{lemma}\label{prop:nabla-1}
	Let $\gamma = \twomatrix{a}{b}{c}{d} \in \SL(2,F)_{\mathrm{reg}}$. Suppose $c \neq 0$, then
	\[ \nabla(\sigma_{\mathrm{LP}}(\gamma))  = \gamma_\psi\left( \lrangle{-c, c(2 + \Tr(\gamma))} \right). \]
\end{lemma}
\begin{proof}
	We have $\sigma_\text{LP}(\gamma) = (\gamma, m(\gamma\ell, \ell) M_\ell[\gamma])$. First calculate
	\begin{equation}\label{eqn:cal-m} \begin{aligned}
		A_{\gamma\ell, \ell} & = \lrangle{ a,c | 1,0 } = -c \bmod F^{\times 2}, \\
		m(\gamma\ell, \ell) & = \gamma_\psi(1)^{1-0-1} \gamma_\psi(-c) = \gamma_\psi(-c).
	\end{aligned}\end{equation}
	Next, consider the following Lagrangians of $W^\Box := W^- \oplus W$:
	\[ \ell_1 := \ell \oplus \ell, \quad \ell_2 := \Gamma_1, \quad \ell_3 := \Gamma_{-\gamma}, \]
	noting that $\ell_3$ is transversal to both $\ell_1, \ell_2$ by the assumptions on $\gamma$. By  into \cite[Lemme 1.4.2]{Per81}, $\tau(\ell_1, \ell_2, \ell_3)$ is Witt-equivalent to the quadratic form on $\Gamma_1$ defined by $q_{123}: v \mapsto \lrangle{\pi_1 v | \pi_3 v}_{W^\Box}$ where $\pi_1, \pi_3$ are the projections attached to $W^\Box = \ell_1 \oplus \ell_3$. This form is degenerate: we will soon see that its radical is $\Gamma_1 \cap (\ell \oplus \ell) \simeq \ell$. We contend that $q_{123}$ is Witt-equivalent to $\lrangle{ -c(2 + \Tr(\gamma)) }$.
	
	To see this, represent the elements of $W^\Box$ as $(x,y;x',y')$. Let $v = (\alpha,\beta; \alpha, \beta) \in \Gamma_1$. There is a unique decomposition
	\[ v = \underbracket{(r, 0; s, 0)}_{\in \ell \oplus \ell} + \underbracket{(w, -\gamma w)}_{\in \Gamma_{-\gamma}} = (r, 0; s, 0) + \left( t, u; -at-bu, -ct-du \right) \]
	where $w = (t,u) \in W$, so that $q_{123}(v) = -ru - s(ct+du)$. The resulting linear system
	\begin{align*}
		r + t = \alpha & = s - at - bu \\
		u = \beta & = -ct - du
	\end{align*}
	entails $u=\beta$ and $-ru - s(ct+du) = -r\beta + s\beta = (t-\alpha+s)\beta$. Furthermore,
	\[ (s, t) = (a t + b\beta + \alpha, \; t) = \left( \frac{a\beta(d+1)}{-c} + b\beta + \alpha, \; \frac{(d+1)\beta}{-c} \right). \]
	This leads to
	\begin{align*}
		q_{123}(v) & = \left( \frac{(d+1)\beta}{-c} - \alpha + \frac{a\beta(d+1)}{-c} + b\beta + \alpha \right) \beta \\
		& = \left( \frac{(a+1)(d+1)\beta}{-c} + b\beta \right) \beta = \dfrac{2 + a + d}{-c} \cdot \beta^2;
	\end{align*}
	here we used $ad-bc=1$. Therefore $q_{123}$ factors through the $\beta$-coordinate, and is Witt-equivalent to $\lrangle{-(2+a+d)/c} \simeq \lrangle{-c(2+a+d)}$ as asserted.
	
	Now apply the dihedral symmetry \cite[p.532]{Li11} of $\tau(\cdots)$ to deduce the Witt equivalences
	\[ \tau(\Gamma_{-\gamma}, \Gamma_1, \ell \oplus \ell) \sim -\tau(\ell \oplus \ell, \Gamma_1, \Gamma_{-\gamma}) \sim \lrangle{c(2 + \Tr(\gamma))}. \]
	We conclude by comparing with \eqref{eqn:nabla} and \eqref{eqn:cal-m}.
\end{proof}

\begin{lemma}\label{prop:nabla-2}
	Let $\gamma \in \SL(2,F)_{\mathrm{reg}}$ and $\tilde{\gamma} \in \bm{p}^{-1}(\gamma)$. For every $g_1 \in \GL(2, F)$ with $\nu := \det g_1$ we have
	\[ \nabla(g\tilde{\gamma}g^{-1}) = \nabla(\tilde{\gamma}) \Cali_2(\nu, \gamma) \]
	where $\Cali_2$ is defined as in Definition--Proposition \ref{def:Cali-factor}, and $g \in \PGL(2, F)$ is the image of $g_1$.
\end{lemma}
\begin{proof}
	We may also assume $\tilde{\gamma} = \sigma_\text{LP}(\gamma)$ since $\nabla$ is genuine. Since $\nabla(\tilde{\gamma}) = |\det(\gamma + 1)|^{1/2} (\Theta^+_\psi - \Theta^-_\psi)(\tilde{\gamma})$ is locally constant, upon perturbation we may further assume $\gamma = \twomatrix{a}{b}{c}{d}$ with $c \neq 0$. Parameterize the stable class of $\gamma$ by $(K, F, \lambda)$ as in \S\ref{sec:Kubota}, where $K$ is an étale $F$-algebra of dimension $2$ and $\lambda \in K^1$; denote by $\tau$ the nontrivial $F$-involution on $K$. There exists $\omega \in K^\times$ such that $\lambda = \omega/\tau(\omega)$. Hence
	\[ 2 + \Tr(\gamma) = N_{K/F}(1 + \lambda) = (\omega + \tau(\omega))^2 N_{K/F}(\omega)^{-1}. \]
	By Lemma \ref{prop:nabla-1} and \cite[Proposition 1.3.4]{Per81}, we express $\nabla(\tilde{\gamma})$ as
	\begin{equation}\label{eqn:nabla-new} \begin{aligned}
		\gamma_\psi\left( \lrangle{-c, c(2 + \Tr(\gamma))} \right) & = \gamma_\psi(1) \gamma_\psi\left(- (2 + \Tr(\gamma)) \right) (-c, c(2 + \Tr(\gamma)))_{F,2} \\
		& = \gamma_\psi(1) \gamma_\psi(N_{K/F}(\omega))^{-1} (N_{K/F}(\omega), -c)_{F,2}
	\end{aligned}\end{equation}
	using $(-c,c)_{F,2} = 1$. As $\nabla$ is $G(F)$-invariant, we may assume $g_1 = \twomatrix{1}{}{}{\nu}$. It follows that $\nabla(g_1 \tilde{\gamma} g_1^{-1}) = \nabla(\tilde{\gamma}) (N_{K/F}(\omega), \nu)_{F,2}$ since $N_{K/F}(\omega) \bmod F^{\times 2}$ does not change.
\end{proof}

We switch back to the case of general $n$, identifying $\overline{G}_\psi^{(2)}$ with $\tilde{G}$ by the unique isomorphism.
\begin{theorem}\label{prop:st-conj-twofold}
	Adams' notion of stable conjugacy in Definition \ref{def:Adams-conj}, when restricted to $\overline{G}_\psi^{(2)} \simeq \tilde{G}$, coincides with Definition \ref{def:st-conj-elements} for $m=2$.
\end{theorem}
\begin{proof}
	Let $\tilde{\gamma} \in \overline{G}_{\psi, \text{reg}}^{(2)}$ with image $\gamma$ and $T := Z_G(\gamma)$. Let $\Ad(g): \gamma \mapsto \delta$ be a stable conjugation in $G$. By \cite[Lemme 5.7]{Li11}, there exists a unique $\tilde{\delta} \mapsto \delta$ in $\overline{G}_{\psi, \mathrm{reg}}^{(2)}$ which is Adams-stably conjugate to $\tilde{\gamma}$. We have to show that $\CaliAd(g)(\tilde{\gamma}) = \tilde{\delta}$, or equivalently $\nabla\left(\CaliAd(g)(\tilde{\gamma})\right) = \nabla(\tilde{\gamma})$ by Theorem \ref{prop:nabla}.
	
	Take a factorization pair $(g', g'')$ for $\Ad(g)$ (Definition \ref{def:st-conj}) so that $\CaliAd(g) = \Ad(g'') \CaliAd(g')$. Since $\nabla$ is $G(F)$-invariant, we may assume $g = g' \in G^T_\text{ad}(F)$ and $g''=1$. Take a parameter $(K, K^\sharp, x, c)$ for the conjugacy class of $\gamma$, with $K = \prod_{i \in I} K_i$, etc.; see \S\ref{sec:Sp}.
	
	Write $g = (g_i)_{i \in I}$. Since $\CaliAd(g)$ is the composition of $\CaliAd(g_i)$ in any order by \textbf{AD.4} of Proposition \ref{prop:CAd-prop}, Lemma \ref{prop:metaplectic-reduction} and \textbf{AD.1} reduce the problem to the case $n=1$ upon passing to a finite separable extension of $F$. Hence we may write $\gamma = \twomatrix{a}{b}{c}{d}$ and represent $g$ by $g_1 \in \GL(2,F)$ by choosing a symplectic basis. Note that the $\GL(2,F) = \GSp(W)$ action on $\widetilde{\SL}(2,F)$ mentioned in \S\ref{sec:Adams} coincides with the one from Proposition \ref{prop:BD-adjoint-action}, say by Proposition \ref{prop:lifting-uniqueness}.
	
	Now Lemma \ref{prop:nabla-2} implies $\nabla\left(\CaliAd(g)(\tilde{\gamma})\right) = \nabla(\tilde{\gamma})$ by the very definition of $\CaliAd(g)$, which involves the same factor $\Cali_2(\det g_1, \gamma)$.
\end{proof}

\subsection{Relation with \texorpdfstring{$\Theta$}{Theta}-lifting}\label{sec:theta}
Suppose $\text{char}(F)=0$ and fix $\psi$ to form Weil's metaplectic group $\overline{G}_\psi^{(2)}$ as in \S\ref{sec:Adams}. Denote by $\Pi_-(\overline{G}_\psi^{(2)})$ the genuine admissible dual of $\overline{G}_\psi^{(2)}$, i.e. the set of isomorphism classes of genuine irreducible admissible representations. Also denote by $\tilde{G} \twoheadrightarrow G(F)$ the BD-cover with $m=2$. As mentioned in \S\ref{sec:Adams}, there is a unique topological isomorphism $\tilde{G} \simeq \overline{G}^{(2)}_\psi$. Note that $Y_{Q,2} = Y$, thus the isogenies $\iota_{Q,2}$ are identity maps.

Let $(V,q)$ be a quadratic $F$-vector space with $\dim_F V = 2n+1$ and $d^\pm(V,q) = 1$. Denote by $\Pi(\SO(V,q))$ (resp. $\Pi(\Or(V,q))$) the admissible dual of $\SO(V,q)$ (resp. of $\Or(V,q)$). A fundamental result of Adams--Barbasch \cite{AB98} and Gan--Savin \cite{GS1}, for archimedean and non-archimedean $F$ respectively, says that for every $\pi_{\SO} \in \Pi(\SO(V,q))$, there exists a unique extension $\pi_{\Or}$ to $\Or(V,q) = \SO(V,q) \times \{\pm 1_{\Or}\}$ such that the $\theta$-lift $\theta_\psi(\pi_{\Or})$ to $\overline{G}_\psi^{(2)}$ is nonzero. Furthermore, it asserts that $\pi_{\SO} \mapsto \theta_\psi(\pi_{\Or})$ yields a bijection
\begin{equation}\label{eqn:AB-GS} \begin{aligned}
	\Pi_-(\tilde{G}) \simeq \Pi_-(\overline{G}_\psi^{(2)}) \xleftrightarrow{1:1} & \bigsqcup_{\substack{\dim_F V = 2n+1 \\ d^\pm(V,q) = 1 \\ \bmod \simeq }} \Pi(\SO(V,q)) \\
	\theta_\psi(\pi_{\Or}) \longmapsfrom & \pi_{\SO}
\end{aligned}\end{equation}
preserving discrete series, supercuspidal representations, etc. All these $\SO(V,q)$ are pure inner forms of each other.

Hereafter we assume $F$ is non-archimedean of residual characteristic $p \neq 2$. Observe that $\Lgrp{\SO(V,q)} = \Sp(2n, \CC) \times \Weil_F$. As $\psi$, $(W, \lrangle{\cdot|\cdot})$ are chosen, this can be identified with $\Lgrp{\tilde{G}}$ by \S\ref{sec:L-group}. Also recall that $\SO(2n+1)$ and $\Sp(W)$ share the same Weyl group $\Omega$ relative to some split maximal torus.

Now consider an epipelagic $L$-parameter (Definition \ref{def:epipelagic-parameter})
\[ \phi: \Weil_F \to \Lgrp{\tilde{G}}, \]
which factorizes into
\[ \Weil_F \xrightarrow{\phi_{S_{Q,2}, \Lgrp{j}} } \Lgrp{S_{Q,2}} \xrightarrow{\Lgrp{j}} \Lgrp{\tilde{G}}. \]
As remarked in \S\ref{sec:epipelagic-parameters} (see also \cite[\S 5]{Kal15}), this is done by choosing a $\chi$-datum and works in the exactly same way for both $\tilde{G}$ and $\SO(2n+1) = G_{Q,2}$ on the dual side.  The parameter $\phi_{S_{Q,2}, \Lgrp{j}}$ yields
\begin{compactitem}
	\item the triple $(\mathcal{E}, S, \theta^\flat)$ as in Definition \ref{def:inducing-data}, now with $\theta^\flat: S(F) \to \CC^\times$ an epipelagic character;
	\item a stable class $\mathcal{E}$ of embeddings $j: S \hookrightarrow G$ of maximal tori of type (ER), corresponding to an element of $H^1(F, \Omega)$;
	\item the element in $H^1(F, \Omega)$ also determines a stable class $\mathcal{F}$ of embeddings $k: S \hookrightarrow \SO(V,q)$, where $(V,q)$ varies as in \eqref{eqn:AB-GS}; see \cite[\S 3.2]{Kal19} for generalities on this extension across inner forms.
\end{compactitem}
For each $k \in \mathcal{F}$, we can transport $\theta^\flat$ to a character $\theta_k$ of $kS(F)$. Kaletha's epipelagic supercuspidal $L$-packet is defined as
\begin{equation*}
	\Pi^{\SO}_\phi := \left\{ \pi_{kS, \theta_k \epsilon_{kS}} : k \in \mathcal{F} \right\} \; \subset \bigsqcup_{\substack{\dim_F V = 2n+1 \\ d^\pm(V,q) = 1 \\ \bmod \simeq }} \Pi(\SO(V,q)),
\end{equation*}
where
\begin{compactitem}
	\item $\epsilon_{kS}: kS(F) \to \{\pm 1\}$ is the character in \cite[\S 4.6]{Kal15}, see also \S\ref{sec:toral-invariants};
	\item $\pi_{kS, \theta_k \epsilon_{kS}}$ is the irreducible supercuspidal representation of $\SO(V,q)$ constructed in \cite[\S 3]{Kal15}, see also Theorem \ref{prop:epipelagic-supercuspidal}.
\end{compactitem}

\begin{remark}\label{rem:SO-packet-size}
	By \cite[Proposition 5.7]{Kal15}, $|\Pi^{\SO}_\phi|$ equals $|\pi_0(C_\phi, 1)^\wedge| = |\mathbf{B}(S)|$. More precisely, $H^1(F,S) \simeq \mathbf{B}(S)$ (Kottwitz's isomorphism) acts on the conjugacy classes of embeddings in $\mathcal{E}$ and $\mathcal{F}$, and makes both sets into torsors.
\end{remark}

On the other hand, using the stable system for $m=2$ of Theorem \ref{prop:dagger-SS}, we construct the packet $\Pi_\phi \subset \Pi_-(\tilde{G})$ of Definition \ref{def:epipelagic-packet}. Our aim is to show that $\Pi_\phi = \Pi_\phi^{\SO}$ under \eqref{eqn:AB-GS}. The main tool will be the theory of Loke--Ma--Savin \cite{LMS16}. To this end, we follow \textit{loc. cit.} to assume
\[ \psi|_{\mathfrak{o}_F} \not\equiv 1, \quad \psi|_{\mathfrak{p}_F} = 1. \]
and take $\xi = \psi$ in the constructions of \S\ref{sec:compact-induction}.

Before applying their results, notice that the inducing data of depth $\frac{1}{e}$ in \textit{loc. cit.} take the form $(x, \lambda, \chi)$, where
\[ x \in \mathcal{B}(G, F), \quad \lambda \in \check{\mathsf{V}}_{x,1/e}, \quad \chi \in \Hom(\mathsf{S}_\lambda, \CC^\times), \quad \mathsf{S}_\lambda := \Stab_{G(F)_x}(\lambda) \big/ G(F)_{x, 1/e}. \]
Given $(\mathcal{E}, S, \theta^\flat)$ as above and $j \in \mathcal{E}$, the point $x$ will be associated to $j$, and $\lambda$ arises from $\theta^\flat|_{jS(F)_{1/e}}$ (depending on $\psi$). Before explicating $\chi$, we split the cover over $\Stab_{G(F)_x}(\lambda)$ using the following facts:
\begin{compactitem}
	\item by \cite[\S 3.4]{LMS16} there exists an $\mathfrak{o}_F$-lattice $\Lambda \subset W$ such that $\Lambda = \left\{w \in W: \lrangle{w|\Lambda} \subset \mathfrak{p}_F \right\}$ (i.e. self-dual), and $G(F)_x \subset \Stab_{G(F)}(\Lambda)$;
	\item the lattice model \cite[Chapitre 2, II.8]{MVW87} associated to $\Lambda$ furnishes a splitting $\sigma_\Lambda: \Stab_{G(F)}(\Lambda) \hookrightarrow \overline{G}_\psi^{(8)} \simeq \tilde{G}$, and $\sigma_\Lambda|_{G(F)_x}$ is independent of $\Lambda$ by \cite[Lemma A.3.1]{LMS16}.
\end{compactitem}
The character produced from $(\lambda, \chi)$ of $\Stab_{G(F)_x}(\lambda)$ can thus be lifted to a genuine one of its preimage, and the remaining construction is as in \S\ref{sec:compact-induction}.

By Lemma \ref{prop:pro-p-splitting}, the splitting restricted to $G(F)_{x, 1/e}$ coincides with the one in \S\ref{sec:compact-induction}. From Lemma \ref{prop:S_0} and \eqref{eqn:lambda-stab} we have $\Stab_{G(F)_x}(\lambda) = jS(F)_{p'} \ltimes G(F)_{x, 1/e}$, thus the crux is to compare the splittings over $jS(F)_{p'}$.

\begin{lemma}\label{prop:chi-formula}
	The genuine epipelagic representation $\pi_{jS, \theta_j \epsilon_{jS}}$ corresponds to the datum $(x, \lambda, \chi)$ in \cite{LMS16}, where $x, \lambda$ are determined as above, and
	\[ \chi = \theta^\flat \otimes \theta_j^\dagger \otimes \epsilon_{jS}: \; jS(F)_{p'} \to \CC^\times; \]
	here $\theta^\flat$ are transported from $S(F)$ to $jS(F)$.
\end{lemma}
\begin{proof}
	By inspecting the construction of $\theta_j$ in \S\ref{sec:inducing-data}, it boils down to identify the splittings over $jS(F)_{p'}$ given by
	\begin{inparaenum}[(a)]
		\item the lattice models, and
		\item the recipe of Example \ref{eg:section-S}.
	\end{inparaenum}
	Recall that in Example \ref{eg:section-S}, we have $W = \bigoplus_{i \in I} W_i$ as joint eigenspaces under $jS(F)_{p'} = \{\pm 1\}^I$, and lift each $(-1)_{\Sp(W_i)}$ into $\tilde{G}^\natural \hookrightarrow \overline{G}^{(8)}_\psi$. Let $\Lambda$ be a self-dual lattice stabilized by $jS(F)$. Therefore we have a decomposition $\Lambda = \bigoplus_{i \in I} \Lambda_i$ where $\Lambda_i = \Lambda \cap W_i$ is the joint eigen-lattice since $p > 2$; each $\Lambda_i$ is still self-dual. Cf. \cite[p.551]{Li11}.

	For each $i$, there exist transversal Lagrangians $W_i = \ell_i \oplus \ell'_i$ such that $\Lambda_i = (\Lambda_i \cap \ell_i) \oplus (\Lambda_i \cap \ell'_i)$. Let us compare the splittings $\sigma_{\Lambda_i}$ and $\sigma_{\ell_i}$, the latter being reviewed before Proposition \ref{prop:convenient-Schrodinger}. By \cite[Proposition 2.13]{Li11} they agree over $\Stab(\Lambda_i) \cap \GL(\ell_i) \ni -1_{\Sp(W_i)}$. Proposition \ref{prop:convenient-Schrodinger} ensures that $\sigma_{\ell_i}(-1)$ equals the $\widetilde{-1}_{\Sp(W_i)}$ in Definition \ref{def:lifting-minus-1}, thus completes the proof.
\end{proof}

More notations: given $(V,q)$ with $\dim_F V = 2n+1$, we systematically write $H := \SO(V,q)$. For an inducing datum $(x', \lambda', \chi')$ for epipelagic supercuspidals for $\Or(V,q)$ as in \cite{LMS16}, where $\lambda' \in \check{\mathsf{V}}_{x',r}$, define $\mathsf{S}_{\lambda'}$ as the stabilizer of $\lambda'$ in $\Or(V,q)_{x'}$ modulo $H(F)_{x',r}$. The version defined using $H(F)_{x'}$ is denoted by $\mathsf{S}^H_{\lambda'}$. To distinguish $G$ and $H$, we shall also write $\check{\mathsf{V}}^G_{x,r}$, etc. For both $G, H$ we use the invariant form \eqref{eqn:B_g-LMS} to identify Lie algebras and their duals.

\begin{theorem}\label{prop:Theta-compatibility}
	For every additive character $\psi$ and every epipelagic $L$-parameter $\phi: \Weil_F \to \Lgrp{\tilde{G}}$, in the notation of \eqref{eqn:AB-GS} we have
	\[ \Pi_\phi = \left\{ \theta_\psi(\pi_{\Or}) : \pi_{\SO} \in \Pi^{\SO}_\phi \right\}, \]
	where $\Pi_\phi$ is constructed using the stable system of Theorem \ref{prop:dagger-SS}.
\end{theorem}
\begin{proof}
	To begin with, we reduce the problem to the case that $\psi$ induces a nontrivial character of $\mathfrak{o}_F/\mathfrak{p}_F$. Indeed, $\psi$ enters in the identification $\Lgrp{G} \simeq \tilde{G}^\vee \times \Weil_F$, which is not fixed at this moment. It remains to prove that the packets from both sides are independent of $\psi$: for $\Pi_\phi$ this is Theorem \ref{prop:packet-independence}, whilst for $\left\{ \theta_\psi(\pi_{\Or}) : \pi_{\SO} \in \Pi^{\SO}_\phi \right\}$ this is \cite[Proposition 11.1]{GG} that stems from \cite[Theorem 12.1 (i)]{GS1}.

	Let $Y \in \mathfrak{s}_{\text{reg}}(F) \cap \mathfrak{s}(F)_{-1/e}$ be defined by $\theta^\flat \circ \exp = \psi(\mathbb{B}(Y, \cdot))$ as in \S\ref{sec:Adler-Spice}, now with $\psi = \xi$. Apply Remarks \ref{rem:explicit-moment-map}, \ref{rem:moment-map-basis} and Proposition \ref{prop:T-depth} to see that for any $j: S \hookrightarrow G$ in the given stable class, we have
	\begin{compactitem}
		\item a quadratic $F$-vector space $(V,q) = (W, q\lrangle{Y}) \oplus \lrangle{(-1)^n \det Y}$,
		\item a natural embedding $k: S \hookrightarrow \SO(V,q) =: H$,
		\item $\iota: W \hookrightarrow V$ standing for the natural inclusion,
		\item $\mathcal{L}$, $\mathcal{L}'$: the self-dual lattice functions corresponding to the points $x \in \mathcal{B}(G, F)$, $x' \in \mathcal{B}(\SO(V,q)), F)$ arising from $j,k$ respectively.
	\end{compactitem}
	They satisfy \index{moment map}
	\begin{gather*}
		j\mathfrak{s}(F)_{-1/e} \ni jY \xleftarrow{M_W} \left( \iota \in \Hom_F(W,V) \right) \xrightarrow{M_V} kY \in k\mathfrak{s}(F)_{-1/e}, \\
		\iota\mathcal{L}_s \subset \mathcal{L}'_{s - \frac{1}{2e}}.
	\end{gather*}
	The construction in \S\ref{sec:compact-induction} or \cite[\S 3.3]{Kal15} associates to $jY, kY$ the stable linear functionals $\lambda \in \check{\mathsf{V}}^G_{x, 1/e}$ and $\lambda' \in \check{\mathsf{V}}^H_{x', 1/e}$. The formulas above ``witness'' the matching condition (M) of \cite[Proposition 1.2.1]{LMS16} for $(x, \lambda)$ and $(x', -\lambda')$. The $H$-version of \eqref{eqn:lambda-stab} entails isomorphisms
	\[ \mathsf{S}_\lambda \simeq jS(F)/jS(F)_{1/e} \leftiso S(F)_{p'} \rightiso kS(F)/kS(F)_{1/e} \simeq \mathsf{S}^H_{\lambda'}. \]
	We construct a map $\alpha: \mathsf{S}_{\lambda'} \to \mathsf{S}_\lambda$ as follows. For every $\gamma \in kS(F)$, Remark \ref{rem:explicit-moment-map} guarantees that the action \eqref{eqn:moment-action} satisfies
	\[ (1, k(\gamma)) \iota = k(\gamma) \circ \iota = \iota \circ j(\gamma) = (j(\gamma)^{-1}, 1) \iota; \]
	likewise $-1_{\Or} \in \Or(V,q)$ satisfies $(1, -1_{\Or}) \iota = -\iota = (-1_{\Sp} , 1) \iota$. It is routine to check that the surjective homomorphism
	\begin{align*}
		\alpha: \mathsf{S}_{\lambda'} = \mathsf{S}^H_{\lambda'} \times \{\pm 1_{\Or}\} & \longrightarrow \mathsf{S}_\lambda \\
		k(\gamma) & \longmapsto j(\gamma), \quad (\gamma \in S(F)_{p'}) \\
		-1_{\Or} & \longmapsto -1_{\Sp}.
	\end{align*}
	coincides with the $\alpha$ constructed in \cite[A.2 Proof of Lemma 8.1.1, (ii)]{LMS16}. Now transport $\theta^\dagger_j$, $\epsilon_{jS}$, $\epsilon_{kS}$ to $S(F)_{p'}$ via $j,k$, and invoke Lemma \ref{prop:chi-formula} to write
	\[ \chi' := \chi \circ \alpha = \underbracket{(\theta^\dagger_j \cdot \epsilon_j/\epsilon_k)} \cdot \epsilon_{kS} \cdot \theta^\flat \boxtimes \left[ -1_{\Or} \mapsto \chi(-1_{\Sp}) \right] \]
	modulo transportation among $S$, $jS$ an $kS$. Theorem \ref{prop:interplay} asserts that $\theta^\dagger_j \cdot \epsilon_j/\epsilon_k = 1$, so $\chi' = \epsilon_{kS} \theta^\flat \boxtimes \cdots$. Denote by $\pi(x, \lambda, \chi) \in \Pi_-(\tilde{G})$ and $\pi(x',\lambda',\chi') \in \Pi(\Or(V,q))$ the epipelagic supercuspidals so obtained. We are ready to apply \cite[Theorem 1.2.3 (i)]{LMS16}: keeping track of the sign and contragredient therein, we have
	\[ \theta_\psi\left( \pi(x, \lambda, \chi)\right) = \pi\left( x', -\lambda', \chi'^{-1} \right) \]
	under the $\theta$-lifting backwards to $\Or(V,q)$. 

	In terms of Kaletha's data, $\pi(x', -\lambda', \chi'^{-1})$ is associated to the epipelagic character $(\theta^\flat)^{-1}$ transported to $kS(F)$. Now vary the $G(F)$-conjugacy class of $j \in \mathcal{E}$, so that the $\pi(x, \lambda, \chi)$ exhaust $\Pi_\phi$ without repetition. On the other hand, Theorem \ref{prop:mm} implies that the corresponding $k$ exhausts the set $\mathcal{F}$ (modulo conjugation) of embeddings into pure inner forms of $\SO(2n+1)$, without repetition. In view of Remark \ref{rem:SO-packet-size}, we obtain Kaletha's $L$-packet for $\SO(2n+1)$ induced from $(\theta^\flat)^{-1}$ by discarding the $\{\pm 1_{\Or}\}$-components.

	Finally, $(\theta^\flat)^{-1} = \theta^\flat \circ \Ad(w_0)$ where $w_0 \in \Omega(H,kS)(\bar{F})$ is the longest element, thus defined over $F$. The $\Omega(H,kS)(F)$-action leaves Kaletha's $L$-packet intact. All in all, we obtain $\Pi^{\SO}_\phi$ on the $\SO(2n+1)$-side.
\end{proof}

\subsection{Theory of Hiraga--Ikeda}
Hereafter we assume $m \in 2\Z$. Fix $\epsilon: \mu_m \rightiso \bmu_m$, and fix an additive character $\psi$ when $m \equiv 2 \pmod 4$. Following Hiraga--Ikeda, define two morphisms $\bm{\tau}^\pm: \PGL(2) \to \SL(2)$ by
\[ \bm{\tau}^\pm(\gamma) = \pm 1 \cdot (\det \hat{\gamma})^{-m/2} \hat{\gamma}^m, \quad \hat{\gamma} \in \GL(2): \; \text{representative}. \]
Given $\delta \in \SL(2)_\text{reg}$, put $T := Z_{\SL(2)}(\delta)$. Up to stable conjugacy, the pair $(T, \delta)$ is parameterized by a $2$-dimensional étale $F$-algebra $K$ with nontrivial $F$-involution $\tau$, and $x \in K^1$ as done in \S\ref{sec:Kubota}. Set $T_1 := Z_{\GL(2)}(T)$, then $T(F) \simeq K^1$ and $T_1(F) \simeq K^\times$.

Note that if $\delta = \bm{\tau}^\pm(\gamma) \in \SL(2,F)_\text{reg}$, then $\gamma \in T_1/\Gm \subset \PGL(2)$. Since $\det|_{T_1}$ corresponds to $N_{K/F}$, by choosing a representative $\omega \in K^\times$ of $\gamma$, we have $\det\hat{\gamma} = N_{K/F}(\omega)$ and the definition of $\bm{\tau}^\pm$ translates into
\[ \pm x = \omega^m N_{K/F}(\omega)^{-m/2} = \left( \omega/\tau(\omega)\right)^{m/2}. \]

Now consider the BD-cover $\bmu_m \hookrightarrow \widetilde{\SL}(2,F) \stackrel{\bm{p}}{\twoheadrightarrow} \SL(2,F)$ constructed in \S\ref{sec:BD-Sp} (with $n=1$). We shall employ the preferred section $\bm{s}$ and the cocycle $\bm{c}$ of \eqref{eqn:Kubota-cocycle}. By convention $\delta := \bm{p}(\tilde{\delta})$. Fix a maximal $F$-torus $T \subset \SL(2)$. Recalling \eqref{eqn:isogeny-Sp}, the discussion above leads to a bijection
\[\begin{tikzcd}
	\left\{ (\gamma, \tilde{\delta}) \in \PGL(2,F) \times \tilde{T}_\text{reg} : \delta = \bm{\tau}^\pm(\gamma) \right\} \arrow[leftrightarrow]{r}{1:1} & \tilde{T}^\pm_{Q,m, \text{reg}}.
\end{tikzcd}\]
Here $\tilde{T}^\pm_{Q,m, \text{reg}} := \left\{ (\tilde{\delta}, \delta_0) \in \tilde{T}^\pm_{Q,m} : \delta \in T_\text{reg}(F) \right\}$, see \eqref{eqn:iota-cover}, and $\delta_0$ corresponds to $\omega/\tau(\omega) \in K^1$. We can rephrase the transfer factor of Hiraga--Ikeda as follows.

\begin{definition}[K.\ Hiraga, T.\ Ikeda]\index{transfer factor}
	Let $\sigma \in \{+,-\}$, $T \subset \SL(2)$ be a maximal torus, and $(\tilde{\delta}, \delta_0) \in \tilde{T}^\sigma_{Q,m, \text{reg}}$. Write $\tilde{\delta} = \noyau\bm{s}(\delta)$ for some $\noyau \in \bmu_m$ and take $\omega \in K^\times$ such that $\delta_0$ is parameterized by $\omega/\tau(\omega)$.
	\begin{enumerate}
		\item Suppose $m \equiv 2 \pmod 4$. Set
			\begin{align*}
				\Delta^+(\tilde{\delta}, \delta_0) & := \noyau \cdot \dfrac{\gamma_\psi(1)}{\gamma_\psi(N_{K/F}(\omega))} \cdot \left( N_{K/F}(\omega), -\bm{x}(\delta) \right)_{F,2} & (\sigma = +), \\
				\Delta^-(\tilde{\delta}, \delta_0) & := \gamma_\psi(1)^2 \Delta^+ \left( \bm{s}(-1)\tilde{\delta} , \delta_0 \right) & (\sigma = -).
			\end{align*}
		\item Suppose $4 \mid m$. Set
			\begin{align*}
				\Delta^+(\tilde{\delta}, \delta_0) & := \noyau \cdot \left( N_{K/F}(\omega), -\bm{x}(\delta) \right)_{F,2} & (\sigma = +), \\
				\Delta^-(\tilde{\delta}, \delta_0) & := \Delta^+ \left( \bm{s}(-1)\tilde{\delta} , \delta_0 \right) & (\sigma = -).
			\end{align*}
	\end{enumerate}
\end{definition}

\begin{remark}
	In view of the Remark \ref{rem:Kubota-minus}, the original factor $\Delta^\pm$ of Hiraga--Ikeda lives on the ``opposite'' central extension of $\widetilde{\SL}(2,F)$ by $\bmu_m$, which can be obtained from ours by changing $\epsilon$ to $\epsilon^{-1}$ by Remark \ref{rem:rescaling-Q}. Thus there is no essential difference.
\end{remark}

\begin{proposition}[Hiraga--Ikeda]\label{prop:HI-invariance}
	The factors $\Delta^\pm$ are invariant under adjoint $\SL(2,F)$-action.
\end{proposition}
\begin{proof}
	As $\bm{s}(-1)$ is central by Proposition \ref{prop:BD-adjoint-action}, it suffices to consider $\Delta^+$. Let $N := N_{K/F}(\omega)$. Conjugation does not change $N \bmod F^{\times 2}$, thus it remains to show that for $\delta' = \Ad(g)(\delta)$ and $\noyau \bm{s}(\delta') = \Ad(g)(\bm{s}(\delta))$ where $g \in \SL(2,F)$, we have $\noyau (N, \bm{x}(\delta'))_{F,2} = (N, \bm{x}(\delta))_{F,2}$. Note that $\noyau \in \bmu_m$ is unique since $\delta$ is a good element.
	
	Suppose that $\bm{\tau}^+(\gamma) = (\det\hat{\gamma})^{-m/2} \hat{\gamma}^m = \delta$, where $\gamma \in \PGL(2)$ has representative $\hat{\gamma}$ parameterized by the chosen $\omega \in K^\times$. Put $\gamma' := \Ad(g)(\gamma)$. Since $\left( N^{m/2}, \bm{x}(\delta) \right)^{-1}_{F, m} = (N, \bm{x}(\delta))_{F,2}$ and similarly for $\bm{x}(\delta')$, the following holds in $\widetilde{\GL}(2,F)$:
	\[ (N, \bm{x}(\delta))_{F,2} \bm{s}(\gamma^m) = \bm{s}(N^{m/2}) \bm{s}(\delta), \quad (N, \bm{x}(\delta'))_{F,2} \bm{s}(\gamma'^m) = \bm{s}(N^{m/2}) \bm{s}(\delta'). \]
	By the easy fact that $\gamma \mapsto \bm{s}(\gamma)^m$ respects conjugation and \cite[p.130 + Lemma 1.2.1]{Fl80}, we have $\Ad(g)\left( \bm{s}(\gamma^m) \right) = \bm{s}\left( \Ad(g)(\gamma^m) \right)$. On the other hand, it is easily seen that $\bm{s}(N^{m/2})$ centralizes $\widetilde{\SL}(2,F)$. All these combine into
	\begin{align*}
		(N, \bm{x}(\delta))_{F,2} \bm{s}(\gamma'^m) & = \Ad(g) \left( (N, \bm{x}(\delta))_{F,2} \bm{s}(\gamma^m) \right) = \bm{s}(N^{m/2}) \Ad(g) \left( \bm{s}(\delta) \right) \\
		& = \noyau \bm{s}(N^{m/2}) \bm{s}(\delta') = \noyau (N, \bm{x}(\delta'))_{F,2} \bm{s}(\gamma'^m).
	\end{align*}
	This proves the required equality.
\end{proof}

\begin{proposition}\label{prop:HI-Li}
	When $m=2$, we have $\Delta^+(\tilde{\delta}, \delta_0) = \nabla(\tilde{\delta})$.
\end{proposition}
\begin{proof}
	Write $\delta = \twomatrix{a}{b}{c}{d}$. Both sides being $\SL(2,F)$-invariant (Proposition \ref{prop:HI-invariance}) and genuine, we may adjust $\tilde{\delta}$ to assume $c \neq 0$ and $\tilde{\delta} = \bm{s}(\delta)$. Now compare $\Delta^+(\tilde{\delta}, \delta_0)$ with \eqref{eqn:nabla-new}, identifying $\bm{s}(\delta)$ with $\sigma_\text{LP}(\delta)$.
\end{proof}

\begin{lemma}
	When $m \equiv 2 \pmod 4$, the factors $\Delta^\pm(\tilde{\delta}, \delta_0)$ depend only on $\tilde{\delta}$.
\end{lemma}
\begin{proof}
	It suffices to show that $\Delta^+(\tilde{\delta}, \delta_0)$ is independent of $\delta_0$ or $\omega$. When $T$ is anisotropic (resp. split), Proposition \ref{prop:iota-kernel} implies that $\delta$ determines $N_{K/F}(\omega)$ uniquely (resp. up to $\mu_{m/2}$). Since $m/2$ is odd, $(N_{K/F}(\omega), \bm{x}(\delta))_{F,2}$ is unaffected.
\end{proof}

\begin{remark}\index{transfer factor}
	When $m=2$, in \cite[Définition 5.9]{Li11} are defined the two elliptic endoscopic data $(1,0)$, $(0,1)$ of $\widetilde{\SL}(2,F)$ as well as the transfer factors $\Delta_{1,0}$, $\Delta_{0,1}$, both viewed as genuine functions on $\widetilde{\SL}(2,F)_\text{reg}$. Proposition \ref{prop:HI-Li} amounts to $\Delta^+ = \Delta_{(1,0)}$. Consider $\Delta^-$ next. Embed $\widetilde{\SL}(2,F)$ into $\overline{G}^{(8)}_\psi$, then $\gamma_\psi(1)^2 \bm{s}(-1) \in \widetilde{\SL}(2,F)^\natural$ corresponds to $\frac{\gamma_\psi(1)}{\gamma_\psi(-1)} \sigma_\text{LP}(-1) = \sigma_\ell(-1)$ (Proposition \ref{prop:convenient-Schrodinger}). It follows from \cite[Proposition 5.16]{Li11} that $\Delta^- = \Delta_{(0,1)}$.
\end{remark}

For any maximal torus $T \subset \SL(2)$ we have $\kappa_-$ from Definition \ref{def:kappa-minus}; denote by $\kappa_+$ the trivial character of $H^1(F,T)$.
\begin{theorem}\label{prop:HI-cocycle}
	Suppose $\Ad(g): \delta \mapsto \eta$ is a stable conjugation in $\SL(2,F)_{\mathrm{reg}}$. Let $\tilde{\delta}, \tilde{\eta} \in \widetilde{\SL}(2,F)$ be their preimages.
	\begin{itemize}
		\item When $m \equiv 2 \pmod 4$ and $\tilde{\eta} = \CaliAd(g)(\tilde{\delta})$ (dropping references to $\delta_0, \eta_0$ by Corollary \ref{prop:st-conj-elements-canonical}), we have
			\begin{equation*}
				\Delta^\pm(\tilde{\eta}, \ldots) = \lrangle{\kappa_\pm, \mathrm{inv}(\delta, \eta)} \Delta^\pm(\tilde{\delta}, \ldots)
			\end{equation*}
		\item When $4 \mid m$ and $(\tilde{\eta}, \eta_0) = \CaliAd^\pm(g)(\tilde{\delta}, \delta_0)$, we have
			\begin{equation*}
				\Delta^\pm(\tilde{\eta}, \eta_0) = \Delta^\pm(\tilde{\delta}, \delta_0).
			\end{equation*}
	\end{itemize}
\end{theorem}
\begin{proof}
	By the $\SL(2,F)$-invariance of $\Delta^\pm$ (Proposition \ref{prop:HI-invariance}), it suffices to consider stable conjugacy realized by $\Ad(g): \delta \mapsto \eta$ such that $g = \twomatrix{1}{}{}{\nu} \in \PGL(2,F)$, where $\nu \in F^\times$. If $T$ splits then $\Ad(g)$ reduces to ordinary conjugacy and we conclude by \textbf{AD.3} of Proposition \ref{prop:CAd-prop}. Hereafter assume $T$ anisotropic, so that $\delta, \eta$ take the form $\twomatrix{a}{b}{c}{d}$ with $c \neq 0$. A routine computation in $\widetilde{\GL}(2,F)$ using \eqref{eqn:Kubota-cocycle} shows
	\[ \bm{s}\twobigmatrix{1}{}{}{\nu}^{-1} = \bm{s}\twobigmatrix{1}{}{}{\nu^{-1}}, \quad \Ad(g) \bm{s}\twobigmatrix{a}{b}{c}{d} = \bm{s}\twobigmatrix{a}{\nu^{-1}b}{\nu c}{d} \quad (c \neq 0). \]
	The cases of $\Delta^+$ follow because $\bm{x}(\Ad(g)\delta) = \nu\bm{x}(\delta)$ and $\Cali_m(\nu, \delta_0) = (N_{K/F}(\omega), \nu)_2$. The case of $\Delta^-$ for $4 \mid m$ follows by a comparison of the definitions of $\CaliAd^-(g)$ and $\Delta^-$. The case of $\Delta^-$ for $m \equiv 2 \pmod 4$ is accounted by Proposition \ref{prop:CAd-minus-1}, which says $\bm{s}(-1)\CaliAd(g)(\tilde{\delta}) = \sgn_{K/F}(\nu) \CaliAd(g)(\bm{s}(-1)\tilde{\delta})$.
\end{proof}

When $m=2$, Theorem \ref{prop:HI-cocycle} reduces to \cite[Proposition 5.13]{Li11}.

\section{Errata}\label{sec:Errata}




The computations of the toral invariants $f_{(G, S)}(\alpha)$ in Sect.~7.2 of the paper \emph{Stable conjugacy and epipelagic L-packets for Brylinski-Deligne covers of $\Sp(2n)$} in Selecta Mathematica (N.S.), Vol. 26, (2020), no.~1 (ditto for the arXiv version) are flawed when the root $\alpha$ takes the form $\alpha = \epsilon_i \pm \epsilon_j$, in both the symplectic and the odd orthogonal case. The resulting $f_{(G, S)}(\alpha)$ can vary within a stable conjugacy class $\mathcal{E}$ of embeddings $j: S \hookrightarrow G$ of maximal tori. The main impacts of this mistake are listed below.
\begin{itemize}
	\item The proof of Lemma 7.2.2 no longer works for general tori $S$. This lemma is used to prove the key Theorem 7.6.3 on the stability of epipelagic $L$-packets.
	\item Theorem 8.3.1 is directly built upon Sect.~7.2. Its aim is interpret the \emph{stable system} (a notion of obstruction to stability, see Definition 7.3.3) constructed in Sect.~8.2 for $m \equiv 2 \pmod{4}$, as a ratio $\epsilon_j(\gamma) / \epsilon_k(\gamma)$. This is used in Sect.~9.3 to show the compatibility with $\Theta$-lifting when $m=2$.
\end{itemize}

The goal here is to show that
\begin{itemize}
	\item Lemma 7.2.2 still holds for a restricted class of maximal tori, which suffices for our purposes;
	\item Theorem 8.3.1 can be turned into a definition for the stable system when $m \equiv 2 \pmod{4}$, in order to force the results in Sect.~9.3 to be true. The new definition is more transparent.
\end{itemize}

The impacted results can thus be restored with reasonable modifications. Remarkably, this is a reversion to the strategy in an early draft of this work.

The author is grateful to Chuijia Wang for kindly pointing our these mistakes.

\subsection{Correction to Lemma 7.2.2}
One should assume in Lemma 7.2.2 that $S$ is a maximal torus of type (ER) in $G = \mathrm{Sp}(W)$, as in Definition 6.1.1. This is all what one needs for the epipelagic supercuspidals studied in Sect.~7.3 --- 7.6. The corrected arguments are given below.

Retain the notation in the original proof. The goal is to show $\epsilon_j(\gamma_0)$ depends only on the stable conjugacy class of $j: S \hookrightarrow G$ and $\gamma_0 \in S(F)_{p'}$. Given $j$, define the sign
\[ E_j(S, G) = \prod_{\alpha \in R(G, jS)_{\mathrm{sym}}(\overline{F}) / \Gamma_F} f_{G, jS}(\alpha). \]
Set $G_{j \gamma_0} := Z_G(j(\gamma_0))^\circ$. By \cite[Lemma 4.12]{Kal15},
\begin{equation}\label{eqn:epsilon-character-errata}
	\epsilon_j(\gamma_0) = \prod_{\substack{\alpha \in R(G, jS)_{\mathrm{sym}}(\overline{F}) / \Gamma_F \\ \alpha(\gamma_0) \neq 1}} f_{G, jS}(\gamma_0) = \frac{E_j(S, G)}{E_j(S, G_{j \gamma_0})}.
\end{equation}

Fix any additive character $\psi$ of $F$. Kottwitz's formula \cite[Corollary 4.11]{Kal15} gives
\begin{equation}\label{eqn:Kottwitz-formula-errata}
	E_j(S, G) = \dfrac{e(G) \epsilon(X^*(S)_{\mathbb{C}} - X^*(T)_{\mathbb{C}}, \psi)}{\displaystyle\prod_{\alpha \in R(G, jS)_{\mathrm{sym}}(\overline{F}) / \Gamma_F} \lambda_{F_\alpha / F_{\pm\alpha}}(\psi \circ \mathrm{Tr}_{F_\alpha / F_{\pm\alpha}}) };
\end{equation}
\begin{compactitem}
	\item $e(G)$ is the Kottwitz sign of $G$;
	\item $T$ is a minimal Levi subgroup in the quasisplit inner form $G^*$ of $G$;
	\item $\epsilon(\cdots , \psi)$ is the $\epsilon$-factor with Langlands' normalization associated with the virtual $\Gamma_F$-representation $X^*(S)_{\mathbb{C}} - X^*(T)_{\mathbb{C}}$ of degree zero;
	\item $F_\alpha/F_{\pm\alpha}$ are the quadratic extensions associated with symmetric roots $\alpha$, and $\lambda_{F_\alpha / F_{\pm\alpha}}(\cdots)$ are Langlands' constants.
\end{compactitem}

Formulas \eqref{eqn:epsilon-character-errata} and \eqref{eqn:Kottwitz-formula-errata} apply to connected reductive $F$-groups in general. In particular, we have the formula for $E_j(S, G_{\gamma_0})$ by replacing $T$ with a minimal Levi subgroup $T_j \subset G^*_{j\gamma_0}$. The denominator of \eqref{eqn:Kottwitz-formula-errata} only depends on the stable conjugacy class of $j$.

Return to our setting that $S$ is of type (ER) and $G = \Sp(W)$. It remains to show that
\[ e(G_{j\gamma_0}) \epsilon(X^*(S)_{\mathbb{C}} - X^*(T_j)_{\mathbb{C}}, \psi) \]
depends only on the stable conjugacy class of $j$.

Using Theorem 6.2.2 we know that $j\gamma_0$ can only have eigenvalues $\pm 1$, with multiplicities $2n_\pm$ respectively. Hence
\[ G_{j\gamma_0} = \mathrm{Sp}(W_+) \times \mathrm{Sp}(W_-) \]
where $W_\pm \subset W$ are symplectic vector subspaces of dimension $2n_\pm$. The stable conjugacy class of $j$ and $\gamma_0$ determine the isomorphism class of $G_{j\gamma_0}$, hence determines $e(G_{j\gamma_0})$ and $X^*(T_j)_{\mathbb{C}}$. The proof is now complete.

\subsection{Correction to Section 8.2 and Theorem 8.3.1}
The obscure construction in Sect.~8.2 should be abandoned. Theorem 8.3.1 should instead be turned into a definition for the stable system when $m \equiv 2 \pmod{4}$. Specifically, we shall define
\[ \theta_j^\dagger(\gamma_0) := \frac{\epsilon_k(\gamma_0)}{\epsilon_j(\gamma_0)}, \]
where $\gamma_0 \in S(F)_{p'}$. Then we put $\theta_j := \theta_j^\circ \theta_j^\dagger$ to obtain a stable system as in Theorem 8.2.6.

Let us recall the notation from Sect. 8.2 and 8.3 briefly.
\begin{compactitem}
	\item Fix an additive character $\psi$ of the non-Archimedean local field $F$. Fix a symplectic $F$-vector space $(W, \lrangle{\cdot|\cdot}$ to define $G = \mathrm{Sp}(W)$ and its $m$-fold BD-cover $\tilde{G}$.
	\item Consider an inducing datum $(\mathcal{E}, S, \theta^\flat)$ (Definition 7.3.1) and pick any $Y = Y_\psi \in \mathfrak{s}(F)$ as in (8.3).
	\item For each conjugacy class of embeddings $j: S \hookrightarrow G$ in $\mathcal{E}$, we plug $Y$ into the moment map correspondence to produce $H = \mathrm{SO}(V,q)$ and the conjugacy class of embeddings $k: S \hookrightarrow H$.
\end{compactitem}

This will guarantee the compatibility with $\Theta$-lifting of epipelagic representations \cite{LMS16} when $m = 2$. All the statements in Sect.~9.3 remain valid.

To achieve this, we must check the conditions \textbf{SS.1}---\textbf{SS.3} in Definition 7.3.3, and show that $\theta_j^\dagger$ is independent of the choice of $Y$.

Suppose $Y$ is chosen. We begin with \textbf{SS.2}. It cannot be reduced to $\dim W = 2$ like in the original arguments for Theorem 8.2.6. However, we can still
\begin{itemize}
	\item parameterize $j$ by a datum $(K, K^\sharp, \vec{c})$ with $K = \prod_{i \in I} K_i$ an étale $F$-algebra, $K^\sharp = \prod_{i \in I} K_i^\sharp$ the subalgebra fixed by an involution $\tau$, and $\vec{c} = (c_i)_{i \in I} \in K^\times$ with $\tau(\vec{c}) = -\vec{c}$;
	\item parameterize $jY \in j\mathfrak{s}(F)$ by $(K, K^\sharp, \vec{c}, \vec{y})$ where $\vec{y} = (y_i)_{i \in I} \in K^\times$ with $\tau(\vec{y}) = -\vec{y}$; the part $\vec{y}$ depends only on the stable conjugacy class of $j$.
\end{itemize}
This gives $S(F) \simeq \prod_{i \in I} K_i^1$ where $K_i^1$ is the norm-one subtorus of $K_i^\times$. A similar parameterization for conjugacy classes of maximal tori and their elements exists for $H$ and its pure inner forms. See for example \cite[\S 3.1]{Li11}.

By Theorem 6.2.2, $\gamma_0 = (\gamma_{0, i})_{i \in I} \in \{\pm 1\}^I$. Accordingly,
\[ I = I_+ \sqcup I_-, \quad I_{\pm} := \{i \in I: \gamma_{0, i} = \pm 1\}. \]

The analysis of stable conjugacy in $\tilde{G}$ in the proof of Theorem 8.2.6 reduces \textbf{SS.2} to the equality
\begin{equation}\label{eqn:SS2-errata}
	\theta^\dagger_{j'}(\gamma_0) = \theta^\dagger_j(\gamma_0) \prod_{\substack{i \in I \\ \gamma_{0,i} = -1}} \sgn_{K_i/K_i^\sharp}\left( c_i / c'_i \right)
\end{equation}
for all $j, j' \in \mathcal{E}$, parameterized by $(K, K^\sharp, \vec{c})$ and $(K, K^\sharp, \vec{c}')$ respectively.

Denote by $k': S \hookrightarrow H' = \mathrm{SO}(V', q')$ the embedding arising from $j'$ by moment map correspondence. It is stably conjugate to $k$ (possibly across pure inner forms). By definition of the moment map correspondence in Sect.~8.1, $k$ (resp.\ $k'$) is parameterized by $(K, K^\sharp, \vec{c}\vec{y})$ (resp.\ $(K, K^\sharp, \vec{c}'\vec{y})$).

Let $V_\pm \subset V$ and $V'_\pm \subset V'$ are the quadratic subspaces on which $k\gamma_0$ and $k'\gamma_0$ act by $\pm 1$, respectively. Note that $\dim V_+ = \dim V'_+$ (resp.\ $\dim V_- = \dim V'_-$) are odd (resp.\ even). We have
\[ H_{k\gamma_0} \simeq \SO(V_+) \times \SO(V_-), \quad H'_{k' \gamma_0} \simeq \SO(V'_+) \times \SO(V'_-). \]

Note that $k$ and $k'$ are related by a pure inner twist from $H$ to $H'$ coming from $H^1(F, S)$. In turn, this restricts to a pure inner twists from $H_{j\gamma_0}$ to $H'_{j'\gamma_0}$, and similarly on their $\pm$-factors. The parameter of $kS_+$ in $kS = kS_+ \times kS_-$ is carved out from that of $kS$ by $I_+ \subset I$. Ditto for $k'S_+$.

We have shown that $\epsilon_j(\gamma_0) = \epsilon_{j'}(\gamma_0)$, so \eqref{eqn:SS2-errata} reduces into
\begin{equation}\label{eqn:SS2-errata-1}
	\epsilon_{k'}(\gamma_0) = \epsilon_k(\gamma_0) \prod_{\substack{i \in I \\ \gamma_{0,i} = -1}} \sgn_{K_i/K_i^\sharp}\left( c_i / c'_i \right).
\end{equation}
We shall prove it by Kottwitz's formula. The Langlands constants are stably invariant and $T_k \simeq T'_{k'}$ since $(H_{k\gamma_0})^* = (H'_{k'\gamma_0})^*$. Thus \eqref{eqn:epsilon-character-errata} and \eqref{eqn:Kottwitz-formula-errata} imply
\[ \frac{\epsilon_{k'}(\gamma_0)}{\epsilon_k(\gamma_0)} = \frac{E_{k'}(S, H') / E_{k'}(S, H'_{k'\gamma_0})}{E_k(S, H) / E_k(S, H_{k\gamma_0})} = \frac{e(H')e(H_{k'\gamma_0})}{e(H)e(H_{k\gamma_0})} \]

\begin{enumerate}
	\item First, we claim that
	\[ \frac{e(H')}{e(H)} = \prod_{i \in I} \sgn_{K_i/K_i^\sharp}(c_i/c'_i). \]
	Indeed, $e(H)$ is computed by the diagonal arrow below.
	\[\begin{tikzcd}
		H^1(F, H^*) \arrow[r] \arrow[rrd] & H^2(F, Z) \arrow[r, "\rho"] & H^2(F, \Gm)_{2\text{-tors}} \\
		\{\pm 1\} \arrow[rr] \arrow[u, "\simeq"] & & \{\pm 1\} \arrow[u, "\simeq"']
	\end{tikzcd}\]
	where $Z = \Ker[H^*_{\text{SC}} \to H^*]$, and $\rho$ is induced by the half-sum of positive roots. The diagonal arrow is non-trivial since there does exist a pure inner form $H^{**}$ of $H^*$ with $e(H^{**}) = -1$. Hence the bottom arrow is the identity map.
	
	The pure inner twist from $H$ to $H'$ in question comes from
	\[\begin{tikzcd}
		\left( \sgn_{K_i/K_i^\sharp}(c_i/c'_i) \right)_{i \in I} \arrow[phantom, r, "\in" description, sloped] & \{\pm 1\}^I \arrow[r, "\text{mult.}"] & \{\pm 1\} & \\
		& H^1(F, S) \arrow[r, "\text{natural}"'] \arrow[u, "\simeq"] & H^1(F, H) \arrow[u, "\simeq"] \arrow[r, "\sim", "\text{natural}"'] & H^1(F, H^*) .
	\end{tikzcd}\]
	The claim follows.
	
	\item Since $\dim V_+ = \dim V'_+$ is odd, the same arguments yield
	\[ \frac{e(\SO(V'_+))}{e(\SO(V_+))} = \prod_{i \in I_+} \sgn_{K_i/K_i^\sharp}(c_i/c'_i). \]
	
	\item To establish \eqref{eqn:SS2-errata-1}, it suffices to show $e(\SO(V'_-)) = e(\SO(V_-))$. Apart from the trivial case $\dim V_- = 2$, there are exactly two pure inner forms of $\SO(V_-)$: the other one arises from a quadratic vector space $V_{-, \diamond}$ with the same dimension and discriminant, but opposite Hasse invariant. Denote by $V^{\mathrm{an}}_-$ and $V^{\mathrm{an}}_{-,\diamond}$ their anisotropic kernels. It is known that
	\[ \{ \dim V^{\mathrm{an}}_-, \dim V^{\mathrm{an}}_{-, \diamond} \} = \begin{cases}
		\{0, 4\}, \\
		\{2\}.
	\end{cases}\]
	In either case, the split ranks of $\SO(V_-)$ and $\SO(V_{-,\diamond})$ have the same parity, hence $e(\SO(V_-)) = e(\SO(V_{-,\diamond}))$. This establishes \textbf{SS.2}.
\end{enumerate}

For \textbf{SS.1}, one has $\theta_j = \theta_j^\circ \theta_j^\dagger$ by construction. Rescaling $\psi$ or $\lrangle{\cdot|\cdot}$ by some $a \in F^\times$ amounts to a similar rescaling of $Y$. This amounts to rescale the $q\lrangle{Y}$ in Remark 8.1.2 by $a$. This also amounts to replacing the parameter $(K, K^\sharp, \vec{c})$ (see Sect.~3.1) of $j$ by $(K, K^\sharp, a\vec{c})$, corresponding to some $j' \in \mathcal{E}$. The required sign-change is already predicated by \textbf{SS.2}.

Now consider \textbf{SS.3}. Following the proof of Theorem 8.2.6, one should show that $\theta^\dagger_j$ depends only on the pro-$p$ part of the datum $\theta^\flat$. Granting this fact, one shows $\theta_j^\dagger = \theta_{\Ad(w)j}^\dagger$ for any $w \in \Omega(G,jS)(F)$ by repeating the original arguments.

To prove this, assume for convenience that the $a(\psi)$ in (8.2) is zero. The only ambiguity is that $Y \in \mathfrak{s}(F)_{-1/e}$ can be replaced by any $Y' \in Y + \mathfrak{s}(F)_0$. Consider the parameter $\vec{y}' = (y'_i)_{i \in I}$ of $jY'$. Then
\[ \frac{y'_i}{y_i} - 1 \; \text{is topologically nilpotent in}\; K_i^\sharp \]
for all $i \in I$. Since we assumed $p \neq 2$ from Sect.~7.3 onward, the group of $1$ + topological nilpotents is $2$-divisible in $(K_i^\sharp)^\times$, hence $\sgn_{K_i/K_i^\sharp} = 1$ on them. The resulting $k: S \hookrightarrow H$, etc.\ are thus unaltered up to conjugacy.

The above also shows that $\theta^\dagger_j(\gamma_0)$ depends only on the pro-$p$ part of $\theta^\flat$ determined by $Y$, not on $Y$ itself.




\printindex

\bibliographystyle{abbrv}
\bibliography{Epipelagic}

\begin{thebibliography}{10}

\bibitem{SGA7-1}
{\em Groupes de monodromie en g\'eom\'etrie alg\'ebrique. {I}}.
\newblock Lecture Notes in Mathematics, Vol. 288. Springer-Verlag, Berlin-New
  York, 1972.
\newblock S{\'e}minaire de G{\'e}om{\'e}trie Alg{\'e}brique du Bois-Marie
  1967--1969 (SGA 7 I), Dirig{\'e} par A. Grothendieck. Avec la collaboration
  de M. Raynaud et D. S. Rim.

\bibitem{Ad98}
J.~Adams.
\newblock Lifting of characters on orthogonal and metaplectic groups.
\newblock {\em Duke Math. J.}, 92(1):129--178, 1998.

\bibitem{AB98}
J.~Adams and D.~Barbasch.
\newblock Genuine representations of the metaplectic group.
\newblock {\em Compositio Math.}, 113(1):23--66, 1998.

\bibitem{Adl98}
J.~D. Adler.
\newblock Refined anisotropic {$K$}-types and supercuspidal representations.
\newblock {\em Pacific J. Math.}, 185(1):1--32, 1998.

\bibitem{AS09}
J.~D. Adler and L.~Spice.
\newblock Supercuspidal characters of reductive {$p$}-adic groups.
\newblock {\em Amer. J. Math.}, 131(4):1137--1210, 2009.

\bibitem{ArIntro}
J.~Arthur.
\newblock An introduction to the trace formula.
\newblock In {\em Harmonic analysis, the trace formula, and {S}himura
  varieties}, volume~4 of {\em Clay Math. Proc.}, pages 1--263. Amer. Math.
  Soc., Providence, RI, 2005.

\bibitem{BLR90}
S.~Bosch, W.~L{\"u}tkebohmert, and M.~Raynaud.
\newblock {\em N\'eron models}, volume~21 of {\em Ergebnisse der Mathematik und
  ihrer Grenzgebiete (3) [Results in Mathematics and Related Areas (3)]}.
\newblock Springer-Verlag, Berlin, 1990.

\bibitem{BD01}
J.-L. Brylinski and P.~Deligne.
\newblock Central extensions of reductive groups by {$\mathbf K_2$}.
\newblock {\em Publ. Math. Inst. Hautes \'Etudes Sci.}, (94):5--85, 2001.

\bibitem{BH06}
C.~J. Bushnell and G.~Henniart.
\newblock {\em The local {L}anglands conjecture for {$\rm GL(2)$}}, volume 335
  of {\em Grundlehren der Mathematischen Wissenschaften [Fundamental Principles
  of Mathematical Sciences]}.
\newblock Springer-Verlag, Berlin, 2006.

\bibitem{CGP15}
B.~Conrad, O.~Gabber, and G.~Prasad.
\newblock {\em Pseudo-reductive groups}, volume~26 of {\em New Mathematical
  Monographs}.
\newblock Cambridge University Press, Cambridge, second edition, 2015.

\bibitem{DR09}
S.~DeBacker and M.~Reeder.
\newblock Depth-zero supercuspidal {$L$}-packets and their stability.
\newblock {\em Ann. of Math. (2)}, 169(3):795--901, 2009.

\bibitem{D76}
P.~Deligne.
\newblock Le support du caract\`ere d'une repr\'esentation supercuspidale.
\newblock {\em C. R. Acad. Sci. Paris S\'er. A-B}, 283(4):Aii, A155--A157,
  1976.

\bibitem{D96}
P.~Deligne.
\newblock Extensions centrales de groupes alg\'ebriques simplement connexes et
  cohomologie galoisienne.
\newblock {\em Inst. Hautes \'Etudes Sci. Publ. Math.}, (84):35--89 (1997),
  1996.

\bibitem{Fl80}
Y.~Z. Flicker.
\newblock Automorphic forms on covering groups of {${\rm GL}(2)$}.
\newblock {\em Invent. Math.}, 57(2):119--182, 1980.

\bibitem{GG}
W.~T. Gan and F.~Gao.
\newblock The {L}anglands-{W}eissman program for {B}rylinski-{D}eligne
  extensions.
\newblock {\em Ast\'{e}risque}, (398):187--275, 2018.
\newblock L-groups and the Langlands program for covering groups.

\bibitem{GS1}
W.~T. Gan and G.~Savin.
\newblock Representations of metaplectic groups {I}: epsilon dichotomy and
  local {L}anglands correspondence.
\newblock {\em Compos. Math.}, 148(6):1655--1694, 2012.

\bibitem{SGA3-1}
P.~Gille and P.~Polo, editors.
\newblock {\em Sch\'emas en groupes ({SGA} 3). {T}ome {I}. {P}ropri\'et\'es
  g\'en\'erales des sch\'emas en groupes}.
\newblock Documents Math\'ematiques (Paris) [Mathematical Documents (Paris)],
  7. Soci\'et\'e Math\'ematique de France, Paris, 2011.
\newblock S{\'e}minaire de G{\'e}om{\'e}trie Alg{\'e}brique du Bois Marie
  1962--64. [Algebraic Geometry Seminar of Bois Marie 1962--64], A seminar
  directed by M. Demazure and A. Grothendieck with the collaboration of M.
  Artin, J.-E. Bertin, P. Gabriel, M. Raynaud and J-P. Serre, Revised and
  annotated edition of the 1970 French original.

\bibitem{SGA3-3}
P.~Gille and P.~Polo, editors.
\newblock {\em Sch\'emas en groupes ({SGA} 3). {T}ome {III}. {S}tructure des
  sch\'emas en groupes r\'eductifs}.
\newblock Documents Math\'ematiques (Paris) [Mathematical Documents (Paris)],
  8. Soci\'et\'e Math\'ematique de France, Paris, 2011.
\newblock S{\'e}minaire de G{\'e}om{\'e}trie Alg{\'e}brique du Bois Marie
  1962--64. [Algebraic Geometry Seminar of Bois Marie 1962--64], A seminar
  directed by M. Demazure and A. Grothendieck with the collaboration of M.
  Artin, J.-E. Bertin, P. Gabriel, M. Raynaud and J-P. Serre, Revised and
  annotated edition of the 1970 French original.

\bibitem{HM08}
J.~Hakim and F.~Murnaghan.
\newblock Distinguished tame supercuspidal representations.
\newblock {\em Int. Math. Res. Pap. IMRP}, (2):Art. ID rpn005, 166, 2008.

\bibitem{HC99}
Harish-Chandra.
\newblock {\em Admissible invariant distributions on reductive {$p$}-adic
  groups}, volume~16 of {\em University Lecture Series}.
\newblock American Mathematical Society, Providence, RI, 1999.
\newblock Preface and notes by Stephen DeBacker and Paul J. Sally, Jr.

\bibitem{HI}
T.~Ikeda and K.~Hiraga.
\newblock Stabilization of the trace formula for the covering group of
  {$\mathrm{SL}_2$} and its application to the theory of {Kohnen} plus spaces.
\newblock Slides for PANT 2011.

\bibitem{JL70}
H.~Jacquet and R.~P. Langlands.
\newblock {\em Automorphic forms on {${\rm GL}(2)$}}.
\newblock Lecture Notes in Mathematics, Vol. 114. Springer-Verlag, Berlin-New
  York, 1970.

\bibitem{Kal13}
T.~Kaletha.
\newblock Simple wild {$L$}-packets.
\newblock {\em J. Inst. Math. Jussieu}, 12(1):43--75, 2013.

\bibitem{Kal15}
T.~Kaletha.
\newblock Epipelagic {$L$}-packets and rectifying characters.
\newblock {\em Invent. Math.}, 202(1):1--89, 2015.

\bibitem{Kal19}
T.~Kaletha.
\newblock Regular supercuspidal representations.
\newblock {\em J. Amer. Math. Soc.}, 32(4):1071--1170, 2019.

\bibitem{KP84}
D.~A. Kazhdan and S.~J. Patterson.
\newblock Metaplectic forms.
\newblock {\em Inst. Hautes \'Etudes Sci. Publ. Math.}, (59):35--142, 1984.

\bibitem{Ko86}
R.~E. Kottwitz.
\newblock Stable trace formula: elliptic singular terms.
\newblock {\em Math. Ann.}, 275(3):365--399, 1986.

\bibitem{Ku69}
T.~Kubota.
\newblock {\em On automorphic functions and the reciprocity law in a number
  field.}
\newblock Lectures in Mathematics, Department of Mathematics, Kyoto University,
  No. 2. Kinokuniya Book-Store Co., Ltd., Tokyo, 1969.

\bibitem{LL79}
J.-P. Labesse and R.~P. Langlands.
\newblock {$L$}-indistinguishability for {${\rm SL}(2)$}.
\newblock {\em Canad. J. Math.}, 31(4):726--785, 1979.

\bibitem{LS1}
R.~P. Langlands and D.~Shelstad.
\newblock On the definition of transfer factors.
\newblock {\em Math. Ann.}, 278(1-4):219--271, 1987.

\bibitem{Le19}
S.~Leslie.
\newblock A generalized theta lifting, {CAP} representations, and {A}rthur
  parameters.
\newblock {\em Trans. Amer. Math. Soc.}, 372(7):5069--5121, 2019.

\bibitem{Li11}
W.-W. Li.
\newblock Transfert d'int\'egrales orbitales pour le groupe m\'etaplectique.
\newblock {\em Compos. Math.}, 147(2):524--590, 2011.

\bibitem{Li12b}
W.-W. Li.
\newblock La formule des traces pour les rev\^etements de groupes r\'eductifs
  connexes. {II}. {A}nalyse harmonique locale.
\newblock {\em Ann. Sci. \'Ec. Norm. Sup\'er. (4)}, 45(5):787--859 (2013),
  2012.

\bibitem{LM18}
H.~Y. Loke and J.-J. Ma.
\newblock Local theta correspondences between supercuspidal representations.
\newblock {\em Ann. Sci. \'{E}c. Norm. Sup\'{e}r. (4)}, 51(4):927--991, 2018.

\bibitem{LMS16}
H.~Y. Loke, J.-J. Ma, and G.~Savin.
\newblock Local theta correspondences between epipelagic supercuspidal
  representations.
\newblock {\em Math. Z.}, 283(1-2):169--196, 2016.

\bibitem{Mat69}
H.~Matsumoto.
\newblock Sur les sous-groupes arithm\'etiques des groupes semi-simples
  d\'eploy\'es.
\newblock {\em Ann. Sci. \'Ecole Norm. Sup. (4)}, 2:1--62, 1969.

\bibitem{Mi87T}
J.~Milne.
\newblock The (failure of the) {Hasse} principle for centres of semisimple
  groups.
\newblock available online.

\bibitem{MVW87}
C.~M{\oe}glin, M.-F. Vign{\'e}ras, and J.-L. Waldspurger.
\newblock {\em Correspondances de {H}owe sur un corps {$p$}-adique}, volume
  1291 of {\em Lecture Notes in Mathematics}.
\newblock Springer-Verlag, Berlin, 1987.

\bibitem{MP94}
A.~Moy and G.~Prasad.
\newblock Unrefined minimal {$K$}-types for {$p$}-adic groups.
\newblock {\em Invent. Math.}, 116(1-3):393--408, 1994.

\bibitem{MP96}
A.~Moy and G.~Prasad.
\newblock Jacquet functors and unrefined minimal {$K$}-types.
\newblock {\em Comment. Math. Helv.}, 71(1):98--121, 1996.

\bibitem{Per81}
P.~Perrin.
\newblock Repr\'esentations de {S}chr\"odinger, indice de {M}aslov et groupe
  metaplectique.
\newblock In {\em Noncommutative harmonic analysis and {L}ie groups
  ({M}arseille, 1980)}, volume 880 of {\em Lecture Notes in Math.}, pages
  370--407, Berlin, 1981. Springer.

\bibitem{Pra01}
G.~Prasad.
\newblock Galois-fixed points in the {B}ruhat-{T}its building of a reductive
  group.
\newblock {\em Bull. Soc. Math. France}, 129(2):169--174, 2001.

\bibitem{Rag04}
M.~S. Raghunathan.
\newblock Tori in quasi-split-groups.
\newblock {\em J. Ramanujan Math. Soc.}, 19(4):281--287, 2004.

\bibitem{RY14}
M.~Reeder and J.-K. Yu.
\newblock Epipelagic representations and invariant theory.
\newblock {\em J. Amer. Math. Soc.}, 27(2):437--477, 2014.

\bibitem{Sch85}
W.~Scharlau.
\newblock {\em Quadratic and {H}ermitian forms}, volume 270 of {\em Grundlehren
  der Mathematischen Wissenschaften [Fundamental Principles of Mathematical
  Sciences]}.
\newblock Springer-Verlag, Berlin, 1985.

\bibitem{Spr74}
T.~A. Springer.
\newblock Regular elements of finite reflection groups.
\newblock {\em Invent. Math.}, 25:159--198, 1974.

\bibitem{SS70}
T.~A. Springer and R.~Steinberg.
\newblock Conjugacy classes.
\newblock In {\em Seminar on {A}lgebraic {G}roups and {R}elated {F}inite
  {G}roups ({T}he {I}nstitute for {A}dvanced {S}tudy, {P}rinceton, {N}.{J}.,
  1968/69)}, Lecture Notes in Mathematics, Vol. 131, pages 167--266. Springer,
  Berlin, 1970.

\bibitem{Sus85}
A.~A. Suslin.
\newblock Algebraic k-theory and the norm-residue homomorphism.
\newblock {\em Journal of Soviet Mathematics}, 30(6):2556--2611, 1985.

\bibitem{Th06}
T.~Thomas.
\newblock The {M}aslov index as a quadratic space.
\newblock {\em Math. Res. Lett.}, 13(5-6):985--999, 2006.
\newblock arXiv:math/0505561v3.

\bibitem{113891}
user28172 (http://mathoverflow.net/users/28172/user28172).
\newblock How fine one must choose an affine cover to get weil restriction?
\newblock MathOverflow.
\newblock URL:http://mathoverflow.net/q/113891 (version: 2012-11-20).

\bibitem{Wa97}
J.-L. Waldspurger.
\newblock Le lemme fondamental implique le transfert.
\newblock {\em Compositio Math.}, 105(2):153--236, 1997.

\bibitem{Wa01}
J.-L. Waldspurger.
\newblock Int\'egrales orbitales nilpotentes et endoscopie pour les groupes
  classiques non ramifi\'es.
\newblock {\em Ast\'erisque}, (269):vi+449, 2001.

\bibitem{Wa08}
J.-L. Waldspurger.
\newblock L'endoscopie tordue n'est pas si tordue.
\newblock {\em Mem. Amer. Math. Soc.}, 194(908):x+261, 2008.

\bibitem{Wei13}
C.~A. Weibel.
\newblock {\em The {$K$}-book}, volume 145 of {\em Graduate Studies in
  Mathematics}.
\newblock American Mathematical Society, Providence, RI, 2013.
\newblock An introduction to algebraic $K$-theory.

\bibitem{Weil64}
A.~Weil.
\newblock Sur certains groupes d'op\'erateurs unitaires.
\newblock {\em Acta Math.}, 111:143--211, 1964.

\bibitem{Wei09}
M.~H. Weissman.
\newblock Metaplectic tori over local fields.
\newblock {\em Pacific J. Math.}, 241(1):169--200, 2009.

\bibitem{Weis16}
M.~H. Weissman.
\newblock Covering groups and their integral models.
\newblock {\em Trans. Amer. Math. Soc.}, 368(5):3695--3725, 2016.

\bibitem{Weis18b}
M.~H. Weissman.
\newblock A comparison of {L}-groups for covers of split reductive groups.
\newblock Number 398, pages 277--286. 2018.
\newblock L-groups and the Langlands program for covering groups.

\bibitem{Weis18}
M.~H. Weissman.
\newblock L-groups and parameters for covering groups.
\newblock Number 398, pages 33--186. 2018.
\newblock L-groups and the Langlands program for covering groups.

\bibitem{Yu01}
J.-K. Yu.
\newblock Construction of tame supercuspidal representations.
\newblock {\em J. Amer. Math. Soc.}, 14(3):579--622 (electronic), 2001.

\end{thebibliography}


\vspace{1em}
\begin{flushleft}
	Wen-Wei Li \\
	E-mail address: \href{mailto:wwli@pku.edu.cn}{\texttt{wwli@pku.edu.cn}} \\
	School of Mathematical Sciences / Beijing International Center for Mathematical Research, Peking University \\
	No.\ 5, Yiheyuan Road, 100871 Beijing, People's Republic of China.
\end{flushleft}

\end{document}